\documentclass[11pt]{article}
\usepackage{fullpage}

\usepackage[T1]{fontenc}
\usepackage[english]{babel}

\usepackage{amsmath,amssymb,amsthm,mathtools,esint}
\usepackage{tikz,graphicx}
\usetikzlibrary{arrows.meta}
\usetikzlibrary{decorations.markings}
\usepackage{graphicx,subcaption}
\usepackage{enumerate}
\usepackage{dsfont}
\usepackage{verbatim}
\usepackage{hyperref}

\theoremstyle{plain}
\newtheorem*{theorem*}{Theorem}
\newtheorem{theorem}{Theorem}[section] 
\newtheorem{lemma}[theorem]{Lemma}
\newtheorem{proposition}[theorem]{Proposition}
\newtheorem{corollary}[theorem]{Corollary}

\newtheorem{assumption}[theorem]{Assumption}
\newtheorem*{assumption*}{Assumption}

\theoremstyle{definition}
\newtheorem{definition}[theorem]{Definition}
\newtheorem{remark}[theorem]{Remark}
\newtheorem{fact}[theorem]{Fact}

\numberwithin{equation}{section}

\mathtoolsset{showonlyrefs}

\newcommand{\ZZ}{\mathbb{Z}}
\newcommand{\RR}{\mathbb{R}}
\newcommand{\CC}{\mathbb{C}}

\newcommand{\Ordo}{\mathcal{O}}

\newcommand{\EE}{\mathbb{E}}
\newcommand{\PP}{\mathbb{P}}

\DeclareMathOperator{\im}{Im}
\DeclareMathOperator{\re}{Re}

\DeclareMathOperator{\diag}{diag}
\renewcommand{\d}{\,\mathrm{d}}
\renewcommand{\i}{\mathrm{i}}
\newcommand{\e}{\mathrm{e}}
\newcommand{\quadgraph}{\tikz[scale=.08]{\draw (1,0)--(0,1)--(-1,0)--(0,-1)--(1,0)}}
  \tikzset{
  compass/.pic = {
    \foreach[count=\i,evaluate={\m=div(\i-1,4);\a=90*\i-45*(\m+1)}] \d in {NE,NW,SW,SE,E,N,W,S}{
      \filldraw[pic actions,rotate=\a,scale=.7+.3*\m] (0,0) -- (45:1)--(0:3) node[scale=3, transform shape,rotate=-90,above]{\d};
      \filldraw[pic actions,fill=white,rotate=\a,scale=.7+.3*\m] (0,0) -- (-45:1)--(0:3)--cycle;
    };
  }
}
  \tikzset{
  compassOp/.pic = {
    \foreach[count=\i,evaluate={\m=div(\i-1,4);\a=90*\i-45*(\m+1)}] \d in {NW,NE,SE,SW,W,N,E,S}{
      \filldraw[pic actions,rotate=\a,scale=.7+.3*\m] (0,0) -- (45:1)--(0:3) node[scale=3, transform shape,rotate=-90,above]{\d};
      \filldraw[pic actions,fill=white,rotate=\a,scale=.7+.3*\m] (0,0) -- (-45:1)--(0:3)--cycle;
    };
  }
}

\DeclareMathOperator{\Tr}{Tr}
\DeclareMathOperator{\Log}{Log}
\DeclareMathOperator{\Int}{Int}

\DeclareMathOperator{\adj}{adj}
\DeclareMathOperator{\one}{\mathds{1}}

\title{Geometry of the doubly periodic Aztec dimer model}

\author{Tomas Berggren\footnote{Department of Mathematics, Massachusetts Institute of Technology, 77 Massachusetts Ave., Cambridge, MA 02139, USA. E-mail: tomasb@mit.edu}
	\and Alexei Borodin\footnote{Department of Mathematics, Massachusetts Institute of Technology, 77 Massachusetts Ave., Cambridge, MA 02139, USA. E-mail: borodin@math.mit.edu}}

\date{}

\begin{document}

\maketitle

\begin{abstract}
The purpose of the present work is to provide a detailed asymptotic analysis of the~$k\times\ell$ doubly periodic Aztec diamond dimer model of growing size for any~$k$ and~$\ell$ and under mild conditions on the edge weights. We explicitly describe the limit shape and the 'arctic' curves that separate different phases, as well as prove the convergence of local fluctuations to the appropriate translation-invariant Gibbs measures away from the arctic curves. We also obtain a homeomorphism between the rough region and the amoeba of an associated Harnack curve, and illustrate, using this homeomorphism, how the geometry of the amoeba offers insight into various aspects of the geometry of the arctic curves. In particular, we determine the number of frozen and smooth regions and the number of cusps on the arctic 
curves.

Our framework essentially relies on three somewhat distinct areas: (1) Wiener-Hopf factorization approach to computing dimer correlations; (2) Algebraic geometric `spectral' parameterization of periodic dimer models; and (3) Finite-gap theory of linearization of (nonlinear) integrable partial differential and difference equations on the Jacobians of the associated algebraic curves. In addition, in order to access desired asymptotic results we develop a novel approach to steepest descent analysis on Riemann surfaces via their amoebas. 
\end{abstract}

\tableofcontents

\section{Introduction}

\subsection{Preface}

 \begin{figure}
 \begin{center}
  \begin{tikzpicture}[scale=.5]
    \draw (0,0) node { \includegraphics[trim={16cm, 2cm, 14cm, 1cm}, clip, scale=.3]{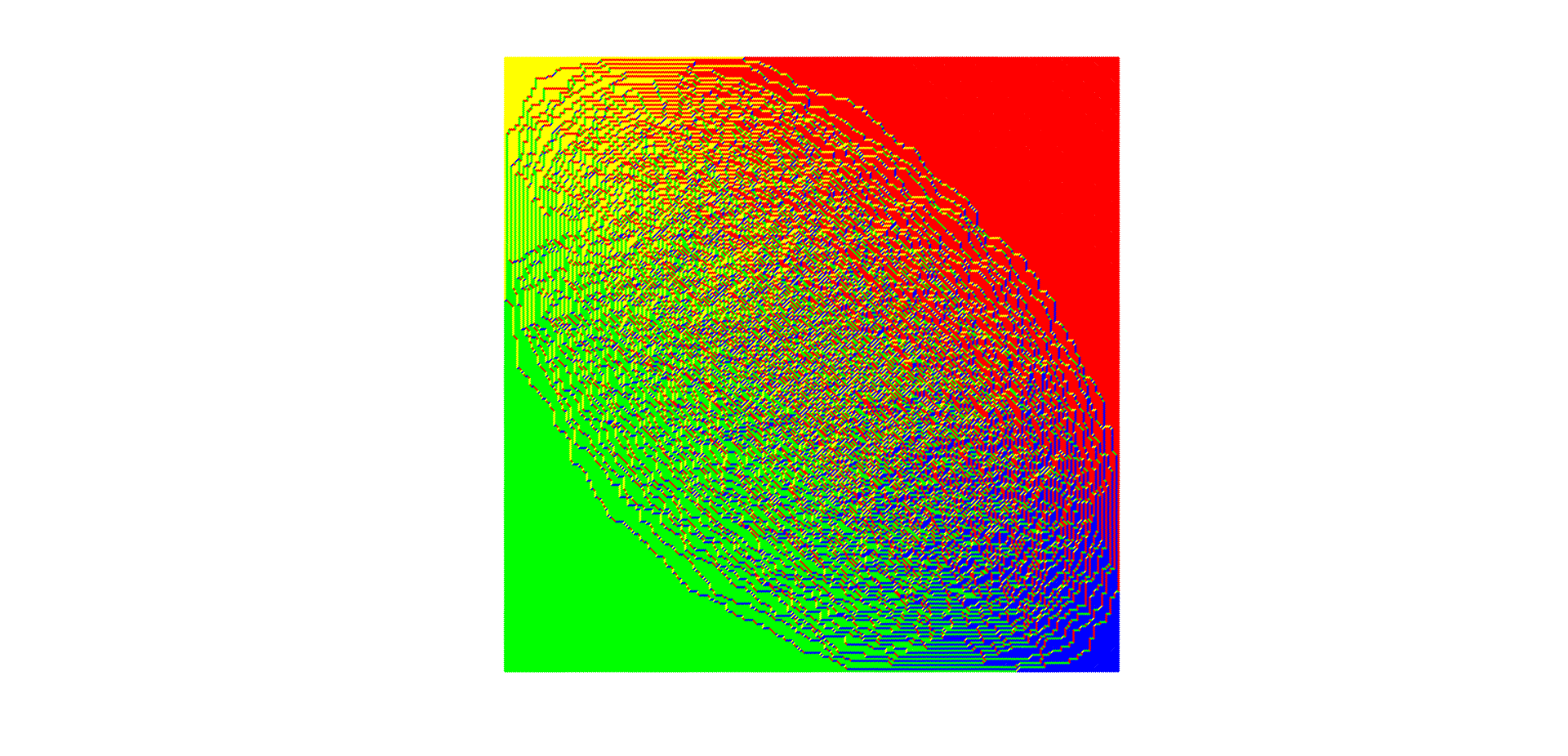}};
    %Compass
    \pic[scale=.2, rotate=-45] at (8,4.5) {compass};
%\draw [ultra thick,->](7.5,3)--(9,4.5);
%\draw [ultra thick,->](8.5,3)--(7,4.5);
%\draw (9,4.5) node[above]{north};
%\draw (7,4.5) node[above] {west};
   \end{tikzpicture}   
     \begin{tikzpicture}[scale=.5]
    \draw (0,0) node { \includegraphics[trim={1cm, 1cm, 1cm, 1cm}, clip, scale=.55]{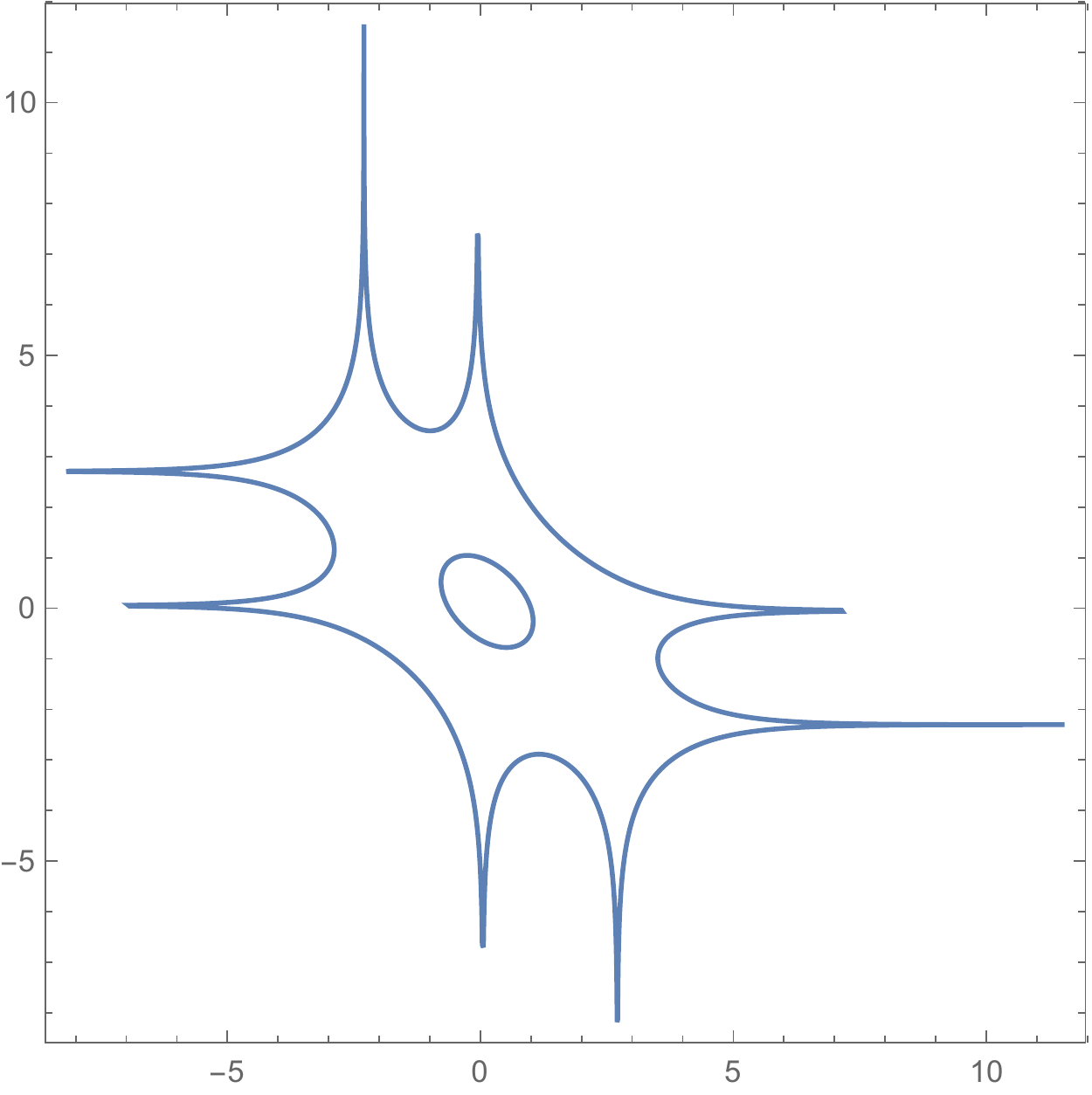}};
       %Compass
       \pic[scale=.2, rotate=-45] at (5,5) {compassOp};
%\draw [ultra thick,->](2.5,4)--(4,5.5);
%\draw [ultra thick,->](2.5,5)--(4,3.5);
%\draw (4,5.5) node[above]{north};
%\draw (4,3.5) node[below] {west};
   \end{tikzpicture}
 \end{center}
  \caption{A randomly generated tiling (via so-called domino-shuffling, see Elkies--Kuperberg--Larsen--Propp~\cite{EKLP92b}, Jockusch-Propp-Shor~\cite{JPS98} and Propp~\cite{Pro03}) and the associated amoeba. The periodicity is~$2\times 2$, the edge weights are given by~$\alpha_{2,2}=1.5$,~$\beta_{1,1}=.95$,~$\beta_{2,2}=.1$,~$\gamma_{2,1}=.95$,~$\gamma_{1,2}=.1$, and all the other weights are~$1$. Note that the orientation of the compass in the right picture is opposite to the orientation of the compass in the left picture. The description of the arctic curves is influenced by this fact, see the discussion after Definition~\ref{def:regions} and Corollaries~\ref{cor:arctic_curve_slopes},~\ref{cor:convex} and~\ref{cor:cusps}. 
  \label{fig:intro:tiling_amoeba1}}
\end{figure}

Planar \emph{dimer models}, also known as \emph{domino tiling models}, have been an active area of research over the past few decades. The origin of this field traces back to the work of Kasteleyn~\cite{Kas61} and Temperley--Fisher~\cite{TF61} in the early 1960s. A \emph{dimer covering}, or a \emph{perfect matching}, of a finite graph is a subset of edges called \emph{dimers} such that each vertex in the graph is covered by exactly one dimer. A dimer model is a probability measure on the set of all possible dimer configurations of the graph. Typically, the probability measure is defined by assigning positive \emph{edge weights} to the graph's edges. The probability of a dimer configuration is then defined to be proportional to the product of the edge weights of the dimers contained in the configuration. 

In the present work we focus on the subgraphs of the square lattice known as \emph{Aztec diamonds}, see Figure~\ref{fig:intro:aztec_diamond}, introduced by Elkies--Kuperberg--Larsen--Propp~\cite{EKLP92a, EKLP92b}, and assign the edge weights in a \emph{doubly periodic} fashion -- the edge weights are periodic in two coordinate directions. We use the notation~$k\times  \ell$ to indicate that the edge weights are~$k$-periodic in the vertical direction and~$\ell$-periodic in the horizontal direction. When both~$k$ and~$\ell$ are at least 2, three types of macroscopic regions might appear in the limit of large domains -- frozen, rough and smooth regions. This is in contrast with the uniform setting on the square lattice, where only frozen and rough regions are expected. The first instance of a model with all three types of regions was discovered in the physics literature by Nienhuis--Hilhorst--Blöte~\cite{NHB84}. In the context of Aztec diamonds, doubly periodic edge weights were first considered by Chhita--Young, Chhita--Johansson~\cite{CJ16, CY14}, and Di Francesco--Soto-Garrido~\cite{FSG14}. For a general introduction to dimer models, we refer to the lecture notes by Kenyon~\cite{Ken04} and Gorin~\cite{Gor21}.

A dimer configuration on a bipartite graph can be characterized by its \emph{height function}, as introduced independently by Thurston~\cite{Thu90} in mathematics, and by Levitov~\cite{Lev90} in the physics literature. When the size of the graph tends to infinity, or equivalently, the lattice mesh size tends to zero, Cohn--Kenyon--Propp~\cite{CKP00} proved that the random height function concentrates around a \emph{limit shape}, and that this limit shape is a solution to a variational problem, see also Kenyon--Okounkov--Sheffield~\cite[Section 6]{KOS06} and references therein, and Kuchumov~\cite{Kuc17}. The study of the variational problem was pioneered by Kenyon--Okounkov~\cite{KO07} and significantly deepened by Astala--Duse--Prause--Zhong~\cite{ADPZ20} providing a complete characterization of the regularity of the limit shape.

Once the existence of a limit shape is established, a natural progression is to investigate fluctuations. The present paper primarily focuses on local fluctuations in the bulk (on the lattice scale). Groundbreaking works of Sheffield~\cite{She05} and Kenyon--Okounkov--Sheffield~\cite{KOS06} proved that in the case of periodic weightings of planar bipartite graphs, for each \emph{slope} within a certain corresponding \emph{Newton polygon}, there exists an \emph{ergodic translation-invariant Gibbs measure} with that slope which is unique for all non-extreme slopes.\footnote{The slope describes the asymptotically linear behavior of the corresponding height function.} The Gibbs measures from three different phases are characterized by the behavior of dimer-dimer correlations: In the \emph{frozen} (or \emph{solid}) phase, the correlation between two dimers does not decay as the distance increases; in the \emph{rough disordered} (or \emph{liquid}) phase, it decays polynomially; and in the \emph{smooth disordered} (or \emph{gas}) phase, it decays exponentially~\cite{KOS06}. All three phases can be observed in Figures~\ref{fig:intro:tiling_amoeba1} and~\ref{fig:intro:tiling_amoeba2}. It was conjectured in~\cite{CKP00} and also in~\cite{KOS06}, that any dimer model must converge, in the local limit, to the ergodic translation-invariant Gibbs measure with slope determined by the slope of the limit shape at the corresponding point. 

For uniform weights, several checks of this conjecture across various, yet rather restrictive classes of domains within square and hexagonal lattices had been achieved over the last two decades, until recently Aggarwal~\cite{Agg19} established the conjecture for the hexagonal lattice for general simply connected domains. On the other hand, research on models with doubly periodic edge weights, particularly those that lead to all phases, is still in an early stage. The same asymptotic techniques seem to be no longer directly applicable, and, in fact, the only model for which the limiting local statistics have been obtained so far is the Aztec diamond with certain (limited) choices of edge weights, see Chhita--Johansson~\cite{CJ16}, Duits--Kuijlaars~\cite{DK21}, Berggren~\cite{Ber21}, and Borodin--Duits~\cite{BD22}. 

The main contribution of the present paper is a detailed asymptotic analysis of the~$k\times \ell$ doubly periodic Aztec diamond dimer model for any~$k$ and~$\ell$, and with mild (likely removable) conditions on the edge weights. We explicitly describe the limit shape and the `arctic' curves that separate different phases, as well as prove the convergence of local fluctuations to the appropriate translation-invariant Gibbs measures away from the arctic curves. In addition, we obtain a homeomorphism between the rough region and an associated amoeba, and illustrate, using this homeomorphism, how the geometry of the amoeba offers insight into various aspects of the geometry of the arctic curves. In particular, we determine the number of frozen and smooth regions, the number of cusps on the arctic curves, and establish local convexity of the rough region\footnote{We are very grateful to Nicolai Reshetikhin for drawing our attention to the latter property.}. To the best of our knowledge, no previous work has been able to achieve any of the above\footnote{With the exception of the local concavity, which was recently established in~\cite{ADPZ20}.} in any setting with generic parameters and with all three phases present. 

Our framework essentially relies on three somewhat distinct areas. Let us provide a brief description of these areas and explain how we utilize them.

The first area is asymptotic analysis of orthogonal polynomials, and it has played a key role in random matrix models since the 1960's, see the classical book of Mehta~\cite{Meh04} for early developments. In the mid-1990's, the introduction of a powerful new method of steepest descent for Riemann--Hilbert problems (close relatives of Wiener--Hopf factorizations) substantially expanded the range of analyzable models, see the book of Deift~\cite{Dei99a} for an exposition. When discrete models of random matrix type, including dimer models, came into light in the late 1990's, both classical and Riemann--Hilbert approaches found immediate applications. In particular, Johansson employed the classical discrete Krawtchouk orthogonal polynomials for asymptotic analysis of the dimers on the Aztec diamond with uniform weights with great success~\cite{Joh01, Joh02, Joh05a}. However, the case of periodic edge weights with both periods greater than 1 turned out to be much more difficult, and it was not until fairly recently that the orthogonal polynomials and Riemann--Hilbert approach was successfully applied by Duits--Kuijlaars~~\cite{DK21}. An essential novelty is that \emph{matrix-valued} orthogonal polynomials are needed. The associated Riemann--Hilbert problem was analyzed in~\cite{DK21} to study the two-periodic Aztec diamond, and later, to study lozenge tilings of hexagons by Charlier--Duits--Kuijlaars--Lenells~\cite{CDKL19} and Charlier~\cite{Cha20a}. It was proved by Berggren--Duits~\cite{BD19} (and re-proved by Chhita--Duits~\cite{CD22} using more of a classical approach, cf. Widom~\cite{Wid74}) that in certain cases, the study of matrix-valued orthogonal polynomials reduces to the study of Wiener--Hopf factorizations of matrix-valued symbols. This was used in~\cite{Ber21} and~\cite{BD22} to analyze the Aztec diamond with certain specific sub-classes of doubly periodic edge weights.

The second area is the algebraic geometric approach to dimer models that was pioneered two decades ago by Kenyon--Okounkov--Sheffield in~\cite{KO06, KOS06}. They established a remarkable correspondence between dimers on weighted periodic graphs and their \emph{spectral data} consisting of a \emph{spectral curve} and a point of its Jacobian. It turned out that the possible spectral curves are exactly the so-called \emph{Harnack} curves, singled out earlier in the work of Mikhalkin~\cite{Mik00}. The invertibility of the \emph{spectral transform}, consisting in associating the spectral data to a dimer model, is a deep fact that required further work and led to new developments, see Goncharov--Kenyon~\cite{GK13}, Fock~\cite{Foc15}, and George--Goncharov--Kenyon~\cite{GGK22}. An elegant exposition of these and further results is available in Boutillier--Cimasoni--de Tilière~\cite{BCT22}. We will use the spectral parameterization in our approach. 

The third area is an algebraic geometric take on integrable partial differential equations. A beautiful geometric theory of the so-called \emph{finite gap solutions} of such equations dates back to the 1970s,  see a book-length exposition by Belokolos--Bobenko--Enol'skii--Its--Matveev~\cite{BBEIM94} with historical notes and references therein. The matrix case, which is most relevant for us, had been originally developed by Dubrovin~\cite{Dub76, Dub77}, Its~\cite{Its76}, and Krichever~\cite{Kri76, Kri77}. In fact, what we mostly need is a later discrete time off-spring of the finite gap theory whose detailed exposition and applications were given by Moser--Veselov~\cite{MV91} (see also Deift--Li--Tomei~\cite{DLT92}). In that work, it was shown how an isospectral flow on certain quadratic matrix polynomials, obtained by repeatedly moving a right linear divisor of the polynomial with a given spectrum to the left side, is linearizable on the Jacobian (or the Prym variety) of the corresponding spectral curve. The main goal of~\cite{MV91} was to describe the dynamics of certain discrete analogs of classical integrable systems in terms of Abelian (or elliptic, for curves of genus 1) functions.

Our starting point is a result of~\cite{BD19} from the first area above that provides an expression for the correlation kernel of a determinantal point process which is equivalent to the Aztec dimer model. The kernel has the form of a double contour integral with the integrand containing the factors of a Wiener--Hopf factorization of a high power of a  matrix-valued polynomial. The key challenge lies in obtaining an expression for this factorization that would enable a steepest descent analysis of the contour integral. In contrast to classical models, where this matrix factorization reduces to a scalar factorization problem, the factors of the polynomial do not have to commute. In~\cite{BD19}, an iterative procedure for obtaining the Wiener--Hopf factorization was discussed. This iterative procedure was later employed in~\cite{Ber21} to study a class of~$2\times \ell$-periodically weighted Aztec diamonds. In that class of models the procedure actually terminated, which led to an expression that was suitable for asymptotic analysis. A~$2\times 2$-periodic Aztec dimer model that falls outside that class was studied in~\cite{BD22}. In the setting of that paper, the iterative procedure does not terminate, necessitating the need for additional ideas to be discussed next.

A novel viewpoint presented in~\cite{BD22} consisted in realizing the above iterative procedure as an integrable system and solving it through a mapping to a linear flow on an elliptic curve. It came straight from the second area (finite-gap solutions) discussed above. Generically, this flow is ergodic. However, at torsion points, when the flow is periodic (which is precisely when the iterative procedure terminates), the authors demonstrated that the resulting expression for the correlation kernel was suitable for asymptotic analysis. In the present paper, that initial approach is brought to fruition on higher levels in three different ways: (1) The algebraic geometric power of the third area from above (spectral transform of the dimer models) allowed us to handle arbitrary periodicity, with higher genus Riemann surfaces\footnote{more specifically, \emph{M-curves}.} replacing the elliptic curve; (2) A novel take on the steepest descent analysis via amoebas removed the complexity in contour deformation on multi-sheeted surfaces; and (3) The eigenvector interpretation of the iterative re-factorization procedure resulted in effective asymptotic control of any arising linear flow on the tori-Jacobians, including the ergodic ones. 

It was noted already in~\cite{Ber21} that the spectral curve naturally associated with the determinantal processes/non-intersecting paths perspective, coincides with the spectral curve from the spectral data of the dimer model, although this observation was not used there. In the present work, we were able to employ it to our advantage. Specifically, it implies, via~\cite{KOS06}, that the spectral curve is always a Harnack curve. In its turn, this allows us to view the spectral curve in terms of the corresponding amoeba, which provides massive benefits for geometric, essentially visual understanding of the critical points of the relevant action functional and the steepest descent and ascent trajectories in the steepest descent analysis. A remarkable fact is that the complexity of the arguments ends up being independent of the complexity (genus) of the underlying Riemann surface.
Furthermore, we demonstrate that the connection between the viewpoints of the non-intersecting paths and the dimers extends beyond the level of spectral curves, roughly speaking, lifting it from determinants of matrices to the matrices themselves. This observation, together with deep facts from area three about zeros of the adjugate of the magnetically altered Kasteleyn matrix, turns out to be critical for our understanding and controlling the flow of the integrable system. 

Let us now discuss our results in more detail.

 \begin{figure}
 \begin{center}
  \begin{tikzpicture}[scale=.5]
    \draw (0,0) node { \includegraphics[trim={16cm, 2cm, 14cm, 1cm}, clip, scale=.3]{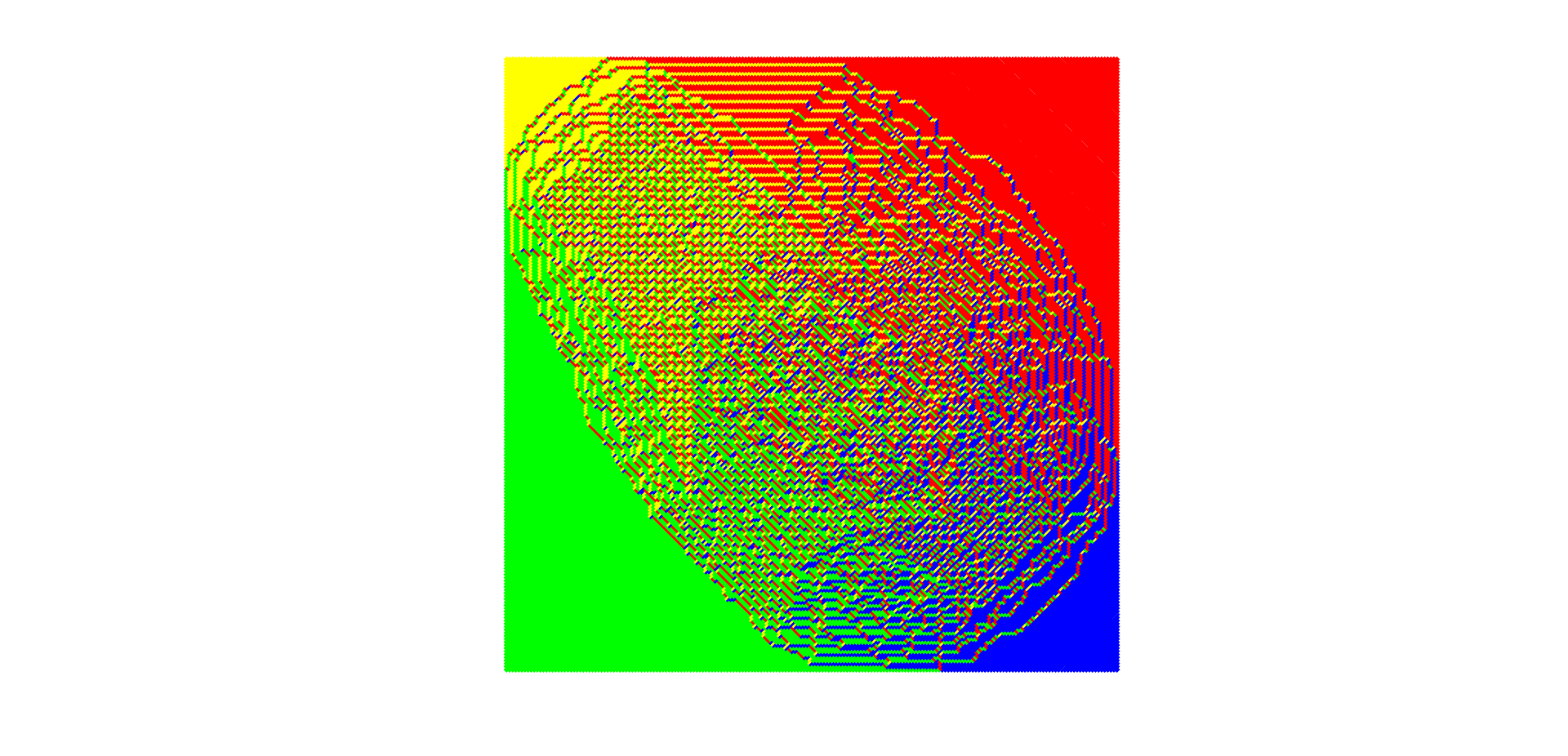}};
    %Compass
    \pic[scale=.2, rotate=-45] at (8,4.5) {compass};
%\draw [ultra thick,->](7.5,3)--(9,4.5);
%\draw [ultra thick,->](8.5,3)--(7,4.5);
%\draw (9,4.5) node[above]{north};
%\draw (7,4.5) node[above] {west};
   \end{tikzpicture}   
 \begin{tikzpicture}[scale=1]
  \tikzset{->-/.style={decoration={
  markings, mark=at position .5 with {\arrow{stealth}}},postaction={decorate}}}
    \draw (0,0) node {\includegraphics[trim={1cm, 1cm, 1cm, 1cm}, clip, angle=180, scale=.45]{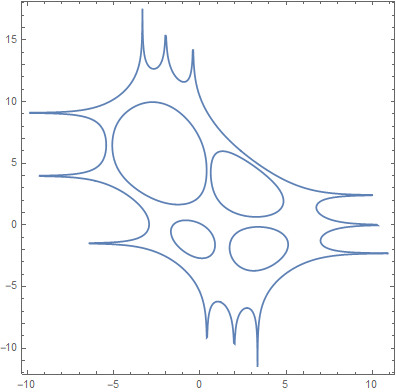}};
   %Compass
    \pic[scale=.2, rotate=-45] at (2.7,2.4) {compassOp};
%\draw [ultra thick,->](1,2)--(1.75,2.75);
%\draw [ultra thick,->](1,2.5)--(1.75,1.75);
%\draw (1.75,2.75) node[above]{north};
%\draw (1.75,1.75) node[below] {west};
	%A_0
%	\draw (.6,1.6) node {$A_{0,1}$};
%	\draw (-1.7,1.9) node {$A_{0,4}$};    
%	\draw (-1.2,-1.) node {$A_{0,7}$};    
%	\draw (2.0,-1.9) node {$A_{0,10}$};    
%	\draw [red,line width=0,-{Stealth[length=2.3mm,blue]}] (.8,.98)--(.7,1.025);
    %A_i
%	\draw (-.8,-.2) node {$A_1$};
%	\draw [red,line width=0,-{Stealth[length=2.3mm,blue]}] (-.3,.02)--(-.4,.16);
%	\draw (-1,.9) node {$A_2$};	
%	\draw [red,line width=0,-{Stealth[length=2.3mm,blue]}] (-.1,1.09)--(-.2,1.09);
%	\draw (.6,-.5) node {$A_3$};
%	\draw [red,line width=0,-{Stealth[length=2.3mm,blue]}] (.8,.055)--(.7,.1);
%	\draw (.1,.75) node {$A_4$};
%	\draw [red,line width=0,-{Stealth[length=2.3mm,blue]}] (-.54,.97)--(-.53,1.);
	%angles
	\draw (.3,2.5) node {$q_{\infty,3}$};
	\draw (-.45,2.6) node {$q_{\infty,2}$};
	\draw (-1.15,3.0) node {$q_{\infty,1}$};
	
	\draw (-3.1,1) node {$p_{0,1}$};
	\draw (-3.1,.5) node {$p_{0,2}$};
	\draw (-3.1,0) node {$p_{0,3}$};

	\draw (0,-2.6) node {$q_{0,3}$};
	\draw (.5,-3.0) node {$q_{0,2}$};
	\draw (1.1,-3.0) node {$q_{0,1}$};

	\draw (2.2,.8) node {$p_{\infty,1}$};
	\draw (3.0,-.4) node {$p_{\infty,2}$};
	\draw (3.2,-1.4) node {$p_{\infty,3}$};
   \end{tikzpicture}
 \end{center}
  \caption{A randomly generated tiling and the associated amoeba. The periodicity is~$3\times 3$ and the edge weights are the same in the two pictures.
 % Here the angles are given by~$\alpha_i^v/\gamma_i^v=(\beta_i^v)^{-1}=d_i^{-1}$,~$\alpha_j^h/\beta_j^h=d_{4-j}^2c_{4-j}$, and~$\gamma_j^h=c_{4-j}$ with~$d_1=1.5$,~$d_2=7.3$,~$d_3=28.1$ and~$c_1=.1$,~$c_2=1$,~$c_3=11.2$.
  \label{fig:intro:tiling_amoeba2}}
\end{figure}
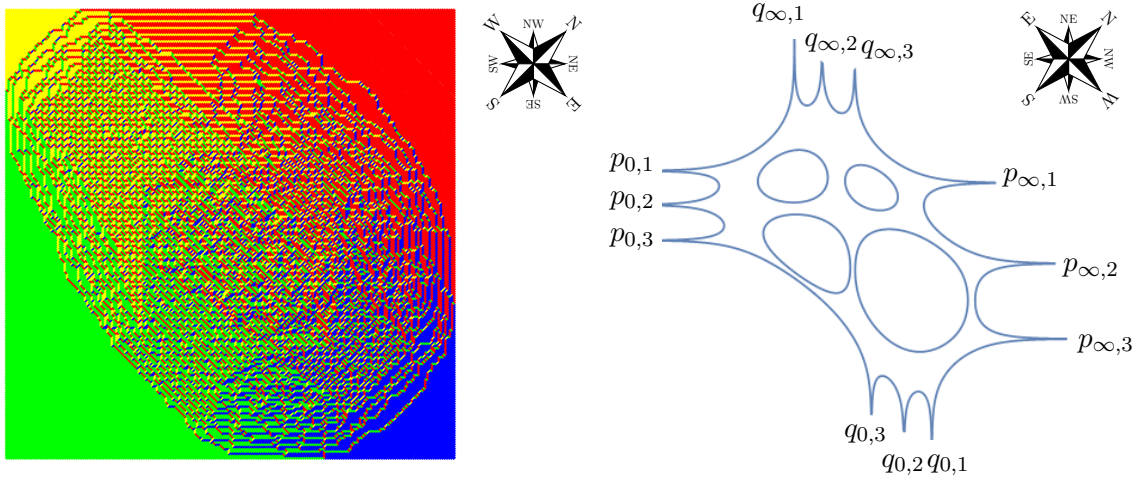

\subsection{The Aztec dimer model and a non-intersecting paths model}\label{sec:intro:model}
The dimer model of interest is defined on the subgraph of the square lattice known as the Aztec diamond, see Figure~\ref{fig:intro:aztec_diamond}. Going forward we consider the rotated square lattice, see Figure~\ref{fig:aztec_diamond} in Section~\ref{sec:model}. The probability measure on the set of all dimer configurations of the Aztec diamond is given by
\begin{equation}\label{eq:intro:dimer_model}
\PP(M)=\frac{1}{Z}\prod_{e\in M}w(e),
\end{equation}
where~$w$ is a positive weight function defined on the edges, and~$Z$ is a normalizing factor known as the partition function. The weight function~$w$ is taken to be~$k$-periodic in the vertical direction and~$\ell$-periodic in the horizontal direction. We say that the model is~$k\times \ell$-periodic. Details are provided in Section~\ref{sec:dimer}.

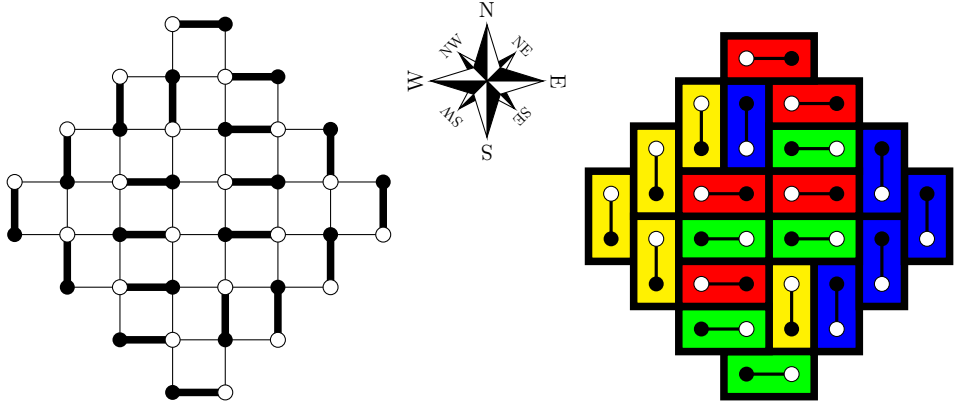
\begin{figure}
\begin{center}
\begin{tikzpicture}[scale=.70]

% Coordinate axes
\foreach \x in {0,1,2,3}
\foreach \y in {0,...,3}
{\draw (-3.5+\x+\y,.5-\x+\y) rectangle (-2.5+\x+\y,1.5-\x+\y);
}
% West dominos
\foreach \x/\y in {-4/0,-3/1,-2/2,-3/-1,0/-2}
{ 
\draw[line width = 1mm] (\x+.5,\y+.5)--(\x+.5,\y+1.5);
\draw (\x+.5,\y+.5) node[circle,fill,inner sep=2pt]{};
\draw (\x+.5,\y+1.5) node[circle,draw=black,fill=white,inner sep=2pt]{};
}

% East dominos
\foreach \x/\y in {1/-2,2/-1,3/0,2/1,-1/2}
{
\draw[line width = 1mm] (\x+.5,\y+.5)--(\x+.5,\y+1.5);
\draw (\x+.5,\y+.5) node[circle,draw=black,fill=white,inner sep=2pt]{};
\draw (\x+.5,\y+1.5) node[circle,fill,inner sep=2pt]{};
}

% South dominos
\foreach \x/\y in {-1/-3,-2/-2,0/2,0/0,-2/0}
{
\draw[line width = 1mm] (\x+.5,\y+.5)--(\x+1.5,\y+0.5);
\draw (\x+.5,\y+.5) node[circle,fill,inner sep=2pt]{};
\draw (\x+1.5,\y+.5) node[circle,draw=black,fill=white,inner sep=2pt]{};
}

% North dominos
\foreach \x/\y in {-2/1,0/1,0/3,-1/4,-2/-1}
{
\draw[line width = 1mm] (\x+.5,\y+.5)--(\x+1.5,\y+0.5);
\draw (\x+.5,\y+.5) node[circle,draw=black,fill=white,inner sep=2pt]{};
\draw (\x+1.5,\y+.5) node[circle,fill,inner sep=2pt]{};
}

\end{tikzpicture}
%\qquad
\begin{tikzpicture}[scale=.6]

% West dominos
\foreach \x/\y in {-4/0,-3/1,-2/2,-3/-1,0/-2}
{ \fill[color=yellow]
(\x,\y) rectangle (\x+1,\y+2); 
\draw [line width = 1mm] (\x,\y) rectangle (\x+1,\y+2);
\draw[very thick,black] (\x+.5,\y+.5)--(\x+.5,\y+1.5);
\draw (\x+.5,\y+.5) node[circle,fill,inner sep=2pt]{};
\draw (\x+.5,\y+1.5) node[circle,draw=black,fill=white,inner sep=2pt]{};
}

% East dominos
\foreach \x/\y in {1/-2,2/-1,3/0,2/1,-1/2}
{  \fill[color=blue]
(\x,\y) rectangle (\x+1,\y+2);
\draw [line width = 1mm] (\x,\y) rectangle (\x+1,\y+2);
\draw[very thick,black] (\x+.5,\y+.5)--(\x+.5,\y+1.5);
\draw (\x+.5,\y+.5) node[circle,draw=black,fill=white,inner sep=2pt]{};
\draw (\x+.5,\y+1.5) node[circle,fill,inner sep=2pt]{};
}

% South dominos
\foreach \x/\y in {-1/-3,-2/-2,0/2,0/0,-2/0}
{  \fill[color=green]
(\x,\y) rectangle (\x+2,\y+1);
\draw [line width = 1mm] (\x,\y) rectangle (\x+2,\y+1);
\draw[very thick,black] (\x+.5,\y+.5)--(\x+1.5,\y+0.5);
\draw (\x+.5,\y+.5) node[circle,fill,inner sep=2pt]{};
\draw (\x+1.5,\y+.5) node[circle,draw=black,fill=white,inner sep=2pt]{};
}

% North dominos
\foreach \x/\y in {-2/1,0/1,0/3,-1/4,-2/-1}
{  \fill[color=red]
(\x,\y) rectangle (\x+2,\y+1);
\draw [line width=1mm] (\x,\y) rectangle (\x+2,\y+1);
\draw[very thick,black] (\x+.5,\y+.5)--(\x+1.5,\y+0.5);
\draw (\x+.5,\y+.5) node[circle,draw=black,fill=white,inner sep=2pt]{};
\draw (\x+1.5,\y+.5) node[circle,fill,inner sep=2pt]{};
}
%Compass
\pic[scale=.25] at (-6.25,4) {compass};
%\draw [ultra thick,->](-4.5,3)--(-4.5,4.5);
%\draw [ultra thick,->](-4,3.5)--(-5.5,3.5);
%\draw (-4.5,4.5) node[above] {north};
%\draw (-5.5,3.7) node[above] {west};

\end{tikzpicture}
\end{center}
\caption{The Aztec diamond of size~$4$. Left: An example of a dimer covering. Right: The corresponding domino tiling of the Aztec diamond. In the rest of the paper and the definition of the model, the Aztec diamond is rotated by~$-\pi/4$ compared with the orientation depicted in this image. \label{fig:intro:aztec_diamond}}
\end{figure}

There is a measure-preserving bijection between the dimer model~\eqref{eq:intro:dimer_model} and a non-intersecting paths model, see Section~\ref{sec:paths}. The non-intersecting paths model is defined by the probability measure~$\PP_\text{path}$ on~$(\ZZ^{kn})^{2k\ell N}$, with weights proportional to 
\begin{equation}\label{eq:intro:measure_on_points}
 \prod_{m=1}^{2k\ell N} \det \left( T_{\phi_m}(u_{m-1}^i,u_{m}^j)\right)_{i,j=1}^{kn}, \quad u^i_m\in \ZZ \,\,\, \text{for all} \,\,\, i,m,
\end{equation} 
where the transition matrices~$T_{\phi_m}$,~$m=1\dots,2k\ell N$, are \emph{block Toeplitz matrices} with \emph{symbols}~$\phi_m$. The transition matrices are~$2\ell$-periodic, that is,~$\phi_m=\phi_{m+2\ell}$, and 
\begin{equation}
\phi_{2i-1}(z)=
\begin{psmallmatrix}
\gamma_{1,i} & 0 & \cdots & 0 & \alpha_{k,i} z^{-1} \\
\alpha_{1,i} & \gamma_{2,i} & \cdots & 0 & 0 \\
\vdots & \vdots & \ddots & \cdots \\
0 & 0 & \cdots & \alpha_{k-1,i} & \gamma_{k,i}
\end{psmallmatrix}, \quad
\phi_{2i}(z)=\frac{1}{1-\beta_i^v z^{-1}}
\begin{psmallmatrix}
1 & \prod_{j=2}^{k}\beta_{j,i}z^{-1} & \cdots & \beta_{k,i}z^{-1}  \\
\beta_{1,i} & 1 &  \cdots & \beta_{k,i}\beta_{1,i}z^{-1}  \\
\vdots & \vdots & \vdots & \ddots & \cdots \\
\prod_{j=1}^{k-1}\beta_{j,i} & \prod_{j=2}^{k-1}\beta_{j,i} & \cdots & 1
\end{psmallmatrix},
\end{equation}
for~$i=1,\dots,\ell$, where~$\alpha_{j,i}$,~$\beta_{j,i}$ and~$\gamma_{j,i}$, for~$j=1,\dots,k$, are the edge weights (see Figure~\ref{fig:weigths} in Section~\ref{sec:model}) and~$\beta_i^v=\prod_{j=1}^k\beta_{j,i}$. The measure~$\PP_\text{path}$ is a determinantal point process and fits into the framework of~\cite{BD19}, which means that the associated \emph{correlation kernel}~$K_\text{path}$ can be expressed as a double contour integral. 

Even though the two point processes~\eqref{eq:intro:dimer_model} and~\eqref{eq:intro:measure_on_points} are equivalent, it is not obvious how to express statistics of dimers in terms of observables from the paths model; even expressing the probability of observing a dimer at a certain location is challenging. Our first result relates the inverse Kasteleyn matrix for the dimer model with the correlation kernel from the paths model.
\begin{theorem}[Theorem~\ref{thm:inverse_kasteleyn}]\label{thm:intro:inverse_kasteleyn}
Let~$K_\text{path}$ be the correlation kernel associated with the point processes~$\PP_\text{path}$, expressed in terms of a matrix with rows and columns indexed by the vertices of the Aztec diamond. Then the inverse Kasteleyn matrix~$K_{G_\text{Az}}^{-1}$ is the transpose of the submatrix of~$K_\text{path}$ constructed from the rows indexed by the white vertices and the columns indexed by the black vertices. 
\end{theorem}
A similar statement was given in~\cite{CJY15} for the biased Aztec diamond, for which~$k=\ell=1$. Analogous statements for the hexagonal lattice (equivalently, lozenge tilings) can be found in~\cite{BF09, BGR10, Pet14}. A somewhat different relation for the Aztec diamond was also given in~\cite{CD22}. 

Given the symbols of the transition matrices~$\phi_m$, we define a matrix-valued function~$\Phi(z)=\prod_{m=1}^{2\ell}\phi_m(z)$, that is, the product of the symbols over one period. The matrix~$\Phi$ and the zero set of its characteristic polynomial, known as the spectral curve, appear naturally in the asymptotic analysis of the correlation kernel~$K_\text{path}$. Let us define~$P$ as 
\begin{equation}\label{eq:intro:characteristic_polynomial}
P(z,w) = \prod_{i=1}^\ell(1-\beta^v_iz^{-1})\det\left(\Phi(z)-w I\right),
\end{equation}
which is, up to a factor, the characteristic polynomial of~$\Phi$. Generically, the spectral curve~$\mathcal R$ is of higher genus, and, if~$k>2$, we leave the realm of hyperelliptic curves. On the other hand, in~\cite{KOS06} the authors introduced a characteristic polynomial defined by the Kasteleyn matrix on the torus defined from the fundamental domain. We prove that in our case the two polynomials coincide.
\begin{proposition}[Proposition~\ref{prop:spectral_curve}]\label{prop:intro:spectral_curve}
Let~$P$ be as in~\eqref{eq:intro:characteristic_polynomial}. Then
\begin{equation}
P(z,w)=\det K_{G_1}(z,w),
\end{equation}
where~$K_{G_1}$ is the magnetically altered Kasteleyn matrix defined by~\eqref{eq:magnetic_kasteleyn_matrix} below.
\end{proposition}
This is the link that allows us to use the theory of spectral data associated with periodically weighted lattices. In particular, it implies that the spectral curve~$\mathcal R$ is a Harnack curve. We also establish another important relation, expressing the Kasteleyn matrix~$K_{G_1}$ in terms of the transition matrices~$\phi_m$, see Lemma~\ref{lem:phi_kasteleyn}. This allows us to use the properties of the adjugate of~$K_{G_1}$ obtained in~\cite{BCT22, KOS06} in order to understand the linear flow discussed in Section~\ref{sec:intro:linear_flow} below. 

Let us recall a few important properties of Harnack curves. Harnack curves are characterized in terms of their amoebas. The amoeba is the image of the curve in~$\RR^2$ under the map~$(z,w)\mapsto (\log|z|,\log|w|)$, and the Harnack curves are the curves for which this map is at most~$2$-to-$1$. In particular, this means that~$\mathcal R$ can be thought of as gluing together two copies of the amoeba along their boundaries. Generically, the boundary of the amoeba is in bijection with the real locus of~$\mathcal R$, and, in our case, it consists of~$g=(k-1)(\ell-1)$ compact ovals that form the inner boundary, and~$2(k+\ell)$ unbounded components that form the outer boundary. The unbounded components are separated by asymptotes,~$2\ell$ vertical and~$2k$ horizontal asymptotes. The parts of the amoeba extending to infinity are called tentacles. See Figures~\ref{fig:intro:tiling_amoeba1} and~\ref{fig:intro:tiling_amoeba2} for examples of amoebas, and Section~\ref{sec:harnack} for detailed definitions. 

The curve~$\mathcal R$ comes with~$2(k+\ell)$ special points which we call angles (following~\cite{GK13}). These are the points for which~$z$ or~$w$ is either~$0$ or~$\infty$, and they are mapped to infinities along different tentacles in the amoeba. The angles are given by
\begin{equation}
q_{0,i}=((-1)^k\alpha^v_i/\gamma^v_i,0), \quad q_{\infty,i}=(\beta^v_i,\infty), \quad
p_{0,j}=(0,(-1)^\ell\alpha^h_j/\beta^h_j), \quad p_{\infty,j}=(\infty,\gamma^h_j),
\end{equation}
for~$i=1,\dots,\ell$ and~$j=1,\dots,k$, where the nonzero/non-infinity entries are alternating products of the edge weights along a vertical or a horizontal \emph{train track}, that is, 
\begin{equation}\label{eq:intro:entries_angles}
\frac{\alpha^v_i}{\gamma^v_i}=\frac{\prod_{j=1}^k\alpha_{j,i}}{\prod_{j=1}^k\gamma_{j,i}}, \,\,\, 
\beta^v_i=\prod_{j=1}^k\beta_{j,i}, \,\,\, 
\frac{\alpha^h_j}{\beta^h_j}=\frac{\prod_{i=1}^\ell \alpha_{j,i}}{\prod_{i=1}^\ell \beta_{j,i}}, \,\,\,
\text{and} \,\,\, \gamma^h_j=\prod_{i=1}^\ell\gamma_{j,i}.
\end{equation}

The equality of Proposition~\ref{prop:intro:spectral_curve} implies that we can complement the spectral data with the data from the action function, in particular, the critical points of the action function. This combination of the two types of data determines the limit shape, the arctic curves and the local fluctuations. Standard arguments that go back to~\cite{Oko03, OR03} suggest that the global picture and local fluctuations should be derived through the action function, however, the novelty here is that the amoeba takes the role of the upper half plane. 

 \begin{figure}[t]
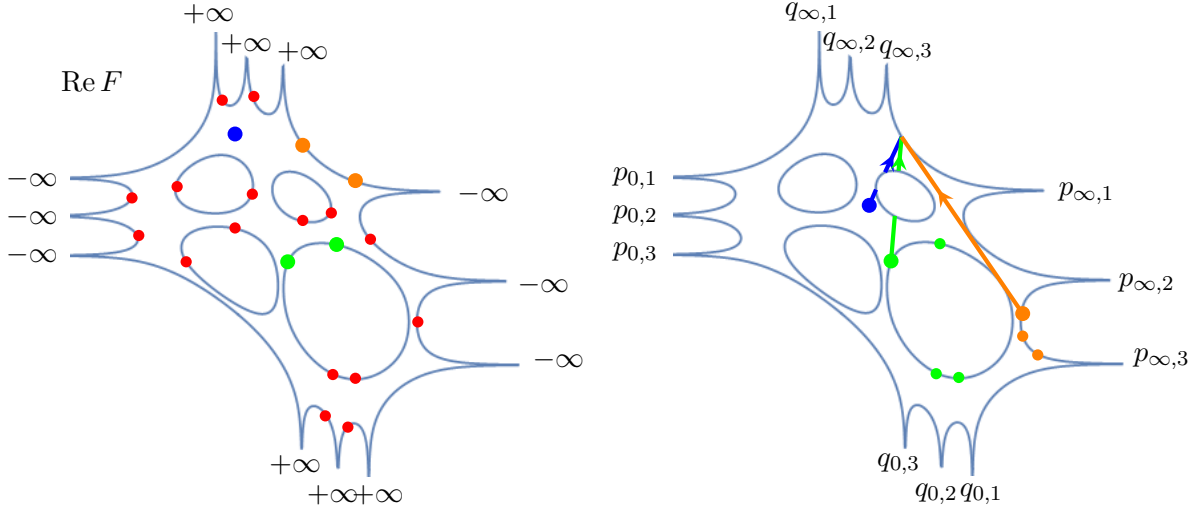

 \begin{center}
 \begin{tikzpicture}[scale=1]
  \tikzset{->-/.style={decoration={
  markings, mark=at position .5 with {\arrow{stealth}}},postaction={decorate}}}
    \draw (0,0) node {\includegraphics[trim={1cm, 1cm, 1cm, 1cm}, clip, angle=180, scale=.5]{amoeba3x3applicable.png}};
	%\re F
	\draw (-2.7,2.3) node {$\re F$};    
    %critical points
    %non-compact oval
	\draw (-.55,2.1) node[circle,fill,red,inner sep=1.5pt]{};
	\draw (-.97,2.05) node[circle,fill,red,inner sep=1.5pt]{};
	
	\draw (-2.17,.75) node[circle,fill,red,inner sep=1.5pt]{};
	\draw (-2.08,.25) node[circle,fill,red,inner sep=1.5pt]{};

	\draw (0.4,-2.15) node[circle,fill,red,inner sep=1.5pt]{};
	\draw (.7,-2.3) node[circle,fill,red,inner sep=1.5pt]{};

	\draw (1.,.2) node[circle,fill,red,inner sep=1.5pt]{};	
	\draw (1.63,-.9) node[circle,fill,red,inner sep=1.5pt]{};

    %compact ovals
	\draw (-1.45,-.1) node[circle,fill,red,inner sep=1.5pt]{};
	\draw (-.8,.35) node[circle,fill,red,inner sep=1.5pt]{};
	\draw (.5,-1.6) node[circle,fill,red,inner sep=1.5pt]{};
	\draw (.8,-1.65) node[circle,fill,red,inner sep=1.5pt]{};	
	\draw (-1.57,.9) node[circle,fill,red,inner sep=1.5pt]{};
	\draw (-.57,.8) node[circle,fill,red,inner sep=1.5pt]{};
	\draw (.1,.45) node[circle,fill,red,inner sep=1.5pt]{};
	\draw (.48,.55) node[circle,fill,red,inner sep=1.5pt]{};
	%special critical points
	%smooth
	\draw (-.1,-.1) node[circle,fill,green,inner sep=2pt]{};
	\draw (.55,.13) node[circle,fill,green,inner sep=2pt]{};
	%rough
	\draw (-.8,1.6) node[circle,fill,blue,inner sep=2pt]{};
	%frozen
	\draw (.1,1.45) node[circle,fill,orange,inner sep=2pt]{};
	\draw (.8,.98) node[circle,fill,orange,inner sep=2pt]{};

	%angles
	\draw (.1,2.7) node {$+\infty$};
	\draw (-.65,2.8) node {$+\infty$};
	\draw (-1.15,3.2) node {$+\infty$};
	
	\draw (-3.5,1) node {$-\infty$};
	\draw (-3.5,.5) node {$-\infty$};
	\draw (-3.5,0) node {$-\infty$};

	\draw (0,-2.8) node {$+\infty$};
	\draw (.5,-3.2) node {$+\infty$};
	\draw (1.1,-3.2) node {$+\infty$};

	\draw (2.5,.8) node {$-\infty$};
	\draw (3.3,-.4) node {$-\infty$};
	\draw (3.5,-1.4) node {$-\infty$};
  \end{tikzpicture}
   \begin{tikzpicture}[scale=1]
  \tikzset{->-/.style={decoration={
  markings, mark=at position .5 with {\arrow{stealth}}},postaction={decorate}}}
  \tikzset{-->-/.style={decoration={
  markings, mark=at position .7 with {\arrow{stealth}}},postaction={decorate}}}
    \draw (0,0) node {\includegraphics[trim={1cm, 1cm, 1cm, 1cm}, clip, angle=180, scale=.5]{amoeba3x3applicable.png}};
   %\gamma_{\xi,\eta}
   %Rough
   \draw[blue,ultra thick](-.39,.64)--(-.3,.85);
   \draw[->-,blue,ultra thick](-.18,1.08)--(0.04,1.55);
   %Smooth
   \draw[green,ultra thick](-.1,-.1)--(-.05,.5);
   \draw[->-,green,ultra thick](.0,1.1)--(0.04,1.55);   
   %Frozen
%   \draw[->-,orange] (0.85,.55)--(0.04,1.55);
   \draw[-->-,orange,ultra thick] (1.65,-.8)--(0.04,1.55);
%   %Curve at angle 
%   \draw[-->-,red,ultra thick] (.9,.6)--(.9,.95);
%   \draw[-->-,red,ultra thick] (1.8,-.5)--(1.65,-.2);
   
   % Points in the amoeba
   % Rough
	\draw (-.39,.64) node[blue,circle,fill,inner sep=2pt]{};
	%smooth
	\draw (-.1,-.1) node[circle,fill,green,inner sep=2pt]{};
	\draw (.55,.13) node[circle,fill,green,inner sep=1.5pt]{};
	\draw (.5,-1.6) node[circle,fill,green,inner sep=1.5pt]{};
	\draw (.8,-1.65) node[circle,fill,green,inner sep=1.5pt]{};	
   % Frozen
   	\draw (1.65,-.8) node[circle,fill,orange,inner sep=2pt]{};
   	\draw (1.65,-1.1) node[circle,fill,orange,inner sep=1.5pt]{};
   	\draw (1.85,-1.35) node[circle,fill,orange,inner sep=1.5pt]{};
   	
	%angles
	\draw (.1,2.7) node {$q_{\infty,3}$};
	\draw (-.65,2.8) node {$q_{\infty,2}$};
	\draw (-1.15,3.2) node {$q_{\infty,1}$};
	
	\draw (-3.5,1) node {$p_{0,1}$};
	\draw (-3.5,.5) node {$p_{0,2}$};
	\draw (-3.5,0) node {$p_{0,3}$};

	\draw (0,-2.8) node {$q_{0,3}$};
	\draw (.5,-3.2) node {$q_{0,2}$};
	\draw (1.1,-3.2) node {$q_{0,1}$};

	\draw (2.5,.8) node {$p_{\infty,1}$};
	\draw (3.3,-.4) node {$p_{\infty,2}$};
	\draw (3.5,-1.4) node {$p_{\infty,3}$};
  \end{tikzpicture}
 \end{center}
  \caption{Left: The behavior of~$\re F$ at the angles, and the locations of the critical points of~$F$ mapped to the amoeba. Right: The curves in the formula for the limit shape. The coloring indicates the region: blue corresponds to the rough region, green corresponds to the smooth region, and orange corresponds to the frozen region. Exactly one of these possibilities is realized at any given point of the limit shape.
 \label{fig:intro:amoeba}}
\end{figure}

Let~$(\xi,\eta)\in (-1,1)^2$ be the global coordinates in the scaled Aztec diamond. For~$q=(z,w)\in \mathcal R$ the action function is defined by
\begin{equation}\label{eq:intro:def_action_function}
F(\tilde q;\xi,\eta) 
= \frac{k}{2}(1-\xi)\log w-\frac{\ell}{2}(1-\eta) \log z - \log \frac{\prod_{i=1}^\ell E(\tilde q_{0,i},\tilde q)^k}{\prod_{j=1}^k E(\tilde p_{0,j},\tilde q)^\ell},
\end{equation}
where~$E$ is the prime form of~$\mathcal R$, and~$\tilde q$ is a lift of~$q$ to the universal cover of~$\mathcal R$, see the end of Section~\ref{sec:abel_theta} and Definition~\ref{def:action_function}. The third term can be thought of as the logarithm of a meromorphic function with zeros of order~$k$ at~$q_{0,i}$,~$j=1,\dots,\ell$, and poles of order~$\ell$ at~$p_{0,j}$,~$j=1,\dots,k$. However, generically such function does not exist (the existence is closely related to torsion points), and instead we need to define it as a function on the universal cover of~$\mathcal R$. The differential~$\d F$ and the real part~$\re F$, however, are well-defined objects on~$\mathcal R$, and these are the objects of interest to us. 

Following~\cite{CJ16, OR03}, we define the limiting regions in terms of the critical points of~$F$, that is, the zeros of~$\d F$. More precisely, we show that all but two critical points are mapped to the boundary of the amoeba, and the location of the final two points determines the region. We classify~$(\xi,\eta)$ as being in the rough region if the final two critical points are mapped to the interior of the amoeba, in a smooth region if they are mapped to the inner boundary of the amoeba, and in a frozen region if they are mapped to the outer boundary of the amoeba. See Figure~\ref{fig:intro:amoeba}, left panel. These critical points provide the connection between the amoeba and the global limit shape.

\subsection{Main results}
The results are stated here in an informal way, with references to the precise statements. The mild assumptions under which these results hold are given and discussed in Section~\ref{sec:ass}.

By analyzing the critical points of the action function~$F$, as a function on the spectral curve~$\mathcal R$, we obtain a homeomorphism between the rough region and the amoeba, which captures the global picture. The homeomorphism is defined as the unique critical point of the action function~$F$ in the interior of the amoeba, see Figure~\ref{fig:intro:amoeba}.
\begin{theorem}[Theorem~\ref{thm:arctic_curves} and Corollary~\ref{cor:arctic_curve_slopes}]\label{thm:intro:phase_diagram}
The critical point map is a homeomorphism from the closure of the rough region to the amoeba. Moreover, the induced map between the boundaries, in a correct coordinate system, preserves the slope of the tangent lines of the respective curves. 
\end{theorem}
\begin{remark}\label{rem:intro:coordinates}
The coordinates~$(\xi,\eta)$ are chosen so that the scaled Aztec diamond is the square $(-1,1)^2$. However, the coordinate system in which the homeomorphism of Theorem~\ref{thm:intro:phase_diagram} preserves the slope, is given by~$u=-\frac{\xi+1}{2\ell}$ and~$v=-\frac{\eta+1}{2k}$. From certain perspectives, these coordinates may be considered as more natural, see Remark~\ref{rem:burgers}.  
\end{remark}
\begin{remark}\label{rem:relation_to_adpz}
In the context of~\cite{ADPZ20}, the authors introduce a map from the rough region to the interior of the amoeba which, when extended to the boundary of the north, east, south and west frozen regions, preserves the slope of the tangent lines\footnote{We are grateful to Erik Duse and István Prause for drawing our attention to this fact.}. It turns out that the two maps are, indeed, the same. See Proposition~\ref{prop:app:maps_relations} and, in particular,~\eqref{eq:maps_relations} where the left hand side is the map from~\cite[Proof of Theorem~1.4]{ADPZ20} and the right hand side is the map of Theorem~\ref{thm:intro:phase_diagram}. 
\end{remark}

The inverse of the homeomorphism from the previous theorem is, in fact, explicit. This explicit formulation is used to prove that the homeomorphism preserves the slope of the tangent line. We have been informed by Nikolai Bobenko that combining our formulas with methods currently under development, which will be described in detail in an upcoming paper of Bobenko--Bobenko--Suris~\cite{BBS}, provides accurate numerical plots of the arctic curves which fit well with the random samples produced from the domino-shuffling algorithm.

The previous theorem yields several corollaries, the most straightforward of which determines the number of smooth and frozen regions.
\begin{corollary}[Corollary~\ref{cor:smooth_frozen_regions}]\label{cor:intro:smooth_frozen_regions}
There is a one-to-one correspondence between the smooth regions and the compact components of the complement of the amoeba. Similarly, there is a one-to-one correspondence between the frozen regions and the unbounded components of the complement of the amoeba.
\end{corollary}
The one-to-one correspondence between the frozen regions and the unbounded components of the complement of the amoeba can be stated more concretely: The angles~$p_{0,j}$,~$j=1,\dots,k$, are mapped to distinct turning points\footnote{These are the points where the arctic curve touches the boundary of the Aztec diamond.} along the south-east side of the Aztec diamond, and the boundary of the amoeba between each pair of angles corresponds to the boundary of a frozen region. Similarly, the angles~$q_{0,i}$,~$p_{\infty,j}$ and~$q_{\infty,i}$ are mapped to the turning points along the south-west, north-west and north-east side of the Aztec diamond, respectively, and the part of the boundary of the amoeba between two angles corresponds to the frozen boundary between the corresponding turning points. See Figures~\ref{fig:intro:tiling_amoeba1} and~\ref{fig:intro:tiling_amoeba2}.

We note that the homeomorphism reverses the orientation of the arctic curve, in comparison to the orientation of the boundary of the amoeba. This property together with the property that the map preserves the slope of the tangent line, has the following consequences.
\begin{corollary}[Corollary~\ref{cor:convex}]\label{cor:intro:convex}
The rough region is locally convex at all smooth points of the arctic curve.
\end{corollary}
\begin{corollary}[Corollary~\ref{cor:cusps}]\label{cor:intro:cusps}
The arctic curve has four cusps in each smooth region, and one cusp in each frozen region, except the north, east, south and west frozen regions.
\end{corollary}
Both these corollaries rely on the fact that we know the geometry of the amoeba, in particular, that the complement of the amoeba is convex. The fact that the rough region is locally convex is a general fact proved in~\cite[Theorem 1.2]{ADPZ20}, however, this provides a rather intuitive reason for it. The second corollary is, to the best of our knowledge, new; its graphical illustration can be found in Figure~\ref{fig:arctic_cruce_amoeba} below.

Our final global result is an expression for the limit shape given in terms of the action function~\eqref{eq:intro:def_action_function}.
\begin{proposition}[Proposition~\ref{prop:limit_shape}]\label{prop:intro:limit_shape}
The limit shape is given in terms of its normalized height function~$\bar h$ by
\begin{equation}\label{eq:intro:limit_shape}
\bar h(\xi,\eta)=\frac{1}{k\ell}\frac{1}{2\pi\i}\int_{\gamma_{\xi,\eta}}\d F+1,
\end{equation}
where the curve~$\gamma_{\xi,\eta}$ is as indicated in Figure~\ref{fig:intro:amoeba}, right panel, and described in Definition~\ref{def:curve_integration}.
\end{proposition}
With the coordinates~$(u,v)$ of Remark~\ref{rem:intro:coordinates}, the slope of the limit shape has a particularly simple form:
\begin{equation}\label{eq:intro:slope_new_coord}
\nabla \bar h(u,v)=\left(\frac{1}{2\pi\i}\int_{\gamma_{\xi,\eta}}\frac{\d w}{w}, -\frac{1}{2\pi\i}\int_{\gamma_{\xi,\eta}}\frac{\d z}{z}\right).
\end{equation}
In the smooth regions the limit shape is flat, and it follows from~\eqref{eq:intro:slope_new_coord} that its slope has a rather nice interpretation in terms of the geometry of the spectral curve. Moreover, the additive shift of the limit shape in a smooth region is closely related to a linear flow on the Jacobian. See Theorem~\ref{thm:intro:height_function_smooth} for the second part, and Corollary~\ref{cor:height_function_smooth} for the combined statement.

Let us continue with local fluctuations, the determination of which can be viewed as the most powerful result of this paper. 
\begin{theorem}[Corollary~\ref{cor:convergence_gibbs_measure}]\label{thm:intro:gibbs_limit}
The local statistics of the dimer model~$\PP$ defined by~\eqref{eq:intro:dimer_model} converge, away from the arctic curves, to those of the ergodic translation-invariant Gibbs measure with slope given by~\eqref{eq:intro:slope_new_coord}.
\end{theorem}
This proves the conjecture of~\cite{CKP00} and~\cite{KOS06} discussed in the preface for dimer coverings of the Aztec diamond with periodic edge weights. 

Within the non-corner frozen regions, the ergodic translation-invariant Gibbs measure is not uniquely determined by its slope, since the slope lies on the boundary of the Newton polygon. It is therefore interesting to determine which measure emerges in the large~$N$ limit of the Aztec diamond. In Section~\ref{sec:frozen_region}, we elaborate on this by illustrating how the frozen configuration is determined from the edge weights, specifically, from the quantities~\eqref{eq:intro:entries_angles}.

The limit of Theorem~\ref{thm:intro:gibbs_limit} is obtained in terms of the correlation kernel~$K_\text{path}$ and is stated in Theorem~\ref{thm:local_limit}. The dimer form of the statement is then recovered using the relations between the models~\eqref{eq:intro:dimer_model} and~\eqref{eq:intro:measure_on_points} discussed in Section~\ref{sec:intro:model}. In the following two sections, we provide a brief overview of the key steps involved in proving Theorem~\ref{thm:local_limit}.

\subsection{The Wiener--Hopf factorization}\label{sec:intro:linear_flow}
An essential part of the asymptotic analysis of the correlation kernel~$K_\text{path}$ is to obtain a suitable expression for the Wiener--Hopf factorization~$\phi=\widetilde \phi_-\widetilde \phi_+$ of~$\phi=\Phi^{kN}$. The process of obtaining the Wiener--Hopf factorization involves two main steps. Firstly, we construct a dynamical system that enables us to construct the Wiener--Hopf factorization, and, secondly, we explore the way the dynamical system affects the eigenvectors of~$\Phi$. We discuss these two steps briefly below.

Let us begin by discussing the first part of the procedure; details are given in Sections~\ref{sec:switching} and~\ref{sec:wh_eigenvectors}. The dynamical system 
\begin{equation}
\Phi_0 \mapsto \Phi_1 \mapsto \Phi_2 \mapsto \Phi_3 \mapsto \dots
\end{equation}
is defined as follows. Given~$\Phi_j$, we construct a factorization~$\Phi_j=(\Phi_j)_-(\Phi_j)_+$, where~$(\Phi_j)_+$ is analytic and non-singular in the interior of the unit circle except, possibly, at~$0$, and~$(\Phi_j)_-$ is analytic and non-singular in the exterior of the unit circle except, possibly, at~$\infty$. The function~$\Phi_{j+1}$ is then defined by~$\Phi_{j+1}=(\Phi_j)_+(\Phi_j)_-$ (the two factors are swapped). As the initial condition we take~$\Phi_0=\Phi$. Informally speaking, the system takes the factor with zeros and poles in the interior of the unit circle and moves it to the right, and the factor with zeros and poles in the exterior of the unit circle is moved to the left. The procedure eventually provides a factorization of~$\phi$,
\begin{equation}
\phi=(\Phi_0)_-(\Phi_1)_-\cdot\dots\cdot (\Phi_{kN-1})_-(\Phi_{kN-1})_+\cdot\dots\cdot (\Phi_1)_+(\Phi_0)_+,
\end{equation}
which, up to the multiplicative factor~$z^{\ell N}C$, for some constant invertible matrix~$C$, is the required Wiener--Hopf factorization.

 \begin{figure}[t]
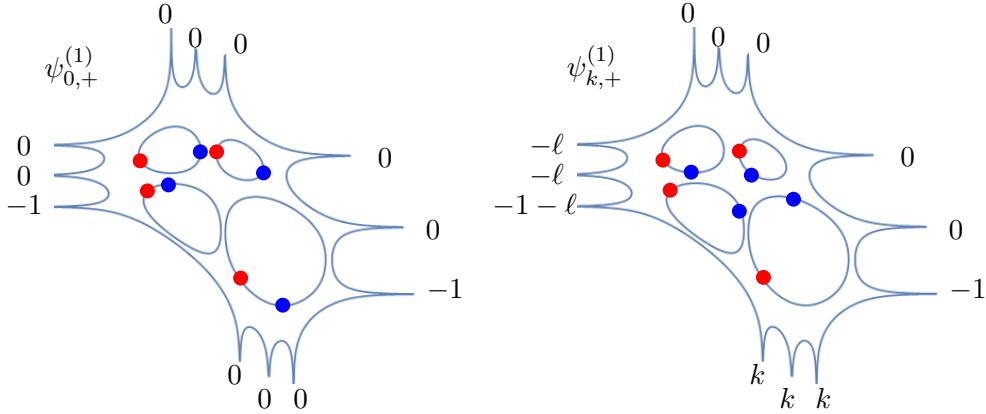

 \begin{center}
 \begin{tikzpicture}[scale=.8]
    \draw (0,0) node {\includegraphics[trim={1cm, 1cm, 1cm, 1cm}, clip, angle=180, scale=.4]{amoeba3x3applicable.png}};
	%nullvector
	\draw (-2.7,2.3) node {$\psi_{0,+}^{(1)}$};    
    %zeros and poles

    %compact ovals
    %Fixed
	\draw (-1.45,.25) node[circle,fill,red,inner sep=2pt]{};
	\draw (-1.1,.35) node[circle,fill,blue,inner sep=2pt]{};
	\draw (.1,-1.2) node[circle,fill,red,inner sep=2pt]{};
	\draw (.8,-1.65) node[circle,fill,blue,inner sep=2pt]{};	
	\draw (-1.57,.75) node[circle,fill,red,inner sep=2pt]{};
	\draw (-.57,.9) node[circle,fill,blue,inner sep=2pt]{};
	\draw (-.3,.9) node[circle,fill,red,inner sep=2pt]{};
	\draw (.48,.55) node[circle,fill,blue,inner sep=2pt]{};

	%angles
	\draw (.1,2.7) node {$0$};
	\draw (-.65,2.8) node {$0$};
	\draw (-1.15,3.2) node {$0$};
	
	\draw (-3.5,1) node {$0$};
	\draw (-3.5,.5) node {$0$};
	\draw (-3.5,0) node {$-1$};

	\draw (0,-2.8) node {$0$};
	\draw (.5,-3.2) node {$0$};
	\draw (1.1,-3.2) node {$0$};

	\draw (2.5,.8) node {$0$};
	\draw (3.3,-.4) node {$0$};
	\draw (3.5,-1.4) node {$-1$};
  \end{tikzpicture}
%	\quad
	 \begin{tikzpicture}[scale=.8]
    \draw (0,0) node {\includegraphics[trim={1cm, 1cm, 1cm, 1cm}, clip, angle=180, scale=.4]{amoeba3x3applicable.png}};
	%nullvector
	\draw (-2.7,2.3) node {$\psi_{k,+}^{(1)}$};    
    %zeros and poles

    %compact ovals
    %Fixed
	\draw (-1.45,.25) node[circle,fill,red,inner sep=2pt]{};
	\draw (-.3,-.1) node[circle,fill,blue,inner sep=2pt]{};
	\draw (.1,-1.2) node[circle,fill,red,inner sep=2pt]{};
	\draw (.6,.1) node[circle,fill,blue,inner sep=2pt]{};	
	\draw (-1.57,.75) node[circle,fill,red,inner sep=2pt]{};
	\draw (-1.1,.55) node[circle,fill,blue,inner sep=2pt]{};
	\draw (-.3,.9) node[circle,fill,red,inner sep=2pt]{};
	\draw (-.1,.5) node[circle,fill,blue,inner sep=2pt]{};

	%angles
	\draw (.1,2.7) node {$0$};
	\draw (-.65,2.8) node {$0$};
	\draw (-1.15,3.2) node {$0$};
	
	\draw (-3.5,1) node {$-\ell$};
	\draw (-3.5,.5) node {$-\ell$};
	\draw (-3.7,0) node {$-1-\ell$};

	\draw (0,-2.8) node {$k$};
	\draw (.5,-3.2) node {$k$};
	\draw (1.1,-3.2) node {$k$};

	\draw (2.5,.8) node {$0$};
	\draw (3.3,-.4) node {$0$};
	\draw (3.5,-1.4) node {$-1$};
  \end{tikzpicture}
 \end{center}
  \caption{The zeros and poles of the first entry of~$\psi_{0,+}$ (left) and~$\psi_{k,+}$ (right) are displayed. The numbers at the angles indicate the order of the zero or pole at each angle. The dots on the inner boundary represent hypothetical zeros. The red zeros remain unchanged, while the blue zeros are modified by the flow. 
 \label{fig:intro:zeros_poles}}
\end{figure}

We then use the fact that~$\det (\Phi_j(z)-wI)$ is independent of~$j$, which implies that the nullvector of~$(\Phi_j(z)-wI)$ is naturally defined as a function on~$\mathcal R$, for all~$j$. Furthermore, for~$(z,w)\in \mathcal R$, if~$\psi_{j,+}(z,w)$ and~$\psi_{j,-}(z,w)$ are right and left nullvectors of~$(\Phi_j(z)-w I)$, then
\begin{equation}\label{eq:intro:eigenvector_j}
\psi_{j+1,+}(z,w)=(\Phi_j)_+(z)\psi_{j,+}(z,w) \quad \text{and} \quad \psi_{j+1,-}(z,w)=\psi_{j,-}(z,w)(\Phi_j)_-(z)
\end{equation}
are right and left nullvectors of~$(\Phi_{j+1}(z)-wI)$, respectively. This observation leads to the following theorem.
\begin{theorem}[Theorem~\ref{thm:wiener-hopf}]\label{thm:intro:wiener-hopf}
For a constant invertible~$k\times k$ matrix~$C$, the following relations hold:
\begin{equation}
z^{\ell N}C\psi_{k N,+}(z,w)=\widetilde \phi_+(z)\psi_{0,+}(z,w), \quad \text{and} \quad \psi_{k N,-}(z,w)z^{-\ell N}C^{-1}=\psi_{0,-}(z,w)\widetilde \phi_-(z),
\end{equation}
where~$\phi=\widetilde \phi_-\widetilde \phi_+$ is the Wiener--Hopf factorization of~$\phi$.
\end{theorem}
This statement allows us to also analyze points at which the flow is ergodic. This is in contrast to~\cite{BD22}, where~$\Phi_j$ were considered for all~$j$, but to obtain an expression that was suitable for asymptotic analysis, it was necessary to consider torsion points. The above formulation is closer to the formulation used in~\cite{MV91}. Taking into account both the right and left nullvectors has not been done before, and, even though it is probably not strictly necessary, it simplifies our calculations significantly.

We continue by explaining the idea of the second part (dynamics of the nullvectors); details are given in Section~\ref{sec:linear_flow}. First, we determine the zeros and poles of~$\psi_{0,+}$ and~$\psi_{0,-}$, which determines them uniquely, up to an irrelevant (for our purposes) multiplicative constant. We use the fact that the adjugate of~$K_{G_1}$ can be expressed in terms of the matrices~$\phi_m$,~$m=1,\dots,2\ell$, as mentioned after Proposition~\ref{prop:intro:spectral_curve}, to express the nullvectors~$\psi_{0,\pm}$ in terms of~$\adj K_{G_1}$. This allows us to use results from~\cite{KOS06} and~\cite{BCT22} to obtain the zeros and poles. Indeed, we find that each entry of~$\psi_{0,\pm}$ has two zeros on each compact oval, and all other zeros and poles are at the angles, see Figure~\ref{fig:intro:zeros_poles}. We then use arguments inspired by those in~\cite{MV91} to obtain the zeros and poles of~$\psi_{m,\pm}$ for all~$m$. It follows from~\eqref{eq:intro:eigenvector_j} that we obtain the zeros and poles of~$\psi_{m+1,\pm}$ by adding the divisor of the zeros and poles of~$(\Phi_j)_\pm$ to the divisor of the zeros and poles of~$\psi_{m,\pm}$. The former divisor introduces zeros and poles to certain angles, and on each compact oval there will be one zero which changes with~$m$, and one which is fixed. The zeros on the compact ovals change to ensure that the nullvectors stay meromorphic on~$\mathcal R$ (cf. Abel's theorem), and their change can be expressed in terms of a linear flow on the Jacobi variety. 

Let us provide a more concrete explanation of the transition from~$\psi_{km,+}$ to~$\psi_{k(m+1),+}$, see Figure~\ref{fig:intro:zeros_poles}. This transition involves adding~$k$ zeros at each angle~$q_{0,i}$,~$i=1,\dots,\ell$, and~$\ell$ poles at each angle~$p_{0,j}$,~$j=1,\dots,k$. These are the zeros and poles appearing in the quotient in the third term of~\eqref{eq:intro:def_action_function}. We denote by~$e_j^{(km)}$ the point on the Jacobian determining the zeros of the~$(j+1)$st entry of~$\psi_{km,+}$ in the compact ovals which changes with~$m$\footnote{Later in the text we will use the more detailed notation~$e_{\mathrm w_{0,j}}^{(km)}$, see Proposition~\ref{prop:zeros_poles}}. The linear flow on the Jacobian is given by
\begin{equation}
e_j^{(km+k)}-e_j^{(km)}=\ell\sum_{i=1}^ku(p_{0,i})-k\sum_{i=1}^\ell u(q_{0,i}) \in \RR^g/\ZZ^g,
\end{equation} 
where~$u$ is the Abel map. In particular 
\begin{equation}
e_j^{(kN)}-e_j^{(0)}=N\left(\ell\sum_{i=1}^ku(p_{0,i})-k\sum_{i=1}^\ell u(q_{0,i})\right) \in \RR^g/\ZZ^g.
\end{equation} 

Surprisingly, at least to us, this linear flow makes a seemingly unrelated appearance in the limit shape in the smooth regions. There the limit shape is flat, and we define~$H$ as the difference between the limit shape and the parallel plane that passes through the origin.
\begin{theorem}[Corollary~\ref{cor:height_function_smooth} and Remark~\ref{rem:linear_flow}]\label{thm:intro:height_function_smooth}
Let~$H$ be as above and let~$(u_n,v_n)$ be points in the~$n$th smooth region,~$n=1,\dots,g$, where~$g=(k-1)(\ell-1)$ is the genus of~$\mathcal R$, and we use the coordinate system of Remark~\ref{rem:intro:coordinates}. Then
\begin{equation}\label{eq:intro:shift_height_function}
e_j^{(kN)}-e_j^{(0)}=-N\big(H(u_1,v_1),\dots,H(u_g,v_g)\big)\!\! \mod \ZZ^g.
\end{equation} 
\end{theorem}
If we consider the linear flow of the left nullvector instead of the right one, we obtain a similar relation but without the minus sign.

\subsection{The asymptotic analysis}

 \begin{figure}[t]
 \begin{center}
 \begin{tikzpicture}[scale=1]
  \tikzset{->-/.style={decoration={
  markings, mark=at position .5 with {\arrow{stealth}}},postaction={decorate}}}
    \draw (0,0) node {\includegraphics[trim={1cm, 1cm, 1cm, 1cm}, clip, angle=180, scale=.5]{amoeba3x3applicable.png}};
    %\tilde \Gamma_l
   \draw[thick](0.04,-2)--(0.04,-1.15);
   \draw[->-,thick](0.04,0.05)--(0.04,.45);
   \draw[thick](0.04,1.1)--(0.04,1.55);
   %\tilde \Gamma_s
   \draw[dashed,thick] (-.1,-1.6)--(-.1,-.9);
   \draw[->-,dashed,thick] (-.1,-.05)--(-.1,.55);
   \draw[dashed,thick] (-.1,1.1)--(-.1,1.8);
	%angles
	\draw (.1,2.7) node {$q_{\infty,3}$};
	\draw (-.65,2.8) node {$q_{\infty,2}$};
	\draw (-1.15,3.2) node {$q_{\infty,1}$};
	
	\draw (-3.5,1) node {$p_{0,1}$};
	\draw (-3.5,.5) node {$p_{0,2}$};
	\draw (-3.5,0) node {$p_{0,3}$};

	\draw (0,-2.8) node {$q_{0,3}$};
	\draw (.5,-3.2) node {$q_{0,2}$};
	\draw (1.1,-3.2) node {$q_{0,1}$};

	\draw (2.5,.8) node {$p_{\infty,1}$};
	\draw (3.3,-.4) node {$p_{\infty,2}$};
	\draw (3.5,-1.4) node {$p_{\infty,3}$};
  \end{tikzpicture}
  \qquad
   \begin{tikzpicture}[scale=1]
	\draw (0,0) node {\includegraphics[trim={0cm, .1cm, .8cm, .5cm}, clip, scale=.65]{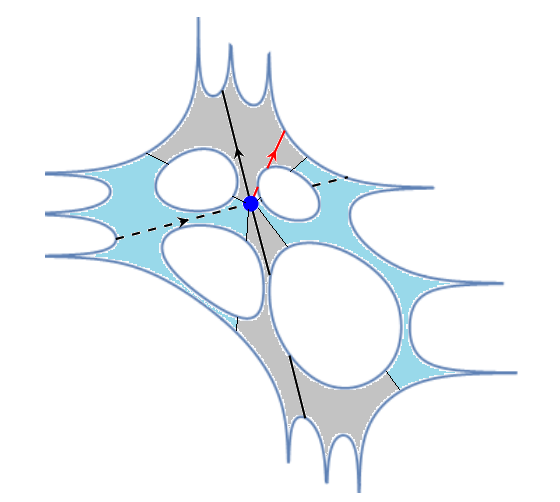}};
	% regions
%	\draw (1.05,.41) node {$C_2$};
%	\draw (-1.7,.7) node {$C_1$};
%	\draw (0,1.8) node {$D_2$};
%	\draw (1.2,-2.1) node {$D_1$};
	
	%angles
	\draw (-1.3,3.0) node {$+\infty$};
	\draw (-2.8,1.4) node {$-\infty$};
	\draw (-.1,-2.7) node {$+\infty$};
	\draw (2.1,1.2) node {$-\infty$};

%	\draw (.1,2.7) node {$+\infty$};
%	\draw (-.65,2.8) node {$+\infty$};
%	\draw (-1.15,3.2) node {$+\infty$};
%	
%	\draw (-3.5,1) node {$-\infty$};
%	\draw (-3.5,.5) node {$-\infty$};
%	\draw (-3.5,0) node {$-\infty$};
%
%	\draw (0,-2.8) node {$+\infty$};
%	\draw (.5,-3.2) node {$+\infty$};
%	\draw (1.1,-3.2) node {$+\infty$};

%	\draw (2.5,.8) node {$-\infty$};
%	\draw (3.3,-.4) node {$-\infty$};
%	\draw (3.5,-1.4) node {$-\infty$};
\end{tikzpicture}
%   \begin{tikzpicture}[scale=1]
%    \tikzset{->-/.style={decoration={
%  markings, mark=at position .5 with {\arrow{stealth}}},postaction={decorate}}}
%    \draw (0,0) node {\includegraphics[trim={1cm, 1cm, 1cm, 1cm}, clip, angle=180, scale=.5]{amoeba3x3applicable.png}};
%    %\gamma_l
%   \draw(0.3,-2.03)--(0.1,-1.25);
%   \draw(-.15,-.25)--(-.39,.64);
%   \draw[->-](-.39,.64)--(-.75,2.05);
%   %\gamma_s
%   \draw[->-,dashed] (-2.1,.2)--(-.25,.68);
%   \draw[dashed] (.39,.86)--(.84,.97);
%   %\gamma_{\xi,\eta}
%   \draw[red](-.39,.64)--(-.3,.85);
%   \draw[->-,red](-.18,1.08)--(0.04,1.55);
%   % Critical point
%   % Rough
%	\draw (-.39,.64) node[blue,circle,fill,inner sep=1.5pt]{};
%  \end{tikzpicture}
 \end{center}
  \caption{Left: The curves~$\Gamma_s$ (dashed) and~$\Gamma_l$ (solid). Right: Curves of steep descent (dashed) and steep ascent (solid). The contribution from the deformation happens along the red curve.  \label{fig:intro:amoeba_curves_rough}}
\end{figure}

In the asymptotic analysis of the correlation kernel~$K_\text{path}$ we express it as a double contour integral on the spectral curve~$\mathcal R$, and to obtain the limit, we perform a steep descent analysis.

In the method of steepest descent, a significant challenge is to understand the curves of steepest descent and ascent and prove that the contours of the integral can be deformed to these curves. This task can become rather technical, especially when dealing with contour integrals on higher genus Riemann surfaces, see, for instance, the previous work on doubly periodic Aztec diamonds~\cite{Ber21, BD22, CJ16, DK21}. However, in our analysis, instead of considering the exact curves of steepest descent and ascent, we focus on regions where curves of steep descent\footnote{The terminology of \emph{steep descent} paths was coined in~\cite{BFPS07} to indicate contours along which the action function was guaranteed to be strictly smaller/bigger than its value at a critical point. We will follow that terminology.} and ascent can be constructed. By combining this approach with the fact that we may think of the spectral curve in terms of the amoeba, we arrive at a relatively non-technical proof. To emphasize this point, we provide here illustrative figures demonstrating the deformation of curves for the rough region. Note, in particular, that the argument is independent of the genus of the Riemann surface.

In Figure~\ref{fig:intro:amoeba_curves_rough}, left panel, we present the image of the curves~$\tilde \Gamma_s=\{(z,w)\in \mathcal R:|z|=r_s\}$ and~$\tilde \Gamma_l=\{(z,w)\in \mathcal R:|z|=r_l\}$ in the amoeba. These correspond to the ones appearing in the double contour integral of~$K_\text{path}$. As curves in the amoeba, they can be freely deformed, except at certain angles. The curve~$\Gamma_s$ cannot be deformed over the angles~$q_{0,i}$ and~$q_{\infty,i}$,~$i=1,\dots,\ell$, and~$\Gamma_l$ cannot be deformed over the angles~$p_{0,j}$ and~$p_{\infty,j}$,~$j=1,\dots,k$. In the right panel of Figure~\ref{fig:intro:amoeba_curves_rough}, we have sketched potential regions where a curve is of steep descent and ascent, respectively. It is apparent from the pictures that the curves in the left panel can be deformed to the respective curves in the right panel. The contribution arising from such deformations occurs only when one curve is deformed across the other. This happens along the red curve in the figure that gives rise to the curve~$\gamma_{\xi,\eta}$ of Proposition~\ref{prop:intro:limit_shape},~\eqref{eq:intro:slope_new_coord} and Theorem~\ref{thm:intro:gibbs_limit}.

\subsection{Connections with previous work}
We conclude the introduction by highlighting a few areas in which our results may be relevant. 

The formula for the inverse Kasteleyn matrix of the two-periodic Aztec diamond, first derived in~\cite{CJ16}, has been utilized to analyze various types of local statistics. It has been applied in the examination of edge fluctuation at the smooth-rough boundary, in the works of Beffara--Chhita--Johansson~\cite{BCJ18, BCJ20}, and dimer-dimer correlations at the smooth-rough boundary, as investigated by Johansson--Mason~\cite{JM22}. Additionally, Johansson--Mason~\cite{JM23} and Bain~\cite{Bai23} have considered limits with weights that vary with the size of the diamond. We believe that the formula obtained in this paper is suitable for other limits beyond those discussed herein. Natural candidates for such limits include the ones mentioned above.

The two viewpoints of~\cite{DK21} and~\cite{BD19} demonstrate that the correlation kernel of the two-periodic Aztec diamond can be expressed in terms of matrix-valued orthogonal polynomials as well as through a Wiener--Hopf factorization. This leads us to the natural question: Can we use the Wiener--Hopf factorization to describe the asymptotic behavior of the orthogonal polynomials? 

Lozenge tilings of a hexagon fall outside the class of models where a Wiener--Hopf factorization yields the dominant behavior of the correlation kernel. Instead, the associated orthogonal polynomials have been employed~\cite{Cha20a, CDKL19}. However, thus far, this approach has only been successful when the underlying spectral curve has genus zero, allowing the use of classical techniques. When the spectral curve has genus greater than zero, and smooth regions are expected, new techniques appear to be necessary. Recently, Bertola--Groot--Kuijlaars~\cite{BGK22} have made significant advancement in this direction. However, despite these advancements, there are still several difficulties to overcome. This raises the question of whether any of the techniques developed in this paper can be applied to the analysis of these orthogonal polynomials.

%\addtocontents{toc}{\protect\setcounter{tocdepth}{1}}
\subsection*{Outline of the paper} 
The paper is organized as follows. In Section~\ref{sec:model}, we introduce the probability measures of interest. We discuss the dimer model for the Aztec diamond in Section~\ref{sec:dimer}, the associated non-intersecting paths model in Section~\ref{sec:paths}, and the ergodic translation-invariant Gibbs measures in Section~\ref{sec:gibbs_measure}. The relation between the correlation kernel and the inverse Kasteleyn matrix is established in Section~\ref{sec:inverse_kasteleyn}. In Section~\ref{sec:characteristic_polynomial}, we discuss the spectral curve, and the proof of Proposition~\ref{prop:intro:spectral_curve} is presented in Section~\ref{sec:spectral_curve}. Sections~\ref{sec:harnack} and~\ref{sec:abel_theta} cover the properties of the spectral curve relevant to our analysis. Our main results are then stated in Section~\ref{sec:main_result}, and the proofs, except for Theorem~\ref{thm:local_limit} and Proposition~\ref{prop:limit_shape}, are provided in the same section. The remaining parts of the paper focus on proving these final two statements. In Section~\ref{sec:wh}, we obtain a suitable expression of the Wiener--Hopf factorization, more precisely, we prove Theorem~\ref{thm:intro:wiener-hopf} in Sections~\ref{sec:switching} and~\ref{sec:wh_eigenvectors}, and describe the integrable flow in Section~\ref{sec:linear_flow}. In the proof of Proposition~\ref{prop:zeros_poles}, we utilize a (previously developed) formalism that is not necessary for the rest of the paper; therefore, we include the proof together with all the relevant definitions and statements in Section~\ref{sec:one_form_thetas}. Finally, the asymptotic analysis is performed in Section~\ref{sec:asymptotic}, with appropriate preparation in Section~\ref{sec:analysis_prep}, and the steep descent analysis is in Section~\ref{sec:steepest_descent}. The final part of the proof of Theorem~\ref{thm:local_limit} is presented in Section~\ref{sec:uniform_limt}, and the proof of Proposition~\ref{prop:limit_shape} is provided in Section~\ref{sec:height_function}.
%\addtocontents{toc}{\protect\setcounter{tocdepth}{2}}

%\addtocontents{toc}{\protect\setcounter{tocdepth}{1}}
\subsection*{Acknowledgements}
We are very grateful to Terrence George for helpful discussions of the spectral transform, to Cédric Boutillier, David Cimasoni, and Béatrice de Tilière for clarifying their work~\cite{BCT22} to us, to Alexander Veselov for giving us confidence in applying the approach of~\cite{MV91} to our problem, to Matthew Nicoletti for important discussions leading to Corollary~\ref{cor:arctic_curve_slopes}, to Nicolai Reshetikhin for drawing our attention to the convexity of the arctic curves (Corollary~\ref{cor:convex}), to Erik Duse and István Prause for explaining possible relations to their work~\cite{ADPZ20} leading to Proposition~\ref{prop:app:maps_relations}, to Semyon Klevtsov for valuable comments on theta functions and prime forms, and to Sunil Chhita and Christophe Charlier for sharing their software implementations of the domino-shuffling sampling algorithm.
TB~was supported by the Knut and Alice Wallenberg Foundation grant KAW~2019.0523. AB~was partially supported by the NSF grant DMS-1853981, and the Simons Investigator program.
%\addtocontents{toc}{\protect\setcounter{tocdepth}{2}}

\section{The model and a first result}\label{sec:model}
In this section we define the probability measures studied in this paper. We begin by introducing the dimer model on the Aztec diamond in Section~\ref{sec:dimer}, and discuss an equivalent non-intersecting paths model in Section~\ref{sec:paths}. The local statistics of the paths model are expressed in terms of a correlation kernel, while the local statistics of the dimer model is described through the inverse Kasteleyn matrix. The relation between these two objects is established in Section~\ref{sec:inverse_kasteleyn}. Finally, in Section~\ref{sec:gibbs_measure} we recall the notion of ergodic translation-invariant Gibbs measures.

\subsection{Transition matrices}\label{sec:transition_matrices}
In this section we define a family of matrices which are used throughout the paper. We also state a few facts about these matrices. For the proofs we refer the reader to~\cite{BD19, LP12, LP13}, but the properties can also be checked by straightforward verification.

Given vectors~$\vec{\alpha},\vec{\beta},\vec{\gamma} \in (\CC^*)^k$, where~$\CC^*=\CC\backslash \{0\}$, with elements~$(\vec \alpha)_i=\alpha_i$,~$(\vec \beta)_i=\beta_i$,~$(\vec \gamma)_i=\gamma_i$, we set
\begin{equation}\label{eq:bernoulli}
\phi^b(z;\vec{\alpha},\vec{\gamma})=
\begin{psmallmatrix}
\gamma_{1} & 0 & \cdots & 0 & \alpha_{k} z^{-1} \\
\alpha_{1} & \gamma_{2} & \cdots & 0 & 0 \\
\vdots & \vdots & \ddots & \vdots & \vdots \\
0 & 0 & \cdots & \gamma_{k-1} & 0 \\
0 & 0 & \cdots & \alpha_{k-1} & \gamma_{k}
\end{psmallmatrix},
\end{equation}and
\begin{equation}\label{eq:geometric}
\phi^g(z;\vec{\beta})=\frac{1}{1-\beta z^{-1}}
\begin{psmallmatrix}
1 & \prod_{j=2}^{k}\beta_{j}z^{-1} & \cdots & \beta_{k}z^{-1}  \\
\beta_{1} & 1 &  \cdots & \beta_{k}\beta_{1}z^{-1}  \\
\vdots & \vdots & \ddots & \cdots \\
\prod_{j=1}^{k-1}\beta_{j} & \prod_{j=2}^{k-1}\beta_{j} & \cdots & 1
\end{psmallmatrix},
\end{equation}
where~$\beta=\prod_{i=1}^k\beta_i$. We also set~$\alpha=\prod_{i=1}^k\alpha_i$,~$\gamma=\prod_{i=1}^k\gamma_i$ and let~$\vec{1}$ denote the vector of length~$k$ consisting of only ones. We will mostly consider vectors in~$\RR_{>0}^k$.
\begin{remark}
In the special case of~$k=1$,
\begin{equation}
\phi^b(z;\vec{\alpha},\vec{\gamma})=\gamma+\alpha z^{-1}, \quad \text{and} \quad \phi^g(z;\vec{\beta})=\frac{1}{1-\beta z^{-1}}.
\end{equation}
\end{remark}
\begin{remark}
In the notation of~\cite{BD19}, 
\begin{equation}
\phi^g(z;\vec{\beta})=\phi^{g \downarrow}(z;\vec{\beta},\vec{1}) \quad \text{and} \quad \phi^b(z;\vec{\alpha},\vec{\gamma})=\phi^{b,\downarrow}(z;\vec{(\alpha/\gamma)},\vec{\gamma}),
\end{equation}
where~$\vec{(\alpha/\gamma)}=(\alpha_1/\gamma_1,\dots,\alpha_k/\gamma_k)$. The notation~$b$ and~$g$ stands for Bernoulli and geometric, and the arrows~$\downarrow$ means that it is the symbol of a transition matrix for a Bernoulli/geometric step down. 
\end{remark}

In the following lemma we collect a number of elementary properties which will be used multiple times throughout the paper.
\begin{lemma}\label{lem:transition_matrices_properties}
With the notation given above, the following properties hold:
\begin{enumerate}[(i)]
\item~$\det \phi^b(z;\vec{\alpha},\vec{\gamma})=\gamma-(-1)^k\alpha z^{-1}$ and~$\det \phi^g(z;\vec{\beta})=(1-\beta z^{-1})^{-1}$, \label{eq:determinant_bern_geo}
\item~$\phi^g(z;\vec{\beta})^{-1}=\phi^b(z;-\vec{\beta},\vec{1})$,\label{eq:inverse_geometric}
\item~$\prod_{j=1}^k\phi^b(z;\vec{\alpha}_{i},\vec{\gamma}_{i})=z^{-1}C'(I+\Ordo(z))$, as~$z\to 0$, for an upper triangular constant matrix~$C'$, \label{eq:bernoulli_zero}
\item~$\phi^g(z;\vec{\beta})=C''(I+\Ordo(z^{-1}))$, as~$z \to \infty$, for a lower triangular constant matrix~$C''$. \label{eq:geometric_infty}
\item~$\phi^b(z;\vec{\alpha},\vec{\gamma})\phi^g(z;\vec{\beta})=C'''(I+\Ordo(z))$, as~$z \to 0$, for an upper triangular constant matrix~$C'''$. \label{eq:prod_zero}
\end{enumerate}
\end{lemma}

\begin{remark}\label{rem:products}
The product in~\eqref{eq:bernoulli_zero} of Lemma~\ref{lem:transition_matrices_properties}, is a product of matrices, which, generally speaking, do not commute. For matrices and scalars~$A_i$ we use the notation 
\begin{equation}
\prod_{i=a+1}^bA_i=
\begin{cases}
A_{a+1}\cdots A_b, & b>a, \\
I, & b=a, \\
A_{b+1}^{-1}\cdots A_a^{-1}, &b<a.
\end{cases}
\end{equation}
Occasionally we will also use the notation
\begin{equation}
\prod_{\substack{i=a+1 \\ \leftarrow}}^bA_i=A_b\cdots A_{a+1}, \,\, \text{if} \,\, b>a, \quad \text{and} \quad \prod_{\substack{i=a+1 \\ \leftarrow}}^bA_i=I, \,\, \text{if} \,\, b=a.
\end{equation}
\end{remark}

\subsection{Dimer model}\label{sec:dimer}
Let us define the Aztec diamond and introduce the dimer model we are interested in. 

We begin by defining a bipartite graph~$G=(\mathcal B,\mathcal W, \mathcal E)$ embedded in the plane. Let 
\begin{equation}
\mathcal B=\{(2i,2j+1): (i,j)\in \ZZ^2\}
\end{equation}
be the black vertices and 
\begin{equation}
\mathcal W=\{(2i+1,2j): (i,j)\in \ZZ^2\}
\end{equation}
be the white vertices. We divide the edges~$\mathcal E$ into four groups,~$\mathrm N$,~$\mathrm E$,~$\mathrm S$ and~$\mathrm W$, called south, west, east and north edges, 
given by 
\begin{equation}
\mathrm N=\{((2i-1,2j),(2i,2j-1)): (i,j)\in \ZZ^2\},
\end{equation}
\begin{equation}
\mathrm E=\{((2i-1,2j),(2i,2j+1)): (i,j)\in \ZZ^2\},
\end{equation}
\begin{equation}
\mathrm S=\{((2i,2j+1),(2i+1,2j)): (i,j)\in \ZZ^2\},
\end{equation}
and
\begin{equation}
\mathrm W=\{((2i,2j-1),(2i+1,2j)): (i,j)\in \ZZ^2\}.
\end{equation}

The \emph{Aztec diamond} of size~$N$ is the subgraph~$G_\text{Az}=(\mathcal B_\text{Az},\mathcal W_\text{Az}, \mathcal E_\text{Az})$ of~$G$ containing all vertices and edges in the square with corners~$(0,0)$,~$(2N,0)$,~$(2N,2N)$ and~$(0,2N)$, including the vertices on the boundary. See Figure~\ref{fig:aztec_diamond}. In the literature it is common that the Aztec diamond is rotated by~$\frac{\pi}{4}$ compared with what is presented here. A \emph{dimer covering} (or a \emph{perfect matching}) of~$G_\text{Az}$ is a subset~$M$ of~$\mathcal E_\text{Az}$, such that each vertex is adjacent to exactly one edge in~$M$. The edges in~$M$ are called \emph{dimers}. A \emph{dimer model} is a probability measure defined on the space of all dimer coverings. 

\begin{figure}
\begin{center}
\begin{tikzpicture}[scale=.70, rotate=-45]

% Coordinate axes
\foreach \x in {0,1,2,3}
\foreach \y in {0,1,2,3}
{\draw (-3.5+\x+\y,.5-\x+\y) rectangle (-2.5+\x+\y,1.5-\x+\y);
}
% West dominos
\foreach \x/\y in {-4/0,-3/1,-2/2,-3/-1,-1/-1,0/-2}
{ 
\draw (\x+.5,\y+.5) node[circle,fill,inner sep=2pt]{};
\draw (\x+.5,\y+1.5) node[circle,draw=black,fill=white,inner sep=2pt]{};
}

% East dominos
\foreach \x/\y in {-2/-1,1/-2,2/-1,3/0,2/1,1/2}
{
\draw (\x+.5,\y+.5) node[circle,draw=black,fill=white,inner sep=2pt]{};
\draw (\x+.5,\y+1.5) node[circle,fill,inner sep=2pt]{};
}

% South dominos
\foreach \x/\y in {-1/-3,-2/-2,-1/3,0/0}
{
\draw (\x+.5,\y+.5) node[circle,fill,inner sep=2pt]{};
\draw (\x+1.5,\y+.5) node[circle,draw=black,fill=white,inner sep=2pt]{};
}

% North dominos
\foreach \x/\y in {-2/1,0/1,-1/2,-1/4}
{
\draw (\x+.5,\y+.5) node[circle,draw=black,fill=white,inner sep=2pt]{};
\draw (\x+1.5,\y+.5) node[circle,fill,inner sep=2pt]{};
}

%Coordinates
\draw (.5,5.5) node {$(2N,2N)$};
\draw (-.1,-3.1) node{$(0,0)$};
\end{tikzpicture}
\begin{tikzpicture}[scale=.70, rotate=-45]
%Compass
\pic[scale=.2, rotate=-45] at (-4.5,4.5) {compass};
%\draw [ultra thick,->](-4,3.5)--(-4,5);
%\draw [ultra thick,->](-3.5,4)--(-5,4);
%\draw (-4,5.5) node{north};
%\draw (-5.5,4) node {west};

%dimers
%north
\draw[line width = 1mm] (-3,3.5)--(-2,3.5);
\draw (-3,3.5) node[circle,draw=black,fill=white,inner sep=2pt]{};
\draw (-2,3.5) node[circle,fill,inner sep=2pt]{};
\draw (-3,3.5) node[left] {north};
%west
\draw[line width = 1mm] (-3,1.5)--(-3,2.5);
\draw (-3,1.5) node[circle,fill,inner sep=2pt]{};
\draw (-3,2.5) node[circle,draw=black,fill=white,inner sep=2pt]{};
\draw (-3,2) node[right] {west};
%south
\draw[line width = 1mm] (-1.5,2.5)--(-.5,2.5);
\draw (-1.5,2.5) node[circle,fill,inner sep=2pt]{};
\draw (-.5,2.5) node[circle,draw=black,fill=white,inner sep=2pt]{};
\draw (-1.5,2.5) node[left] {south};
%east
\draw[line width = 1mm] (-1.5,.5)--(-1.5,1.5);
\draw (-1.5,.5) node[circle,draw=black,fill=white,inner sep=2pt]{};
\draw (-1.5,1.5) node[circle,fill,inner sep=2pt]{};
\draw (-1.5,1) node[right] {east};

\end{tikzpicture}
\quad
\begin{tikzpicture}[scale=.70, rotate=-45]

% Coordinate axes
\foreach \x in {0,1,2,3}
\foreach \y in {0,...,3}
{\draw (-3.5+\x+\y,.5-\x+\y) rectangle (-2.5+\x+\y,1.5-\x+\y);
}
% West dominos
\foreach \x/\y in {-4/0,-3/1,-2/2,-3/-1,0/-2}
{ 
\draw[line width = 1mm] (\x+.5,\y+.5)--(\x+.5,\y+1.5);
\draw (\x+.5,\y+.5) node[circle,fill,inner sep=2pt]{};
\draw (\x+.5,\y+1.5) node[circle,draw=black,fill=white,inner sep=2pt]{};
}

% East dominos
\foreach \x/\y in {1/-2,2/-1,3/0,2/1,-1/2}
{
\draw[line width = 1mm] (\x+.5,\y+.5)--(\x+.5,\y+1.5);
\draw (\x+.5,\y+.5) node[circle,draw=black,fill=white,inner sep=2pt]{};
\draw (\x+.5,\y+1.5) node[circle,fill,inner sep=2pt]{};
}

% South dominos
\foreach \x/\y in {-1/-3,-2/-2,0/2,0/0,-2/0}
{
\draw[line width = 1mm] (\x+.5,\y+.5)--(\x+1.5,\y+0.5);
\draw (\x+.5,\y+.5) node[circle,fill,inner sep=2pt]{};
\draw (\x+1.5,\y+.5) node[circle,draw=black,fill=white,inner sep=2pt]{};
}

% North dominos
\foreach \x/\y in {-2/1,0/1,0/3,-1/4,-2/-1}
{
\draw[line width = 1mm] (\x+.5,\y+.5)--(\x+1.5,\y+0.5);
\draw (\x+.5,\y+.5) node[circle,draw=black,fill=white,inner sep=2pt]{};
\draw (\x+1.5,\y+.5) node[circle,fill,inner sep=2pt]{};
}

\end{tikzpicture}
\end{center}
\caption{The Aztec diamond of size~$4$. Left: The Aztec diamond graph. Right: An example of a dimer covering. The thick edges are the dimers.  \label{fig:aztec_diamond}}
\end{figure}
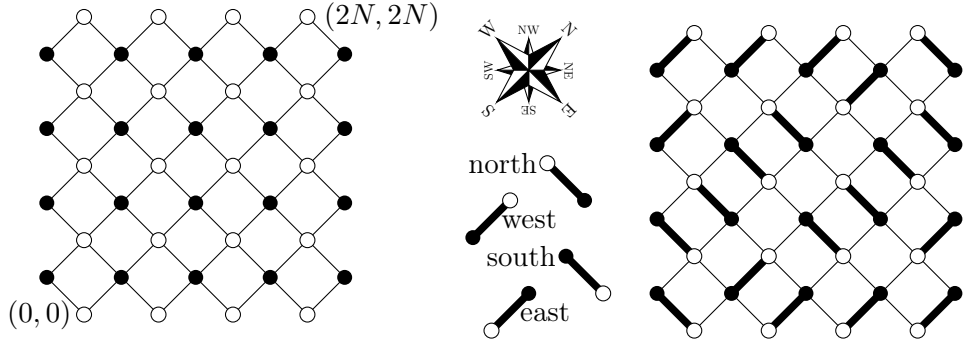

To define the probability measure we introduce weights on the edges -- a function from the edges to the positive real numbers,~$w:\mathcal E\to \RR_{>0}$. The probability measure is defined by
\begin{equation}\label{eq:measure_dimer}
\PP(M)=\frac{1}{Z}\prod_{e\in M}w(e),
\end{equation}
where~$Z=\sum_{M'}\prod_{e\in M'}w(e)$, with the sum running over all possible perfect matchings, is the \emph{partition function}. Let~$\sigma_1$ and~$\sigma_2$ be functions on the white and black vertices, respectively. With the notation~$e=\mathrm w\mathrm b\in \mathcal E$, the weight function~$\tilde w(\mathrm w\mathrm b)=\sigma_1(\mathrm w)w(\mathrm w\mathrm b)\sigma_2(\mathrm b)$ defines the same probability measure as~$w$, and~$\tilde w$ and~$w$ are said to be \emph{gauge equivalent}.

Given a vertex~$(2i-1,2j)\in \mathcal W$, let~$e_{\mathrm S}\in \mathrm S$,~$e_{\mathrm E}\in \mathrm E$,~$e_{\mathrm W}\in \mathrm W$ and~$e_{\mathrm N}\in \mathrm N$, be adjacent to the given vertex. We label the edge weights by
\begin{equation}
\alpha_{j,i}=w(e_{\mathrm S}), \quad \beta_{j,i}=w(e_{\mathrm E}),\quad \gamma_{j,i}=w(e_{\mathrm W}),\quad \text{and}\quad \delta_{j,i}=w(e_{\mathrm N}).
\end{equation} 
In the present paper we consider models for which the edge weights are \emph{doubly periodic}, more precisely, periodic in the vertical direction with period~$k$, and periodic in the horizontal direction with period~$\ell$. That is,~$\alpha_{j,i}=\alpha_{j+k,i+\ell}$ for all~$i,j$, and similarly~$\beta_{j,i}=\beta_{j+k,i+\ell}$ and so on, see Figure~\ref{fig:weigths}. Using a gauge transformation we may, and will, assume, without loss of generality, that~$\delta_{j,i}=1$ for all~$i$,~$j$. For simplicity of notations, we will, throughout the paper, consider the Aztec diamond of size~$k\ell N$ with~$N\in \ZZ_{>0}$. 

Given the edge weights, we define the following quantities:
\begin{equation}\label{eq:edge_weights_product}
\alpha^v_i=\prod_{j=1}^k\alpha_{j,i}, \,\,\, \alpha^h_j=\prod_{i=1}^\ell \alpha_{j,i}, \,\,\, \beta^v_i=\prod_{j=1}^k\beta_{j,i}, \,\,\, \beta^h_j=\prod_{i=1}^\ell \beta_{j,i}, \,\,\,
\gamma^v_i=\prod_{j=1}^k\gamma_{j,i}, \,\,\, \text{and} \,\,\, \gamma^h_j=\prod_{i=1}^\ell\gamma_{j,i}.
\end{equation}
These quantities, or more precisely~$\alpha^h_j/\beta^h_j$,~$\gamma^h_j$,~$\alpha^v_i/\gamma^v_i$,~$\beta^v_i$, are important objects in our analysis, see~\eqref{eq:angles_1} and~\eqref{eq:angles_2}. Moreover, they are closely related to the so-called angles introduced in~\cite{GK13}, see also Section~\ref{sec:one_form_thetas}.

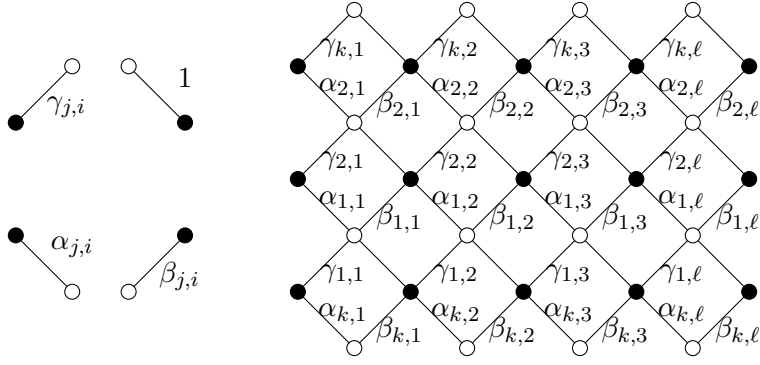
\begin{figure}
\begin{center}
\begin{tikzpicture}[scale=1.5]

% def of edge weights
\draw (-2.5,2-.5)--(-2,2.5-.5);
\draw (-2.5,2-.5) node[circle,draw=black,fill=black,inner sep=2pt]{};
\draw (-2,2.5-.5) node[circle,draw=black,fill=white,inner sep=2pt]{};
\draw (-2.05,2.14-.5) node {$\gamma_{j,i}$};

\draw (-1.5,2.5-.5)--(-1,2-.5);
\draw (-1.5,2.5-.5) node[circle,draw=black,fill=white,inner sep=2pt]{};
\draw (-1,2-.5) node[circle,draw=black,fill=black,inner sep=2pt]{};
\draw (-1,2.4-.5) node {$1$};

\draw (-2.5,1-.5)--(-2,.5-.5);
\draw (-2,.5-.5) node[circle,draw=black,fill=white,inner sep=2pt]{};
\draw (-2.5,1-.5) node[circle,draw=black,fill=black,inner sep=2pt]{};
\draw (-2,.9-.5) node {$\alpha_{j,i}$};

\draw (-1.5,.5-.5)--(-1,1-.5);
\draw (-1,1-.5) node[circle,draw=black,fill=black,inner sep=2pt]{};
\draw (-1.5,.5-.5) node[circle,draw=black,fill=white,inner sep=2pt]{};
\draw (-1.05,.64-.5) node {$\beta_{j,i}$};

% black points
\foreach \x in {0,...,4}
{\foreach \y in {0,...,2}
{\draw (\x,\y) node[circle,draw=black,fill=black,inner sep=2pt]{};
}
}

% white points and lines
\foreach \x in {1,...,4}
{\foreach \y in {1,2}
{\draw (\x-1,\y-1)--(\x,\y);
\draw (\x-1,\y)--(\x,\y-1);
\draw (\x-.5,\y-.5) node[circle,draw=black,fill=white,inner sep=2pt]{};
}
\foreach \y in {0}
{\draw (\x-.5,\y-.5)--(\x,\y);
\draw (\x-1,\y)--(\x-.5,\y-.5);
\draw (\x-.5,\y-.5) node[circle,draw=black,fill=white,inner sep=2pt]{};
}
\foreach \y in {3}
{\draw (\x-1,\y-1)--(\x-.5,\y-.5);
\draw (\x-.5,\y-.5)--(\x,\y-1);
\draw (\x-.5,\y-.5) node[circle,draw=black,fill=white,inner sep=2pt]{};
}
}

\foreach \x in {1,...,3}
{
 \draw (\x+.4-1,0-.2) node {$\alpha_{k,\x}$};
 \draw (\x+.9-1,0-.35) node {$\beta_{k,\x}$};
 \draw (\x+.4-1,3-.85) node {$\gamma_{k,\x}$};
 \draw (\x+.4-1,2-.2) node {$\alpha_{2,\x}$};
 \draw (\x+.9-1,2-.35) node {$\beta_{2,\x}$};
 \draw (\x+.4-1,2-.85) node {$\gamma_{2,\x}$};
 \draw (\x+.4-1,1-.2) node {$\alpha_{1,\x}$};
 \draw (\x+.9-1,1-.35) node {$\beta_{1,\x}$};
 \draw (\x+.4-1,1-.85) node {$\gamma_{1,\x}$};
}

{
 \draw (3.4,-.2) node {$\alpha_{k,\ell}$};
 \draw (3.9,-.35) node {$\beta_{k,\ell}$};
 \draw (3.4,3-.85) node {$\gamma_{k,\ell}$};
 \draw (3.4,2-.2) node {$\alpha_{2,\ell}$};
 \draw (3.9,2-.35) node {$\beta_{2,\ell}$};
 \draw (3.4,2-.85) node {$\gamma_{2,\ell}$};
 \draw (3.4,1-.2) node {$\alpha_{1,\ell}$};
 \draw (3.9,1-.35) node {$\beta_{1,\ell}$};
 \draw (3.4,1-.85) node {$\gamma_{1,\ell}$};
}
\end{tikzpicture}
\end{center}
\caption{The edge weights are~$\alpha_{j,i},\beta_{j,i}, \gamma_{j,i}>0$ for~$i=1,\dots,\ell$,~$j=1,\dots,k$. Here~$k=3$ and~$\ell=5$. \label{fig:weigths}}
\end{figure}

The \emph{Kasteleyn matrix} is a matrix with the rows indexed by the white vertices and the columns indexed by the black vertices. We coordinatize the white vertices by~$(\ell x+i,ky+j)$, where~$\ell x+i=0,\dots,k\ell N-1$, with~$i=0,\dots,\ell-1$, and~$ky+j=-1,0,\dots,k\ell N-1$, with~$j=0,\dots,k-1$. For~$\mathrm w\in \mathcal W_\text{Az}$ we write 
\begin{equation}\label{eq:enumeration_white}
\mathrm{w}_{\ell x+i,ky+j}=\mathrm w=(2\ell x+2i+1,2ky+2j+2) \in \mathcal W_\text{Az}.
\end{equation}
Similarly, for~$\mathrm b\in \mathcal B_\text{Az}$, we write~$\mathrm{b}_{\ell x+i,ky+j}=\mathrm b$ if
\begin{equation}\label{eq:enumeration_black}
\mathrm b=(2\ell x+2i,2ky+2j+1) \in \mathcal B_\text{Az},  
\end{equation}
for~$\ell x+i=0,\dots,k\ell N$, with~$i=0,\dots,\ell-1$, and~$ky+j=0,\dots,k\ell N-1$, with~$j=0,\dots,k-1$. To define the Kasteleyn matrix, we take the \emph{Kasteleyn signs} to be~$-1$ on the north edges, and~$1$ on the other edges. The Kasteleyn matrix~$K_{G_\text{Az}}:\CC^{\mathcal{B}_\text{Az}} \to \CC^{\mathcal{W}_\text{Az}}$ is defined by
\begin{equation}\label{eq:def_kasteleyn_aztec}
(K_{G_\text{Az}})_{\mathrm{w}_{\ell x+i,ky+j}\mathrm{b}_{\ell x'+i',ky'+j'}} \\
=
\begin{cases}
\alpha_{j+1,i+1}, & (\ell x'+i',ky'+j')=(\ell x+i,ky+j+1), \\
\gamma_{j+1,i+1}, & (\ell x'+i',ky'+j')=(\ell x+i,ky+j), \\
\beta_{j+1,i+1}, & (\ell x'+i',ky'+j')=(\ell x+i+1,ky+j+1), \\
-1, & (\ell x'+i',ky'+j')=(\ell x+i+1,ky+j), \\
0, & \text{otherwise.}
\end{cases}
\end{equation}
The matrix is only defined up to permutations of the rows and columns. For definiteness we fix an order of the black and white vertices. We take the order so that~$(\ell x+i,ky+j)<(\ell x'+i',ky'+j')$ if~$\ell x+i<\ell x'+i'$ or if~$\ell x+i=\ell x'+i'$ and~$ky+j<ky'+j'$. 

The Kasteleyn signs are defined in such a way that the partition function of the dimer model is the absolute value of the determinant of~$K_{G_\text{Az}}$. It follows that the edge probabilities can be expressed in terms of the inverse Kasteleyn matrix, see, \emph{e.g.},~\cite{Ken04}.
\begin{theorem}[\cite{Ken97}]\label{thm:kenyon_formula}
Let~$e_m=\mathrm{w}_m\mathrm{b}_m$,~$m=1,\dots,p$, be edges in~$G_\text{Az}$, connecting~$\mathrm{w}_m\in \mathcal W_\text{Az}$ and~$\mathrm{b}_m\in \mathcal B_\text{Az}$, and let~$\PP$ be the measure~\eqref{eq:measure_dimer}. Then 
\begin{equation}\label{eq:kenyon_formula}
\PP[e_1,\dots,e_p\in M]=\prod_{m=1}^p(K_{G_\text{Az}})_{\mathrm{w}_m\mathrm{b}_m}\det\left((K_{G_\text{Az}}^{-1})_{\mathrm{b}_m\mathrm{w}_{m'}}\right)_{1\leq m,m' \leq p}.
\end{equation} 
\end{theorem}
This theorem tells us that the dimer model forms a determinantal point process. See, \emph{e.g.},~\cite{Bor09, Joh17}   for an introduction to determinantal point processes.

In general, it is a difficult problem to find an expression for the matrix~$K_{G_\text{Az}}^{-1}$ suitable for asymptotic analysis. In Section~\ref{sec:inverse_kasteleyn} we give a double integral formula for the inverse Kasteleyn matrix, which we will later show to be suitable for asymptotic analysis. To obtain the formula for the inverse Kasteleyn matrix we employ a non-intersecting paths model.

\begin{figure}
\begin{center}
\begin{tikzpicture}[scale=.75, rotate=-45]

% Coordinate axes
\foreach \x in {0,1,2,3}
\foreach \y in {0,...,3}
{\draw (-3.5+\x+\y,.5-\x+\y) rectangle (-2.5+\x+\y,1.5-\x+\y);
}
% West dominos
\foreach \x/\y in {-4/0,-3/1,-2/2,-3/-1,0/-2}
{ 
\draw[line width = 1mm] (\x+.5,\y+.5)--(\x+.5,\y+1.5);
\draw (\x+.5,\y+.5) node[circle,fill,inner sep=2pt]{};
\draw (\x+.5,\y+1.5) node[circle,draw=black,fill=white,inner sep=2pt]{};
}

% East dominos
\foreach \x/\y in {1/-2,2/-1,3/0,2/1,-1/2}
{
\draw[line width = 1mm] (\x+.5,\y+.5)--(\x+.5,\y+1.5);
\draw (\x+.5,\y+.5) node[circle,draw=black,fill=white,inner sep=2pt]{};
\draw (\x+.5,\y+1.5) node[circle,fill,inner sep=2pt]{};
}

% South dominos
\foreach \x/\y in {-1/-3,-2/-2,0/2,0/0,-2/0}
{
\draw[line width = 1mm] (\x+.5,\y+.5)--(\x+1.5,\y+0.5);
\draw (\x+.5,\y+.5) node[circle,fill,inner sep=2pt]{};
\draw (\x+1.5,\y+.5) node[circle,draw=black,fill=white,inner sep=2pt]{};
}

% North dominos
\foreach \x/\y in {-2/1,0/1,0/3,-1/4,-2/-1}
{
\draw[line width = 1mm] (\x+.5,\y+.5)--(\x+1.5,\y+0.5);
%\draw (\x+.5,\y+.5) node[circle,draw=black,fill=white,inner sep=2pt]{};
%\draw (\x+1.5,\y+.5) node[circle,fill,inner sep=2pt]{};
}

% North edges
\foreach \x/\y in {-1/-2,0/-1,1/0,2/1, -2/-1,-1/0,0/1,1/2, -3/0,-2/1,-1/2,0/3, -4/1,-3/2,-2/3,-1/4}
{
\draw[line width = .4mm, red] (\x+.5,\y+.5)--(\x+1.5,\y+0.5);
\draw (\x+.5,\y+.5) node[circle,draw=black,fill=white,inner sep=2pt]{};
\draw (\x+1.5,\y+.5) node[circle,fill,inner sep=2pt]{};
}

%values
%boundary
\foreach \x in {0,1,2,3,4}
{\draw (\x,\x-3) node {\x};
\draw (-\x,5-\x) node {4};
}
\foreach \y in {1,2,3}
{\draw (-\y,\y-3) node {\y};
\draw (\y,5-\y) node {4};
}
%interior
\draw (0,-2) node {1};
\draw (0,-1) node {2};
\draw (0,0) node {2};
\draw (0,1) node {3};
\draw (0,2) node {3};
\draw (0,3) node {4};
\draw (0,4) node {4};

\draw (-1,-1) node {2};
\draw (-1,0) node {2};
\draw (-1,1) node {3};
\draw (-1,2) node {3};
\draw (-1,3) node {3};

\draw (1,-1) node {1};
\draw (1,0) node {2};
\draw (1,1) node {3};
\draw (1,2) node {3};
\draw (1,3) node {4};

\draw (-2,0) node {2};
\draw (-2,1) node {3};
\draw (-2,2) node {3};

\draw (2,0) node {2};
\draw (2,1) node {3};
\draw (2,2) node {3};

\draw (-3,1) node {3};

\draw (3,1) node {3};
\end{tikzpicture}
\end{center}
\caption{A dimer configuration of the Aztec diamond of size~$4$ together with the values of the height function. The black edges are the dimers in the dimer covering, and the red edges are the reference set~$\mathrm N$ of the north edges. The union of the dimers and the reference set form the level curves of the height function. \label{fig:height_function}}
\end{figure}
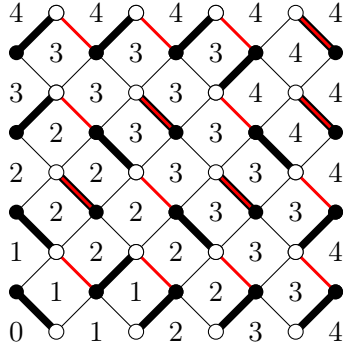

Before we proceed to discuss the non-intersecting paths, we conclude this section by defining the height function. The height function will not be the central object in the present paper, we only introduce it to compare the slope of the limiting Gibbs measures with the slope of the limit shape. The \emph{height function} is defined on the faces of~$G_\text{Az}$, that is, on 
\begin{equation}
\{(2i,2j):i,j=0,\dots,k\ell N\}\cup\{(2i+1,2j+1):i,j=0,\dots,k\ell N-1\}.
\end{equation}
If~$\mathrm f$ and~$\mathrm f'$ are two faces in~$G_\text{Az}$, we define the height function~$h$ for a dimer covering~$M$ so that
\begin{equation}\label{eq:height_difference_aztec}
h(\mathrm f')-h(\mathrm f)=\sum_{e=\mathrm w\mathrm b}(\pm)\left(1_{e\in M}-1_{e\in \mathrm N}\right),
\end{equation}
where the sum runs over the edges intersecting the edges of a dual path of~$G_\text{Az}$ going from~$f$ to~$f'$, and the sign is~$+$ if the path intersects the edge~$e$ with the white vertex on the right, and~$-1$ if it is on the left. Recall that~$\mathrm N$ is the set of north edges. This determines~$h$ up to an additive constant. We fix the constant by setting~$h(0,0)=0$. See Figure~\ref{fig:height_function}. It is not difficult to see that going around a vertex does not change the value of~$h$, which shows that~$h$ is a well-defined function. 

\subsection{Non-intersecting paths model}\label{sec:paths}
In this section we change our viewpoint. Instead of viewing the dimers as particles in a determinantal point process, as in Theorem~\ref{thm:kenyon_formula}, we consider points, or particles, coming from non-intersecting paths on a directed graph. The two point processes are equivalent, in the sense that there exists a measure-preserving bijection between them. The purpose of this change of perspective is that we can then rely on~\cite{BD19}, which gives us a formula for the correlation kernel, see Theorem~\ref{thm:bd_thm} below. The connection between dimer coverings of the Aztec diamond and families of non-intersecting paths has been used many times, see for instance~\cite{Joh02, Joh17}. 

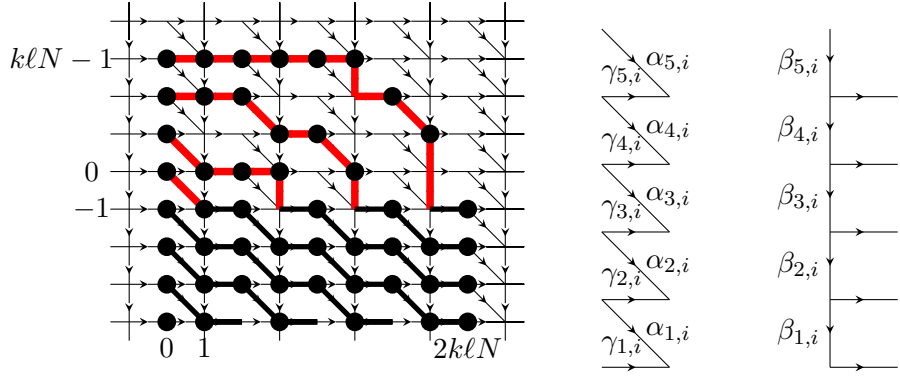
\begin{figure}
\begin{center}
 \begin{tikzpicture}[scale=0.5]
 \tikzset{->-/.style={decoration={
  markings, mark=at position .5 with {\arrow{stealth}}},postaction={decorate}}}
  %Horisontal lines
 \foreach \y in {-3,-2,...,5}
 {
 \foreach \x in {-1,0,...,8}
 {\draw[->,>=stealth] (\x-.5,\y)--(\x+.5,\y);
 \draw (8.5,\y)--(9.5,\y);
 }}
 %Vertical lines
 \foreach \x in {-1,1,...,9}
 {\foreach \y in {-3,-2,...,4}
 {\draw[-<,>=stealth] (\x,\y-.5)--(\x,\y+.5);
 \draw (\x,4.5)--(\x,5.5);}}
 %diagonal lines
 \foreach \y in {-2,-1,...,5}
 {\foreach \x in {-1,1,...,7}
 \draw [->-](\x+1,\y)to(\x+2,\y-1);}

% East dominos
\foreach \x/\y in {5/4,7/2,7/1,5/1,3/1}
{
\draw[line width = 1mm,red] (\x,\y)--(\x,\y-1);
}

% Artificial dominos
\foreach \x/\y in {1/4,3/4,5/3,1/3,3/2,1/1}
{
\draw[line width = 1mm,red] (\x,\y)--(\x+1,\y);
}

% West dominos
\foreach \x/\y in {0/4,2/4,4/4,0/3,2/1}
{
\draw[line width = 1mm,red] (\x,\y)--(\x+1,\y);
%points
\draw (\x,\y) node[circle,fill,inner sep=2.5pt]{};
\draw (\x+1,\y) node[circle,fill,inner sep=2.5pt]{};
}

% South dominos
\foreach \x/\y in {6/3,2/3,4/2,0/2,0/1}
{
\draw[line width = 1mm,red] (\x,\y)--(\x+1,\y-1);
%points
\draw (\x,\y) node[circle,fill,inner sep=2.5pt]{};
\draw (\x+1,\y-1) node[circle,fill,inner sep=2.5pt]{};
}

% lines 
\foreach \y in {-2,-1,0}
{\foreach \x in {1,3,5,7}
{\draw[line width = .7mm] (\x-1,\y)--(\x,\y-1);
\draw[line width = .7mm] (\x,\y)--(\x+1,\y);
\draw[line width = .7mm] (\x,-3)--(\x+1,-3);
%points
\draw (\x-1,\y) node[circle,fill,inner sep=2.5pt]{};
\draw (\x,\y-1) node[circle,fill,inner sep=2.5pt]{};
}
\draw (8,\y) node[circle,fill,inner sep=2.5pt]{};
}

%extra points
\draw (0,-3) node[circle,fill,inner sep=2.5pt]{};
\draw (8,-3) node[circle,fill,inner sep=2.5pt]{};

%coordinates
\draw (0,-3.7) node {$0$};
\draw (1,-3.7) node {$1$};
\draw (8,-3.7) node {$2k\ell N$};
\draw (-2,0) node {$-1$};
\draw (-2,1) node {$0$};
\draw (-2.8,4) node {$k\ell N-1$};
\end{tikzpicture}
\qquad
 \begin{tikzpicture}[scale=.9]
 \tikzset{->-/.style={decoration={
  markings, mark=at position .5 with {\arrow{stealth}}},postaction={decorate}}}
  %Horisontal lines
 \foreach \y in {1,2,...,5}
 {
 \foreach \x in {0}
 {\draw[->-,>=stealth] (\x,\y-1)--(\x+1,\y-1);
 \draw [above](\x+.3,\y-1) node {$\gamma_{\y,i}$};
 }}
 %diagonal lines
 \foreach \y in {1,2,...,5}
 {\foreach \x in {0}
 {\draw [->-](\x,\y)to(\x+1,\y-1);
 \draw [right](\x+.5,\y-.5) node {$\alpha_{\y,i}$};
 }
 }

\end{tikzpicture}
\qquad
\begin{tikzpicture}[scale=.9]
 \tikzset{->-/.style={decoration={
  markings, mark=at position .5 with {\arrow{stealth}}},postaction={decorate}}}
  %Horisontal lines
 \foreach \y in {1,2,...,5}
 {
 \foreach \x in {0}
 {\draw[->-,>=stealth] (\x,\y-1)--(\x+1,\y-1);
 }}
 %Vertical lines
 \foreach \x in {0}
 {\foreach \y in {1,2,...,5}
 {\draw[->-,>=stealth] (\x,\y)--(\x,\y-1);
  \draw [left](\x,\y-.5) node {$\beta_{\y,i}$};
 }}
\end{tikzpicture}

\end{center}
\caption{The directed graph with a family of non-intersecting paths. The weighted directed graph is constructed by gluing together smaller graphs, by alternating the two types of graphs on the right. The black parts of the paths are fixed, due to the non-intersecting property of the paths and the chosen boundary conditions, and the red parts of the paths correspond to a dimer configuration of the Aztec diamond. \label{fig:directed_graph}}
\end{figure}

Consider the following directed graph. We take the vertices as the integer lattice~$\ZZ^2$ and the edges as
\begin{equation}
((2i,j),(2i+1,j)), \,\,\, ((2i,j),(2i+1,j-1)), \,\,\, ((2i+1,j),(2i+2,j)) \,\,\, \text{and} \,\,\, ((2i+1,j),(2i+1,j-1)),
\end{equation}
where the order of the points determines the direction of the edge, see Figure~\ref{fig:directed_graph}. Let~$n\geq \ell N$. We consider~$kn$ paths on this directed graph coordinatized by 
\begin{equation}
\{u_i^j\}_{i=0,\dots,2k\ell N;j=-k(n-\ell N),\dots,k\ell N-1}\subset \left(\ZZ^{2k\ell N}\right)^{kn},
\end{equation}
with starting points~$u_0^j$ and endpoints~$u_{2k \ell N}^j$, where the boundary conditions are taken as
\begin{equation}\label{eq:boundary_condition}
 u_0^j=j \text{ and } u_{2k \ell N}^j=-k\ell N+j, \quad j=-kn+k\ell N,\ldots,k\ell N-1.
\end{equation} 
We require that the paths form a family of \emph{non-intersecting paths}, meaning that no two paths pass through the same vertex. To each such family we associate particles by assigning a particle to~$(2i,j)$ and~$(2i+1,j)$ if the edge connecting the two points belongs to a path in the  family, and to the points~$(2i,j)$ and~$(2i+1,j-1)$ if the edge connecting them belongs to a path in the family.

Due to the non-intersection property, the top~$i$th path,~$1\leq i \leq k\ell N$, has to contain the point~$(2i-1,-1)$. Similarly, the lowest~$i$th path,~$1\leq i\leq k\ell N$, has to contain the point~$(2i-1,k\ell N-kn-1)$. Everything in between these points is fixed. See Figure~\ref{fig:directed_graph}. The part of the non-intersecting paths between the left boundary and the points~$(2i-1,-1)$,~$i=1,\dots,k\ell N$, corresponds to a dimer covering of the Aztec diamond in the following way. If we remove the edges~$((2i+1,j),(2i+2,j))$ in the directed graph, then the non-intersecting paths correspond precisely to the so-called DR-paths\footnote{after Dana Randell, see~\cite[Example 6.49]{Sta97}}. The DR-paths are constructed from a dimer covering as follows. To each dimer we associate a domino, a~$1\times 2$ block, see Figure~\ref{fig:dr-paths}. On the west dominos we draw a horizontal line, on the south dominos we draw a line diagonally down, on the east dominos we draw a vertical line, and we do not draw anything on the north dominos, see Figure~\ref{fig:dr-paths}. The lines form~$k\ell N$ non-intersecting paths. Moreover, if we follow the particles associated with the non-intersecting paths, then we get two particles associated with each west dimer, and two particles associated with each south dimer. We recover the point process where we take any vertex in~$G_\text{Az}$ covered by a west or south dimer as a particle, which is in bijection with the original dimer model.

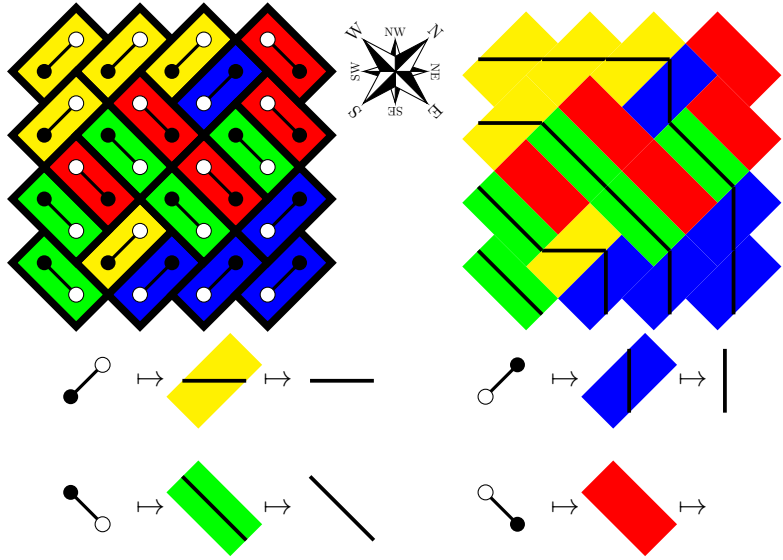
\begin{figure}
\begin{center}
\begin{tikzpicture}[scale=.6, rotate=-45]

% West dominos
\foreach \x/\y in {-4/0,-3/1,-2/2,-3/-1,0/-2}
{ \fill[color=yellow]
(\x,\y) rectangle (\x+1,\y+2); 
\draw [line width = 1mm] (\x,\y) rectangle (\x+1,\y+2);
\draw[very thick,black] (\x+.5,\y+.5)--(\x+.5,\y+1.5);
\draw (\x+.5,\y+.5) node[circle,fill,inner sep=2pt]{};
\draw (\x+.5,\y+1.5) node[circle,draw=black,fill=white,inner sep=2pt]{};
}

% East dominos
\foreach \x/\y in {1/-2,2/-1,3/0,2/1,-1/2}
{  \fill[color=blue]
(\x,\y) rectangle (\x+1,\y+2);
\draw [line width = 1mm] (\x,\y) rectangle (\x+1,\y+2);
\draw[very thick,black] (\x+.5,\y+.5)--(\x+.5,\y+1.5);
\draw (\x+.5,\y+.5) node[circle,draw=black,fill=white,inner sep=2pt]{};
\draw (\x+.5,\y+1.5) node[circle,fill,inner sep=2pt]{};
}

% South dominos
\foreach \x/\y in {-1/-3,-2/-2,0/2,0/0,-2/0}
{  \fill[color=green]
(\x,\y) rectangle (\x+2,\y+1);
\draw [line width = 1mm] (\x,\y) rectangle (\x+2,\y+1);
\draw[very thick,black] (\x+.5,\y+.5)--(\x+1.5,\y+0.5);
\draw (\x+.5,\y+.5) node[circle,fill,inner sep=2pt]{};
\draw (\x+1.5,\y+.5) node[circle,draw=black,fill=white,inner sep=2pt]{};
}

% North dominos
\foreach \x/\y in {-2/1,0/1,0/3,-1/4,-2/-1}
{  \fill[color=red]
(\x,\y) rectangle (\x+2,\y+1);
\draw [line width=1mm] (\x,\y) rectangle (\x+2,\y+1);
\draw[very thick,black] (\x+.5,\y+.5)--(\x+1.5,\y+0.5);
\draw (\x+.5,\y+.5) node[circle,draw=black,fill=white,inner sep=2pt]{};
\draw (\x+1.5,\y+.5) node[circle,fill,inner sep=2pt]{};
}
%Compass
%\pic[scale=.2, rotate=-45] at (-5.25,-1.25) {compass};
\pic[scale=.2, rotate=-45] at (2,6) {compass};
%\draw [ultra thick,->](-4.5,-2)--(-4.5,-.5);
%\draw [ultra thick,->](-4,-1.5)--(-5.5,-1.5);
%\draw (-4.5,-.5) node[above] {north};
%\draw (-5.5,-1.5) node[above] {west};
\end{tikzpicture}
%\qquad
\begin{tikzpicture}[scale=.6, rotate=-45]

% West dominos
\foreach \x/\y in {-4/0,-3/1,-2/2,-3/-1,0/-2}
{ \fill[color=yellow]
(\x,\y) rectangle (\x+1,\y+2); 
%\draw [line width = 1mm] (\x,\y) rectangle (\x+1,\y+2);
\draw[line width = .5 mm, black] (\x,\y+.5)--(\x+1,\y+1.5);
}

% East dominos
\foreach \x/\y in {1/-2,2/-1,3/0,2/1,-1/2}
{  \fill[color=blue]
(\x,\y) rectangle (\x+1,\y+2);
%\draw [line width = 1mm] (\x,\y) rectangle (\x+1,\y+2);
\draw[line width=.5mm,black] (\x,\y+1.5)--(\x+1,\y+.5);
}

% South dominos
\foreach \x/\y in {-1/-3,-2/-2,0/2,0/0,-2/0}
{  \fill[color=green]
(\x,\y) rectangle (\x+2,\y+1);
%\draw [line width = 1mm] (\x,\y) rectangle (\x+2,\y+1);
\draw[line width=.5mm,black] (\x,\y+.5)--(\x+2,\y+0.5);
}

% North dominos
\foreach \x/\y in {-2/1,0/1,0/3,-1/4,-2/-1}
{  \fill[color=red]
(\x,\y) rectangle (\x+2,\y+1);
%\draw [line width=1mm,black] (\x,\y) rectangle (\x+2,\y+1);
}

\end{tikzpicture}
\qquad
\begin{tikzpicture}[scale=.6, rotate=-45]

% West dominos
\foreach \x/\y in {-.5/-.5}
{\draw[very thick,black] (\x+.5,\y+.5)--(\x+.5,\y+1.5);
\draw (\x+.5,\y+.5) node[circle,fill,inner sep=2pt]{};
\draw (\x+.5,\y+1.5) node[circle,draw=black,fill=white,inner sep=2pt]{};
\draw (\x+1.5,\y+2) node {$\mapsto$};
 \fill[color=yellow]
(\x+2,\y+2) rectangle (\x+3,\y+4); 
%\draw [line width = 1mm] (\x,\y) rectangle (\x+1,\y+2);
\draw[line width = .5 mm, black] (\x+2,\y+2.5)--(\x+3,\y+3.5);
\draw (\x+3.5,\y+4) node {$\mapsto$};
\draw[line width = .5 mm, black] (\x+4,\y+4.5)--(\x+5,\y+5.5);
}

% East dominos
\foreach \x/\y in {6/6}
{  \draw[very thick,black] (\x+.5,\y+.5)--(\x+.5,\y+1.5);
\draw (\x+.5,\y+.5) node[circle,draw=black,fill=white,inner sep=2pt]{};
\draw (\x+.5,\y+1.5) node[circle,fill,inner sep=2pt]{};
\draw (\x+1.5,\y+2) node {$\mapsto$};
\fill[color=blue]
(\x+2,\y+2) rectangle (\x+3,\y+4);
%\draw [line width = 1mm] (\x,\y) rectangle (\x+1,\y+2);
\draw[line width=.5mm,black] (\x+2,\y+3.5)--(\x+3,\y+2.5);
\draw (\x+3.5,\y+4) node {$\mapsto$};
\draw[line width=.5mm,black] (\x+3.5,\y+5)--(\x+4.5,\y+4);
}

% South dominos
\foreach \x/\y in {1./-2.}
{ \draw[very thick,black] (\x+.5,\y+.5)--(\x+1.5,\y+0.5);
\draw (\x+.5,\y+.5) node[circle,fill,inner sep=2pt]{};
\draw (\x+1.5,\y+.5) node[circle,draw=black,fill=white,inner sep=2pt]{};
\draw (\x+2,\y+1.5) node {$\mapsto$};
 \fill[color=green]
(\x+2,\y+2) rectangle (\x+4,\y+3);
%\draw [line width = 1mm] (\x,\y) rectangle (\x+2,\y+1);
\draw[line width=.5mm,black] (\x+2,\y+2.5)--(\x+4,\y+2.5);
\draw (\x+4,\y+3.5) node {$\mapsto$};
\draw[line width=.5mm,black] (\x+4,\y+4.5)--(\x+6,\y+4.5);
}

% North dominos
\foreach \x/\y in {7.5/4.5}
{
\draw[very thick,black] (\x+.5,\y+.5)--(\x+1.5,\y+0.5);
\draw (\x+.5,\y+.5) node[circle,draw=black,fill=white,inner sep=2pt]{};
\draw (\x+1.5,\y+.5) node[circle,fill,inner sep=2pt]{};
\draw (\x+2,\y+1.5) node {$\mapsto$};
  \fill[color=red]
(\x+2,\y+2) rectangle (\x+4,\y+3);
%\draw [line width=1mm,black] (\x,\y) rectangle (\x+2,\y+1);
\draw (\x+4,\y+3.5) node {$\mapsto$};
}

\end{tikzpicture}
\end{center}
\caption{The construction of the DR-paths from dimers. Left: A dimer covering together with the corresponding domino tiling. Right: The domino tiling together with the DR-paths. The DR-paths are closely related to the level lines of the height function defined in Section~\ref{sec:dimer}. Indeed, this relation is clearly visible when comparing the above paths with the level lines in Figure~\ref{fig:height_function}, the underlying dimer configuration is the same in the two pictures. \label{fig:dr-paths}}
\end{figure}

We proceed by defining a probability measure on the set of non-intersecting paths which is consistent with the probability measure~\eqref{eq:measure_dimer}. Since an edge~$((2(i-1),j-1),(2i-1,j-1))$ corresponds to a west dimer, we assign the weight~$\gamma_{j,i}$ to such edge. Similarly, an edge~$((2(i-1),j),(2i-1,j-1))$ corresponds to a south dimer, and we assign the weight~$\alpha_{j,i}$ to it, and an edge~$((2i-1,j),(2i-1,j-1))$ corresponds to an east dimer, and we assign the weight~$\beta_{j,i}$ to it. To the last (north) type of edges we assign the weight~$1$. We end up with a weighted directed graph. The weight of a collection of non-intersecting paths is defined as the product of weights of the covered edges. The Lindström--Gessel--Viennot theorem, see~\cite{GV85, Lin73}, implies that this probability measure~$\PP_\text{path}$ has weights proportional to 
\begin{equation}\label{eq:measure_on_points}
 \prod_{m=1}^{2k\ell N} \det \left( T_{\phi_m}(u_{m-1}^i,u_{m}^j)\right)_{i,j=1}^{kn},
\end{equation} 
where the transition matrix~$T_{\phi_m}$,~$m=1\dots,2k\ell N$, is a \emph{block Toeplitz matrix} with \emph{symbol}~$\phi_m$. That is, the~$(i,j)$ block of~$T_{\phi_m}$ is given by
\begin{equation}
\frac{1}{2\pi \i}\int_{|z|=1}\phi_m(z)z^{i-j}\frac{\d z}{z},
\end{equation}
where the symbols are matrices of size~$k\times k$, given by
\begin{equation}\label{eq:def_phi_bernoulli}
\phi_{2i-1}(z)= \phi^b(z;\vec{\alpha_i},\vec{\gamma_i}),
\end{equation}
and
\begin{equation}\label{eq:def_phi_geometric}
\phi_{2i}(z)=\phi^g(z;\vec{\beta_i}),
\end{equation}
for~$i = 1,2,\dots,k\ell N$, see~\eqref{eq:bernoulli} and~\eqref{eq:geometric} for~$\phi^b$ and~$\phi^g$. Here~$\vec{\alpha}_i$,~$\vec{\beta}_i$, and~$\vec{\gamma}_i$ are the vectors with elements~$(\vec \alpha_i)_j=\alpha_{j,i}$,~$(\vec \beta_i)_j=\beta_{j,i}$, and~$(\vec \gamma_i)_j=\gamma_{j,i}$. Given the transition matrices, we define the matrix-valued functions
\begin{equation}\label{eq:matrix_valued_function}
\Phi(z) = \prod_{i=1}^{2\ell}\phi_i(z),
\end{equation}
and~$\phi(z) = \Phi(z)^{kN}$, which we will now use to express the correlation kernel of the point process~$\PP_\text{path}$ defined by~\eqref{eq:measure_on_points}. 

The point process~$\PP_\text{path}$, defined on~$\{(m,u)\}_{m,u} \subset \{0,\dots,2k\ell N\} \times \ZZ$ with~$u_0^j$ and~$u_{2k\ell N}^j$ as in~\eqref{eq:boundary_condition}, is, by the Eynard--Mehta theorem~\cite{EM98}, a \emph{determinantal point process}. That is, there exists a function~$K_\text{path}$, called the \emph{correlation kernel}, such that for all~$(m_i,u_i)\in \{0,\dots,2k \ell N\}\times \ZZ$,~$i=1,\dots,p$,
\begin{equation}\label{eq:dpp}
\PP_\text{path}\left(\text{there are particles at } (m_1,u_1),\dots,(m_p,u_p)\right)=\det (K_\text{path}(m_i,u_i;m_j,u_j))_{j,i=1}^p.
\end{equation}

Recall that we are only interested in the points of the point process that lie above the horizontal line at~$-1$ (the red part of Figure~\ref{fig:directed_graph}), and their distribution is independent of the parameter~$n$ in~\eqref{eq:measure_on_points} if~$n\geq \ell N$. We may therefore take the limit~$n\to \infty$. This setting falls into that of~\cite[Theorem 3.1]{BD19}, which means that we can express the correlation kernel in terms of a (matrix) Wiener--Hopf factorization, if such a factorization exists. 

Assume that the matrix-valued functions~$\phi$ and~$\phi^{-1}$ are analytic in a neighborhood of the unit circle. We say that~$\phi$ admits a \emph{Wiener--Hopf factorization} if
\begin{equation}\label{eq:wh_factorization}
\phi(z)=\phi_+(z)\phi_-(z)=\widetilde \phi_-(z)\widetilde \phi_+(z),
\end{equation}
for~$z$ on the unit circle. Here we require that~$\phi_+^{\pm1}$,~$\widetilde \phi_+^{\pm1}$ are analytic in the closed unit disc, and~$\phi_-^{\pm1}$,~$\widetilde \phi_-^{\pm1}$ are analytic in the complement of the unit disc, with the behavior~$\phi_-, \widetilde \phi_- \sim z^{-\ell N}I$ as~$z\to \infty$. See Section~\ref{sec:wh} for details.

The following theorem is~\cite[Theorem 3.1]{BD19} specialized to our setting. For the arguments of the correlation kernel in~\eqref{eq:dpp} we write~$m=2\ell x+i$,~$x=0,\dots,kN-1$,~$i=0,\dots,2\ell-1$, and~$u=ky+j$,~$y\in \ZZ$,~$j=0,\dots,k-1$. 
\begin{theorem}[\cite{BD19}]\label{thm:bd_thm}
Assume that~$\phi=\Phi^{kN}$ admits a Wiener--Hopf factorization, as described above. Then the correlation kernel~$K_\text{path}$ associated to the~$n\to \infty$ limit of the probability measure~$\PP_\text{path}$ defined by~\eqref{eq:measure_on_points} is given by
\begin{multline}\label{eq:bd_thm}
 \left[K_\text{path}(2\ell x+i,ky+j;2\ell x'+i',ky'+j')\right]_{j',j=0}^{k-1}
 = -\frac{\one_{2\ell x+i>2\ell x'+i'}}{2\pi\i}\int_{\Gamma} \prod_{m=2\ell x'+i'+1}^{2\ell x+i}\phi_m(z)z^{y'-y}\frac{\d z}{z} \\
 + \frac{1}{(2\pi\i)^2}\int_{\Gamma_s}\int_{\Gamma_l} \left(\prod_{m=1}^{i'}\phi_m(z_1)\right)^{-1} \Phi(z_1)^{-x'}\widetilde \phi_-(z_1)\widetilde\phi_+(z_2)\Phi(z_2)^{x-kN}\prod_{m=1}^i\phi_m(z_2)\frac{z_1^{y'}}{z_2^{y}}\frac{\d z_2\d z_1}{z_2(z_2-z_1)},
\end{multline}
for~$x=0,\dots, kN-1$,~$i=0,\dots 2\ell-1$,~$y\in \ZZ$ and~$j=0,\dots,k-1$. The contours are given by~$\Gamma=\{z:|z|=1\}$,~$\Gamma_s=\{z:|z|=r_s\}$ and~$\Gamma_l=\{z:|z|=r_l\}$, oriented in the positive direction, with~$r_s<1<r_l$ such that~$\phi$ and~$\phi^{-1}$ are analytic for~$r_s\leq |z|\leq r_l$.
\end{theorem}
The function~$\phi$ does not always admit a Wiener--Hopf factorization. See the discussion in Section~\ref{sec:ass} for a condition on the edge weights which is equivalent to the existence of the Wiener--Hopf factorization in our setting.

We are only interested in the above formula when~$(x,y)$ lies in the Aztec diamond, which means that~$y=0,\dots,\ell N-1$.

\begin{remark}
In~\cite{BD19} two different correlation kernels are given in the statement, depending on how the limit~$n\to\infty$ is taken. The form of the second correlation kernel, which is not given here, is similar to the one given above, however, the factorization~$\phi=\phi_+\phi_-$ of~\eqref{eq:wh_factorization} is used instead of the factorization~$\phi=\widetilde \phi_-\widetilde \phi_+$, which is present above. 
\end{remark}

\begin{remark}
The correlation kernel of the previous theorem naturally appears in block form, as presented here. However, if we are interested in, say, the probability of finding a particle at the location~$(2\ell x+i,ky+j)$, then it is given by the~$(j+1,j+1)$ entry of the right hand side of~\eqref{eq:bd_thm}. We emphasize also that the rows are indexed by~$j'$ and the columns are indexed by~$j$.  
\end{remark}

\subsection{The inverse Kasteleyn matrix}\label{sec:inverse_kasteleyn}
At this point we have two equivalent determinantal point processes, with two different correlation kernels, the inverse Kasteleyn matrix~$K_{G_\text{Az}}^{-1}$, and~$K_\text{path}$. We took the road via the non-intersecting paths since it yielded the correlation kernel, while there is no known way to obtain a suitable formula for the inverse Kasteleyn matrix. It turns out, however, that~$K_{G_\text{Az}}^{-1}$ is simply a submatrix of~$K_\text{path}$. More precisely, the kernel~$K_\text{path}$ can be viewed as a matrix with both the rows and columns indexed by the vertices in the Aztec diamond, while the inverse Kasteleyn matrix~$K_{G_\text{Az}}^{-1}$ is a matrix where the rows are indexed by the black vertices and the columns are indexed by the white vertices. If we restrict the matrix~$K_\text{path}$ to the rows which are indexed by the white vertices and the columns which are indexed by the black vertices, and then take the transpose, we obtain~$K_{G_\text{Az}}^{-1}$.

\begin{theorem}\label{thm:inverse_kasteleyn}
Let~$K_{G_\text{Az}}$ be the Kasteleyn matrix defined in~\eqref{eq:def_kasteleyn_aztec} and let~$K_\text{path}$ be the correlation kernel  given in Theorem~\ref{thm:bd_thm}. For~$\mathrm{b}_{\ell x+i,ky+j}\in \mathcal B_\text{Az}$ and~$\mathrm{w}_{\ell x'+i',ky'+j'}\in \mathcal W_\text{Az}$,
\begin{equation}
\left(K_{G_\text{Az}}^{-1}\right)_{\mathrm{b}_{\ell x+i,ky+j}\mathrm{w}_{\ell x'+i',ky'+j'}}=K_\text{path}(2\ell x+2i,ky+j;2\ell x'+2i'+1,ky'+j').
\end{equation}
\end{theorem}

\begin{proof}
Let us denote the right hand side of the equality in the statement by~$\widetilde K_{\mathrm{b}_{\ell x+i,ky+j}\mathrm{w}_{\ell x'+i',ky'+j'}}$. We will show that~$K_{G_\text{Az}}\widetilde K=I$. 

We divide the computations into three cases: the case when the white vertex indexing the row is in the interior of the Aztec diamond, when it is on the top boundary, and when it is on the bottom boundary (the two white boundaries in Figure~\ref{fig:aztec_diamond}). Throughout the proof we take~$\mathrm{w}_{\ell x'+i',ky'+j'} \in \mathcal W_\text{Az}$. Recall the labeling of the white and black vertices in~\eqref{eq:enumeration_white} and~\eqref{eq:enumeration_black}.

\paragraph{Case 1)} Let~$\mathrm{w}_{\ell x+i,ky+j}\in \mathcal W_\text{Az}$, with~$-1<ky+j<k\ell N-1$. The matrix~$K_{G_\text{Az}}$ is sparse, and~$(K_{G_\text{Az}}\widetilde K)_{\mathrm{w}_{\ell x+i,ky+j}\mathrm{w}_{\ell x'+i',ky'+j'}}$ consists of four terms, corresponding to the four edges attached to~$\mathrm{w}_{\ell x+i,ky+j}$. 

Using the explicit expressions~\eqref{eq:bernoulli} and~\eqref{eq:bd_thm} for~$\phi_{2i+1}$ and~$K_\text{path}$, it is not difficult to see that for two of the terms one obtains
\begin{multline}\label{eq:inverse_bulk_1}
\gamma_{j+1,i+1}\widetilde K_{\mathrm{b}_{\ell x+i,ky+j}\mathrm{w}_{\ell x'+i',ky'+j'}}
+\alpha_{j+1,i+1}\widetilde K_{\mathrm{b}_{\ell x+i,ky+j+1}\mathrm{w}_{\ell x'+i',ky'+j'}}\\
=-\frac{\one_{\ell x+i\geq \ell x'+(i'+1)}}{2\pi\i}\int_{\Gamma} \left(\prod_{m=2\ell x'+2(i'+1)}^{2\ell x+2i}\phi_{m}(z)\cdot\phi_{2i+1}(z)\right)_{j'+1,j+1}z^{y'-y}\frac{\d z}{z}\\
+\frac{1}{(2\pi\i)^2}\int_{\Gamma_s}\int_{\Gamma_l} \left(\left(\prod_{m=1}^{2i'+1}\phi_m(z_1)\right)^{-1} \Phi(z_1)^{-x'}\widetilde \phi_-(z_1)\right.\\
 \left.\times \widetilde\phi_+(z_2)\Phi(z_2)^{x-kN}\left(\prod_{m=1}^{2i}\phi_m(z_2)\right)\phi_{2i+1}(z_2)\right)_{j'+1,j+1}\frac{z_1^{y'}}{z_2^{y}}\frac{\d z_2\d z_1}{z_2(z_2-z_1)}.
\end{multline}
%The variable~$z$ appears when~$y$ increases with one.
Similarly, recall~\eqref{eq:inverse_geometric} of Lemma~\ref{lem:transition_matrices_properties}, the other two terms give
\begin{multline}\label{eq:inverse_bulk_2}
-\widetilde K_{\mathrm{b}_{\ell x+i+1,ky+j}\mathrm{w}_{\ell x'+i',ky'+j'}}
+\beta_{j+1,i+1}\widetilde K_{\mathrm{b}_{\ell x+i+1,ky+j+1}\mathrm{w}_{\ell x'+i',ky'+j'}} \\
=\frac{\one_{\ell x+i\geq \ell x'+i'}}{2\pi\i}\int_{\Gamma} \left(\prod_{m=2\ell x'+2(i'+1)}^{2\ell x+2(i+1)}\phi_{m}(z)\cdot\phi_{2(i+1)}(z)^{-1}\right)_{j'+1,j+1}z^{y'-y}\frac{\d z}{z} \\
 -\frac{1}{(2\pi\i)^2}\int_{\Gamma_s}\int_{\Gamma_l} \left(\left(\prod_{m=1}^{2i'+1}\phi_m(z_1)\right)^{-1} \Phi(z_1)^{-x'}\widetilde \phi_-(z_1)\right.\\
 \left.\times \widetilde\phi_+(z_2)\Phi(z_2)^{x-kN}\left(\prod_{m=1}^{2(i+1)}\phi_m(z_2)\right)\phi_{2(i+1)}(z_2)^{-1}\right)_{j'+1,j+1}\frac{z_1^{y'}}{z_2^{y}}\frac{\d z_2\d z_1}{z_2(z_2-z_1)}.
\end{multline}
The terms of the left hand sides of~\eqref{eq:inverse_bulk_1} and~\eqref{eq:inverse_bulk_2} are the four nonzero terms in the product, so
\begin{equation}
(K_{G_\text{Az}}\widetilde K)_{\mathrm{w}_{\ell x+i,ky+j}\mathrm{w}_{\ell x'+i',ky'+j'}}=\frac{\one_{\ell x+i=\ell x'+i'}}{2\pi\i}\int_{\Gamma} I_{j'+1,j+1}z^{y'-y}\frac{\d z}{z}=\one_{(\ell x'+i',ky'+j')=(\ell x+i,ky+j)}.
\end{equation}

\paragraph{Case 2)}
Let us proceed by considering the top boundary of the Aztec diamond, that is, considering the white vertices of the form~$\mathrm{w}_{\ell x+i,k\ell N-1}$. There are now only two terms in the matrix product, and they do not sum up as nicely as in the previous case. Instead, we use that~$\mathrm{w}_{\ell x+i,k\ell N-1}$ lies on the boundary of the Aztec diamond, which allows us to compute the sum using residue calculus. The sum is given by 
\begin{multline}\label{eq:inverse_boundary_top}
\gamma_{k,i+1}\widetilde K_{\mathrm{b}_{\ell x+i,k\ell N-1}\mathrm{w}_{\ell x'+i',ky'+j'}}
-\widetilde K_{\mathrm{b}_{\ell x+i+1,k\ell N-1}\mathrm{w}_{\ell x'+i',ky'+j'}}\\
=\frac{\one_{\ell x+i=\ell x'+i'}}{2\pi\i}\int_{\Gamma} \left(\prod_{m=2\ell x'+2(i'+1)}^{2\ell x+2(i+1)}\phi_m(z)\right)_{j'+1,k}z^{y'-\ell N+1}\frac{\d z}{z}\\
-\frac{\one_{\ell x+i\geq \ell x'+(i'+1)}}{2\pi\i}\int_{\Gamma} \left(\prod_{m=2\ell x'+2(i'+1)}^{2\ell x+2i}\phi_m(z)\cdot\left(I\gamma_{k,i+1}-\phi_{2\ell x+2i+1}(z)\phi_{2\ell x+2i+2}(z)\right)\right)_{j'+1,k}z^{y'-\ell N+1}\frac{\d z}{z}\\
+\frac{1}{(2\pi\i)^2}\int_{\Gamma_s}\int_{\Gamma_l} \left(\left(\prod_{m=1}^{2i'+1}\phi_m(z_1)\right)^{-1} \Phi(z_1)^{-x'}\widetilde \phi_-(z_1)\widetilde\phi_+(z_2)\Phi(z_2)^{x-kN}\right.\\
 \left.\times \left(\prod_{m=1}^{2i}\phi_m(z_2)\right)\left(I\gamma_{k,i+1}-\phi_{2\ell x+2i+1}(z_2)\phi_{2\ell x+2i+2}(z_2)\right)\right)_{j'+1,k}\frac{z_1^{y'}}{z_2^{\ell N-1}}\frac{\d z_2\d z_1}{z_2(z_2-z_1)}.
\end{multline}
All but the first term on the right hand side are zero. Indeed, the behavior of~\eqref{eq:bernoulli},~\eqref{eq:geometric}, and~$\widetilde \phi_-$ at infinity implies that 
\begin{equation}
\widetilde\phi_+(z_2)\Phi(z_2)^{x-kN}=\widetilde\phi_-(z_2)^{-1}\Phi(z_2)^{x}=\Ordo(z_2^{\ell N}),
\end{equation}
and 
\begin{equation}
\left(I\gamma_{k,i+1}-\phi_{2\ell x+2i+1}(z_2)\phi_{2\ell x+2i+2}(z_2)\right)_{j'+1,k}=\Ordo(z_2^{-1})
\end{equation} 
as~$z_2\to \infty$, which means tat we may expand~$\Gamma$ and~$\Gamma_l$ to infinity. Hence,~\eqref{eq:inverse_boundary_top} is equal to
\begin{equation}
\frac{\one_{\ell x+i= \ell x'+i'}}{2\pi\i}\int_{\Gamma} \phi_{2\ell x+2(i+1)}(z)_{j'+1,k}z^{y'-\ell N}\d z=\one_{(\ell x'+i',ky'+j')=(\ell x+i,k\ell N-1)}.
\end{equation}
Here we have used that~$\phi_{2\ell x+2(i+1)}(z)_{j'+1,k}=\one_{j'+1=k}+\Ordo(z^{-1})$, as~$z\to \infty$. 

\paragraph{Case 3)}
The remaining part to check is the bottom boundary, that is, when the white vertex is given by~$\mathrm{w}_{\ell x+i,-1}$. This case is similar to the previous one.

The sum is now given by 
\begin{multline}\label{eq:inverse_boundary_bottom}
\alpha_{k,i+1}\widetilde K_{\mathrm{b}_{\ell x+i,0}\mathrm{w}_{\ell x'+i',ky'+j'}}
+\beta_{k,i+1}\widetilde K_{\mathrm{b}_{\ell x+i+1,0}\mathrm{w}_{\ell x'+i',ky'+j'}}\\
=-\beta_{k,i+1}\frac{\one_{\ell x+i= \ell x'+i'}}{2\pi\i}\int_{\Gamma} \left(\prod_{m=2\ell x'+2(i'+1)}^{2\ell x+2(i+1)}\phi_m(z)\right)_{j'+1,1}z^{y'}\frac{\d z}{z}\\
-\frac{\one_{\ell x+i\geq \ell x'+(i'+1)}}{2\pi\i}\int_{\Gamma} \left(\prod_{m=2\ell x'+2(i'+1)}^{2\ell x+2i}\phi_m(z)\cdot\left(I\alpha_{k,i+1}+\beta_{k,i+1}\phi_{2\ell x+2i+1}(z)\phi_{2\ell x+2i+2}(z)\right)\right)_{j'+1,1}z^{y'}\frac{\d z}{z}\\
+\frac{1}{(2\pi\i)^2}\int_{\Gamma_s}\int_{\Gamma_l} \left(\left(\prod_{m=1}^{2i'+1}\phi_m(z_1)\right)^{-1} \Phi(z_1)^{-x'}\widetilde \phi_-(z_1)\widetilde\phi_+(z_2)\Phi(z_2)^{x-kN}\right.\\
 \left.\times \left(\prod_{m=1}^{2i}\phi_m(z_2)\right)\left(I\alpha_{k,i+1}+\beta_{k,i+1}\phi_{2\ell x+2i+1}(z_2)\phi_{2\ell x+2i+2}(z_2)\right)\right)_{j'+1,1}z_1^{y'}\frac{\d z_2\d z_1}{z_2(z_2-z_1)}.
\end{multline}
In a similar way as in the second case, using~\eqref{eq:prod_zero} of Lemma~\ref{lem:transition_matrices_properties} and the behavior of~\eqref{eq:geometric} at zero, we may deform~$\Gamma$ and~$\Gamma_l$ through zero, to see that all terms on the right hand side, except the first one, are zero. In the deformation of~$\Gamma_l$ we pick up a residue at~$z_2=z_1$ along~$\Gamma_s$, however, that integral is also zero. Hence,~\eqref{eq:inverse_boundary_bottom} is equal to
\begin{equation}
-\beta_{k,i+1}\frac{\one_{\ell x+i=\ell x'+i'}}{2\pi\i}\int_{\Gamma} \phi_{2\ell x+2(i+1)}(z)_{j'+1,1}z^{y'}\frac{\d z}{z}=\one_{(\ell x'+i',ky'+j')=(\ell x+i,-1)}.
\end{equation}
We have used that~$-\beta_{k,i+1}\phi_{2\ell x+2(i+1)}(z)_{j'+1,1}=c_{j'+1}z+\Ordo(z^2)$ as~$z\to 0$, where~$c_k=1$.
\end{proof}

\begin{remark}
The proof merely checks the formula, it does not offer any insight into how we obtained it. The observation leading to the guess was that the formulas given in~\cite[Theorem 4.3]{KOS06} turned out to be a submatrix of the limiting bulk correlation kernel which we obtain in Theorem~\ref{thm:local_limit}.
\end{remark}

\subsection{Ergodic translation-invariant Gibbs measures}\label{sec:gibbs_measure}
A complete characterization of ergodic translation-invariant Gibbs measures, with non-extreme slopes, was accomplished in~\cite{KOS06} for planar periodic bipartite graphs with periodic edge weights. These measures constitute a two-parameter family, naturally parametrized by~$(r_1,r_2)\in \RR^2$. As mentioned in the introduction, these measures are conjectured to represent the limiting measure in periodic planar dimer models. In the present section, we proceed to construct this family for our specific graph~$G$.

We follow~\cite{KOS06} and embed the fundamental domain, the smallest non-repeating part of the weighted graph~$G$, in the torus, by identifying~$\mathrm b_{\ell x+i,ky+j}$ with~$\mathrm b_{\ell x'+i,k y'+j}$ for all~$x,x',y,y'\in \ZZ$, and similarly for the white vertices and the edges. Recall that here~$i=0,\dots,\ell-1$ and~$j=0,\dots,k-1$. We denote the resulting bipartite graph on the torus by~$G_1=(\mathcal B_1,\mathcal W_1, \mathcal E_1)$ with the vertices~$\mathrm b_{i,j}$ and~$\mathrm w_{i,j}$, and order them consistently with the order of the vertices in~$G_\text{Az}$. We define a closed loop~$\gamma_u$ on the torus, which intersects the edges between~$\mathrm b_{i,0}$ and~$\mathrm w_{i',k-1}$,~$i,i'=0,\dots,\ell-1$, oriented so it has the black vertices on the right. Similarly, we define~$\gamma_v$ as a closed loop on the torus intersecting the edges between~$\mathrm w_{\ell-1,j}$ and~$\mathrm b_{0,j'}$,~$j,j'=0,\dots,k-1$, oriented so it has the black vertices on the left. See Figure~\ref{fig:magnetic_weights}. For the edges which are intersected by~$\gamma_u$, we multiply the edge weights with~$z^{-1}$, and we multiply the edge weights for the edges which are intersected by~$\gamma_v$ by~$w$. The Kasteleyn signs are defined as in~$G_\text{Az}$, that is, we take them to be~$-1$ for the north edges and~$1$ for the other ones. The matrix~$K_{G_1}(z,w):\CC^{\mathcal B_1}\to \CC^{\mathcal W_1}$, called magnetically altered Kasteleyn matrix, is defined as the Kasteleyn matrix of the weighted graph~$G_1$. Explicitly~$K_{G_1}(z,w)$ is given by
\begin{equation}\label{eq:magnetic_kasteleyn_matrix}
K_{G_1}(z,w) = 
\begin{psmallmatrix}
A_1 & B_1 & \mathbf{0} & \cdots & \mathbf{0} \\
\mathbf{0} & A_2 & B_2 & \cdots & \mathbf{0} \\
\mathbf{0} & \mathbf{0} & A_3 & \cdots & \mathbf{0} \\
 & \vdots & & \ddots & \\
wB_\ell & \mathbf{0} & \mathbf{0} & \cdots & A_\ell 
\end{psmallmatrix},
\end{equation}
where~$A_i= \phi^b(z;\vec{\alpha_i},\vec{\gamma_i})^T$ and~$B_i= -\phi^b(z;-\vec{\beta_i},\vec{1})^T$. In fact, recall item~\eqref{eq:inverse_geometric} of Lemma~\ref{lem:transition_matrices_properties},~$A_i(z)=\phi_{2i-1}(z)^T$ and~$B_i=-(\phi_{2i}(z)^{-1})^T$, where~$\phi_{2i-1}$ and~$\phi_{2i}$ are the matrices in~\eqref{eq:def_phi_bernoulli} and~\eqref{eq:def_phi_geometric}. This relation between the Kasteleyn matrix and the transition matrices is an important relation that helps us to relate the non-intersecting paths picture with the dimer picture.

\begin{figure}
\begin{center}
\begin{tikzpicture}[scale=1.6]

% def of edge weights
\draw (-2.5,2-.5)--(-2,2.5-.5);
\draw (-2.5,2-.5) node[circle,draw=black,fill=black,inner sep=2pt]{};
\draw (-2,2.5-.5) node[circle,draw=black,fill=white,inner sep=2pt]{};
\draw (-2.05,2.14-.5) node {$\gamma_{j,i}$};

\draw (-1.5,2.5-.5)--(-1,2-.5);
\draw (-1.5,2.5-.5) node[circle,draw=black,fill=white,inner sep=2pt]{};
\draw (-1,2-.5) node[circle,draw=black,fill=black,inner sep=2pt]{};
\draw (-1,2.4-.5) node {$-1$};

\draw (-2.5,1-.5)--(-2,.5-.5);
\draw (-2,.5-.5) node[circle,draw=black,fill=white,inner sep=2pt]{};
\draw (-2.5,1-.5) node[circle,draw=black,fill=black,inner sep=2pt]{};
\draw (-2,.9-.5) node {$\alpha_{j,i}$};

\draw (-1.5,.5-.5)--(-1,1-.5);
\draw (-1,1-.5) node[circle,draw=black,fill=black,inner sep=2pt]{};
\draw (-1.5,.5-.5) node[circle,draw=black,fill=white,inner sep=2pt]{};
\draw (-1.05,.64-.5) node {$\beta_{j,i}$};

% black points
\foreach \x in {0,...,4}
{\foreach \y in {0,...,2}
{\draw (\x,\y) node[circle,draw=black,fill=black,inner sep=2pt]{};
}
}

% white points and lines
\foreach \x in {1,...,4}
{\foreach \y in {1,2}
{\draw (\x-1,\y-1)--(\x,\y);
\draw (\x-1,\y)--(\x,\y-1);
\draw (\x-.5,\y-.5) node[circle,draw=black,fill=white,inner sep=2pt]{};
}
\foreach \y in {0}
{\draw (\x-.5,\y-.5)--(\x,\y);
\draw (\x-1,\y)--(\x-.5,\y-.5);
\draw (\x-.5,\y-.5) node[circle,draw=black,fill=white,inner sep=2pt]{};
}
\foreach \y in {3}
{\draw (\x-1,\y-1)--(\x-.5,\y-.5);
\draw (\x-.5,\y-.5)--(\x,\y-1);
\draw (\x-.5,\y-.5) node[circle,draw=black,fill=white,inner sep=2pt]{};
}
}

% vertices
\foreach \y in {0,1,2}
{\foreach \x in {0,1,2,3}
{\draw (\x,\y) node[right] {$\mathrm{b}_{\x,\y}$};
\draw (\x+.5,\y+.5) node [below] {$\mathrm{w}_{\x,\y}$};
}
}

\foreach \y in {0,1,2}
{\draw (4,\y) node [right] {$\mathrm{b}_{0,\y}$};
}

\foreach \x in {0,1,2,3}
{
\draw (\x+.5,-.5) node[below] {$\mathrm{w}_{\x,2}$};
}

\foreach \x in {1,2,3,4}
{
 \draw (\x+.4-1.1,-.35) node {$z^{-1}$};
 \draw (\x+.85-1.1,-.05) node {$z^{-1}$};
}

{
 \draw (0.9+3.,3-.65) node {$w$};
 \draw (0.9+3.,2-.4) node {$w$};
 \draw (0.9+3.,2-.65) node {$w$};
 \draw (0.9+3.,1-.4) node {$w$};
 \draw (0.9+3.,1-.65) node {$w$};
  \draw (.9+3.,-.4) node {$w$};
}

 % loops
 \draw [<-](-.25,-.25)--(4.25,-.25);
 \draw (-.27,-.25) node[below] {$\gamma_u$};
  \draw [<-](3.75,-.75)--(3.75,2.75);
 \draw (3.75,-.75) node[right] {$\gamma_v$};
 \end{tikzpicture}
\end{center}
\caption{The fundamental domain with the magnetic weights, with the periodicity being~$3\times 4$. The lines~$\gamma_u$ and~$\gamma_v$ are the two loops of the torus.\label{fig:magnetic_weights}}
\end{figure}

We define an infinite Kasteleyn matrix~$K_G:\CC^{\mathcal B}\to \CC^{\mathcal W}$ from the infinite graph~$G$, with the Kasteleyn signs as in~$K_{G_\text{Az}}$ and~$K_{G_1}$. Explicitly,~$K_G$ is given by the same formula~\eqref{eq:def_kasteleyn_aztec} as~$K_{G_\text{Az}}$, but not restricted just to~$G_\text{Az}$. In~\cite{KOS06} the ergodic translation-invariant Gibbs measures were expressed in terms of~$K_G$ and~$K_{G_1}^{-1}$. Namely, for~$(r_1,r_2)\in \RR^2$ we define~$K_{G,(r_1,r_2)}^{-1}:\CC^{\mathcal W}\to \CC^{\mathcal B}$ by
\begin{equation}\label{eq:inverse_infinite_kasteleyn}
\left(K_{G,(r_1,r_2)}^{-1}\right)_{\mathrm{b}_{\ell x+i,ky+j}\mathrm{w}_{\ell x'+i',ky'+j'}}=\frac{1}{(2\pi\i)^2}\int_{|z|=\e^{r_1}}\int_{|w|=\e^{r_2}}\left(K_{G_1}(z,w)^{-1}\right)_{\mathrm{b}_{i,j}\mathrm{w}_{i',j'}}\frac{z^{y'-y}}{w^{x'-x}} \frac{\d w}{w}\frac{\d z}{z}.
\end{equation}
\begin{remark}
It is a straightforward check, which we will not pursue here, that 
\begin{equation}
K_GK_{G,(r_1,r_2)}^{-1}=I,
\end{equation}
for all~$(r_1,r_2)\in \RR^2$. A general argument can be found in the proof of~\cite[Theorem 40]{BCT22}
\end{remark}
The ergodic translation-invariant Gibbs measures, indexed by~$(r_1,r_2)\in \RR^2$, are given by   
\begin{equation}\label{eq:gibbs_measure}
\PP_{(r_1,r_2)}[e_1,\dots,e_p\in M]=\prod_{m=1}^p(K_{G})_{\mathrm{w}_m\mathrm{b}_m}\det\left((K_{G,(r_1,r_2)}^{-1})_{\mathrm{b}_m\mathrm{w}_{m'}}\right)_{1\leq m,m' \leq p},
\end{equation} 
where~$e_m=\mathrm{w}_m\mathrm{b}_m$,~$m=1,\dots,p$, are arbitrary edges in~$G$. It is well known, see~\cite{KOS06}, that the Gibbs measures come in different phases, called \emph{rough}, \emph{frozen}, and \emph{smooth}, also known as \emph{liquid}, \emph{solid}, and \emph{gas} in the literature. The phase depends on the location of~$(r_1,r_2)$, whether it lies in the interior of the amoeba, the unbounded complement or the bounded complement of the amoeba. We define the amoeba in Section~\ref{sec:harnack}, and it is closely related to whether~$K_{G_1}(z,w)^{-1}$ is singular on the torus in~\eqref{eq:inverse_infinite_kasteleyn} or not.

Another way to parametrize the ergodic translation-invariant Gibbs measures is via their slopes. We fix a dimer configuration~$M'$ of~$G_1$ and periodically extend it to~$G$. For a given dimer configuration~$M$, the height difference~$H(f')-H(f)$ between two faces~$f$ and~$f'$ of~$G$ is given by
\begin{equation}\label{eq:height_difference_torus}
H(\mathrm f')-H(\mathrm f)=\sum_{e=\mathrm w\mathrm b}(\pm)\left(1_{e\in M}-1_{e\in M'}\right),
\end{equation}
where the sum runs over the edges intersecting a dual path going from~$f$ to~$f'$, and the sign is~$+$ if the path intersects the edge~$e$ with the white vertex on the right, and~$-1$ if it is on the left. The slope~$(s,t)$ of~$\PP_{(r_1,r_2)}$ is defined as
\begin{equation}
s=\EE\left[H(\mathrm f')-H(\mathrm f)\right] \quad \text{and} \quad t=\EE\left[H(\mathrm f'')-H(\mathrm f)\right],
\end{equation}
where~$\mathrm f'$ is equal to~$\mathrm f$ as faces on the torus, but where we have traveled once around the torus along~$\gamma_u$, and similarly for~$\mathrm f''$ but where we have traveled along~$\gamma_v$ instead. 

Let us take~$M'$ as the configuration containing all north dimers. With this choice of reference configuration, the height difference~\eqref{eq:height_difference_torus} coincides, locally, with the height difference of the height function of the Aztec diamond, cf.~\eqref{eq:height_difference_aztec}. The slope~$(s,t)$ can be expressed as a sum of edge probabilities, and by~\eqref{eq:gibbs_measure} we obtain 
\begin{equation}\label{eq:slope_1}
s=-\sum_{e=\mathrm w\mathrm b}(K_{G})_{\mathrm w\mathrm b}\left(K_{G,(r_1,r_2)}^{-1}\right)_{\mathrm b\mathrm w},
\end{equation}
where the sum is taken over all edges intersecting~$\gamma_u$, with the edges projected down to~$G$, and
\begin{equation}\label{eq:slope_2}
t=\sum_{e=\mathrm w\mathrm b}(K_{G})_{\mathrm w\mathrm b}\left(K_{G,(r_1,r_2)}^{-1}\right)_{\mathrm b\mathrm w}-k,
\end{equation}
where the sum now runs over all edges intersecting~$\gamma_v$.

\begin{remark}\label{rem:extremal_slopes}
If~$(r_1,r_2)$ is such that~$\PP_{(r_1,r_2)}[e\in M]=1$ for every west edge~$e$, then, since there is no west edge intersecting~$\gamma_u$ or~$\gamma_v$, the slope is equal to~$(s,t)=(0,-k)$. It is not difficult to see that if~$(s,t)=(0,-k)$, then~$\PP_{(r_1,r_2)}[e\in M]=1$ for any west edge~$e$. Similarly, if~$e$ is a north, east, or south edge, then the same statement holds but with the slope equal to~$(s,t)=(0,0)$,~$(s,t)=(-\ell,0)$, and~$(s,t)=(-\ell,-k)$, respectively. We will see that these slopes form the vertices of the rectangle consisting of all possible slopes, see Remark~\ref{rem:possible_slopes}.
\end{remark}

\section{The spectral curve}\label{sec:characteristic_polynomial}

In the previous section we obtained a double contour integral formula for the correlation kernel~$K_\text{path}$, and hence, via Theorem~\ref{thm:inverse_kasteleyn}, for the inverse Kasteleyn~$K_{G_\text{Az}}^{-1}$. In this section we define a characteristic polynomial which will be central in the asymptotic analysis of the correlation kernel. In Section~\ref{sec:spectral_curve} we show that this characteristic polynomial coincides with the characteristic polynomial defined in~\cite{KOS06}, which is known to be closely related to the Gibbs measures~\eqref{eq:gibbs_measure}. By a result of the same paper, we conclude that the spectral curve defined by such a polynomial is a so-called Harnack curve. The definition of a Harnack curve is rather technical, and the goal of Section~\ref{sec:harnack} is to recall that definition and discuss important properties of the spectral curve. In Section~\ref{sec:abel_theta} we discuss the Abel map, theta functions and prime forms of a Riemann surface. We will use these objects in Sections~\ref{sec:wh} and~\ref{sec:asymptotic} to describe the integrand of the correlation kernel from Theorem~\ref{thm:bd_thm} more explicitly.

\subsection{The characteristic polynomial}\label{sec:spectral_curve}

In the analysis of the correlation kernel in Theorem~\ref{thm:bd_thm} the zero locus of a polynomial plays an important role. More precisely, the steep descent analysis of the double integral will be performed on the Riemann surface defined via this zero set. The polynomial in question is given by
\begin{equation}\label{eq:characteristic_polynomial}
P(z,w) = \prod_{i=1}^\ell(1-\beta^v_iz^{-1})\det\left(\Phi(z)-w I\right)
\end{equation}
with~$\beta^v_i$ and~$\Phi$ given by~\eqref{eq:edge_weights_product} and~\eqref{eq:matrix_valued_function} above.

In~\cite{KOS06} a central object is the characteristic polynomial~$\det K_{G_1}(z,w)$, and in particular the spectral curve -- the zero set of the characteristic polynomial. It was noted in~\cite{Ber21} that in the setting of~\cite{Ber21} the polynomial~\eqref{eq:characteristic_polynomial} coincides with the characteristic polynomial~$\det K_{G_1}(z,w)$. The following proposition tells us that this is also true in the more general setting of the present paper.

\begin{proposition}\label{prop:spectral_curve}
Let~$P$ be the polynomial defined in~\eqref{eq:characteristic_polynomial}, and let~$K_{G_1}$ be the Kasteleyn matrix defined by~\eqref{eq:magnetic_kasteleyn_matrix}. Then
\begin{equation}
P(z,w)=\det K_{G_1}(z,w).
\end{equation}
\end{proposition}

\begin{proof}
The proposition follows from Schur's determinant formula, which states that the determinant of a block matrix is given by
\begin{equation}
\det 
\begin{pmatrix}
A & B \\
C & D 
\end{pmatrix}
=\det (D)\det (A-BD^{-1}C),
\end{equation}
if~$D$ is invertible. 

We write~$K_{G_1}(z,w)$ in block form, 
\begin{equation}
K_{G_1}(z,w)=
\begin{pmatrix}
A & B \\
C & D 
\end{pmatrix}
\end{equation}
with 
\begin{equation}
A=\phi_{1}^T \quad B=(-(\phi_2^{-1})^T,\mathbf{0},\dots,\mathbf{0}), \quad 
C=
\begin{psmallmatrix}
\mathbf{0} \\
\vdots \\
\mathbf{0}\\
-w(\phi_{2\ell}^{-1})^T
\end{psmallmatrix},
\quad\text{and} \quad 
D=
\begin{psmallmatrix}
\phi_3^T & -(\phi_4^{-1})^T & \cdots & \mathbf{0} \\
\mathbf{0} & \phi_5^T & \cdots & \mathbf{0} \\
 & \vdots & & \ddots & \\
\mathbf{0} & \mathbf{0} & \cdots & \phi_{2\ell-1}^T 
\end{psmallmatrix}.
\end{equation}
It is straightforward to confirm that
\begin{equation}
D^{-1}=
\begin{psmallmatrix}
(\phi_3^{-1})^T & \prod_{m=3}^5(\phi_m^{-1})^T & \cdots & \prod_{m=3}^{2\ell-1}(\phi_m^{-1})^T \\
\mathbf{0} & (\phi_5^{-1})^T & \cdots & \prod_{m=5}^{2\ell-1}(\phi_m^{-1})^T \\
 & \vdots & & \ddots & \\
\mathbf{0} & \mathbf{0} & \cdots & (\phi_{2\ell-1}^{-1})^T 
\end{psmallmatrix},
\end{equation}
which yields
\begin{equation}
A-BD^{-1}C=\phi_1^T-w\prod_{m=2}^{2\ell}(\phi_m^{-1})^T=\left((\Phi-wI)\left(\prod_{m=2}^{2\ell}\phi_m\right)^{-1}\right)^T.
\end{equation}
Hence, using~\eqref{eq:characteristic_polynomial} and Lemma~\ref{lem:transition_matrices_properties}~\eqref{eq:determinant_bern_geo},
\begin{equation}
\det K_{G_1}(z,w) = \prod_{m=2}^\ell\det \phi_{2m-1}(z)\det(\Phi(z)-wI)\prod_{m=2}^{2\ell}\det \phi_m(z)^{-1} 
=P(z,w),
\end{equation}
which proves the statement.
\end{proof}
\begin{remark}
The above statement is closely related to~\cite[Theorem 1.2]{AGR21} and~\cite{Izo22}.
\end{remark}
In~\cite{KOS06} it was proved that the zero set of~$\det K_{G_1}(z,w)$ defines a so-called Harnack curve. Thus, Proposition~\ref{prop:spectral_curve} shows that this is also true for the surface that appears in the steep descent analysis of the correlation kernel in Theorem~\ref{thm:bd_thm}.

\subsection{Harnack curves and their amoebas}\label{sec:harnack}
In this section we describe the structure of the curve defined by the zero set of the polynomial~$P$. It follows from the previous section and~\cite{KOS06} that it is a so-called Harnack curve, which is a particularly nice type of algebraic curve. The aim of this section is to define and describe three objects: the Newton polygon~$N(P)$, the (closure of the) spectral curve~$\mathcal R$, and the amoeba~$\mathcal A$. We follow~\cite{MR00} and~\cite[Chapter 5 and 6]{GKZ94}, see also~\cite{Mik00}.

Given the Laurent polynomial~$P$ defined by~\eqref{eq:characteristic_polynomial}, we set 
\begin{equation}
A=\left\{(i,j)\in \ZZ^2: a_{ij}\neq 0, \, P(z,w)=\textstyle{\sum_{i,j}}a_{ij}z^iw^j\right\},
\end{equation}
and define the \emph{Newton polygon}~$N(P)\subset \RR^2$ as the convex hull of~$A$. The Newton polygon is the rectangle with vertices~$(-\ell,0)$,~$(0,0)$,~$(0,k)$ and~$(-\ell,k)$, which can be seen from the explicit expression of~$P=\det K_{G_1}$ where~$K_{G_1}$ is given by~\eqref{eq:magnetic_kasteleyn_matrix}. Let~$\delta_1,\dots,\delta_4$ be the sides of~$N(P)$, given in cyclic order with~$\delta_1$ as the line segment from~$(-\ell,0)$ to~$(0,0)$,~$\delta_2$ from~$(0,0)$ to~$(0,k)$, and so on. We let~$d_i$ be the integer length of~$\delta_i$ and~$g$ be the number of integer points in the interior of~$N(P)$. That is,~$d_1=d_3=\ell$,~$d_2=d_4=k$ and~$g=(k-1)(\ell-1)$. 

Given the Newton polygon, we define the space in which we will embed the zero locus of~$P$.  We set
\begin{equation}
\CC T_A^\circ=\{(z^{0}w^{0}:z^{-1}w^0:\dots:z^{-\ell}w^k):(z,w)\in (\CC^*)^2\} \subset \CC P^{|A|-1}, 
\end{equation}
where~$\CC P^{|A|-1}$ is the~$(|A|-1)$-dimensional projective space over~$\CC$, and~$\CC^*=\CC\backslash \{0\}$. We denote the closure of~$\CC T_A^{\circ}$ by~$\CC T_A$. The group~$(\CC^*)^2$ acts on~$\CC T_A^\circ$, where the action is free and transitive, so~$\CC T_A^\circ$ and~$(\CC^*)^2$ are isomorphic.
Hence,~$(\CC^*)^2\subseteq \CC T_A$. Similarly,~$(\RR^*)^2\subseteq \RR T_A$, where~$\RR T_A$ is the real part of~$\CC T_A$. 

The \emph{moment map},
\begin{equation}
\CC T_A\ni (z^{0}w^{0}:z^{-1}w^0:\dots:z^{-\ell}w^k) \mapsto  \frac{\sum_{(i,j)\in A} |z^iw^j|\cdot(i,j)}{\sum_{(i,j)\in A} |z^iw^j|} \in N(P),
\end{equation}
is a surjective map that maps~$(\CC^*)^2$ to the interior of~$N(P)$, and the complement of~$(\CC^*)^2$ to the boundary of~$N(P)$. See~\cite[Theorem 1.11 in chapter 6]{GKZ94}. The complement of~$(\RR^*)^2\subseteq \RR T_A$ is the union of four (the number of sides of~$N(P)$) lines~$l_i$,~$i=1,\dots,4$, called the \emph{axes} of~$\RR T_A$. The moment map sends each axes to a side of the Newton polygon, and we enumerate the axes so that it takes the axis~$l_i$ to~$\delta_i$. Explicitly, the axes corresponds to taking~$z\in \{0,\infty\}$ or~$w\in \{0,\infty\}$. For instance, the line corresponding to~$w=0$ is the closure of 
\begin{equation}\label{eq:axis}
\RR^* \ni z\mapsto (z^0w^0:z^{-1}w^0:\dots:z^{-\ell}w^0:0:\dots:0),
\end{equation} 
and the moment map takes this line to the side~$\delta_1$, connecting~$(0,0)$ and~$(-\ell,0)$, by 
\begin{equation}
\RR^* \ni z \mapsto (z^0w^0:z^{-1}w^0:\dots:z^{-\ell}w^0:0:\dots:0)\mapsto \frac{\sum_{i=0}^{-\ell} |z^i|\cdot(i,0)}{\sum_{i=0}^{-\ell} |z^i|}. 
\end{equation}
Hence,~$l_1$ corresponds to~$w=0$. Similarly we observe that~$l_2$ corresponds to~$z=\infty$,~$l_3$ to~$w=\infty$, and~$l_4$ corresponds to~$z=0$.

We are ready to define a Harnack curve. Let~$Q$ be any real polynomial~$Q(z,w)$ such that~$N(Q)=N(P)$. We define~$\mathcal R_Q^\circ$ as the zero set
\begin{equation}\label{eq:spectral_curve}
\{(z,w)\in (\CC^*)^2: Q(z,w)=0\},
\end{equation}
and denote the closure of~$\mathcal R_Q^\circ$ in~$\CC T_A$ by~$\mathcal R_Q$. If the curve~$\mathcal R_Q^\circ$ is non-singular, we say that the curve is a \emph{Harnack curve} if:
\begin{itemize}
\item the real part of~$\mathcal R_Q$ consists of~$g+1$ connected components,
\item all those components but one do not intersect the lines~$l_1, \dots,l_4$,
\item the component intersecting the lines can be divided into four consecutive arcs~$a_1,\dots, a_4$ so that~$|a_i\cap l_i|=d_i$ and~$|a_i\cap l_j|=0$ if~$j\neq i$.  
\end{itemize}
If~$\mathcal R_Q^\circ$ is singular, it is a \emph{singular Harnack curve} if (at least) one of the components which do not intersect the axes is contracted to a point. See~\cite[Definitions 2 and 3]{MR00}.

We take~$Q=P$ and refer to~$\mathcal R^\circ=\mathcal R_P^\circ$ as the \emph{spectral curve} of the model, denoting the closure by~$\mathcal R=\mathcal R_P$. The fact that~$P=\det K_{G_1}$ implies, by~\cite[Theorem 5.1]{KOS06}, that~$\mathcal R^\circ$ is a (possibly singular) Harnack curve.

 \begin{figure}[t]
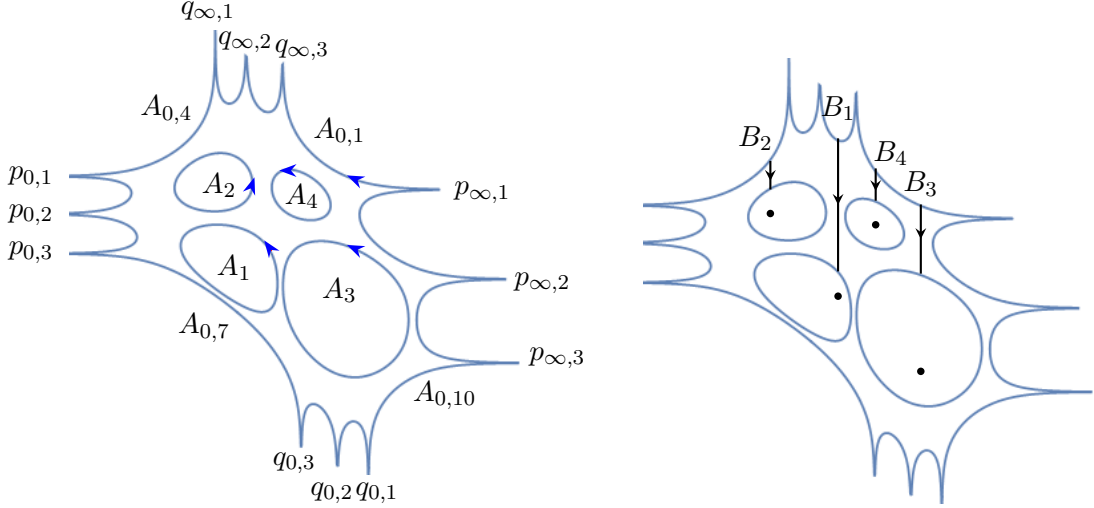

 \begin{center}
 \begin{tikzpicture}[scale=1]
  \tikzset{->-/.style={decoration={
  markings, mark=at position .5 with {\arrow{stealth}}},postaction={decorate}}}
    \draw (0,0) node {\includegraphics[trim={1cm, 1cm, 1cm, 1cm}, clip, angle=180, scale=.5]{amoeba3x3applicable.png}};
	%A_0
	\draw (.6,1.6) node {$A_{0,1}$};
	\draw (-1.7,1.9) node {$A_{0,4}$};    
	\draw (-1.2,-1.) node {$A_{0,7}$};    
	\draw (2.0,-1.9) node {$A_{0,10}$};    
	\draw [red,line width=0,-{Stealth[length=2.3mm,blue]}] (.8,.98)--(.7,1.025);
    %A_i
	\draw (-.8,-.2) node {$A_1$};
	\draw [red,line width=0,-{Stealth[length=2.3mm,blue]}] (-.3,.02)--(-.4,.16);
	\draw (-1,.9) node {$A_2$};	
	\draw [red,line width=0,-{Stealth[length=2.3mm,blue]}] (-.1,1.09)--(-.2,1.09);
	\draw (.6,-.5) node {$A_3$};
	\draw [red,line width=0,-{Stealth[length=2.3mm,blue]}] (.8,.055)--(.7,.1);
	\draw (.1,.75) node {$A_4$};
	\draw [red,line width=0,-{Stealth[length=2.3mm,blue]}] (-.54,.97)--(-.53,1.);
	%angles
	\draw (.1,2.7) node {$q_{\infty,3}$};
	\draw (-.65,2.8) node {$q_{\infty,2}$};
	\draw (-1.15,3.2) node {$q_{\infty,1}$};
	
	\draw (-3.5,1) node {$p_{0,1}$};
	\draw (-3.5,.5) node {$p_{0,2}$};
	\draw (-3.5,0) node {$p_{0,3}$};

	\draw (0,-2.8) node {$q_{0,3}$};
	\draw (.5,-3.2) node {$q_{0,2}$};
	\draw (1.1,-3.2) node {$q_{0,1}$};

	\draw (2.5,.8) node {$p_{\infty,1}$};
	\draw (3.3,-.4) node {$p_{\infty,2}$};
	\draw (3.5,-1.4) node {$p_{\infty,3}$};
	% weights
	% 3\times 3-periodic
	% \alpha_{31}=d_1c_1, \alpha_{22}=d_2c_2, \alpha_{13}=d_3c_3
	% \gamma_{31}=c_1, \gamma_{22}=c_2, \gamma_{13}=c_3
	% \beta_{31}=d_1^{-1}, \beta_{22}=d_2^{-1}, \beta_{13}=d_3^{-1}
	% \alpha_i^v/\gamma_i^v=d_i, \beta_i^v=d_i^{-1}
	% \alpha_i^h/\beta_i^h=d_{4-i}^2c_{4-i}, \gamma_i^h=c_{4-i}	
	%d_1 = 1.5; d_2 = 7.3; d_3 = 28.1;
	%c_1 = .1; c_2 = 1; c_3 = 11.2;
  \end{tikzpicture}
  \quad
   \begin{tikzpicture}[scale=1]
  \tikzset{->-/.style={decoration={
  markings, mark=at position .5 with {\arrow{stealth}}},postaction={decorate}}}
    \tikzset{-->-/.style={decoration={
  markings, mark=at position .7 with {\arrow{stealth}}},postaction={decorate}}}
    \draw (0,0) node {\includegraphics[trim={1cm, 1cm, 1cm, 1cm}, clip, angle=180, scale=.5]{amoeba3x3applicable.png}};
    %(r_1,r_2)
	\draw (-.4,-.2) node[circle,fill,inner sep=1pt]{};
	\draw (.7,-1.2) node[circle,fill,inner sep=1pt]{};
%	\draw (.7,-1.2) node[above] {$(r_1^{(2)},r_2^{(2)})$};
	\draw (-1.3,.9) node[circle,fill,inner sep=1pt]{};
	\draw (.1,.75) node[circle,fill,inner sep=1pt]{};
	%B_i
	\draw (-.4,2.3) node {$B_1$};
%	\draw (-.7,1.5) node {$B_1$};
	\draw[->-,thick](-.4,1.9)--(-.4,.13);
	\draw (.7,1.3) node {$B_3$};
	\draw[->-,thick](.7,1.02)--(.7,.1);
	\draw (-1.5,1.9) node {$B_2$};
	\draw[-->-,thick](-1.3,1.6)--(-1.3,1.22);
	\draw (.3,1.7) node {$B_4$};
	\draw[-->-,thick](.1,1.5)--(.1,1.07);
  \end{tikzpicture}
 \end{center}
  \caption{The amoeba of a dimer model with a~$3\times 3$-periodic fundamental domain. The outer boundary corresponds to~$A_0=\cup_{i=1}^{2(k+\ell)}A_{0,i}$, with tentacles corresponding to each angle. The holes corresponds to~$A_i$,~$i=1,\dots,g$. In the right picture the curves represents the~$B$-cycles, and the dots represent the points~$(r_1^{(i)},r_2^{(i)})$, see Section~\ref{sec:abel_theta}. 
 %\textcolor{blue}{The orientation is so that~$A_i\wedge B_j=\delta_{ij}$, the crossing means (I think)~$x$-axes$\wedge y$-axes$=1$, while~$y$-axes$\wedge x$-axes$=-1$}
 \label{fig:amoeba}}
\end{figure}

We denote the connected components of the real part of~$\mathcal R$ by~$A_0, A_1, \dots, A_g$. If the curve is non-singular then these connected components are disjoint simple loops, and if the curve is singular then~$A_i$ is a point, for some~$i$. We take~$A_0$ as the component intersecting the axes~$l_1,\dots,l_4$ and define the orientation of~$A_0$ to be consistent with the arcs~$a_1,\dots,a_4$. The orientations of~$A_i$, for~$i=1,\dots,g$, are defined so that they are consistent with~$A_0$, in the sense that if we deform a positively oriented curve from~$A_0$ to~$A_i$,~$i=1,\dots,g$, then the resulting curves are still oriented in positive direction, see Figure~\ref{fig:amoeba}. We will refer to~$A_i$,~$i=0,\dots,g$ as \emph{ovals}, and~$A_i$,~$i=1,\dots,g$ as \emph{compact ovals}. The axes~$l_1,\dots,l_4$ intersect~$A_0$ at~$d=\sum_{i=1}^4d_i=2(\ell+k)$ points, so~$A_0\cap \mathcal R^\circ$ consists of~$d$ connected components. We denote the components by~$A_{0,i}$,~$i=1,\dots,d$, so that the enumeration follows the orientation of~$A_0$. We take~$A_{0,1}$ to be the component with boundary points on the arcs~$a_2$ and~$a_3$, that is, the boundary points are given by~$p_{\infty,j}$ and~$q_{\infty,i}$, defined below, for some~$j$ and~$i$. See Figure~\ref{fig:amoeba} for the image of the ovals under the~$\Log$ map defined below.

The intersection points of the component~$A_0$ with the axes play an important role, and we can in fact describe them in terms of the edge weights. Recall that~$l_1$ and~$l_3$ correspond to~$w=0$ and~$w=\infty$, respectively. The intersection points on~$a_1$ and~$a_3$ are therefore obtained by taking~$w\to 0$ and~$w\to \infty$ in~\eqref{eq:spectral_curve} with~$Q=P$, respectively. We denote them by~$q_{0,i}$ and~$q_{\infty,i}$. Similarly we get the points on~$a_2$ and~$a_4$, which we denote by~$p_{\infty,j}$ and~$p_{0,j}$, by taking~$z\to \infty$ and~$z\to0$, respectively. For instance, taking~$w\to 0$ in the equation~$P(z,w)=0$ gives us~$q_{0,i}=(z,0)$ where, using item~\eqref{eq:determinant_bern_geo} of Lemma~\ref{lem:transition_matrices_properties} and~\eqref{eq:spectral_curve},~$z$ solves the equation
\begin{equation}\label{eq:equation_angles}
\prod_{i=1}^\ell(\gamma^v_i-(-1)^k\alpha^v_iz^{-1})=0.
\end{equation}
We obtain
\begin{equation}\label{eq:angles_1}
q_{0,i}=((-1)^k\alpha^v_i/\gamma^v_i,0), \quad q_{\infty,i}=(\beta^v_i,\infty), \quad i=1,\dots,\ell,
\end{equation}
and
\begin{equation}\label{eq:angles_2}
p_{0,j}=(0,(-1)^\ell\alpha^h_j/\beta^h_j), \quad p_{\infty,j}=(\infty,\gamma^h_j), \quad j=1,\dots,k,
\end{equation}
where~$\alpha_i^{h/v}$,~$\beta_i^{h/v}$ and~$\gamma_i^{h/v}$, are given in~\eqref{eq:edge_weights_product}.
We refer to these points as \emph{angles}. In Section~\ref{sec:one_form_thetas} we will see that these points coincide with the angles as used in~\cite{BCT22}. 

We proceed by defining the amoeba of the Laurent polynomial~$P$. Let us define the function~$\Log:(\CC^*)^2\mapsto \RR^2$ given by
\begin{equation}\label{eq:log_map}
\Log(z,w)=(\log|z|,\log|w|).
\end{equation}
We write~$(r_1,r_2)=(\log|z|,\log|w|)$. The \emph{amoeba}~$\mathcal A$ is defined by the image of~$\mathcal R^\circ$ under the map~$\Log$,~$\mathcal A=\Log(\mathcal R^\circ)$. The Harnack curves can be characterized in terms of the amoebas, namely, the curve~$\mathcal R^\circ$ is a (possibly singular) Harnack curve if and only if the map~$\Log|_{\mathcal R^\circ}:\mathcal R^\circ \to \mathcal A$ is at most~$2$-to-$1$,~\cite[Theorem 1]{MR00}. The polynomial~$P$ is a real polynomial, so~$(z,w)\in \mathcal R^\circ$ if and only if~$(\bar z,\bar w)\in \mathcal R^\circ$, which implies that the map is~$2$-to-$1$ away from the real part of~$\mathcal R^\circ$. Moreover, the map is~$1$-to-$1$ on the non-singular part of the real part of~$\mathcal R^\circ$, which is equal to~$\mathcal R^\circ\cap \Log^{-1}(\partial \mathcal A)$, see~\cite[Theorem 1 and Corollary 3]{MR00}. The orientation of~$A_i$ defines an orientation of~$\partial \mathcal A$, namely, the curve~$\Log(A_i)$ is oriented counterclockwise, see Figure~\ref{fig:amoeba}. 

 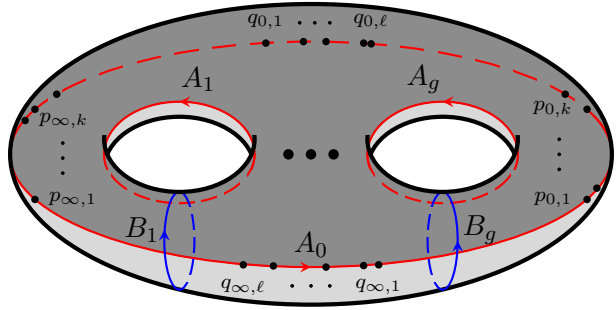
\begin{figure}[t]
 \begin{center}
 \begin{tikzpicture}[scale=1]
   \tikzset{->-/.style={decoration={
  markings, mark=at position .25 with {\arrow{stealth}}},postaction={decorate}}}
     \tikzset{-<-/.style={decoration={
  markings, mark=at position .25 with {\arrow{stealth[reversed]}}},postaction={decorate}}}
     \tikzset{>-/.style={decoration={
  markings, mark=at position .0 with {\arrow{stealth}}},postaction={decorate}}}
\tikzset{-</.style={decoration={
  markings, mark=at position .5 with {\arrow{stealth[reversed]}}},postaction={decorate}}}
 \colorlet{colorTop}{gray!90}
 \colorlet{colorBottom}{gray!30}
 \colorlet{colorAcycles}{red}
 \colorlet{colorBcycles}{blue}
 %Coloring
  \begin{scope}
    \clip (-4,-2) rectangle (4,0);
    \draw[fill=colorBottom] (0,0) ellipse (4cm and 2cm);
\end{scope}
\foreach \y in {-.25}
\foreach \x in {.25}
{{
  \begin{scope}
    \clip (-4,-2) rectangle (4,.2);
    \draw[fill=colorTop] (0,0) ellipse (4cm and 1.5cm);
	\draw[fill=white] (-1.5+\y,.2) ellipse (1cm and .7cm);
	\draw[fill=white] (1.5+\x,.2) ellipse (1cm and .7cm);
\end{scope} 
\begin{scope}
    \clip (-4,2) rectangle (4,0);
    \draw[fill=colorTop] (0,0) ellipse (4cm and 2cm);
	\draw[fill=white] (-1.5+\y,0) ellipse (1cm and .7cm);
	\draw[fill=white] (1.5+\x,0) ellipse (1cm and .7cm);
\end{scope} 
\begin{scope}
    \clip (-4,2) rectangle (4,0);
    \draw[fill=colorBottom] (-1.5+\y,0) ellipse (1cm and .7cm);
    \draw[fill=colorBottom] (1.5+\x,0) ellipse (1cm and .7cm);
	\draw[fill=white] (-1.5+\y,-.2) ellipse (1cm and .7cm);
	\draw[fill=white] (1.5+\x,-.2) ellipse (1cm and .7cm);
\end{scope} 
}}
  %Outer real curve
  \begin{scope}
    \clip (-4,-2) rectangle (4,.5);
    \draw[thick, colorAcycles,->-,rotate=180] (0,0) ellipse (4cm and 1.5cm);
\end{scope} 
\foreach \x in {-3.5,-3,...,4}
{  \begin{scope}
    \clip (\x-.31,0) rectangle (\x,2);
    \draw[thick, colorAcycles] (0,0) ellipse (4cm and 1.5cm);
\end{scope} 
}
 %Outer boundery
  \draw[ultra thick, rotate=180] (0,0) ellipse (4cm and 2cm);
%Left
%Inner real curve left
 \foreach \y in {-.25}
{ \begin{scope}
    \clip (-3+\y,1) rectangle (\y,-.2);
    \draw[thick, colorAcycles, ->-] (-1.5+\y,0) ellipse (1cm and .7cm);
\end{scope} 
\foreach \x in {-2.5,-2.2,...,0}
{
  \begin{scope}
    \clip (\x-.2+\y,0) rectangle (\x+\y,-1);
    \draw[thick, colorAcycles] (-1.5+\y,0) ellipse (1cm and .65cm);
\end{scope} 
}
%B-cycle left
  \begin{scope}
    \clip (-2+\y,-2) rectangle (-1.5+\y,0);
    \draw[thick, colorBcycles, -<] (-1.5+\y,-1.15) ellipse (.2cm and .65cm);
\end{scope} 
\foreach \x in {-2,-1.7,...,0}
{
  \begin{scope}
    \clip (-1+\y,\x-.2) rectangle (-1.5+\y,\x);
    \draw[thick, colorBcycles] (-1.5+\y,-1.15) ellipse (.2cm and .65cm);
\end{scope}
}
%Inner boundary left
  \begin{scope}
    \clip (-3+\y,-1) rectangle (\y,.25);
    \draw[ultra thick] (-1.5+\y,.2) ellipse (1cm and .7cm);
\end{scope} 
  \begin{scope}
    \clip (-3+\y,1) rectangle (\y,0);
    \draw[ultra thick] (-1.5+\y,-.2) ellipse (1cm and .7cm);
\end{scope} 
}
%Right
%Inner real curve right
 \foreach \y in {+.25}
{
  \begin{scope}
    \clip (3+\y,1) rectangle (+\y,-.2);
    \draw[thick, colorAcycles, ->-] (1.5+\y,0) ellipse (1cm and .7cm);
\end{scope} 
\foreach \x in {2.5,2.2,...,0}
{
  \begin{scope}
    \clip (\x+.2+\y,0) rectangle (\x+\y,-1);
    \draw[thick, colorAcycles] (1.5+\y,0) ellipse (1cm and .65cm);
\end{scope} 
}
%B-cycle right
  \begin{scope}
    \clip (2+\y,-2) rectangle (1.5+\y,0);
    \draw[thick, colorBcycles, >-] (1.5+\y,-1.15) ellipse (.2cm and .65cm);
\end{scope}
\foreach \x in {-2,-1.7,...,0}
{
  \begin{scope}
    \clip (1+\y,\x-.2) rectangle (1.5+\y,\x);
    \draw[thick, colorBcycles] (1.5+\y,-1.15) ellipse (.2cm and .65cm);
\end{scope}
}
%Inner boundary right
  \begin{scope}
    \clip (3+\y,-1) rectangle (\y,.25);
    \draw[ultra thick] (1.5+\y,.2) ellipse (1cm and .7cm);
\end{scope} 
  \begin{scope}
    \clip (3+\y,1) rectangle (\y,0);
    \draw[ultra thick] (1.5+\y,-.2) ellipse (1cm and .7cm);
\end{scope} 
}
%dots
%ovals
\foreach \x in {-.3,0,.3}
{\draw (\x,0) node[circle,fill,inner sep=1.4pt]{};}
%q_\infty
\foreach \x in {-.25,0,.25}
{\draw (\x,-1.75) node[circle,fill,inner sep=.5pt]{};}
%p_0
\foreach \y in {-.2,0,.2}
{\draw (3.3,\y) node[circle,fill,inner sep=.5pt]{};}
%q_0
\foreach \x in {-.15,.05,.25}
{\draw (\x,1.75) node[circle,fill,inner sep=.5pt]{};}
%p_\infty
\foreach \y in {-.25,-.05,.15}
{\draw (-3.3,\y) node[circle,fill,inner sep=.5pt]{};}
%ovals and cycles
\draw (0,-1.2) node {$A_0$}; 
\draw (-1.5,1) node {$A_1$}; 
\draw (1.5,1) node {$A_g$}; 
\draw (-2.25,-1.) node {$B_1$}; 
\draw (2.25,-1.) node {$B_g$}; 
% Angles
\draw (-.9,-1.47) node[circle,fill,inner sep=1pt,label=below:$\scriptstyle{q_{\infty,\ell}}$]{};
\draw (-.5,-1.485) node[circle,fill,inner sep=1pt]{};
\draw (.2,-1.5) node[circle,fill,inner sep=1pt]{};
\draw (.7,-1.48) node[circle,fill,inner sep=1pt]{};
\draw (.9,-1.47) node[circle,fill,inner sep=1pt,label=below:$\scriptstyle{q_{\infty,1}}$]{};
\draw (3.67,-.6) node[circle,fill,inner sep=1pt,label=left:$\scriptstyle{p_{0,1}}$]{};
\draw (3.8,-.45) node[circle,fill,inner sep=1pt]{};
\draw (3.38,.8) node[circle,fill,inner sep=1pt]{};
\draw (3.67,.6) node[circle,fill,inner sep=1pt,label=left:$\scriptstyle{p_{0,k}}$]{};
\draw (.8,1.47) node[circle,fill,inner sep=1pt,label=above:$\scriptstyle{q_{0,\ell}}$]{};
\draw (.7,1.48) node[circle,fill,inner sep=1pt]{};
\draw (.2,1.5) node[circle,fill,inner sep=1pt]{};
\draw (-.1,1.5) node[circle,fill,inner sep=1pt]{};
\draw (-.6,1.48) node[circle,fill,inner sep=1pt,label=above:$\scriptstyle{q_{0,1}}$]{};
\draw (-3.8,.45) node[circle,fill,inner sep=1pt,label=right:$\scriptstyle{p_{\infty,k}}$]{};
\draw (-3.67,.6) node[circle,fill,inner sep=1pt]{};
\draw (-3.38,.8) node[circle,fill,inner sep=1pt]{};
\draw (-3.67,-.6) node[circle,fill,inner sep=1pt,label=right:$\scriptstyle{p_{\infty,1}}$]{};
 \end{tikzpicture}
 \end{center}
  \caption{An illustration of the Riemann surface. If we remove the angles from the picture above, the resulting surface can be visualized as the union of two copies of the amoeba joined together along their boundaries. In this depiction, the red curves represent the boundary. The inner red curves correspond to the~$A$-cycles, while the blue curves represent the~$B$-cycles (cf. Figure~\ref{fig:amoeba}).
 \label{fig:riemann_surface}}
\end{figure}

The real part of~$\mathcal R$ divides~$\mathcal R$ into two connected components. Each component is in bijection, under the restriction of the~$\Log$ map, with the interior of the amoeba minus~$\Log(A_i)$ for all~$i$ such that~$A_i$ is reduced to a point. In particular, if~$\mathcal R$ is non-singular, the components are in bijection with the interior of the amoeba. We denote the component with positively oriented boundary~$A_0-(A_1+\dots A_g)$ by~$\mathcal R_0$. We will often visualize the Harnack curve~$\mathcal R$ via the amoeba, see Figure~\ref{fig:riemann_surface}.

Let us describe what the amoeba in our setting looks like. The image of~$A_i$,~$i=1,\dots,g$, under the map~$\Log$, are simple closed disjoint curves. In particular, the curves are bounded and form the inner boundary of the amoeba. If~$A_i$ is reduced to a point, so is the corresponding curve in the amoeba. The image of~$A_0\cap \mathcal R^\circ$ under~$\Log$ consists of simple unbounded curves, with so-called \emph{tentacles} going to infinity, and they form the outer boundary of~$\mathcal A$. Each tentacle corresponds to an angle. Indeed, there are at most~$\ell$ tentacles with the~$r_2$-coordinate going to~$+\infty$ and at most~$\ell$ ones with the~$r_2$-coordinate going to~$-\infty$. The asymptotes are the lines~$r_1=\log\beta^v_i$, and the lines~$r_1=\log(\alpha_i^v/\gamma_i^v)$,~$i=1\dots,\ell$, respectively. Similarly, there are at most~$k$ tentacles with the~$r_1$-coordinate going to~$+\infty$ and~$-\infty$ with asymptotes~$r_2=\log\gamma_j^h$ and~$r_2=\log(\alpha_j^h/\beta_j^h)$,~$j=1,\dots,k$, respectively. In particular, there are a maximal number of tentacles precisely when all angles are distinct. Finally, each connected component of the complement of~$\mathcal A$ is convex,~\cite[Corollary 1.6, Chapter 6]{GKZ94}. See Figure~\ref{fig:amoeba} for an example. For a more general description, see, for instance,~\cite{Mik00}.

 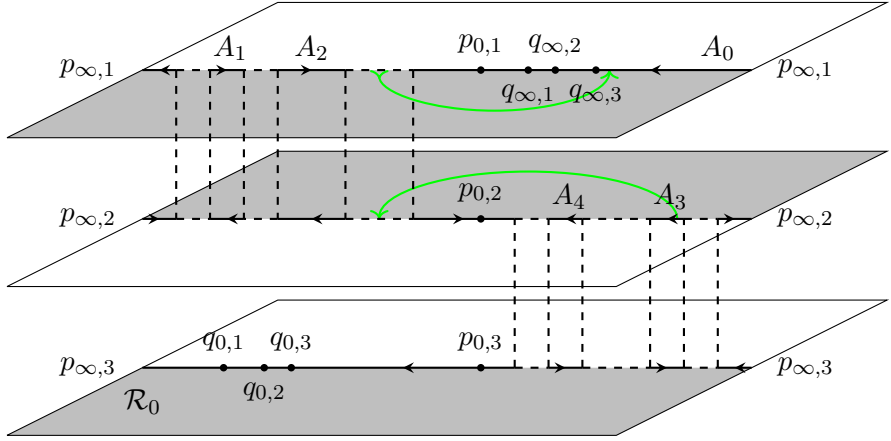
\begin{figure}[t]
 \begin{center}
 \begin{tikzpicture}[scale=.9]
   \tikzset{->--/.style={decoration={
  markings, mark=at position .3 with {\arrow{stealth}}},postaction={decorate}}}
   \tikzset{->-/.style={decoration={
  markings, mark=at position .5 with {\arrow{stealth}}},postaction={decorate}}}
    %first sheet
    % Plane 
   \fill[color=lightgray](-4,0)--(5,0)--(3,-1)--(-6,-1);
   \draw (-6,-1)--(3,-1)--(7,1)--(-2,1)--(-6,-1);
   % curve
   \draw[green,thick,>->] (-.5,0) arc(-180:0:1.7cm and .6cm);
   % Points
   \draw (1,0) node[circle,fill,inner sep=1pt,label=above:$p_{0,1}$]{};
   \draw (5.8,0) node{$p_{\infty,1}$};
   \draw (-4.8,0) node{$p_{\infty,1}$};
   \draw (1.7,0) node[circle,fill,inner sep=1pt,label=below:$q_{\infty,1}$]{};
   \draw (2.1,0) node[circle,fill,inner sep=1pt,label=above:$q_{\infty,2}$]{};
   \draw (2.7,0) node[circle,fill,inner sep=1pt,label=below:$q_{\infty,3}$]{};
   \draw (4.5,0) node[above]{$A_0$};
   \draw (-2.7,0) node[above]{$A_1$};
   \draw (-1.5,0) node[above]{$A_2$};
   % real ovals
   \draw[thick,->--] (5,0)--(0,0);
   \draw[thick,->-] (-2,0)--(-1,0);
   \draw[thick,->-] (-3,0)--(-2.5,0);
   \draw[thick,->-] (-3.5,0)--(-4,0);
   % Cuts
   \draw [thick,dashed] (-1,0)--(0,0);
   \draw [thick,dashed] (-2,0)--(-2.5,0);
   \draw [thick,dashed] (-3.5,0)--(-3,0);
   % Second Sheet
   % Plane 
   \fill[color=lightgray]
   (-2,-1.2)--(7,-1.2)--(5,-2.2)--(-4,-2.2);
   \draw (-6,-3.2)--(3,-3.2)--(7,-1.2)--(-2,-1.2)--(-6,-3.2);
   % curve
   \draw[green,thick,<-] (-.5,-2.2) arc(180:0:2.2cm and .7cm);
   % Points
   \draw (1,-2.2) node[circle,fill,inner sep=1pt,label=above:$p_{0,2}$]{};
   \draw (5.8,-2.2) node{$p_{\infty,2}$};
   \draw (-4.8,-2.2) node{$p_{\infty,2}$};
   \draw (3.8,-2.2) node[above]{$A_3$};
   \draw (2.3,-2.2) node[above]{$A_4$};
   % real ovals
   \draw[thick,->-] (4.5,-2.2)--(5,-2.2);   
   \draw[thick,->-] (4,-2.2)--(3.5,-2.2);
   \draw[thick,->-] (2.5,-2.2)--(2,-2.2);
   \draw[thick,->-] (0,-2.2)--(1.5,-2.2);
   \draw[thick,->-] (-2.5,-2.2)--(-3,-2.2);
   \draw[thick,->-] (-1,-2.2)--(-2,-2.2);
   \draw[thick,->-] (-4,-2.2)--(-3.5,-2.2);
   % Cuts
   \draw [thick,dashed] (-1,-2.2)--(0,-2.2);
   \draw [thick,dashed] (-2,-2.2)--(-2.5,-2.2);
   \draw [thick,dashed] (-3.5,-2.2)--(-3,-2.2);
   \draw [thick,dashed] (1.5,-2.2)--(2,-2.2);
   \draw [thick,dashed] (2.5,-2.2)--(3.5,-2.2);
   \draw [thick,dashed] (4,-2.2)--(4.5,-2.2);
   % Connect
   \draw[thick,dashed] (0,0)--(0,-2.2);
   \draw[thick,dashed] (-1,0)--(-1,-2.2);
   \draw[thick,dashed] (-2,0)--(-2,-2.2);
   \draw[thick,dashed] (-2.5,0)--(-2.5,-2.2);
   \draw[thick,dashed] (-3,0)--(-3,-2.2);
   \draw[thick,dashed] (-3.5,0)--(-3.5,-2.2);
   % third Sheet
   % Plane 
   \fill[color=lightgray]
   (-6,-5.4)--(3,-5.4)--(5,-4.4)--(-4,-4.4);
   \draw (-6,-5.4)--(3,-5.4)--(7,-3.4)--(-2,-3.4)--(-6,-5.4);
   % Points
   \draw (1,-4.4) node[circle,fill,inner sep=1pt,label=above:$p_{0,3}$]{};
   \draw (5.8,-4.4) node{$p_{\infty,3}$};
   \draw (-4.8,-4.4) node{$p_{\infty,3}$};
   \draw (-1.8,-4.4) node[circle,fill,inner sep=1pt,label=above:$q_{0,3}$]{};
   \draw (-2.2,-4.4) node[circle,fill,inner sep=1pt,label=below:$q_{0,2}$]{};
   \draw (-2.8,-4.4) node[circle,fill,inner sep=1pt,label=above:$q_{0,1}$]{};
   \draw (-4,-4.9) node {$\mathcal R_0$};
   % real ovals
   \draw[thick,->-] (5,-4.4)--(4.5,-4.4);   
   \draw[thick,->-] (3.5,-4.4)--(4,-4.4);
   \draw[thick,->-] (2,-4.4)--(2.5,-4.4);
   \draw[thick,->--] (1.5,-4.4)--(-4,-4.4);
   % Cuts
   \draw[thick] [dashed] (1.5,-4.4)--(2,-4.4);
   \draw[thick] [dashed] (2.5,-4.4)--(3.5,-4.4);
   \draw[thick] [dashed] (4,-4.4)--(4.5,-4.4);
   % Connect
   \draw[thick,dashed] (1.5,-4.4)--(1.5,-2.2);
   \draw[thick,dashed] (2,-4.4)--(2,-2.2);
   \draw[thick,dashed] (2.5,-4.4)--(2.5,-2.2);
   \draw[thick,dashed] (3.5,-4.4)--(3.5,-2.2);
   \draw[thick,dashed] (4,-4.4)--(4,-2.2);   
   \draw[thick,dashed] (4.5,-4.4)--(4.5,-2.2);   
  \end{tikzpicture}
 \end{center}
  \caption{The Riemann surface represented in terms of~$k$ sheets. The real ovals and the cuts are located along the real axes. The real ovals are the solid parts and the cuts are the dashed parts. The cuts connecting two consecutive sheets are alternating between the sides of the origin. The grey shaded region is the connected component~$\mathcal R_0$. The green curve is an oriented curve in~$\mathcal R_0$ going from~$A_3$ to~$A_{0,1}$, compare with the green curve in Figure~\ref{fig:amoeba_curves}.\label{fig:rs_sheets}}
\end{figure}

In this paper we will, with the exception of Section~\ref{sec:limit_shape_smooth}, visualize the Riemann surface in terms of the amoeba, as described above. However, it is common to represent a Riemann surface through the sheets defined by the projections to one of the variables. Therefore, for the sake of completeness, we describe~$\mathcal R$ via these sheets as well. See Figure~\ref{fig:rs_sheets} for an example where~$k=\ell=3$. 

The~$k$ sheets~$\mathcal R_j$,~$j=1,\dots,k$, are defined as~$\mathcal R_j=\{(z,w_j(z))\in \mathcal R:z\in \CC\}$, where~$|w_1(z)|>\dots>|w_k(z)|$ for almost all~$z$. The cuts connecting~$\mathcal R_j$ with~$\mathcal R_{j+1}$ are defined as the sets~$\{(z,w_j(z))\in \mathcal R:z\in \CC, \,|w_j(z)|=|w_{j+1}(z)|\}$. The fact that~$\mathcal R$ is a Harnack curve implies that the description through these sheets is particularly easy; If~$(z,w_j)\in \mathcal R_j$, then~$(z,w_j)$ is in a real oval or in a cut if and only if~$z\in \RR$. Indeed, if~$(z,w_j)$ is in a cut so that~$(z,w_{j+1})\in \mathcal R$ and~$|w_j|=|w_{j+1}|$, then, since~$P$ is real,~$(z,w_j), (z,w_{j+1}),(\bar z,\bar w_j),(\bar z,\bar w_{j+1})\in \mathcal R$, and all points are mapped, under the map~$\Log$, to the same point in the amoeba. The~$2$-to-$1$ property of~$\Log$ implies, since~$w_j\neq w_{j+1}$ for almost all~$z$ in the cut, that~$z=\bar z$. Conversely, if~$z\in \RR$, then~$(z,\bar w_j)\in \mathcal R$, and either~$w_j\in \RR$, and~$(z,w_j)$ is in a real oval, or~$w_j\notin \RR$, and~$(z,w_j)$ is in a cut, since~$(z,\bar w_j)\in \mathcal R$, and~$|w_j|=|\bar w_j|$. Furthermore, we know that all angles lie on the non-compact oval, and we know their order along~$A_0$, cf. Figure~\ref{fig:amoeba}. Note that the above implies that if there is a cut with~$z<0$ connecting~$\mathcal R_{j-1}$ with~$\mathcal R_j$, then~$\mathcal R_j$ and~$\mathcal R_{j+1}$ are connected only at points~$(z,w)$ with~$z>0$, and \emph{vice versa}.  

\subsection{The Abel map, theta function and prime form}\label{sec:abel_theta}
The goal of this section is to discuss the theta functions and the prime form on~$\mathcal R$. These objects can be thought of as building blocks for meromorphic functions on~$\mathcal R$ in the same way as linear functions are building blocks of meromorphic functions on the Riemann sphere. Moreover, we will define the Abel map and recall Abel's theorem. Most of the material is standard, and we refer to the first chapter in~\cite{BK11} as a general reference, and as an excellent introduction to Riemann surfaces. We will at some points use the fact that we are dealing with a Harnack curve, which, in particular, is a so-called~$M$-curve with involution~$\sigma(z,w)=(\bar z, \bar w)$, cf.~\cite[Example 10]{BCT22}. The theory of~$M$-curves is less standard, however, the material we need is very well summarized in~\cite{BCT22}. We will assume throughout this section that~$\mathcal R$ is non-singular. In particular, this means that the genus of~$\mathcal R$ is~$g=(k-1)(\ell-1)$.

The first aim is to define the Abel map and recall Abel's theorem. We take the compact ovals~$A_1,\dots, A_g$ as the set of \emph{$A$-cycles}. As the set of \emph{$B$-cycles} we take
\begin{equation}
B_i = \{(z,w)\in \mathcal R:|z|=\e^{r_1^{(i)}}, |w|>\e^{r_2^{(i)}}\},
\end{equation}
for~$i=1,\dots,g$ where~$(r_1^{(i)},r_2^{(i)})\in \RR^2$ is any point in the~$i$th compact component of the complement of~$\mathcal A$. The orientation of the~$B$-cycles is such that the projection of~$B_i$ to the~$z$ variable is oriented in the negative direction. The set of~$A$-cycles and~$B$-cycles form a \emph{canonical basis} of cycles. See Figure~\ref{fig:amoeba} for the image of the~$A$-cycles and~$B$-cycles in the amoeba, see also Figure~\ref{fig:riemann_surface}. We pick a basis of holomorphic differential forms~$\omega_i$,~$i=1,\dots,g$, such that
\begin{equation}\label{eq:basis_differentials}
\int_{A_j} \omega_i = \delta_{ij},
\end{equation}
that is, a \emph{canonical basis of differentials}. We use the notation~$\vec{\omega}=(\omega_1,\dots,\omega_g)$. The fact that we consider an~$M$-curve implies that the integral over~$\omega_i$ along a part of an oval is real, in fact, if~$\Gamma \subset A_j$ is a segment of one of the ovals oriented in positive direction, then
\begin{equation}\label{eq:positive_differentials}
\int_\Gamma\omega_i>0,
\end{equation}
see~\cite[Lemma 14]{BCT22}. We define the \emph{period matrix}~$B \in \CC^{g\times g}$ by
\begin{equation}
B_{ij} = \int_{B_i}\omega_j.
\end{equation}
The imaginary part of the matrix~$B$ is positive definite and since~$\mathcal R$ is an~$M$-curve~$B$ is purely imaginary,~\cite[Lemma 11]{BCT22}.

The \emph{Jacobi variety} is defined as~$J(\mathcal R)=\CC^g/ (\ZZ^g+B\ZZ^g)$, and the \emph{Abel map}~$u:\mathcal R \to J(\mathcal R)$ is given by
\begin{equation}\label{eq:abel_map}
u(q) = \int_{q_0}^q\vec{\omega}\!\! \mod (\ZZ^g+B\ZZ^g),
\end{equation} 
where we pick~$q_0 \in A_0$.

A divisor is a formal linear combination of points in~$\mathcal R$,~$D = \sum_{i}a_iq_i$, with integer coefficients~$a_i$. The \emph{degree} of the divisor is the sum of the coefficients,~$\deg(D)=\sum_i a_i$. We define the \emph{Abel map of divisors} as
\begin{equation}
u(D) = \sum_ia_i u(q_i).
\end{equation}
A divisor~$D=\sum_ia_iq_i-\sum_ib_ip_i$ is called \emph{principal} if there is a meromorphic function~$f$ with zeros precisely at~$q_i$ of order~$a_i$ and poles at~$p_i$ of order~$b_i$. We write~$D=(f)$. Similarly, for a differential form~$\omega$ we write~$(\omega)=D$ where~$D$ is the divisor defined as the linear combination of the zeros and poles of~$\omega$. Recall Abel's theorem: a divisor~$D=\sum_i a_iq_i$ is principal if and only if 
\begin{equation}\label{eq:abels_thm}
\deg(D)=0 \quad \text{and} \quad u(D)=0. 
\end{equation} 
For a differential form~$\omega$ and its divisor~$D=(\omega)$ we instead have 
\begin{equation}\label{eq:abels_thm_diff}
\deg(D)=2g-2 \quad \text{and} \quad u(D)=2\Delta,
\end{equation}
where~$\Delta$ is the \emph{vector of Riemann constants}.
We say that a divisor~$D$ is a \emph{standard divisor} if 
\begin{equation}\label{eq:standard_divisor}
D=\sum_{i=1}^g q_i,
\end{equation}
where~$q_i\in A_i$. 

We proceed to define the theta function. Given the period matrix~$B$, we define the \emph{theta function} by 
\begin{equation}
\theta(z)=\theta(z;B)=\sum_{n\in \ZZ^g}\e^{\i\pi\left(n\cdot B n+2n\cdot z\right)}.
\end{equation}
Since the imaginary part of~$B$ is positive-definite, the series converges uniformly and defines an entire function on~$\CC^g$. The theta function is a quasi-periodic function. Indeed, for~$m,n\in \ZZ^g$
\begin{equation}\label{eq:quasi-periodic}
\theta(z+m+B n)=\e^{-2\i\pi n\cdot (z+\frac{1}{2}B n)}\theta(z).
\end{equation}

Let~$\widetilde{\mathcal R}$ be the universal cover of~$\mathcal R$. For~$e\in \CC^g$ we define the function~$\Theta: \widetilde{\mathcal R} \mapsto \CC$ by
\begin{equation}
\Theta(\tilde q;e)=\theta\left(\int_{\tilde q_0}^{\tilde q}\vec\omega + e\right),
\end{equation}
where~$\tilde q_0$ is an arbitrary lift of~$q_0\in A_0$, chosen in~\eqref{eq:abel_map}, to the universal cover~$\widetilde{\mathcal R}$.\footnote{We use tildes to denote both general points on~$\widetilde{\mathcal R}$ and lifts of points on~$\mathcal R$ given by the same letter without the tilde.} For any~$e\in \CC^g$ such that~$\tilde q\mapsto \Theta(\tilde q;e)$ is not identically zero, there exist~$g$ points~$q_1,\dots,q_g \in \mathcal R$ such that 
\begin{equation}
\Theta(\tilde q_i;e)=0, \quad i=1,\dots,g,
\end{equation}
for any lift~$\tilde q_i$ of~$q_i$, and these points satisfy the equality
\begin{equation}\label{eq:jacobi_inverse}
\sum_{j=1}^g u(q_j)=-e+\Delta,
\end{equation}
in~$J(\mathcal R)$, where~$\Delta$ is, as above, the vector of Riemann constants. Moreover, the divisor~$\sum_{i=1}^gq_i$ is uniquely determined by~\eqref{eq:jacobi_inverse}. The fact that~$\mathcal R$ is an~$M$-curve implies that~$\Delta\in \RR^g+\frac{1}{2}B\mathbf{1}$, where~$\mathbf{1}\in \ZZ^g$ is the vector with only ones,~\cite[Lemma 19]{BCT22}. Moreover, if~$e\in \RR^g$, then the zeros lie on different compact ovals, that is, after possible relabeling,~$q_j\in A_j$. In fact, for~$M$-curves the zero divisor of~$\tilde q\mapsto \Theta(\tilde q;e)$ is a standard divisor if and only if~$e\in \RR^g$.

The theta function can be used to express meromorphic functions on~$\mathcal R$. However, there is actually an even better object for that purpose, namely the \emph{prime form}. The prime form is a differential form defined on~$\widetilde{\mathcal R}\times \widetilde{\mathcal R}$. For a definition we refer the reader to~\cite{BCT22, Mum07b}. For~$\tilde p,\tilde q \in \widetilde{\mathcal R}$ we denote the prime form by~$E(\tilde p,\tilde q)$. The prime form can be used as building blocks of meromorphic functions on~$\mathcal R$. Namely, if~$h$ is a meromorphic function then
\begin{equation}
h(q)=c\prod_{j=1}^n \frac{E(p_j,q)}{E(q_j,q)},
\end{equation}
where~$p_j$ are the zeros and~$q_j$ the poles of~$f$, and~$c$ is some constant. The drawback, compared with the theta function, is that~$E$ is not a function. The quotient
\begin{equation}
\tilde q\mapsto \frac{E(\tilde p_1,\tilde q)}{E(\tilde p_2,\tilde q)},
\end{equation}
however, is a function on~$\widetilde{\mathcal R}$.

We list the properties that will be relevant for us and refer the reader to~\cite[Section 2.5]{BCT22} and~\cite[Section 1 (p.\,\,3.207)]{Mum07b}. 
\begin{fact}\label{fact:prop_prime_form}
Let~$\tilde p,\tilde p_1,\tilde p_2,\tilde q,\tilde q',\tilde q'' \in \widetilde{\mathcal R}$, where~$\tilde q'$ and~$\tilde q''$ are obtained by adding the lift of the cycle~$A_i$ and~$B_i$, respectively, to~$\tilde q$. That is,~$\int_{\tilde q_0}^{\tilde q'}\omega_j=\int_{\tilde q_0}^{\tilde q}\omega_j+\int_{A_i}\omega_j$, and~$\int_{\tilde q_0}^{\tilde q''}\omega_j=\int_{\tilde q_0}^{\tilde q}\omega_j+\int_{B_i}\omega_j$, for all~$j$, where~$\tilde q_0$ is an arbitrary lift of~$q_0\in A_0$ from~\eqref{eq:abel_map}. Then: 
\begin{enumerate}[(1)]
\item~$E(\tilde p,\tilde q)=0$ if and only the projection to~$\mathcal R$ of~$\tilde p$ is equal to the projection of~$\tilde q$, and the zero is simple, \label{eq:prime_form_zero}
\item~$E(\tilde p,\tilde q)=-E(\tilde q,\tilde p)$, \label{eq:prime_form_antisym}
\item~$E(\bar{\tilde p},\bar{\tilde q})=\overline{E(\tilde p,\tilde q)}$, \label{eq:prime_form_conjugation}
\item~$\dfrac{E(\tilde p_1,\tilde q')}{E(\tilde p_2,\tilde q')}=\dfrac{E(\tilde p_1,\tilde q)}{E(\tilde p_2,\tilde q)}$, \label{eq:prime_form_lift_a}
\item~$\dfrac{E(\tilde p_1,\tilde q'')}{E(\tilde p_2,\tilde q'')}=\e^{-2\pi\i\int_{\tilde p_1}^{\tilde p_2}\omega_i}\dfrac{E(\tilde p_1,\tilde q)}{E(\tilde p_2,\tilde q)}$.\label{eq:prime_form_lift_b}
\end{enumerate}
\end{fact}

\section{Main results}\label{sec:main_result}
In this section we present our main results. In Section~\ref{sec:global_limit} we will describe the global structure. In particular, we will define the limiting regions and prove that the global limit is described in terms of the amoeba, proving Theorem~\ref{thm:intro:phase_diagram} along the way. In Section~\ref{sec:local_limit} we will then describe the local limit. The local limit is expressed in terms of the limit of the correlation kernel of Theorem~\ref{thm:bd_thm}. This limit is then used to prove Theorem~\ref{thm:intro:gibbs_limit}.

The assumptions under which these statements are proved are given in Section~\ref{sec:ass}.

\subsection{Assumptions}\label{sec:ass}
We will impose three assumptions on the model. 
\begin{assumption}\label{ass:main_ass}
\begin{enumerate}[(a)]
\item The size of the Aztec diamond is~$k\ell N$ with an integer~$N\geq 1$. \label{ass:size}
\item The edge weights fulfill the inequality~$\beta_i^v<1<\alpha_i^v/\gamma_i^v$ for~$i=1,\dots,\ell$. \label{ass:wh}
\item The associated Harnack curve,~$\mathcal R^\circ$, is a non-singular Harnack curve. In particular this means that the genus of the curve is~$(k-1)(\ell-1)$ and the angles are distinct.\label{ass:non-singularity}
\end{enumerate}
\end{assumption}
These three assumptions will be assumed throughout the paper unless explicitly stated otherwise. Let us comment briefly on these assumptions before we proceed to our results.

Assumption~\eqref{ass:size} is used for convenience. Removing this assumption would mainly make the notation heavier. 

Assumption~\eqref{ass:wh} is equivalent to the existence of the Wiener--Hopf factorization in Theorem~\ref{thm:bd_thm}, which is essential for that theorem. However, by analytic continuation of the parameters in the correlation kernel it is possible to extend the theorem to a larger class of edge weights. This was done, for instance, in~\cite{BD19}. We believe it is possible to extend the result of Theorem~\ref{thm:bd_thm} so that~\eqref{ass:wh} could be removed. However, for the analytic continuation one needs to deform the curves of integration in the correlation kernel. In the form given in Theorem~\ref{thm:bd_thm} it is not possible, at least not when~$k$ is even, to reach all edge weights by deforming the contours. Instead, it is Proposition~\ref{prop:finite_kernel} that should be more suitable for such an analytic continuation. Furthermore, it seems plausible that under such continuation the assumption is no longer necessary in the proof of Lemma~\ref{lem:limit_kernel_rough}.

Assumption~\eqref{ass:non-singularity} consists of two parts. We require the genus to be maximal because we want to avoid working on a singular Riemann surface. If instead the Harnack curve is singular, one can use the amoeba to define the normalization of the curve, see~\cite[Corollary 3]{MR00}. The assumption that the angles are distinct is mainly to simplify some of the arguments, however, at this point we do not claim to know alternative arguments in case the angles are not different. Non-singularity of the Harnack curve holds generically. 

We find it natural to conjecture that all the results in the present paper, after appropriate adjustments, hold without the above assumptions.

\subsection{The global limit}\label{sec:global_limit}
In the asymptotic analysis of the correlation kernel given in Theorem~\ref{thm:bd_thm} we use the method of steep descent. In this section we define an action function that will later be used in Section~\ref{sec:asymptotic}. In the steep descent analysis, the critical points of the action function will play an important role, and it is therefore natural to define the limiting regions of the dimer covering of the Aztec diamond in terms of these critical points, see Definition~\ref{def:regions}. Using this definition and the definition of the action function, we prove that there is a homeomorphism between the rough region and the interior of the amoeba. Moreover, we prove that there is a one-to-one correspondence between the set of smooth regions and the set of compact components of the complement of the amoeba, and a one-to-one correspondence between the set of frozen regions and the unbounded components of the complement of the amoeba.

Given~$x=0,\dots,kN$ and~$y=0,\dots, \ell N$ as in Section~\ref{sec:paths}, we assume that 
\begin{equation}
 \frac{x}{N}\to\frac{k}{2}(\xi+1), \quad \text{and} \quad \frac{y}{N}\to\frac{\ell }{2}(\eta+1),
\end{equation}
as~$N\to \infty$, for some~$(\xi,\eta)\in (-1,1)^2$. We refer to~$(\xi,\eta)$ as \emph{global coordinates}. Recall the definition of the angles,~\eqref{eq:angles_1} and~\eqref{eq:angles_2}.
\begin{definition}\label{def:action_function}
Let~$(\xi,\eta) \in (-1,1)^2$ be the global coordinates of the Aztec diamond, as defined above. For a point~$\tilde q \in \widetilde{\mathcal R}$ in the universal cover of~$\mathcal R$, we denote the projection of~$\tilde q$ to~$\mathcal R$ by~$q=(z,w)$. We define the \emph{action function}~$F:\widetilde {\mathcal R} \to \CC$ as
\begin{equation}\label{eq:def_action_function}
F(\tilde q;\xi,\eta) 
= \frac{k}{2}(1-\xi)\log w-\frac{\ell}{2}(1-\eta) \log z - \log f(\tilde q),
\end{equation}
where 
\begin{equation}\label{eq:def_f}
f(\tilde q)=\frac{\prod_{i=1}^\ell E(\tilde q_{0,i},\tilde q)^k}{\prod_{j=1}^k E(\tilde p_{0,j},\tilde q)^\ell},
\end{equation}
for some arbitrary lift of the angles~$q_{0,i}$ and~$p_{0,j}$ to~$\widetilde{\mathcal R}$, and the branch of the logarithm is the principle one.
\end{definition}

\begin{remark}
The action function~$F$ can also be expressed in terms of theta functions instead of prime forms. Indeed, if~$r\in \CC^g$ with~$\theta(r)=0$, then
\begin{equation}
f(\tilde q)=c\,\frac{\prod_{i=1}^\ell \theta\left(\int_{\tilde q_{0,i}}^{\tilde q}\vec{\omega}+r\right)^k}{\prod_{j=1}^k \theta\left(\int_{\tilde p_{0,j}}^{\tilde q}\vec{\omega}+r\right)^\ell},
\end{equation}
for some nonzero constant~$c$, which would play no role in our analysis below. See~\cite[Chapter~II~\S~$3$]{Mum07a}. Indeed, let~$q_1,\dots,q_{g-1}$ be~$g-1$ generic points in~$\mathcal R$, and set
\begin{equation}
r=-\sum_{j=1}^{g-1} \int_{\tilde q_0}^{\tilde q_j}\vec{\omega}+\Delta,
\end{equation}
where~$q_0$ is the base point chosen in~\eqref{eq:abel_map},~$\tilde q_j$,~$j=0,\dots,g-1$, are arbitrary lifts of~$q_j$ to~$\widetilde{\mathcal R}$, and~$\Delta$ is the vector of Riemann constants as in~\eqref{eq:abels_thm_diff}. It follows from~\eqref{eq:jacobi_inverse} that~$\theta(r)=0$. For~$p,q\in \mathcal R$ with lifts~$\tilde p,\tilde q\in \widetilde{\mathcal R}$ set~$e_p=r-\int_{\tilde q_0}^{\tilde p}\vec{\omega}$. Then~\eqref{eq:jacobi_inverse} implies that the zeros of the function
\begin{equation}
\tilde q\mapsto \theta\left(\int_{\tilde p}^{\tilde q}\vec{\omega}+r\right)=\theta\left(\int_{\tilde q_0}^{\tilde q}\vec{\omega}+e_p\right)
\end{equation} 
are~$\tilde q=\tilde p, \tilde q_1,\dots,\tilde q_{g-1}$. Moreover, if~$p_1,p_2\in \mathcal R$, then, by~\eqref{eq:prime_form_zero},~\eqref{eq:prime_form_lift_a} and~\eqref{eq:prime_form_lift_b} of Fact~\ref{fact:prop_prime_form} and by~\eqref{eq:quasi-periodic}, the function
\begin{equation}
q\mapsto \frac{\theta\left(\int_{p_2}^{q}\vec{\omega}+r\right)}{\theta\left(\int_{p_1}^{q}\vec{\omega}+r\right)}\frac{E(p_1,q)}{E(p_2,q)}
\end{equation}
is meromorphic on~$\mathcal R$ with no zeros or poles, and hence, a nonzero constant.
\end{remark}

In fact, for our purposes it is not the function~$F$ by itself that will be relevant; instead, it is the differential~$\d F$ and the real part~$\re F$ that will be of importance to us. Both these objects are well-defined on~$\mathcal R$, and not just on the universal cover.
\begin{lemma}\label{lem:conjugate_functions}
Let~$F$ be as in Definition~\ref{def:action_function}. The differential form~$\d F$ may be viewed as a meromorphic differential form of~$q=(z,w)\in \mathcal R$, and~$\re F$ may be viewed as a harmonic function of~$q=(z,w)\in \mathcal R$.
%Let~$F$ be as in Definition~\ref{def:action_function}. The differential~$\d F$ defines a well-defined meromorphic differential on~$\mathcal R$, and~$\re F$ defines a well-defined harmonic function on~$\mathcal R$. We will therefore think of them as objects of~$q=(z,w)\in \mathcal R$. 
Moreover, 
\begin{equation}\label{eq:conjugate_functions}
\d F(\bar z,\bar w;\xi,\eta)=\overline{\d F(z,w;\xi,\eta)} \quad \text{and} \quad \re F(\bar z,\bar w;\xi,\eta)=\re F(z,w;\xi,\eta).
\end{equation}
\end{lemma} 
\begin{proof}
For the first part of the statement we show that~$\d F$ and~$\re F$ are independent of the branch of the logarithm or the lift of~$q$. Indeed, changing the branch of the logarithm adds a purely imaginary constant to~$F$. Such a constant does not effect either~$\d F$ or~$\re F$. If we add cycles~$A_{m'}$ and~$B_m$ to~$q$ in the right hand side of~\eqref{eq:def_action_function}, then items~\eqref{eq:prime_form_lift_a} and~\eqref{eq:prime_form_lift_b} in Fact~\ref{fact:prop_prime_form} imply that we add terms of the form 
\begin{equation}
2\pi\i\int_{\tilde p_{0,j}}^{\tilde q_{0,i}}\omega_m,
\end{equation}
to~$F$. These terms are purely imaginary by~\eqref{eq:positive_differentials} and, moreover, they are independent of~$q$. This proves that~$\d F$ and~$\re F$ are well-defined on~$\mathcal R$. 

The relations in~\eqref{eq:conjugate_functions} follow from item~\eqref{eq:prime_form_conjugation} in Fact~\ref{fact:prop_prime_form}.	\end{proof}

\begin{remark}
To be precise, the differential form~$\d F$ is not defined along discontinuities of~$F$ coming from the branch of the logarithm. The statement should be interpreted that~$\d F$ can be extended to a meromorphic differential on~$\mathcal R$. 
\end{remark}

To determine the location of the zeros of~$\d F$ we analyze~$\re F$ at the angles, where~$\re F$ is unbounded. By part~\eqref{eq:prime_form_zero} of Fact~\ref{fact:prop_prime_form} we know the order of zeros and poles of~$f$ at the angles. Indeed,~$f$ behaves as~$z^{-\ell}$ at~$p_{0,j}$ and as~$w^k$ at~$q_{0,i}$, and is bounded at~$p_{\infty,j}$ and~$q_{\infty,i}$. We find that
\begin{align}
\re F(z,w;\xi,\eta)
&= \frac{\ell}{2}(\eta+1)\log|z|+\Ordo(1) \to -\infty, & \text{as} & \quad (z,w)\to p_{0,j}, \label{eq:limit_f_1} \\
\re F(z,w;\xi,\eta)&= -\frac{\ell}{2}(1-\eta)\log|z|+\Ordo(1) \to -\infty, & \text{as} & \quad (z,w)\to p_{\infty,j}, \label{eq:limit_f_2} \\
\re F(z,w;\xi,\eta)  
&= -\frac{k}{2}(1+\xi)\log|w|+\Ordo(1) \to +\infty, & \text{as} & \quad (z,w)\to q_{0,i}, \label{eq:limit_f_3}\\
\re F(z,w;\xi,\eta) & = \frac{k}{2}(1-\xi)\log|w|+\Ordo(1) \to +\infty, & \text{as} & \quad (z,w) \to q_{\infty,i}. \label{eq:limit_f_4}
\end{align}
Using the above asymptotics, a standard argument now gives us the location of all but two zeros of~$\d F$. See Figure~\ref{fig:critical_points} for a schematic picture of the location of the critical points and the behavior of~$\re F$ at the angles.
\begin{lemma}\label{lem:zeros_poles_df}
The differential~$\d F$ defined above is meromorphic with~$2(k+\ell)$ poles and, hence,~$2g-2+2(k+\ell)$ zeros. There are two distinct zeros on each compact oval~$A_i$,~$i=1,\dots,g$, and a zero on each component~$A_{0,i}$, for~$i\in \{1,\dots,2(k+\ell)\}\backslash \{1,\ell+1,k+\ell+1,k+2\ell+1\}$.
\end{lemma}

 \begin{figure}[t]
 \begin{center}
 \begin{tikzpicture}[scale=1]
  \tikzset{->-/.style={decoration={
  markings, mark=at position .5 with {\arrow{stealth}}},postaction={decorate}}}
    \draw (0,0) node {\includegraphics[trim={1cm, 1cm, 1cm, 1cm}, clip, angle=180, scale=.5]{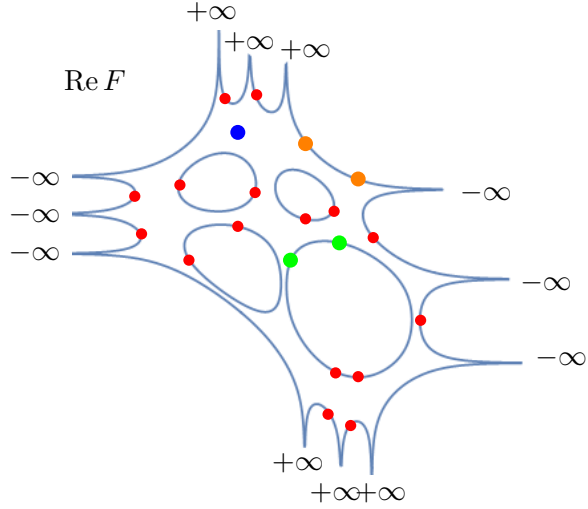}};
	%\re F
	\draw (-2.7,2.3) node {$\re F$};    
    %critical points
    %non-compact oval
	\draw (-.55,2.1) node[circle,fill,red,inner sep=1.5pt]{};
	\draw (-.97,2.05) node[circle,fill,red,inner sep=1.5pt]{};
	
	\draw (-2.17,.75) node[circle,fill,red,inner sep=1.5pt]{};
	\draw (-2.08,.25) node[circle,fill,red,inner sep=1.5pt]{};

	\draw (0.4,-2.15) node[circle,fill,red,inner sep=1.5pt]{};
	\draw (.7,-2.3) node[circle,fill,red,inner sep=1.5pt]{};

	\draw (1.,.2) node[circle,fill,red,inner sep=1.5pt]{};	
	\draw (1.63,-.9) node[circle,fill,red,inner sep=1.5pt]{};

    %compact ovals
	\draw (-1.45,-.1) node[circle,fill,red,inner sep=1.5pt]{};
	\draw (-.8,.35) node[circle,fill,red,inner sep=1.5pt]{};
	\draw (.5,-1.6) node[circle,fill,red,inner sep=1.5pt]{};
	\draw (.8,-1.65) node[circle,fill,red,inner sep=1.5pt]{};	
	\draw (-1.57,.9) node[circle,fill,red,inner sep=1.5pt]{};
	\draw (-.57,.8) node[circle,fill,red,inner sep=1.5pt]{};
	\draw (.1,.45) node[circle,fill,red,inner sep=1.5pt]{};
	\draw (.48,.55) node[circle,fill,red,inner sep=1.5pt]{};
	%special critical points
	%smooth
	\draw (-.1,-.1) node[circle,fill,green,inner sep=2pt]{};
	\draw (.55,.13) node[circle,fill,green,inner sep=2pt]{};
	%rough
	\draw (-.8,1.6) node[circle,fill,blue,inner sep=2pt]{};
	%frozen
	\draw (.1,1.45) node[circle,fill,orange,inner sep=2pt]{};
	\draw (.8,.98) node[circle,fill,orange,inner sep=2pt]{};

	%angles
	\draw (.1,2.7) node {$+\infty$};
	\draw (-.65,2.8) node {$+\infty$};
	\draw (-1.15,3.2) node {$+\infty$};
	
	\draw (-3.5,1) node {$-\infty$};
	\draw (-3.5,.5) node {$-\infty$};
	\draw (-3.5,0) node {$-\infty$};

	\draw (0,-2.8) node {$+\infty$};
	\draw (.5,-3.2) node {$+\infty$};
	\draw (1.1,-3.2) node {$+\infty$};

	\draw (2.5,.8) node {$-\infty$};
	\draw (3.3,-.4) node {$-\infty$};
	\draw (3.5,-1.4) node {$-\infty$};
  \end{tikzpicture}
 \end{center}
  \caption{The picture indicates the behavior of~$\re F$ at the angles, and the locations of the critical points of~$F$, when viewed as a function on the amoeba. The red dots are the critical points that are given in Lemma~\ref{lem:zeros_poles_df}, and the green, blue and orange dots are possible locations of the last two critical points in the smooth, rough and frozen regions, respectively. \label{fig:critical_points}}
\end{figure}

\begin{proof}
Similarly to~\eqref{eq:limit_f_1}-\eqref{eq:limit_f_4}, we see that the differential~$\d F$ is meromorphic with simple poles at the angles. Thus,~$\d F$ has~$2(k+\ell)$ poles. Recall that the degree of~$\d F$ is~$2g-2$, see~\eqref{eq:abels_thm_diff}. Hence,~$\d F$ has~$2g-2+2(k+\ell)$ zeros.

Each compact oval contains two zeros. This follows from a standard argument, see, for instance,~\cite[Proposition 6.4]{DK21}. Indeed, let~$q\in A_i$. Then 
\begin{equation}
\int_{A_i}\d F=F(\tilde q;\xi,\eta)-F(\tilde q';\xi,\eta)+c\i,
\end{equation}
where~$\tilde q$ and~$\tilde q'$ are lifts of~$q$ to the universal cover, and~$c$ is an integer multiple of~$2\pi$ coming from the logarithm. It follows from Lemma~\ref{lem:conjugate_functions}, by taking the real part of the previous equality, that
\begin{equation}
\re \left(\int_{A_i} \d F\right)=0.
\end{equation} 
Since~$\d F$ is real and continuous on~$A_i$ (as~$A_i$ is in the real part of~$\mathcal R$), we conclude that~$\d F$ has at least two distinct zeros on~$A_i$. 

Let us consider~$A_{0,i}$ as in the statement. The imaginary part of~$F$ is constant along~$A_{0,i}$ (cf. Fact~\ref{fact:prop_prime_form}\eqref{eq:prime_form_conjugation}), which means that a critical point of~$\re F$ in~$A_{0,i}$, if it exists, is a zero of~$\d F$. Moreover, by~\eqref{eq:limit_f_1}--\eqref{eq:limit_f_4},~$\re F \to \pm\infty$ at the endpoints, and by definition of~$i$ the sign is the same at both endpoints. This implies that~$\re F$ does have a critical point in~$A_{0,i}$. See Figure~\ref{fig:critical_points}.
\end{proof}

The previous lemma localizes all but two critical points of~$F$. By~\eqref{eq:conjugate_functions} the last two critical points either come as a conjugate pair, or they are both real. If the critical points are real, then they have to lie on the same oval, and if the oval is~$A_0$, then they have to lie on the same component~$A_{0,i}$ for some~$i$. This observation follows by considering the graph of~$\re F$ along the real part of~$\mathcal R$, see Figure~\ref{fig:critical_points}. The location of these two final critical points defines the different limiting regions.
\begin{definition}\label{def:regions}
Let~$(\xi,\eta)\in (-1,1)$ be the global coordinates and let~$\d F$ be as above. We say that~$(\xi,\eta)$ lies: 
\begin{itemize}
\item in the \emph{frozen region} if there are two or three simple zeros of~$\d F$ in~$A_{0,i}$, for some~$i=1,\dots,2(k+\ell)$,
\item in the \emph{smooth disordered region} if there are four simple zeros of~$\d F$ in~$A_i$ for some~$i=1,\dots,g$,
\item in the \emph{rough disordered region} if there is a zero of~$\d F$ in~$\mathcal R_0$. 
\item in the \emph{boundary between the rough and frozen region} if there is a zero of order two or three of~$\d F$ in~$A_0$,
\item in the \emph{boundary between the smooth and rough region} if there is a zero of order two or three of~$\d F$ in~$A_i$ for some~$i=1,\dots,g$.
\end{itemize}
\end{definition}
The boundary of the rough region is known as the \emph{arctic curve}. We will refer to the smooth region corresponding to~$A_i$, as the~$i$th smooth region. Similarly, we will refer to the frozen region corresponding to~$A_{0,i}$ as the~$i$th frozen region. In the special case~$i=1$ we will refer to it as the \emph{north frozen region}, as the \emph{east frozen region} if~$i=\ell+1$, as the \emph{south frozen region} if~$i=\ell+k+1$, and as the \emph{west frozen region} if~$i=2\ell+k+1$. Corollary~\ref{cor:smooth_frozen_regions} below implies that these correspondences are, in fact, one-to-one. We will also see in Remark~\ref{rem:limiting_slopes} that the north region only consists of north dimers in the limit, and similarly for the east, south and west regions. Note that the cyclic order of the regions is reversed compared to the order of the corresponding components of the boundary of the amoeba, see Figures~\ref{fig:intro:tiling_amoeba1} and~\ref{fig:intro:tiling_amoeba2}. This observation is important for Corollaries~\ref{cor:convex} and~\ref{cor:cusps}.

By definition of the rough region there exists a natural map from the rough region to~$\mathcal R_0$, one of the two connected components of~$\mathcal R\backslash \left(\cup_{i=0}^gA_i\right)$, as defined in Section~\ref{sec:harnack}.
\begin{definition}\label{def:omega}
Let~$\mathcal F_R$ be the rough region from Definition~\ref{def:regions}. We define~$\Omega:\mathcal F_R \to \mathcal R_0$ as the unique zero of~$\d F$ in~$\mathcal R_0$.  
\end{definition}
We will see below that the map~$\Omega$ is a homeomorphism. This is in agreement with many models previously studied, see, \emph{e.g.},~\cite{BF14}. Most commonly, the corresponding map is a homeomorphism to the upper half plane. However, in the presence of a smooth region, or if the domain is not simply connected, the map defines a homeomorphism to a multiply connected domain, see~\cite{Ber21, BG19}. 
\begin{theorem}\label{thm:homeomorphism}
The map~$\Omega$ from Definition~\ref{def:omega} is a homeomorphism.
\end{theorem}
\begin{proof}
We prove that~$\Omega$ is a bijection by computing the inverse. 

To find the inverse we extend the definition of~$F$ from~$(\xi,\eta)\in (-1,1)^2$ to~$(\xi,\eta)\in \RR^2$, by the formula~\eqref{eq:def_action_function}. Let~$(\xi,\eta)\in \RR^2$ and~$(z,w)\in \mathcal R_0$. Then~$\d F(z,w;\xi,\eta)=0$ if and only if
\begin{equation}
A(z,w)
\begin{pmatrix}
\xi \\
\eta
\end{pmatrix}
=B(z,w),
\end{equation}
where
\begin{equation}
A(z,w)=
\begin{pmatrix}
\frac{k}{2}\im \frac{\d w}{w} & -\frac{\ell}{2}\im \frac{\d z}{z} \\
\frac{k}{2}\re \frac{\d w}{w} & -\frac{\ell}{2}\re \frac{\d z}{z}
\end{pmatrix}
\end{equation}
and
\begin{equation}
B(z,w)=
\begin{pmatrix}
\im \left(\frac{k}{2}\frac{\d w}{w}-\frac{\ell}{2}\frac{\d z}{z}-\d \log f(z,w)\right)\\
\re \left(\frac{k}{2}\frac{\d w}{w}-\frac{\ell}{2}\frac{\d z}{z}-\d \log f(z,w)\right)
\end{pmatrix}.
\end{equation}

For~$(z,w)\in \mathcal R_0$ the matrix~$A$ is invertible. Indeed,
\begin{equation}\label{eq:non-zero_determinant}
\det A(z,w)=\frac{k\ell}{4}\left|\frac{\d z}{z}\right|^2 \im \left(-\frac{\d w}{\d z}\frac{z}{w}\right),
\end{equation}
and it suffices to show that the right hand side is nonzero. Since~$\frac{\d z}{z}\neq 0$ when~$(z,w)\in \mathcal R_0$, we only need to consider the rightmost factor. By differentiating~$P(z,w)=0$ with respect to~$z$, we obtain~$-\frac{\d w}{\d z}=\frac{P_z}{P_w}$, which means that the rightmost factor is the imaginary part of the \emph{logarithmic Gauss map}. The logarithmic Gauss map~$\gamma:\mathcal R^\circ \mapsto \CC P^1$ is defined by
\begin{equation}
\gamma(z,w)=\left(zP_z(z,w):wP_w(z,w)\right).
\end{equation}
By~\cite[Lemmas 3 and 5]{Mik00} the preimage of~$\RR P^1 \subset \CC P^1$ under the logarithmic Gauss map is the real part of~$\mathcal R^\circ$, that is,~$\gamma^{-1}(\RR P^1)= \cup_{i=0}^g A_i\cap \mathcal R^\circ$. Hence, for~$(z,w)\in \mathcal R_0$, 
\begin{equation}\label{eq:non-zero_determinant_2}
\im \left(-\frac{\d w}{\d z}\frac{z}{w}\right)=\im \left(\frac{zP_z}{wP_w}\right)\neq 0.
\end{equation}
In Appendix~\ref{app:alt_proofs} we include a direct proof that~\eqref{eq:non-zero_determinant} is nonzero without relying on the logarithmic Gauss map.

This tells us that for all~$(z,w)\in \mathcal R_0$ there exists a point~$(\xi,\eta) \in \RR^2$, given by
\begin{equation}
\begin{pmatrix}
\xi \\
\eta
\end{pmatrix}
=A(z,w)^{-1}B(z,w)
\end{equation}
such that~$\d F(z,w;\xi,\eta)=0$. It turns out, as we show below, that if~$(\xi,\eta) \notin (-1,1)^2$, then all zeros of~$\d F$ are real. In particular, this implies that~$A(z,w)^{-1}B(z,w)\in (-1,1)^2$ if~$(z,w)\in \mathcal R_0$. Hence,~$\Omega^{-1}=A^{-1}B$.  

The argument that all zeros of~$\d F$ are real when~$(\xi,\eta)\notin (-1,1)^2$ is similar to the proof of Lemma~\ref{lem:zeros_poles_df}. Assume first that~$\eta>1$ while~$\xi\in (-1,1)$. Then the limit~\eqref{eq:limit_f_2} changes, namely,~$\re F\to \infty$ as~$(z,w)\to p_{\infty,j}$. The other limits,~\eqref{eq:limit_f_1},~\eqref{eq:limit_f_3} and~\eqref{eq:limit_f_4}, however, stay the same. Repeating the argument in Lemma~\ref{lem:zeros_poles_df} tells us that at least all but two critical points are real. But in comparison to the setting of Lemma~\ref{lem:zeros_poles_df}, the final two critical points also have to be real. Namely, one has to be in~$A_{0,1}$ and the other in~$A_{0,k+2\ell+1}$. Hence, there cannot be an additional critical point in~$\mathcal R_0$. If~$\eta=1$,~$\re F$ is bounded at~$p_{\infty,j}$ and~$\d F$ has no poles there. This means that~$\d F$ has~$k$ fewer poles and also~$k$ fewer zeros. All forced zeros of~$\d F$ stay in the same places except for~$(k-1)$ zeros on~$A_{0,i}$'s between~$p_{\infty,j}$'s.  As any potential zeros of~$\d F$ come with its conjugate, we conclude that there is no zero of~$\d F$ in~$\mathcal R_0$. The other situations work similarly, and we omit their detailed descriptions.

The continuity of~$\Omega$ follows from the continuity of zeros of analytic functions. The continuity of~$\Omega^{-1}$ is clear from the explicit expression~$\Omega^{-1}=A^{-1}B$.
\end{proof}

\begin{remark}\label{rem:burgers}
As previously mentioned, similar maps have appeared numerous times before, see, \emph{e.g.},~\cite{Ber21, BF14, Nic22}. These maps are expected to satisfy (a version of) the complex Burgers equation, see~\cite[Theorem 1]{KO07}. This is also the case in our setting. Indeed, with the coordinates~$u=-\frac{\xi+1}{2\ell}$ and~$v=-\frac{\eta+1}{2k}$, we let~$(z(u,v),w(u,v))=\Omega(u,v)=\Omega(\xi,\eta)$. The functions~$z$ and~$w$ satisfy the equation
\begin{equation}
\frac{z_u}{z}+\frac{w_v}{w}=0.
\end{equation}
We include a proof in Appendix~\ref{app:burgers}. The variable change~$(\xi,\eta)\mapsto (u,v)$ is natural from the perspective of the fundamental domain, see the discussion leading up to~\eqref{eq:slope_new_coord}.
\end{remark}

The map~$\Omega$ extends nicely up to the boundary, as we will see below. However, it turns out to be more appropriate, for certain aspects, to view the homeomorphism as a map to the amoeba instead of the spectral curve, that is, to consider the composition of the map~$\Omega$ with the~$\Log$ map defined by~\eqref{eq:log_map}. Recall that the amoeba~$\mathcal A$ is a closed set. 
\begin{theorem}\label{thm:arctic_curves}
Let~$\Omega$ be as in Definition~\ref{def:omega} and let~$\Log$ be as in~\eqref{eq:log_map}. The map~$\Log \circ \, \Omega$ extends to a homeomorphism~$\Log \circ \, \Omega:\overline{\mathcal F_R}\cap(-1,1)^2 \to \mathcal A$ from the closure of the rough region to the amoeba. 
\end{theorem} 
\begin{proof}
Let us first prove that there is a homeomorphism between the boundary of the amoeba and the boundary of the rough region. The proof is similar to the proof of Theorem~\ref{thm:homeomorphism}. Let~$(r_1(t),r_2(t))$ be a parameterization of one of the components of~$\partial \mathcal A$ and let~$(z(t),w(t))=\left(\left.\Log\right|_{\mathcal R^\circ}\right)^{-1}(r_1(t),r_2(t))$. Recall (see Section~\ref{sec:harnack}) that the map~$\Log|_{\mathcal R^\circ}$, the restriction of~$\Log$ to~$\mathcal R^\circ$, is a bijection on the boundary of~$\mathcal A$. Since~$\d F$ is real on the real part of~$\mathcal R$, cf.~\eqref{eq:conjugate_functions}, we get that~$(\xi,\eta)\in \partial \mathcal F_R$ if and only if~$t\mapsto \re F(z(t),w(t);\xi,\eta)$ has a critical point of order at least two. This happens if and only if
\begin{equation}\label{eq:equation_boundary}
\tilde A(t)
\begin{pmatrix}
\xi \\
\eta
\end{pmatrix}
=\tilde B(t),
\end{equation}
where (with~$(r_1,r_2)=(\log|z|,\log|w|)$)
\begin{equation}
\tilde A(t)=
\begin{pmatrix}
\frac{k}{2}r_2'(t) & -\frac{\ell}{2}r_1'(t) \\
\frac{k}{2}r_2''(t) & -\frac{\ell}{2}r_1''(t)
\end{pmatrix}
\end{equation}
and
\begin{equation}
\tilde B(t)=
\begin{pmatrix}
\frac{k}{2}r_2'(t) - \frac{\ell}{2}r_1'(t)-\frac{\d }{\d t}\left(\log |f(z(t),w(t))|\right)\\
\frac{k}{2}r_2''(t) - \frac{\ell}{2}r_1''(t)-\frac{\d^2 }{\d t^2}\left(\log |f(z(t),w(t))|\right)
\end{pmatrix}.
\end{equation}
We claim that this system of equations has a unique solution~$(\xi,\eta)\in \RR^2$. We assume this to be true for the moment, and postpone the proof slightly.

It follows from the argument given in the proof of Theorem~\ref{thm:homeomorphism} that 
\begin{equation}
\tilde A(t)^{-1}\tilde B(t) \in (-1,1)^2.
\end{equation}
Indeed, we saw that if~$(\xi,\eta)\in \RR^2\backslash (-1,1)^2$, then we can determine the location of all zeros of~$\d F$, in particular, we saw that all zeros are real and simple. Hence,~\eqref{eq:equation_boundary} has no solution in~$\RR^2\backslash(-1,1)^2$, and the map~$(\xi,\eta)\mapsto (r_1,r_2)$ is bijective with inverse 
\begin{equation}
\begin{pmatrix}
\xi \\
\eta
\end{pmatrix}
=\tilde A(t)^{-1}\tilde B(t).
\end{equation}
By continuity of zeros of analytic functions it follows that~$(\xi,\eta)\mapsto (r_1,r_2)$ is continuous, and so is the inverse, which can be seen from the explicit expression.

We continue by establishing the postponed claim. The determinant
\begin{equation}\label{eq:det_boundary}
\det \tilde A=\frac{k\ell}{4}(r_1'r_2''-r_2'r_1'')=\frac{k\ell}{4}(r_1')^2\left(\frac{r_1'}{r_2'}\right)'
\end{equation}
is identically zero in some interval~$t\in I$ if and only if~$t\mapsto (r_1(t),r_2(t))$ is a line segment, which we will see cannot happen. Assume, for instance, that~$r_2(t)=cr_1(t)+d$, for some constants~$c$ and~$d$ and for~$t \in I$. Mapping this back to~$\mathcal R$ means that~$w=\pm \e^d(\pm z)^c$, for one choice of the signs depending on the connected component. The map~$z\mapsto P(z,\pm \e^d(\pm z)^c)$ is analytic and zero for all~$z(t)$,~$t\in I$. Hence, it is identical zero. In particular,~$(z,\pm \e^d(\pm z)^c)\in \mathcal R$ for all~$z>0$, and the entire line~$r_2=cr_1+d$ is a subset of~$\partial \mathcal A$, which is not the case. If~\eqref{eq:det_boundary} is zero at an isolated point~$t_0$, then, by taking the limit~$t\to t_0$ in~\eqref{eq:equation_boundary} we obtain a solution of the equation for~$t=t_0$, and therefore a whole line of solutions. However, as noted above,~\eqref{eq:equation_boundary} has no solutions in~$\RR^2\backslash (-1,1)^2$, and we conclude that~\eqref{eq:det_boundary} has no zeros.

The continuity of zeros of analytic functions allows us to establish that the mapping~$(\xi,\eta)\mapsto (r_1,r_2)$ is the limit of~$\Log \circ \, \Omega$ as it approaches the boundary. We denote also the extended map by~$\Log \circ \, \Omega$. Furthermore, since~$\d F$ is both analytic and real on the boundary of~$\mathcal R_0$, that is, the real part of~$\mathcal R$, it is clear that (two conjugate simple zeros of~$\d F$ are colliding)
\begin{equation}
A(z(r_1,r_2),w(r_1,r_2))^{-1}B(z(r_1,r_2),w(r_1,r_2))\to \tilde A(t_0)^{-1}\tilde B(t_0), 
\end{equation}
as~$(r_1,r_2)\to (r_1(t_0),r_2(t_0)) \in \partial \mathcal A$, where~$A$ and~$B$ are as in the proof of Theorem~\ref{thm:homeomorphism} and $(z(r_1,r_2),w(r_1,r_2))$ is the unique preimage in~$\mathcal R_0$ of~$(r_1,r_2)$ under~$\Log$. Hence, the inverse of the extended function~$\Log \circ \, \Omega$ is continuous.
\end{proof}

The previous theorem has several corollaries describing the geometry of the global limit of the Aztec diamond dimers.

\begin{corollary}\label{cor:smooth_frozen_regions}
There is a one-to-one correspondence between the smooth regions and the compact components of the complement of the amoeba. Similarly, there is a one-to-one correspondence between the frozen regions and the unbounded components of the amoeba.
\end{corollary}
\begin{remark}
The proof of Theorem~\ref{thm:arctic_curves} also provides us with a parameterization of the arctic curve in terms of the boundary of the amoeba. By differentiating~\eqref{eq:equation_boundary} we get that~$(\xi'(t),\eta'(t))=(0,0)$ if and only if~$F$ has a critical point of order three. 
These points are expected to be the locations where the arctic curve develops cusps, thus, exhibiting Pearcey behavior, see~\cite{DK21, OR07}.
\end{remark}
The following corollary follows from the explicit parameterization mentioned in the previous remark.
\begin{corollary}\label{cor:arctic_curve_slopes}
Let~$(u,v)$ be the coordinates of Remark~\ref{rem:burgers}. If the arctic curve is smooth at~$(u,v)$, then the tangent line at that point is parallel to the tangent line of the amoeba at~$\Log \circ \,\Omega(u,v)$. 
\end{corollary}
\begin{remark}
The arctic curve is not smooth at cusps. By interpreting the tangent line at these points as the continuation of the tangent lines approaching them, the statement holds true there as well.
\end{remark}
\begin{proof}
From the proof of Theorem~\ref{thm:arctic_curves} we know that given a parameterization~$(r_1(t),r_2(t))$ of one of the components of~$\partial \mathcal A$, a parameterization of the corresponding part of the arctic curve is given by
\begin{equation}\label{eq:arctic_curve_new_coordinates}
\begin{pmatrix}
-(2\ell u(t)+1) \\
-(2kv(t)+1)
\end{pmatrix}
=\tilde A(t)^{-1}\tilde B(t).
\end{equation}
Differentiating this relation yields
\begin{equation}\label{eq:arctic_curve_slope}
\begin{pmatrix}
u'(t) \\
v'(t)
\end{pmatrix}
=-
\begin{psmallmatrix}
\frac{1}{2\ell} & 0 \\
0 & \frac{1}{2k}
\end{psmallmatrix}
\tilde A(t)^{-1}\left(\tilde A'(t)\tilde A(t)^{-1}\tilde B(t)-\tilde  B'(t)\right).
\end{equation}
The proof now follows by simplifying the right hand side.

Let~$b(t)$ be such that~$\tilde B(t)=(b(t),b'(t))^T$. Using that
\begin{equation}
\tilde A'\tilde A^{-1}=
\begin{pmatrix}
0 & 1 \\
\frac{r_1'''r_2''-r_1''r_2'''}{r_1'r_2''-r_1''r_2'} & \frac{r_1'r_2'''-r_1'''r_2'}{r_1'r_2''-r_1''r_2'}
\end{pmatrix},
\end{equation}
we obtain that
\begin{equation}\label{eq:arctic_curve_slope}
\begin{pmatrix}
u' \\
v'
\end{pmatrix}
=-\frac{1}{k\ell}(d c^2b+c'b'+cb'')
\begin{pmatrix}
-r_1'' & r_1' \\
-r_2'' & r_2'
\end{pmatrix}
\begin{pmatrix}
0 \\
1
\end{pmatrix}
=-\frac{1}{k\ell}(d c^2b+c'b'+cb'')
\begin{pmatrix}
r_1' \\
r_2'
\end{pmatrix},
\end{equation}
where
\begin{equation}
c=-\frac{1}{r_1'r_2''-r_1''r_2'}, \quad \text{and} \quad d=r_1'''r_2''-r_1''r_2'''.
\end{equation}
The conclusion of the statement follows since~$(r_1',r_2')$ is the tangent vector of the amoeba.
\end{proof}

 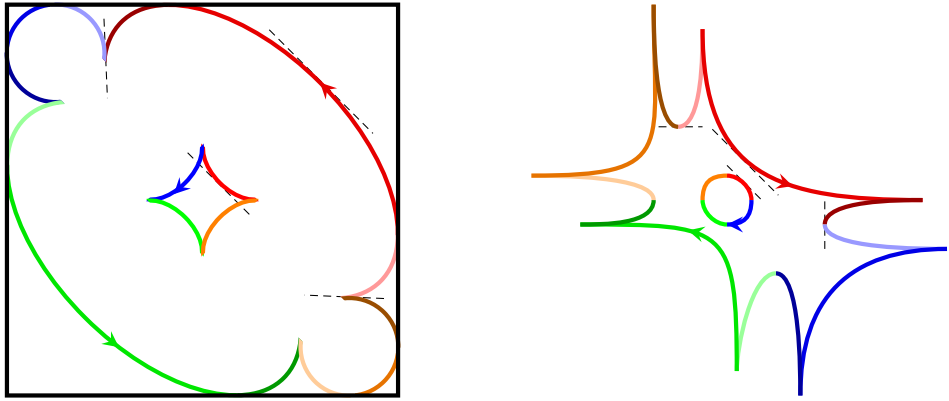
\begin{figure}[t]
  \vspace{-1.5cm}
 \begin{center}
\begin{subfigure}[c]{0.45\textwidth}
\centering
 \begin{tikzpicture}[scale=.65, rotate=-45]
     \tikzset{->-/.style={decoration={
  markings, mark=at position .8 with {\arrow{stealth}}},postaction={decorate}}}
   \tikzset{->--/.style={decoration={
  markings, mark=at position .25 with {\arrow{stealth}}},postaction={decorate}}}
      \tikzset{-<-/.style={decoration={
  markings, mark=at position .0 with {\arrow{stealth[reversed]}}},postaction={decorate}}}
 \colorlet{color1}{red!40}
 \colorlet{color2}{red!90!black}
 \colorlet{color3}{red!60!black}
 \colorlet{color4}{blue!40}
 \colorlet{color5}{blue!90!black}
 \colorlet{color6}{blue!60!black}
 \colorlet{color7}{green!40}
 \colorlet{color8}{green!90!black}
 \colorlet{color9}{green!60!black}
 \colorlet{color10}{orange!40}
 \colorlet{color11}{orange!90!black}
 \colorlet{color12}{orange!60!black}
 \colorlet{color13}{red}
 \colorlet{color14}{blue}
 \colorlet{color15}{green}
 \colorlet{color16}{orange}
 %slopes
    \draw [densely dashed] (4.05,1.2)--(2.85,.1);
    \draw [densely dashed] (-1.5,3.43)--(1.5,3.43);
    \draw [densely dashed] (-4.05,1.2)--(-2.85,.1);
    \draw [densely dashed] (-.9,.47)--(.9,.47);    
%Arctic curve
%top part
\begin{scope}
    \clip (3.3,.6) rectangle (5,4);
    \draw[ultra thick, color1] (0,1.4) ellipse (3.75 and 2);
\end{scope} 
\begin{scope}
    \clip (-3.3,.6) rectangle (3.3,4);
    \draw[ultra thick, color2, ->--] (0,1.4) ellipse (3.75 and 2);
\end{scope} 
\begin{scope}
    \clip (-5,.6) rectangle (-3.3,4);
    \draw[ultra thick, color3] (0,1.4) ellipse (3.75 and 2);
\end{scope}
%left part
\begin{scope}
    \clip (-5,.6) rectangle (-3,2);
    \draw[ultra thick, color4] (-4.25,0) ellipse (1 and 1);
\end{scope} 
\begin{scope}
    \clip (-6,-.6) rectangle (-5,.6);
    \draw[ultra thick, color5] (-4.25,0) ellipse (1 and 1);
\end{scope} 
\begin{scope}
    \clip (-5,-.6) rectangle (-3,-2);
    \draw[ultra thick, color6] (-4.25,0) ellipse (1 and 1);
\end{scope} 
%bottom part
\begin{scope}
    \clip (-5,-.6) rectangle (-3.3,-4);
    \draw[ultra thick, color7] (0,-1.4) ellipse (3.75 and 2);
\end{scope}
\begin{scope}
    \clip (-3.3,-.6) rectangle (3.3,-4);
    \draw[ultra thick, color8, ->-] (0,-1.4) ellipse (3.75 and 2);
\end{scope} 
\begin{scope}
    \clip (3.3,-.6) rectangle (5,-4);
    \draw[ultra thick, color9] (0,-1.4) ellipse (3.75 and 2);
\end{scope} 
%right part
\begin{scope}
    \clip (5,-.6) rectangle (3,-2);
    \draw[ultra thick, color10] (4.25,0) ellipse (1 and 1);
\end{scope} 
\begin{scope}
    \clip (6,-.6) rectangle (5,.6);
    \draw[ultra thick, color11] (4.25,0) ellipse (1 and 1);
\end{scope} 
\begin{scope}
    \clip (5,.6) rectangle (3,2);
    \draw[ultra thick, color12] (4.25,0) ellipse (1 and 1);
\end{scope}
%smooth
\begin{scope}
    \clip (-2,0) rectangle (3,.8);
    \draw[ultra thick, color13] (0,1) ellipse (.85 and .5);
\end{scope} 
\begin{scope}
    \clip (-.8,-2) rectangle (0,2);
    \draw[ultra thick, color14,-<-] (-1,0) ellipse (.5 and .85);
\end{scope} 
\begin{scope}
    \clip (-2,0) rectangle (3,-.8);
    \draw[ultra thick, color15] (0,-1) ellipse (.85 and .5);
\end{scope} 
\begin{scope}
    \clip (.8,-2) rectangle (0,2);
    \draw[ultra thick, color16] (1,0) ellipse (.5 and .85);
\end{scope} 
%Boundary
    \draw[ultra thick, rotate=45](-4,-4) rectangle (4,4);
 \end{tikzpicture}
 \end{subfigure}
\begin{subfigure}[c]{0.45\textwidth}
\centering
 \begin{tikzpicture}[scale=.65]
    \tikzset{->-/.style={decoration={
  markings, mark=at position .6 with {\arrow{stealth}}},postaction={decorate}}}
      \tikzset{-->/.style={decoration={
  markings, mark=at position 1 with {\arrow{stealth}}},postaction={decorate}}}
 \colorlet{color1}{red!40}
 \colorlet{color2}{red!90!black}
 \colorlet{color3}{red!60!black}
 \colorlet{color4}{blue!40}
 \colorlet{color5}{blue!90!black}
 \colorlet{color6}{blue!60!black}
 \colorlet{color7}{green!40}
 \colorlet{color8}{green!90!black}
 \colorlet{color9}{green!60!black}
 \colorlet{color10}{orange!40}
 \colorlet{color11}{orange!90!black}
 \colorlet{color12}{orange!60!black}
 \colorlet{color13}{red}
 \colorlet{color14}{blue}
 \colorlet{color15}{green}
 \colorlet{color16}{orange}
%slopes
    \draw [densely dashed] (-1.4,1.5)--(-.5,1.5);
    \draw [densely dashed] (-.3,1.45)--(1.05,.1);
    \draw [densely dashed] (2,-1)--(2,0);
    \draw [densely dashed] (0,.71)--(.71,0);    
%top part
\draw [ultra thick, color1]   (-1,1.5) to[out=0,in=-90, distance=.5cm] (-.5,3.5);
\draw [ultra thick, color2,->-]   (-.5,3.5) to[out=-90,in=180, distance=3cm] (4,0);
\draw [ultra thick, color3]   (4,0) to[out=180,in=90, distance=.5cm] (2,-.5);
%left part
\draw [ultra thick, color4]   (2,-.5) to[out=-90,in=180, distance=.5cm] (4.5,-1);
\draw [ultra thick, color5]   (4.5,-1) to[out=180,in=90, distance=2cm] (1.5,-4);
\draw [ultra thick, color6]   (1.5,-4) to[out=90,in=0, distance=.5cm] (1,-1.5);
%bottom part
\draw [ultra thick, color7]   (1,-1.5) to[out=180,in=90, distance=.5cm] (0.2,-3.5);
\draw [ultra thick, color8,->-]   (.2,-3.5) to[out=90,in=0, distance=3cm] (-3,-.5);
\draw [ultra thick, color9]   (-3,-.5) to[out=0,in=-90, distance=.5cm] (-1.5,0);
%right part
\draw [ultra thick, color10]   (-1.5,0) to[out=90,in=0, distance=.5cm] (-4,.5);
\draw [ultra thick, color11]   (-4,.5) to[out=0,in=-90, distance=3cm] (-1.5,4);
\draw [ultra thick, color12]   (-1.5,4) to[out=-90,in=180, distance=.5cm] (-1,1.5);
%smooth
\draw [ultra thick, color13]   (0,.5) to[out=0,in=90, distance=.25cm] (.5,0);
\draw [ultra thick, color14,-->]   (.5,0) to[out=-90,in=0, distance=.35cm] (0,-.5);
\draw [ultra thick, color15]   (0,-.5) to[out=180,in=-90, distance=.25cm] (-.5,0);
\draw [ultra thick, color16]   (-.5,0) to[out=90,in=180, distance=.35cm] (0,.5);
 \end{tikzpicture}
  \end{subfigure}
 \end{center}
 \vspace{-1.2cm}
\caption{A schematic illustration depicting the arctic curve and the boundary of the amoeba. The coloring represents the homeomorphism from Theorem~\ref{thm:homeomorphism}. Note that the slopes at the cusps in this illustration are depicted as parallel to the boundary, although we do not possess precise information regarding their orientation.
\label{fig:arctic_cruce_amoeba}}
\end{figure}

The homeomorphism~$\Log \circ \, \Omega$ reverses the orientation of the boundary, that is, if the arctic curve is oriented so that the rough region lies to the left, then under the inherited orientation, the interior of the amoeba lies to the right. See Figures~\ref{fig:intro:tiling_amoeba1},~\ref{fig:intro:tiling_amoeba2} and~\ref{fig:arctic_cruce_amoeba} and the discussion following Definition~\ref{def:regions}. Since the map preserves the slope of the tangent line along the boundary, this means that if the tangent line lies to, say, the right of the boundary of the amoeba, it lies to the left of the boundary of the rough region. Since the components of the complement of the amoeba are convex, see Section~\ref{sec:harnack}, we conclude that the rough region is locally convex.
\begin{corollary}\label{cor:convex}
The rough region is locally convex at all smooth points of the arctic curve.
\end{corollary}

The fact that the homeomorphism preserves the slope of the tangent line while reversing the orientation has additional implications.
\begin{corollary}\label{cor:cusps}
The arctic curve has four cusps at each smooth region, and one cusp at each frozen region, except the north, east, south and west frozen regions.
\end{corollary}
\begin{proof}
A normal vector of the image of~$A_i$,~$i=1,\dots,g$, in the amoeba, rotates by~$2\pi$, while going around the curve once, recall the orientation which is indicated in Figure~\ref{fig:amoeba}. By Corollary~\ref{cor:arctic_curve_slopes} the change of direction of a normal vector of the arctic curve bounding the~$i$th smooth region along all smooth points is also~$2\pi$. However, since the orientation of the arctic curve is reversed under the map~$\Log \circ \, \Omega$, the total change accumulated by going around a component of the arctic curve is~$-2\pi$. This means that the change coming from the singular points is~$-4\pi$, and at each singular point it can only change by~$\pi$. We conclude that there are exactly four cusps. See Figure~\ref{fig:amoeba_curves_rough}.

Similarly, when moving from one tentacle of the amoeba to a consecutive tentacle of the same type, the normal vector undergoes a change of~$-\pi$. For the corresponding frozen region, the change of the normal vector is~$0$, as it remains orthogonal to the boundary of the Aztec diamond both at the start and end points. Therefore, the change of the normal at the possible singular points is~$-\pi$, which proves that there is exactly one cusp.
\end{proof}

\subsection{Local statistics and the slope of the height function}\label{sec:local_limit}
In the previous section we considered the global regions. In this section we will focus on local statistics. In particular, we state our main local result Theorem~\ref{thm:intro:gibbs_limit}. We first state it as the limit of the correlation kernel given in Theorem~\ref{thm:bd_thm}. Using Theorem~\ref{thm:inverse_kasteleyn} we then deduce that the probability measure~\eqref{eq:measure_dimer} converges to the ergodic translation-invariant Gibbs measures. We also compute the limit shape of the Aztec diamond and prove that the slope is equal to the slope of the limiting Gibbs measure.
 
Given the global coordinates~$(\xi,\eta)\in (-1,1)^2$, we introduce the \emph{local coordinates}~$(\kappa,\zeta),(\kappa',\zeta') \in \ZZ^2$. We assume that~$(-1,1)^2\ni (\xi_N,\eta_N)\to (\xi,\eta)$ as~$N\to \infty$, and that~$x,x'=0,\dots,kN$ and~$y,y'=0,\dots,\ell N$ where 
\begin{equation}\label{eq:local_coordinates_x}
 x=\frac{kN}{2}(\xi_N+1)+\kappa, \quad x'=\frac{kN}{2}(\xi_N+1)+\kappa',
\end{equation}
and
\begin{equation}\label{eq:local_coordinates_y}
 y=\frac{\ell N}{2}(\eta_N+1)+\zeta, \quad y'=\frac{\ell N}{2}(\eta_N+1)+\zeta'.
\end{equation}
We emphasize that~$x$ and~$y$ depend on~$N$ while~$\xi$,~$\eta$,~$\kappa$ and~$\zeta$ do not.

 \begin{figure}[t]
 \begin{center}
 \begin{tikzpicture}[scale=1]
  \tikzset{->-/.style={decoration={
  markings, mark=at position .5 with {\arrow{stealth}}},postaction={decorate}}}
  \tikzset{-->-/.style={decoration={
  markings, mark=at position .7 with {\arrow{stealth}}},postaction={decorate}}}
    \draw (0,0) node {\includegraphics[trim={1cm, 1cm, 1cm, 1cm}, clip, angle=180, scale=.5]{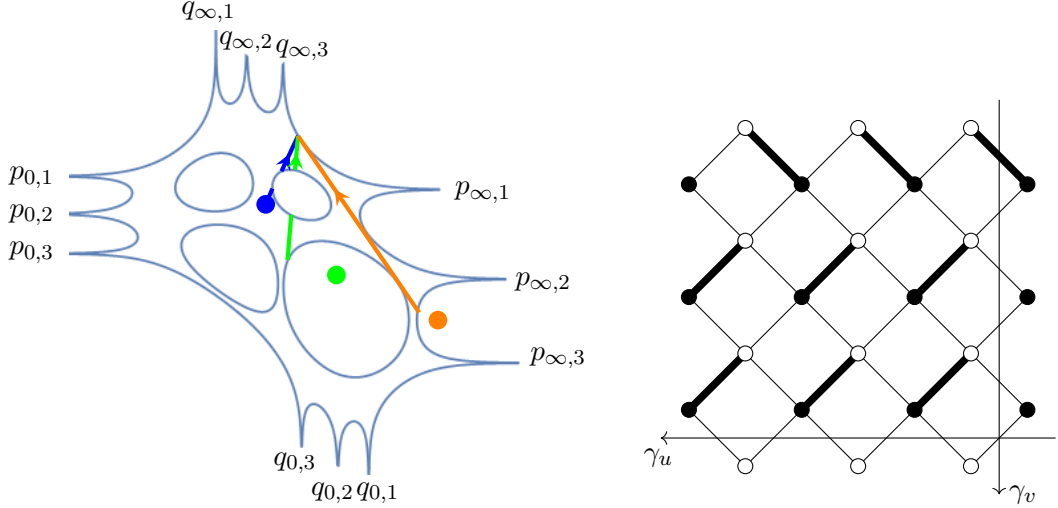}};
   %\gamma_{\xi,\eta}
   %Rough
   \draw[blue,ultra thick](-.39,.64)--(-.3,.85);
   \draw[->-,blue,ultra thick](-.18,1.08)--(0.04,1.55);
   %Smooth
   \draw[green,ultra thick](-.1,-.1)--(-.05,.5);
   \draw[->-,green,ultra thick](.0,1.1)--(0.04,1.55);   
   %Frozen
%   \draw[->-,orange] (0.85,.55)--(0.04,1.55);
   \draw[-->-,orange,ultra thick] (1.65,-.8)--(0.04,1.55);
%   %Curve at angle 
%   \draw[-->-,red,ultra thick] (.9,.6)--(.9,.95);
%   \draw[-->-,red,ultra thick] (1.8,-.5)--(1.65,-.2);
   
   % Points in the amoeba
   % Rough
	\draw (-.39,.64) node[blue,circle,fill,inner sep=2.5pt]{};
	%smooth
	\draw (.55,-.3) node[circle,fill,green,inner sep=2.5pt]{};
   % Frozen
   	\draw (1.9,-.9) node[circle,fill,orange,inner sep=2.5pt]{};

	%angles
	\draw (.1,2.7) node {$q_{\infty,3}$};
	\draw (-.65,2.8) node {$q_{\infty,2}$};
	\draw (-1.15,3.2) node {$q_{\infty,1}$};
	
	\draw (-3.5,1) node {$p_{0,1}$};
	\draw (-3.5,.5) node {$p_{0,2}$};
	\draw (-3.5,0) node {$p_{0,3}$};

	\draw (0,-2.8) node {$q_{0,3}$};
	\draw (.5,-3.2) node {$q_{0,2}$};
	\draw (1.1,-3.2) node {$q_{0,1}$};

	\draw (2.5,.8) node {$p_{\infty,1}$};
	\draw (3.3,-.4) node {$p_{\infty,2}$};
	\draw (3.5,-1.4) node {$p_{\infty,3}$};
  \end{tikzpicture}
  \quad
  \begin{tikzpicture}[scale=1.5]
% West dominos
\foreach \y in {1,2}
\foreach \x in {2,3,4}
{ 
\draw[line width = 1mm] (\x-1,\y-1)--(\x-.5,\y-.5);
}

% North dominos
\foreach \y in {3}
\foreach \x in {2,3,4}
{
\draw[line width = 1mm] (\x,\y-1)--(\x-.5,\y-.5);
}

% black points
\foreach \x in {1,...,4}
{\foreach \y in {0,...,2}
{\draw (\x,\y) node[circle,draw=black,fill=black,inner sep=2pt]{};
}
}

% white points and lines
\foreach \x in {2,...,4}
{\foreach \y in {1,2}
{\draw (\x-1,\y-1)--(\x,\y);
\draw (\x-1,\y)--(\x,\y-1);
\draw (\x-.5,\y-.5) node[circle,draw=black,fill=white,inner sep=2pt]{};
}
\foreach \y in {0}
{\draw (\x-.5,\y-.5)--(\x,\y);
\draw (\x-1,\y)--(\x-.5,\y-.5);
\draw (\x-.5,\y-.5) node[circle,draw=black,fill=white,inner sep=2pt]{};
}
\foreach \y in {3}
{\draw (\x-1,\y-1)--(\x-.5,\y-.5);
\draw (\x-.5,\y-.5)--(\x,\y-1);
\draw (\x-.5,\y-.5) node[circle,draw=black,fill=white,inner sep=2pt]{};
}
}

 % loops
 \draw [<-](.75,-.25)--(4.25,-.25);
 \draw (.73,-.25) node[below] {$\gamma_u$};
  \draw [<-](3.75,-.75)--(3.75,2.75);
 \draw (3.75,-.75) node[right] {$\gamma_v$};
 \end{tikzpicture}
 \end{center}
  \caption{Left picture: An example of the curves and points from Definitions~\ref{def:curve_integration} and~\ref{def:point_integration}. The curves and points correspond to the coordinates~$(\xi,\eta)$ in different regions, the rough (blue), smooth (green) and frozen region (orange). Adding a segment along the boundary of the amoeba to any of the curves does not change the integral, since we integrate along that segment in both directions. Right picture: The fundamental domain with the limiting dimer configuration in the frozen region corresponds to the orange curve in the left picture. The slope in this example is~$(0,-2)$. \label{fig:amoeba_curves}}
\end{figure}

In Theorem~\ref{thm:local_limit} below, we will describe the limiting kernel both as an integral over contours in~$\mathcal R$ and as a double integral over a torus. The curves of integration and the radius of the circles depend on~$(\xi,\eta)$, as described in the following definitions. Recall the definitions in Section~\ref{sec:harnack} and~$\Omega$ from Definition~\ref{def:omega}.
\begin{definition}\label{def:curve_integration}
Given~$(\xi,\eta)\in(-1,1)^2$, we define a curve~$\gamma_{\xi,\eta}$ in~$\mathcal R$. The curve is, up to orientation, invariant under complex conjugation. By the invariance it is enough to specify the part of the curve that lies in~$\mathcal R_0$, which is done as follows. If~$(\xi,\eta)$ is in the~$i$th frozen region, then we take a curve in~$\mathcal R_0$ going from~$A_{0,i}$ to~$A_{0,1}$, if~$(\xi,\eta)$ is in the~$i$th smooth region, then we take a curve going from~$A_i$ to~$A_{0,1}$, and if~$(\xi,\eta)$ is in the rough region, then we take a curve going from~$\Omega(\xi,\eta)$ to~$A_{0,1}$. See Figure~\ref{fig:amoeba_curves} for a sketch of the projection of the curves to the amoeba. 
\end{definition}

\begin{definition}\label{def:point_integration}
Given~$(\xi,\eta)\in (-1,1)^2$, we define~$(r_1,r_2)=(r_1(\xi,\eta),r_2(\xi,\eta))\in \RR^2$ as follows. If~$(\xi,\eta)$ is in the~$i$th frozen region, then we take~$(r_1,r_2)$ as a point in the connected component of the complement of the amoeba~$\mathcal A$ with boundary~$\Log A_{0,i}$. If~$(\xi,\eta)$ is in the~$i$th smooth region, then we take~$(r_1,r_2)$ as a point in the connected component of the complement of~$\mathcal A$ with boundary~$\Log A_{i}$. Finally, if~$(\xi,\eta)$ is in the rough region, then we take~$(r_1,r_2)=\Log(\Omega(\xi,\eta))$. See Figure~\ref{fig:amoeba_curves}.
\end{definition}

Note that if~$(\xi,\eta)$ is in the rough region, then~$(r_1,r_2)$ is the image under~$\Log$ of the endpoints of~$\gamma_{\xi,\eta}$, and if~$(\xi,\eta)$ is in the frozen or smooth regions, then~$(r_1,r_2)$ is in the interior of the component of the complement of the amoeba, with boundary which intersects the (image of the) curve~$\gamma_{\xi,\eta}$.

We are ready to state one of our main results. Recall the matrix functions~$\phi_m$ and~$\Phi$ from Section~\ref{sec:paths}.
\begin{theorem}\label{thm:local_limit}
Let~$x,x',y,y',\kappa,\kappa',\zeta,\zeta'$ be as in~\eqref{eq:local_coordinates_x} and~\eqref{eq:local_coordinates_y}, and assume that~$(\xi_N,\eta_N)\to(\xi,\eta)\in (-1,1)\backslash \partial \mathcal F_R$, as~$N\to \infty$, where~$\partial \mathcal F_R$ is the boundary of the frozen region. Let~$K_\text{path}$ be the correlation kernel associated with the probability measure~$\PP_\text{path}$~\eqref{eq:measure_on_points} given in Theorem~\ref{thm:bd_thm}. Then
\begin{multline}
\lim_{N\to\infty}\left[K_\text{path}(2\ell x+i,ky+j;2\ell x'+i',ky'+j')\right]_{j',j=0}^{k-1} \\
= \left[K_{(\xi,\eta)}(2\ell\kappa+i,k\zeta+j;2\ell\kappa'+i',k\zeta'+j')\right]_{j',j=0}^{k-1},
\end{multline}
for~$i,i' = 0,1,\dots,2\ell-1$, where
\begin{align}
&\left[K_{(\xi,\eta)}(2\ell\kappa+i,k\zeta+j;2\ell\kappa'+i',k\zeta'+j')\right]_{j',j=0}^{k-1} \label{eq:limiting_kernel_lhs}\\
& = -\frac{\one_{2\ell\kappa+i>2\ell\kappa'+i'}}{2\pi\i}\int_{|z|=1} \left(\prod_{m=1}^{i'}\phi_m(z)\right)^{-1}\Phi(z)^{\kappa-\kappa'}\prod_{m=1}^{i}\phi_m(z)z^{\zeta'-\zeta}\frac{\d z}{z} \\
&  - \frac{1}{2\pi\i}\int_{\gamma_{\xi,\eta}}\left(\prod_{m=1}^{i'}\phi_m(z)\right)^{-1}\frac{\adj (\Phi(z)-wI)}{\partial_{w}\det (\Phi(z)-wI)}\prod_{m=1}^{i}\phi_m(z)\frac{z^{\zeta'-\zeta}}{w^{\kappa'-\kappa}} \frac{\d z}{z} \label{eq:limiting_kernel}\\
&=\frac{1}{(2\pi\i)^2}\int_{|z|=\e^{r_1}}\int_{|w|=\e^{r_2}}\left(\prod_{m=1}^{i'}\phi_m(z)\right)^{-1} \\
& \times (\Phi(z)-wI)^{-1}\Phi(z)^{\one_{i'\geq i}}w^{\one_{i'< i}}
\prod_{m=1}^{i}\phi_m(z)\frac{z^{\zeta'-\zeta}}{w^{\kappa'-\kappa}} \frac{\d w}{w}\frac{\d z}{z},\label{eq:local_limit}
\end{align}
and the contours~$|z|=1$,~$|z|=\e^{r_1}$, and~$|w|=\e^{r_2}$ are positively oriented. The convergence is uniform on compact subsets of the respective region. 
\end{theorem}
All the following sections, except Section~\ref{sec:height_function}, are devoted to the proof of this theorem. In Section~\ref{sec:wh} we obtain an expression for the Wiener--Hopf factorization in Theorem~\ref{thm:bd_thm} which is suitable for asymptotic analysis. In Section~\ref{sec:asymptotic} we then rewrite the correlation kernel as a double integral on~$\mathcal R$. After that we proceed to perform a steep descent analysis of the double integral on the Riemann surface. 
\begin{remark}
In Definitions~\ref{def:curve_integration} and~\ref{def:point_integration} there is some freedom in the choice of~$\gamma_{\xi,\eta}$ and~$(r_1,r_2)$. However, the integrals~\eqref{eq:limiting_kernel} and~\eqref{eq:local_limit} are independent of such choices.
\end{remark}

Since Theorem~\ref{thm:inverse_kasteleyn} expresses the inverse Kasteleyn matrix in terms of the correlation kernel of the non-intersecting paths model, we may formulate Theorem~\ref{thm:local_limit} in terms of the dimer model instead of the paths model. Before doing so we change our focus slightly and consider the limit shape of the dimer model. The expected value of the height function can naturally be expressed in terms of~$K_\text{path}$, in fact, the choice of height function in~\eqref{eq:height_difference_aztec} was made so this would be true. It is therefore enough to adjust the proof of the previous theorem slightly to obtain the limit of the expectation of the normalized height function, see Section~\ref{sec:height_function}. 
\begin{proposition}\label{prop:limit_shape}
Let~$x$ and~$y$ be as in~\eqref{eq:local_coordinates_x} and~\eqref{eq:local_coordinates_y}, and assume that~$(\xi_N,\eta_N)\to(\xi,\eta)\in (-1,1)\backslash \partial \mathcal F_R$, as~$N\to \infty$. Set
\begin{equation}\label{eq:limit_shape}
\bar h(\xi,\eta)=\frac{1}{k\ell}\frac{1}{2\pi\i}\int_{\gamma_{\xi,\eta}}\d F+1,
\end{equation}
where~$F$ is given in Definition~\ref{def:action_function} and~$\gamma_{\xi,\eta}$ in Definition~\ref{def:curve_integration}. Let~$h$ be the height function defined by~\eqref{eq:height_difference_aztec}. Then
\begin{equation}
\lim_{N\to \infty}\frac{1}{k\ell N}\EE\left[h(2\ell x+i,2k y+j)\right]=\bar h(\xi,\eta),
\end{equation}
for~$i=0,\dots,2\ell-1$,~$j=0,\dots,2k-1$ with~$i=j\!\! \mod 2$, and
\begin{equation}
\partial_\xi \bar h(\xi,\eta)=-\frac{1}{2\ell}\frac{1}{2\pi\i}\int_{\gamma_{\xi,\eta}}\frac{\d w}{w}, \quad \text{and} \quad \partial_\eta \bar h(\xi,\eta)=\frac{1}{2k}\frac{1}{2\pi\i}\int_{\gamma_{\xi,\eta}}\frac{\d z}{z}.
\end{equation}
\end{proposition}
\begin{remark}
It is known that the scaled height function converges in probability to a deterministic limit, cf.~\cite{CKP00, Gor21, Kuc17}. The function~$\bar h$ in the previous proposition is therefore the limit shape. 
\end{remark}
\begin{remark}\label{rem:height_function_frozen}
Using residue calculus we may simplify the expression of~$\bar h$ when~$(\xi,\eta)$ is in any of the frozen regions. In particular,
\begin{equation}
\bar h(\xi,\eta)=
\begin{cases}
1, & \text{for } \, (\xi,\eta) \, \text{ in the north frozen region,} \\
\frac{1}{2}(1+\xi), & \text{for } \, (\xi,\eta) \, \text{ in the east frozen region,} \\
1+\frac{1}{2}(\xi+\eta), & \text{for } \, (\xi,\eta) \, \text{ in the south frozen region,}\\
\frac{1}{2}(1+\eta), & \text{for } \, (\xi,\eta) \, \text{ in the west frozen region.}
\end{cases}
\end{equation}
This is consistent with having only north dimers in the north frozen region, only east dimers in the east frozen region and so on. 
\end{remark}

The main reason for us to discuss the limit shape is to compare its slope with the slope of the limiting Gibbs measure. To that end, we write the height function in the coordinate system with the projections of~$\gamma_u$ and~$\gamma_v$, defined in Section~\ref{sec:gibbs_measure}, to~$\RR^2$ as the unit vectors. In other words, we divide~$x$ and~$y$ by the size of the Aztec diamond and change their sign. The new coordinates are denoted by~$(u,v)$ and are given by~$u=-\frac{\xi+1}{2\ell}$ and~$v=-\frac{\eta+1}{2k}$, as in Remark~\ref{rem:burgers}. We write~$\bar h(u,v)=\bar h(\xi,\eta)$. The slope of~$\bar h$ in the new coordinate system is
\begin{equation}\label{eq:slope_new_coord}
\nabla \bar h(u,v)=\left(\frac{1}{2\pi\i}\int_{\gamma_{\xi,\eta}}\frac{\d w}{w}, -\frac{1}{2\pi\i}\int_{\gamma_{\xi,\eta}}\frac{\d z}{z}\right).
\end{equation}

The limits in Theorem~\ref{thm:local_limit} are given in a form that resembles the integral formula for the ergodic translation-invariant Gibbs measures discussed in Section~\ref{sec:gibbs_measure}. We will make this relation clearer by relating the integrands in~\eqref{eq:limiting_kernel} and~\eqref{eq:local_limit} with the Kasteleyn matrix on the torus~$K_{G_1}$. The following lemma will also be useful later on. 
\begin{lemma}\label{lem:phi_kasteleyn}
Let~$K_{G_1}$ be the magnetically altered Kasteleyn matrix of the graph~$G_1$ defined in~\eqref{eq:magnetic_kasteleyn_matrix}. Then, for~$(z,w)\in \CC^2$, its inverse is given by 
\begin{equation}
\left(K_{G_1}(z,w)^{-1}\right)_{\mathrm b_{i,j}\mathrm w_{i',j'}}
=\left(\left(\prod_{m=1}^{2i'+1}\phi_m(z)\right)^{-1}(\Phi(z)-wI)^{-1}\Phi(z)^{\one_{i'\geq i}}w^{\one_{i'< i}}\prod_{m=1}^{2i}\phi_m(z)\right)^T_{j+1,j'+1}.
\end{equation}
Consequently, for~$(z,w)\in \mathcal R$,
\begin{equation}
\frac{\adj K_{G_1}(z,w)_{\mathrm b_{i,j}\mathrm w_{i',j'}}\d z}{zw\partial_w P(z,w)}
=\frac{\left(\left(\prod_{m=1}^{2i'+1}\phi_m(z)\right)^{-1}\adj(\Phi(z)-wI)\prod_{m=1}^{2i}\phi_m(z)\d z\right)^T_{j+1,j'+1}}{z\partial_w \det (\Phi(z)-wI)}.
\end{equation}
\end{lemma}

\begin{proof}
The first statement follows by a verification using the explicit formula~\eqref{eq:magnetic_kasteleyn_matrix}. The second statement then follows since
\begin{equation}
\frac{P(z,w)}{\partial_w P(z,w)}=\frac{\det (\Phi(z)-wI)}{\partial_w \det (\Phi(z)-wI)},
\end{equation}
and
\begin{equation}
\adj(\Phi(z)-wI)\Phi(z)=\det(\Phi(z)-wI)I+\adj(\Phi(z)-wI)w=\adj(\Phi(z)-wI)w,
\end{equation}
for~$(z,w)\in \mathcal R$.
\end{proof}

Theorem~\ref{thm:intro:gibbs_limit} now follows as a corollary to Theorem~\ref{thm:local_limit} together with Theorem~\ref{thm:inverse_kasteleyn} and Lemma~\ref{lem:phi_kasteleyn}.
\begin{corollary}\label{cor:convergence_gibbs_measure}
Let~$(\xi,\eta)\in (-1,1)$ and let~$(r_1,r_2)$ be as in Definition~\ref{def:point_integration}. Let 
\begin{equation}
e_m^{(N)}=\mathrm{b}_{\ell x_m+i_m,ky_m+j_m}\mathrm{w}_{\ell x_m'+i_m',ky_m'+j_m'}, \quad m=1,\dots,p,
\end{equation}
be edges in~$G_\text{Az}$, with
\begin{equation}
 x_m=\frac{kN}{2}(\xi_N+1)+\kappa_m, \quad x_m'=\frac{kN}{2}(\xi_N+1)+\kappa_m',
\quad y_m=\frac{\ell N}{2}(\eta_N+1)+\zeta_m, \quad y_m'=\frac{\ell N}{2}(\eta_N+1)+\zeta_m',
\end{equation}
and set~$e_m=\mathrm{b}_{\ell \kappa_m+i_m,k\zeta_m+j_m}\mathrm{w}_{\ell \kappa_m'+i_m',k\zeta_m'+j_m'}$. Assume that~$(\xi_N,\eta_N)\to(\xi,\eta)\in (-1,1)\backslash \partial \mathcal F_R$, as~$N\to \infty$. Then
\begin{equation}\label{eq:limit_to_gibbs_measure}
\lim_{N\to\infty}\PP\left[e_1^{(N)},\dots,e_p^{(N)}\in M_N\right]=\PP_{(r_1,r_2)}\left[e_1,\dots,e_p\in M\right],
\end{equation}
where the probability measures on the left and right hand side are defined by~\eqref{eq:measure_dimer} and~\eqref{eq:gibbs_measure}, respectively. 

Moreover, the slope~$(s,t)$ of the limiting measure, cf.~\eqref{eq:slope_1} and~\eqref{eq:slope_2}, is given by
\begin{equation}\label{eq:equal_slopes}
(s,t)=\nabla \bar h(u,v),
\end{equation}  
where the right hand side is the gradient of the limit shape in scaled coordinates as in~\eqref{eq:slope_new_coord}.
\end{corollary}

\begin{proof}
The convergence of the inverse Kasteleyn matrix of~$G_{Az}$ to that of the infinite graph~$G$, cf.~\eqref{eq:inverse_infinite_kasteleyn},
\begin{equation}\label{eq:limit_inverse_kasteleyn}
\lim_{N\to \infty}(K_{G_\text{Az}}^{-1})_{\mathrm{b}_{\ell x_m+i_m,ky_m+j_m}\mathrm{w}_{\ell x_m'+i_m',ky_m'+j_m'}}=(K_{G,(r_1,r_2)}^{-1})_{\mathrm{b}_{\ell \kappa_m+i_m,k\zeta_m+j_m}\mathrm{w}_{\ell \kappa_m'+i_m',k\zeta_m'+j_m'}},
\end{equation}
is a direct consequence of Theorems~\ref{thm:inverse_kasteleyn} and~\ref{thm:local_limit} together with Lemma~\ref{lem:phi_kasteleyn}. We know by~\eqref{eq:kenyon_formula} and~\eqref{eq:gibbs_measure} that the left and right hand side of~\eqref{eq:limit_to_gibbs_measure} can be expressed in terms of respective inverse Kasteleyn matrices. Hence, the limit~\eqref{eq:limit_to_gibbs_measure} follows from~\eqref{eq:limit_inverse_kasteleyn}.

What remains is to match the two slopes in~\eqref{eq:equal_slopes}. The slopes of the ergodic translation-invariant Gibbs measures have been computed both in~\cite{KOS06} and~\cite{BCT22}. To make sure to match our convention, we include a short computation here as well. Recall the expression of the slope~$(s,t)$ in terms of the sum of edge probabilities~\eqref{eq:slope_1} and~\eqref{eq:slope_2},
\begin{equation}
s=-\sum_{e=\mathrm w\mathrm b}(K_{G})_{\mathrm w\mathrm b}\left(K_{G,(r_1,r_2)}^{-1}\right)_{\mathrm b\mathrm w}, \quad \text{and} \quad 
t=\sum_{e=\mathrm w\mathrm b}(K_{G})_{\mathrm w\mathrm b}\left(K_{G,(r_1,r_2)}^{-1}\right)_{\mathrm b\mathrm w}-k,
\end{equation}
where the sum is taken over all edges intersecting~$\gamma_u$ and~$\gamma_v$, respectively, which are the two loops on the torus, see Section~\ref{sec:gibbs_measure}. Let us focus on the expression for~$s$, the expression for~$t$ follows in a similar way. 

To match the two slopes, we want to write~$K_{G,(r_1,r_2)}$ from~\eqref{eq:gibbs_measure} as an integral on the Riemann surface~$\mathcal R$. We obtain such an expression by using the equality of~\eqref{eq:limiting_kernel} and~\eqref{eq:local_limit}. Indeed, we use the first equality in Lemma~\ref{lem:phi_kasteleyn}, then the equality of~\eqref{eq:limiting_kernel} and~\eqref{eq:local_limit}, and finally the second equality in Lemma~\ref{lem:phi_kasteleyn}. We obtain an expression for~$K_{G,(r_1,r_2)}$ in terms of an integral along~$\gamma_{\xi,\eta}$, and an integral along the unit circle. 

If~$e\in G_1$ crosses~$\gamma_u$, then~$e=\mathrm w_{i',k-1}\mathrm b_{i,0}$, for~$i'=0,\dots,\ell-1$ and~$i=i'$ or~$i=i'+1$. For such edges, 
\begin{equation}\label{eq:kasteleyn_torus_plane}
K_{G_1}(z,w)_{\mathrm w_{i,k-1}\mathrm b_{i',0}}=
\begin{cases}
z^{-1}w(K_{G})_{\mathrm w_{\ell-1,-1}\mathrm b_{\ell,0}}, & i=i'+1=\ell, \\
z^{-1}(K_{G})_{\mathrm w_{i,-1}\mathrm b_{i',0}}, & \text{otherwise}.
\end{cases}
\end{equation}
We obtain
\begin{multline}\label{eq:sum_edge_weights_gamma_x}
s=-\sum_{i'=0}^{\ell-1}\sum_{i=i'}^{i'+1}(K_{G})_{\mathrm w_{i',-1}\mathrm b_{i,0}}\left(K_{G,(r_1,r_2)}^{-1}\right)_{\mathrm b_{i,0}\mathrm w_{i',-1}} 
=\sum_{i=1}^\ell \int_{|z|=1}\phi_{2i}(z)_{k,1}\beta_{k,i}z^{-1}\frac{\d z}{z} \\
+\frac{1}{2\pi\i}\int_{\gamma_{\xi,\eta}}\sum_{i'=0}^{\ell-1}\sum_{i=i'}^{i'+1}K_{G_1}(z,w)_{\mathrm w_{i',k-1}\mathrm b_{i,0}}\frac{\adj K_{G_1}(z,w)_{\mathrm b_{i,0}\mathrm w_{i',k-1}}}{zw\partial_{w}P(z,w)}\d z,
\end{multline}
where the second term in the right hand side comes from the second term in~\eqref{eq:limiting_kernel} and the first term comes from the respective first term, which is only present when~$e$ is an east edge, due to the indicator function. We may simplify the integrands in both integrals in~\eqref{eq:sum_edge_weights_gamma_x}. Indeed, the integrand in the first integral is equal to~$\frac{\beta_i^v\d z}{z(z-\beta_i^v)}$ by~\eqref{eq:geometric} and~\eqref{eq:def_phi_geometric}, so the integral is zero (as~$\beta_i^v<1$ by Assumption~\ref{ass:main_ass}\eqref{ass:wh}), and we will see below that the integrand in the other integral is equal to~$\frac{\d w}{w}$. 

We first observe that
\begin{equation}\label{eq:tr_derivative}
\Tr \left(\adj K_{G_1}(z,w)\partial_zK_{G_1}(z,w)\right)=-z^{-1}\sum_{i'=0}^{\ell-1}\sum_{i=i'}^{i'+1}\adj K_{G_1}(z,w)_{\mathrm b_{i,0}\mathrm w_{i',k-1}}K_{G_1}(z,w)_{\mathrm w_{i',k-1}\mathrm b_{i,0}}.
\end{equation}
This follows from the fact that
\begin{equation}
\partial_zK_{G_1}(z,w)_{\mathrm w_{i',k-1}\mathrm b_{i,0}}=-z^{-1}K_{G_1}(z,w)_{\mathrm w_{i',k-1}\mathrm b_{i,0}},
\end{equation}
if~$i=0,\dots,\ell-1$ and~$i'=i$ or~$i'=i-1$, and the derivative is zero otherwise. Recall Jacobi's formula: for a differentiable square matrix~$A(z)$ the equality
\begin{equation}\label{eq:jacobis_formula}
\frac{\d}{\d z}\det A(z)=\Tr\left(\adj A(z)A'(z)\right)
\end{equation}
holds. Jacobi's formula implies that the left hand side of~\eqref{eq:tr_derivative} is equal to~$\partial_z \det K_{G_1}(z,w)$. Hence, the second integral in~\eqref{eq:sum_edge_weights_gamma_x} is equal to
\begin{equation}
-\frac{1}{2\pi\i}\int_{\gamma_{\xi,\eta}}\frac{\partial_z \det K_{G_1}(z,w)}{w\partial_{w}P(z,w)}\d z
=\frac{1}{2\pi\i}\int_{\gamma_{\xi,\eta}}\frac{\d w}{w},
\end{equation}
where we used that~$P(z,w)=\det K_{G_1}(z,w)$ and~$\partial_wP(z,w)\d w+\partial_zP(z,w)\d z=0$. Hence,~$s$ is equal to the first component of~$\nabla \bar h(u,v)$.

The computation for~$t$ is similar. The main difference is that the contribution from the first term in~\eqref{eq:limiting_kernel} is now nonzero. The sum, corresponding to the first sum on the right hand side of~\eqref{eq:sum_edge_weights_gamma_x}, is instead given by
\begin{equation}
-\sum_{j'=0}^{k-1}\frac{1}{2\pi\i}\int_{|z|=1}\left(\left(\prod_{m=1}^{2\ell-1}\phi_m(z)\right)^{-1}\Phi(z)\left(-\phi_{2\ell}(z)\right)^{-1}\right)_{j'+1,j'+1} \frac{\d z}{z}=k.
\end{equation} 
\end{proof}

\begin{remark}\label{rem:limiting_slopes}
If~$(\xi,\eta)$ is in the north, east, south or west frozen region, then Corollary~\ref{cor:convergence_gibbs_measure} implies that the slope~$(s,t)$ of the limiting Gibbs measure is equal to~$(0,0)$,~$(-\ell,0)$,~$(-\ell,-k)$ and~$(0,-k)$, respectively. These slopes are precisely the slopes in Remark~\ref{rem:extremal_slopes}. Hence, in the limit the north frozen region consists of only north edges, and similarly for the east, south and west frozen regions.
\end{remark}

\begin{remark}\label{rem:possible_slopes}
The expression of~$(s,t)$ in Corollary~\ref{cor:convergence_gibbs_measure} implies that~$-\ell\leq s\leq 0$ and~$-k\leq t\leq 0$ for all~$(\xi,\eta)\in (-1,1)^2$ (see the proof of Corollary~\ref{cor:height_function_smooth} below for a proof of a similar implication), that is,~$(s,t)+(0,k)\in N(P)$, where~$N(P)$ is the Newton polygon. If we would choose the dimer configuration consisting of only west edges instead of only north edges as the reference dimer configuration in the definition of the height function~\eqref{eq:height_difference_aztec}, then the set of slopes would coincide with~$N(P)$.
\end{remark}

\subsubsection{Frozen regions}\label{sec:frozen_region}
We have seen that the north, east, south and west frozen regions consist of only north, east, south and west edges, respectively. Here we will discuss the other frozen regions. Let us focus our discussion on the~$(2\ell+k+1+m)$th frozen region with~$m=1,\dots,k-2$, that is, the frozen region corresponding to~$A_{0,2\ell+k+1+m}$. In this region the slope is~$(0,m-k)$, and the configuration there consists of only north and west dimers. It is not difficult to see that if~$e$ is a west edge with~$e=\mathrm b_{i,j}\mathrm w_{i,j}\in M$, then~$\mathrm b_{i+1,j}\mathrm w_{i+1,j}\in M$, and similarly for a north edge. That is, each row consists of only west edges or only north edges. If the slope is~$(0,m-k)$ then the number of rows in the fundamental domain consisting of north edges is~$m$, and hence the number of rows consisting of west edges is~$k-m$, cf.~\eqref{eq:slope_1} and~\eqref{eq:slope_2}. In fact, any choice of~$k-m$ rows consisting of west edges gives us an ergodic translation-invariant Gibbs measure with slope~$(0,m-k)$. A natural question is which of these appear in the limit of the Aztec diamond. 

Let~$w_1$ be the weight of a dimer configuration on~$G_1$ where the~$(j+1)$st row,~$j=0,\dots,k-1$, consists of west dimers and the~$(j'+1)$st row,~$j'\neq j$, consists of north dimers. Let~$w_2$ be the weight of the same configuration but with north dimers in the~$(j+1)$st row and east dimers in the~$(j'+1)$st. Then~$w_1=w_2\gamma_{j+1}^h/\gamma_{j'+1}^h$, where~$\gamma_{j+1}^h$ is defined in~\eqref{eq:edge_weights_product}. Recall that the edge weights on the north edges are~$1$. (If they instead were~$\delta_{j+1,i}$, then the relation would have been~$w_1=w_2(\gamma_{j+1}^h/\delta_{j+1}^h)(\delta_{j'+1}^h/\gamma_{j'+1}^h)$.) Let us order these factors by magnitude, 
\begin{equation}\label{eq:frozen:order_angles}
\gamma_{j_1+1}^h>\dots>\gamma_{j_k+1}^h.
\end{equation}
It seems reasonable that the~$k-m$ rows consisting of west dimers should be the rows numbered~$i_n$ for~$n=1,\dots,k-m$. Using Corollary~\ref{cor:convergence_gibbs_measure} we show that this indeed is correct.

Consider the same setting as in Corollary~\ref{cor:convergence_gibbs_measure}, with the west edges~$e^{(N)}=\mathrm{b}_{\ell x+i,ky+j}\mathrm{w}_{\ell x+i,ky+j}$ and~$e=\mathrm{b}_{\ell \kappa+i,k\zeta+j}\mathrm{w}_{\ell \kappa+i,k\zeta+j}$,~$i=0,\dots,\ell-1$,~$j=0,\dots,k-1$, and~$(\xi,\eta)$ in the~$2\ell+k+1+m$ frozen region. Then
\begin{equation}\label{eq:limit_west_edge}
\lim_{N\to \infty}\PP\left[e^{(N)}\in M_N\right]=\PP_{(r_1,r_2)}\left[e\in M\right].
\end{equation}
We claim that the right hand side is equal to~$1$ if~$p_{\infty,j+1}$ is in the interior of~$\gamma_{\xi,\eta}$ and zero otherwise. This means that with probability one, the~$(j+1)$st row contains only west edges, in the limit, if~$p_{\infty,j+1}$ is contained in the interior of~$\gamma_{\xi,\eta}$, and with probability zero if it is not. The angles lying in the interior of~$\gamma_{\xi,\eta}$ are the angles between~$A_{0,2\ell+k+1+m}$ and~$A_{0,1}$. Recall that~$p_{\infty,j'+1}=(\infty,\gamma_{j'+1}^h)$. Hence, the angles that are contained in~$\gamma_{\xi,\eta}$ are~$p_{\infty,j_n+1}$, for~$n=1,\dots,k-m$, where~$j_n$ follows the ordering of~\eqref{eq:frozen:order_angles}, which proves the above description. See Figure~\ref{fig:amoeba_curves}.

Let us end this section by proving the claim. Adjusted to the setting of~\cite{BCT22}, the claim is given in Remark 53 of that paper. We compute the right hand side of~\eqref{eq:limit_west_edge} using~\eqref{eq:gibbs_measure}, which can be written as an integral on~$\mathcal R$ using Theorem~\ref{thm:local_limit} and Lemma~\ref{lem:phi_kasteleyn}. The probability is then equal to
\begin{equation}\label{eq:frozen:edge_prob_integral}
-\gamma_{j+1,i+1}\left(\frac{1}{2\pi\i}\int_{\gamma_{\xi,\eta}}\left(\prod_{m=1}^{2i+1}\phi_m(z)\right)^{-1}\frac{\adj (\Phi(z)-wI)}{\partial_{w}\det (\Phi(z)-wI)}\prod_{m=1}^{2i}\phi_m(z)\frac{\d z}{z}\right)_{j+1,j+1}.
\end{equation}
Lemma~\ref{lem:kasteleyn_thetas} below determines the zeros and poles of the integrand above, and it shows that the integrand has simple poles at~$p_{\infty,j+1}$ and~$q_{0,i+1}$. Hence, if~$p_{\infty,j+1}$ is in the exterior of~$\gamma_{\xi,\eta}$, then~\eqref{eq:frozen:edge_prob_integral} is zero, see Figure~\ref{fig:amoeba_curves}. If~$p_{\infty,j+1}$ is in the interior of~$\gamma_{\xi,\eta}$, then, by the residue theorem,~\eqref{eq:frozen:edge_prob_integral} is equal to
\begin{multline}
-\gamma_{j+1,i+1}\left.\left(\left(\prod_{m=1}^{2i+1}\phi_m(z)\right)^{-1}\frac{\adj (\Phi(z)-wI)}{\partial_{w}\det (\Phi(z)-wI)}\prod_{m=1}^{2i}\phi_m(z)\right)_{j+1,j+1}\right|_{(z,w)=p_{\infty,j+1}} \\
=\gamma_{j+1,i+1}\left(\left(\prod_{m=1}^{2i+1}\phi_m(\infty)\right)^{-1}\prod_{m=1}^{2i}\phi_m(\infty)\right)_{j+1,j+1}=\gamma_{j+1,i+1}\gamma_{j+1,i+1}^{-1}=1.
\end{multline}
In the first equality we have used the fact that all matrices in the product evaluated at~$p_{\infty,j+1}$ are lower triangular, and that
\begin{equation}
\partial_w \det(\Phi(z)-wI)|_{p_{\infty,j+1}}=-\Tr\adj(\Phi(z)-wI)|_{p_{\infty,j+1}}=-\adj(\Phi(z)-wI)_{j+1,j+1}|_{p_{\infty,j+1}}.
\end{equation}
which follows from Jacobi's formula~\eqref{eq:jacobis_formula}, as well as relations~$\Phi(\infty)_{j+1,j+1}=\gamma^h_{j+1}$ and~$p_{\infty,j+1}=(\infty,\gamma^h_{j+1})$.

\subsection{The limit shape in a smooth region}\label{sec:limit_shape_smooth}

The pieces of the limit shape in the smooth regions are flat. In this section we express the gradient in terms of the location of the corresponding compact oval, and we express the additive shift as an integral between some of the angles. Later in Section~\ref{sec:wh} we will see that this shift is closely related to the linear flow in the computation of the Wiener--Hopf factorization. 

Let us, as in the end of Section~\ref{sec:harnack}, view~$\mathcal R$ as a~$k$-sheeted Riemann surface with sheets~$\mathcal R^{(z)}_j=\{(z,w_j(z))\in \mathcal R:z\in \CC\}$,~$j=1,\dots,k$, where~$|w_1(z)|>\dots>|w_k(z)|$ for almost all~$z$. Similarly, we can view it as an~$\ell$-sheeted Riemann surface and denote the sheets by~$\mathcal R^{(w)}_i=\{(z_i(w),w)\in \mathcal R:w\in \CC\}$,~$i=1,\dots,\ell$, where~$|z_1(w)|>\dots>|z_\ell(w)|$ for almost all~$w$. Given~$m=1,\dots,g$, we define~$n$ and~$n'$ so that the compact oval~$A_m$ intersects the two sheets~$\mathcal R^{(z)}_n$ and~$\mathcal R^{(z)}_{n+1}$, and the two sheets~$\mathcal R^{(w)}_{n'}$ and~$\mathcal R^{(w)}_{n'+1}$.

\begin{corollary}\label{cor:height_function_smooth}
Let~$\bar h(u,v)$ be the limit shape defined in Proposition~\ref{prop:limit_shape} with coordinates as in~\eqref{eq:slope_new_coord}. If~$(u,v)$ is in the~$m$th smooth region, that is, the one corresponding to the compact oval~$A_m$, then 
\begin{equation}
\nabla \bar h(u,v)=(-n',-n).
\end{equation}
Let~$H$ be the difference between the limit shape and the parallel plane passing through the origin defined by 
\begin{equation}
H(u,v)=k\ell \left(\bar h(u,v)+(n' u+nv)\right).
\end{equation}
Then 
\begin{equation}
H(u,v)=\int_{\ell\sum_{j=1}^{k}p_{0,j}}^{k\sum_{i=1}^{\ell}q_{0,i}}\omega_m-kn'-\ell n+k\ell,
\end{equation}
where~$\omega_m$ is the~$m$th element of the canonical basis of differentials defined in Section~\ref{sec:abel_theta}.
\end{corollary}
\begin{remark}
The integral in the statement should be interpreted as 
\begin{equation}
\int_{\ell\sum_{j=1}^{k}p_{0,j}}^{k\sum_{i=1}^{\ell}q_{0,i}}\omega_m=k\sum_{i=1}^{\ell}\int_{q_0}^{q_{0,i}}\omega_m-\ell\sum_{j=1}^{k}\int_{q_0}^{p_{0,j}}\omega_m,
\end{equation}
for some point~$q_0\in \mathcal R$, and the curves of integration are taken so they do not cross any of the~$A$-cycles or~$B$-cycles.
\end{remark} 
\begin{proof}[Proof of Corollary~\ref{cor:height_function_smooth}]
Recall that the second entry of~$\nabla \bar h(u,v)$ is equal to the integral 
\begin{equation}
-\frac{1}{2\pi\i}\int_{\gamma_{\xi,\eta}}\frac{\d z}{z},
\end{equation}
see~\eqref{eq:slope_new_coord}. The curve~$\gamma_{\xi,\eta}$ intersects~$\mathcal R^{(z)}_j$ for~$j=1,\dots n$, see Figure~\ref{fig:rs_sheets}. Moreover, the projection of~$\gamma_{\xi,\eta}\cap \mathcal R^{(z)}_j$ to the~$z$ variable is a simple closed loop around zero oriented in the positive direction.\footnote{The green curve in Figure~\ref{fig:rs_sheets} shows only half of~$\gamma_{\xi,\eta}$, the other half is the conjugate of the pictured one.} This follows since two consecutive sheets are glued together along either (part of) the positive part of the real line, or (part of) the negative part of the real line, and which of the two alternates as~$j$ increases, see the discussion in Section~\ref{sec:harnack}. Hence, the second entry of~$\nabla \bar h(u,v)$ is~$-n$. Similarly, one shows that the first entry is~$-n'$. The minus sign appears since the curve~$\gamma_{\xi,\eta}\cap \mathcal R^{(w)}_i$ projected to the~$w$ coordinate is oriented in the negative direction. This proves the first statement. 

The limit shape~$\bar h$ given in Proposition~\ref{prop:limit_shape} is expressed in terms of~$\d F$, where~$F$ is as in Definition~\ref{def:action_function}. Using the coordinates~$(u,v)=(-\frac{\xi+1}{2\ell},-\frac{\eta+1}{2k})$ in the definition~\eqref{eq:def_action_function} of~$F$ we obtain
\begin{equation}
H(u,v)=\frac{k}{2\pi\i}\int_{\gamma_{\xi,\eta}}\frac{\d w}{w}-\frac{\ell}{2\pi\i}\int_{\gamma_{\xi,\eta}}\frac{\d z}{z}-\frac{1}{2\pi\i}\int_{\gamma_{\xi,\eta}}\d \log f +k\ell.
\end{equation}
The sum of the first two terms of the right hand side is equal to~$-(kn'+\ell n)$. The contour in the third integral can be deformed to~$-B_m$, the oppositely oriented~$B$-cycle that crosses~$A_m$, without changing the integral, see Figures~\ref{fig:amoeba} and~\ref{fig:amoeba_curves}. Moreover, since~$f$ is a quotient of prime forms, we have~$\d \log f=\omega_{k\sum_{i=1}^{\ell}q_{0,i}-\ell\sum_{j=1}^{k}p_{0,j}}$, where the right hand side is the unique meromorphic differential~$1$-form with zero period along the~$A$-cycles, and simple poles at all~$q_{0,i}$ with residue~$k$ and simple poles at all~$p_{0,j}$ with residue~$-\ell$, see~\cite[Equation (21)]{Fay73} and~\cite[Section 8.2]{BK11}. The integral of this form along the~$B$-cycle~$B_m$ can be expressed in terms of an integral of~$\omega_m$, as follows:
\begin{equation}
\int_{B_m}\omega_{k\sum_{i=1}^{\ell}q_{0,i}-\ell\sum_{j=1}^{k}p_{0,j}}=\int_{\ell\sum_{j=1}^{k}p_{0,j}}^{k\sum_{i=1}^{\ell}q_{0,i}}\omega_m,
\end{equation}
see~\cite[Equation (7)]{Fay73} and~\cite[Equation (2)]{BCT22}. This is known as a \emph{reciprocity law}, and it concludes our proof.
\end{proof}

If we take~$(u_m,v_m)$ in the~$m$th smooth region, for each~$m=1,\dots,g$, then Corollary~\ref{cor:height_function_smooth} implies that
\begin{equation}
(H(u_1,v_1),\dots,H(u_g,v_g))\!\!\!\mod \ZZ^g=k\sum_{i=1}^\ell u(q_{0,i})-\ell\sum_{j=1}^ku(p_{0,j}),
\end{equation}
where~$u$ is the Abel map defined in Section~\ref{sec:abel_theta}. The right hand side will also appear as a linear flow on the Jacobi variety in our calculations, which makes this equality rather interesting. See Remarks~\ref{rem:linear_flow} and~\ref{rem:linear_flow_discussion}.

\section{The Wiener--Hopf factorization}\label{sec:wh}
In this section we obtain an expression for the Wiener--Hopf factorization in the correlation kernel of Theorem~\ref{thm:bd_thm} which is suitable for asymptotic analysis. 

We say that a matrix-valued function~$\phi$ admits a Wiener--Hopf factorization if it factorizes on the unit circle as
\begin{equation}\label{eq:wh}
\phi(z)=\phi_+(z)\phi_-(z)=\widetilde \phi_-(z)\widetilde \phi_+(z),
\end{equation}
where~$\phi_+^{\pm1}$,~$\widetilde \phi_+^{\pm1}$ are analytic in the unit disc and continuous up to the boundary, and~$\phi_-^{\pm1}$,~$\widetilde \phi_-^{\pm1}$ are analytic in the complement of the closed unit disc and continuous up to the boundary, with the behavior~$\phi_-, \widetilde \phi_- \sim z^{-\ell N}I$ as~$z\to \infty$.\footnote{In different settings the normalization at infinity may vary. We use the one that is convenient for the present work} The problem of obtaining a suitable representation of such a factorization, if it exists, is a classical problem and it has been studied extensively. For an overview we refer to~\cite{GMS03} and references therein. See also, e.g.,~\cite{IMM08, KT12} for treatments of Wiener--Hopf factorizations on higher genus Riemann surfaces. In our setting with~$\phi$ given in Section~\ref{sec:paths}, the question of existence of a Wiener--Hopf factorization is answered in~\cite[Theorem 4.8]{BD19}. Namely, if~$\alpha^v_i/\gamma^v_i, \beta^v_i\neq 1$ for all~$i$, where~$\alpha^v_i$,~$\beta^v_i$ and~$\gamma^v_i$ are defined in~\eqref{eq:edge_weights_product}, then~$\phi(z)$ admits a Wiener--Hopf factorization if and only if the winding number of~$\det (\phi(z)z^{\ell N})$ with respect to the unit circle is equal to~$0$. Vanishing of the winding number is equivalent to the condition~$\beta^v_i<1<\alpha^v_i/\gamma^v_i$ for all~$i$, which is exactly~\eqref{ass:wh} in Assumption~\ref{ass:main_ass}. It is also easy to see that the factorization is unique. Indeed, if there are two solutions~$\phi_{i,\pm}$,~$i=1,2$, then~$\phi_{1,+}^{-1}\phi_{2,+}=\phi_{1,-}\phi_{2,-}^{-1}$ on the unit circle, and Liouville's theorem implies that~$\phi_{1,\pm}=\phi_{2,\pm}$. However, to find such a factorization, and, in particular, to find a factorization suitable for asymptotic analysis is, in general, a difficult problem.

\subsection{The iterative re-factorization procedure}\label{sec:switching}

We begin by constructing the Wiener--Hopf factorization recursively. A similar discussion can be found in~\cite[Section 4.4]{BD19} and also, for a biased~$2\times 2$-periodic Aztec diamond, in~\cite[Section 2.4]{BD22}.

We will focus only on the factorization~$\phi(z)=\widetilde \phi_-(z)\widetilde \phi_+(z)$ in~\eqref{eq:wh}, since the other factorization is not present in the formula for the correlation kernel in Theorem~\ref{thm:bd_thm}. Recall that~$\phi=\Phi^{k N}$ where
\begin{equation}\label{eq:phi_recall}
\Phi=\prod_{i=1}^{2\ell} \phi_i,
\end{equation}
with~$\phi_i$ given in~\eqref{eq:def_phi_bernoulli} and~\eqref{eq:def_phi_geometric}. Recall also our convention of the product notation explained in Remark~\ref{rem:products}. First we construct a factorization with the factors being analytic and non-singular in the desired regions, except at zero and infinity. When we have such a factorization, it will be straightforward to obtain the correct normalization. 

Set~$\Phi_0=\Phi$. Let~$\Phi_0=(\Phi_0)_-(\Phi_0)_+$ be a factorization on the unit circle, where~$(\Phi_0)_-^{\pm 1}$ and~$(\Phi_0)_+^{\pm 1}$ are analytic in the exterior of the unit disc and in the closed unit disc, except, possibly, at zero and infinity. We will later specify this factorization. Then
\begin{equation}
\phi=\Phi_0^{kN}=(\Phi_0)_-\left((\Phi_0)_+(\Phi_0)_-\right)^{kN-1} (\Phi_0)_+=(\Phi_0)_-\left(\Phi_1\right)^{kN-1} (\Phi_0)_+,
\end{equation}
where~$\Phi_1=(\Phi_0)_+(\Phi_0)_-$. The idea is to proceed by obtaining a similar factorization of~$\Phi_1=(\Phi_1)_-(\Phi_1)_+$, use these factors to define~$\Phi_2$ and so on. More precisely, given~$\Phi_j$ and a factorization~$\Phi_{j}=(\Phi_j)_-(\Phi_j)_+$, we define~$\Phi_{j+1}=(\Phi_j)_+(\Phi_j)_-$. Inductively, we obtain
\begin{equation}
\phi=\left(\prod_{i=0}^{j-1}(\Phi_i)_-\right)\Phi_j^{kN-j}\left(\prod_{\substack{i=0 \\ \leftarrow}}^{j-1}(\Phi_i)_+\right)=\left(\prod_{i=0}^{j}(\Phi_i)_-\right)\Phi_{j+1}^{kN-(j+1)}\left(\prod_{\substack{i=0 \\ \leftarrow}}^j(\Phi_i)_+\right),
\end{equation}
and hence,
\begin{equation}
\phi=\left(\prod_{i=0}^{kN-1}(\Phi_i)_-\right)\left(\prod_{\substack{i=0 \\ \leftarrow}}^{kN-1}(\Phi_i)_+\right).
\end{equation}
The factors in this factorization of~$\phi$ have the right properties, except at zero and infinity. Before we normalize these factors to have the correct behavior also at zero and infinity, we describe the factorization of~$\Phi_j$,~$j=0,1,\dots$.

In the setting of~\cite{BD19}, there is a particularly concrete way of obtaining the factorization~$\Phi_j$. We describe this procedure when specialized to our setting. Recall the definitions of~$\phi^b$ and~$\phi^g$ in~\eqref{eq:bernoulli} and~\eqref{eq:geometric}. Let~$\vec{\alpha},\vec{\beta},\vec{\gamma} \in \RR_{>0}^k$ be vectors with elements~$(\vec \alpha)_i=\alpha_i$,~$(\vec \beta)_i=\beta_i$,~$(\vec \gamma)_i=\gamma_i$. The factorization relies on the algebraic equation
\begin{equation}\label{eq:switching_rule}
\phi^b(z;\vec{\alpha},\vec{\gamma})\phi^g(z;\vec{\beta})=\phi^g(z;\vec{\beta}')\phi^b(z;\vec{\alpha}',\vec{\gamma})
\end{equation}
where
\begin{equation}\label{eq:update_rule}
\alpha_i'=\alpha_{i-1}\frac{\alpha_i+\gamma_{i+1}\beta_i}{\alpha_{i-1}+\gamma_i\beta_{i-1}} \quad \text{and} \quad \beta_i'=\beta_{i-1}\frac{\alpha_i+\gamma_{i+1}\beta_i}{\alpha_{i-1}+\gamma_i\beta_{i-1}},
\end{equation}
for~$i=1,\dots,k$ and the subscript is taken modulo~$k$. The equality is given in~\cite{BD19, LP12} and is immediate to verify. Moreover,~$\prod_{i=1}^k\alpha'_i=\prod_{i=1}^k\alpha_i$ and~$\prod_{i=1}^k\beta'_i=\prod_{i=1}^k\beta_i$, which implies, using part~\eqref{eq:determinant_bern_geo} of Lemma~\ref{lem:transition_matrices_properties}, that
\begin{equation}\label{eq:invariant_determinant}
\det \phi^b(z;\vec{\alpha},\vec{\gamma})=\det \phi^b(z;\vec{\alpha}',\vec{\gamma}') \quad \text{and} \quad \det \phi^g(z;\vec{\beta})=\det \phi^g(z;\vec{\beta}').
\end{equation}
In particular,~$\phi^b(z;\vec{\alpha},\vec{\gamma})$ is analytic and non-singular if and only if~$\phi^b(z;\vec{\alpha}',\vec{\gamma}')$ is, and~$\phi^g(z;\vec{\beta})$ is analytic and non-singular if and only if~$\phi^g(z;\vec{\beta}')$ is.

Using~\eqref{eq:switching_rule} we obtain the factorization~$\Phi_j=(\Phi_j)_-(\Phi_j)_+$,~$j=0,1,\dots$. By definition,~\eqref{eq:def_phi_bernoulli} and~\eqref{eq:def_phi_geometric}, the factors~$\phi_i$ of~$\Phi_0=\prod_{i=1}^{2\ell} \phi_i$ are equal to~$\phi^b(z;\vec{\alpha}_{i},\vec{\gamma}_{i})$ if~$i$ is odd and to~$\phi^g(z;\vec{\beta}_{i})$ if~$i$ is even, where~$\vec{\alpha}_i$,~$\vec{\beta}_i$ and~$\vec{\gamma}_i$ are vectors constructed from the edge weights. Moreover,~$\phi^b(z;\vec{\alpha}_{i},\vec{\gamma}_{i})$ is analytic and non-singular in the closed unit disk, excepts at zero, and~$\phi^g(z;\vec{\beta}_{i})$ is analytic and non-singular in the complement of the unit disk. This follows from the assumption~$\beta^v_i<1<\alpha^v_i/\gamma^v_i$. We use~\eqref{eq:switching_rule} to rearrange the factors of~$\Phi_0$ so that all factors~$\phi^b$ are moved to the right in the product and all factors~$\phi^g$ are moved to the left. We get that~$\Phi_0=(\Phi_0)_-(\Phi_0)_+$ with
\begin{equation}
(\Phi_0)_-(z)=\prod_{i=1}^\ell \phi^g(z;\vec{\beta}_{0,i}) \quad \text{and} \quad (\Phi_0)_+(z)=\prod_{\substack{i=1 \\ \leftarrow}}^\ell \phi^b(z;\vec{\alpha}_{0,i},\vec{\gamma}_{0,i})
\end{equation}
for some vectors~$\vec{\alpha}_{0,i}, \vec{\beta}_{0,i}, \vec{\gamma}_{0,i}\in \RR^k_{>0}$. We emphasize that the subscript~$(0,i)$ labels the vector and is not a reference to the entries of the vector. Then we repeat the above procedure. Given
\begin{equation}
\Phi_j(z)=\prod_{\substack{i=1 \\ \leftarrow}}^\ell \phi^b(z;\vec{\alpha}_{j-1,i},\vec{\gamma}_{j-1,i})\prod_{i=1}^\ell \phi^g(z;\vec{\beta}_{j-1,i}),
\end{equation}
for some vectors~$\vec{\alpha}_{j-1,i}, \vec{\beta}_{j-1,i}, \vec{\gamma}_{j-1,i}\in \RR^k_{>0}$, we rearrange the terms in the product by moving all factors~$\phi^g$ to the left and all factors~$\phi^b$ to the right using~\eqref{eq:switching_rule}. We obtain the factorization~$\Phi_j=(\Phi_j)_-(\Phi_j)_+$, with
\begin{equation}\label{eq:phi_+_phi_-}
(\Phi_j)_-(z)=\prod_{i=1}^\ell \phi^g(z;\vec{\beta}_{j,i}), \quad \text{and} \quad (\Phi_j)_+(z)=\prod_{\substack{i=1 \\ \leftarrow}}^\ell \phi^b(z;\vec{\alpha}_{j,i},\vec{\gamma}_{j,i}),
\end{equation}
for some vectors~$\vec{\alpha}_{j,i}$,~$\vec{\beta}_{j,i}$ and~$\vec{\gamma}_{j,i}$. 

In principle, one can compute the new constants,~$\vec{\alpha}_{j,i}$,~$\vec{\beta}_{j,i}$ and~$\vec{\gamma}_{j,i}$ using~\eqref{eq:update_rule}. However, in general it seems difficult to obtain a suitable expression as~$j$ tends to infinity, and that is not the path we will take. As we will show in the subsequent sections, it is enough to know the structure of the matrices in~$(\Phi_j)_+$ and~$(\Phi_j)_-$.

To summarize, we have the factorization
\begin{equation}\label{eq:prefactorization}
\phi=\left(\prod_{j=0}^{kN-1}(\Phi_j)_-\right)\left(\prod_{\substack{j=0 \\ \leftarrow}}^{kN-1}(\Phi_j)_+\right),
\end{equation}
where~$(\Phi_j)_+$ and~$(\Phi_j)_-$ are given by~\eqref{eq:phi_+_phi_-}. The factor~$\prod_{j=0}^{kN-1}(\Phi_j)_-$ is analytic and non-singular in the complement of the unit disk, and~$\prod_{\substack{j=0 \\ \leftarrow}}^{kN-1}(\Phi_j)_+$ is analytic and non-singular in the closed unit disk away from zero. What remains is to normalize the factors at zero and infinity.

Lemma~\ref{lem:transition_matrices_properties}~\eqref{eq:bernoulli_zero} implies that~$z^{\ell N}\prod_{\substack{j=0 \\ \leftarrow}}^{kN-1}(\Phi_j)_+$ is analytic and non-singular at zero, and item~\eqref{eq:geometric_infty} of the same lemma implies that there exists a constant invertible matrix~$C$ such that
\begin{equation}
z^{-\ell N}\left(\prod_{j=0}^{kN-1}(\Phi_j)_-\right)C^{-1}=z^{-\ell N}(I+\Ordo(z^{-1})), \quad z \to \infty.
\end{equation}
Hence, the factors of the desired Wiener--Hopf factorization~$\phi=\widetilde \phi_-\widetilde \phi_+$ are given by
\begin{equation}\label{eq:wiener_+}
\widetilde \phi_+(z)=z^{\ell N}C\left(\prod_{\substack{j=0 \\ \leftarrow}}^{kN-1}\Phi_j(z)_+\right),
\end{equation}
and
\begin{equation}\label{eq:wiener_-}
\widetilde \phi_-(z)=z^{-\ell N}\left(\prod_{j=0}^{kN-1}\Phi_j(z)_-\right)C^{-1}.
\end{equation}

\subsection{Eigenvectors and the Wiener--Hopf factorization}\label{sec:wh_eigenvectors}

As mentioned above, keeping track of the constants~$\vec{\alpha}_{j,i}$,~$\vec{\beta}_{j,i}$ and~$\vec{\gamma}_{j,i}$ defined by iterating~\eqref{eq:update_rule} seems, in general, to be a very difficult problem. Instead, we use the construction of the previous section to express the Wiener--Hopf factorization in terms of the eigenvectors of~$\Phi_j$, viewed as functions on the spectral curve. In the following section we will then describe the zeros and poles of the eigenvectors, which is sufficient to recover the Wiener--Hopf factorization.

Recall that~$\Phi_j=(\Phi_j)_-(\Phi_j)_+$ and, by definition,~$\Phi_{j+1}=(\Phi_j)_+(\Phi_j)_-$. Hence,
\begin{equation}
\Phi_{j+1}=(\Phi_j)_+\Phi_j(\Phi_j)_+^{-1}=(\Phi_j)_-^{-1}\Phi_j(\Phi_j)_-.
\end{equation}
These equalities have a direct, yet still essential consequence. Namely, if~$\psi_{j,+}(z,w)$ and~$\psi_{j,-}(z,w)$ are right and left nullvectors of~$\Phi_j(z)-wI$, respectively, then 
\begin{equation}\label{eq:eigenvector_j}
\psi_{j+1,+}(z,w)=(\Phi_j)_+(z)\psi_{j,+}(z,w) \quad \text{and} \quad \psi_{j+1,-}(z,w)=\psi_{j,-}(z,w)(\Phi_j)_-(z)
\end{equation}
are right and left nullvectors of~$\Phi_{j+1}(z)-wI$. We emphasize that~$\psi_{j,+}$ is a column vector while~$\psi_{j,-}$ is a row vector. The definitions in~\eqref{eq:eigenvector_j} lead us to the equalities 
\begin{equation}\label{eq:eigenvector_right_n}
\psi_{k N,+}(z,w)=\left(\prod_{\substack{j=0 \\ \leftarrow}}^{k N-1} \Phi_j(z)_+\right)\psi_{0,+}(z,w),
\end{equation}
and
\begin{equation}\label{eq:eigenvector_left_n}
\psi_{k N,-}(z,w)=\psi_{0,-}(z,w)\left(\prod_{j=0}^{k N-1} \Phi_j(z)_-\right).
\end{equation}
The products on the right hand sides are the factors in~\eqref{eq:prefactorization}, which we know are, up to a factor, equal to~$\widetilde \phi_\pm$. Hence, by multiplying~\eqref{eq:eigenvector_right_n} by~$z^{\ell N}C$ from the left and~\eqref{eq:eigenvector_left_n} by~$z^{-\ell N}C^{-1}$ from the right, and comparing with~\eqref{eq:wiener_+} and~\eqref{eq:wiener_-}, we have shown the following.
\begin{theorem}\label{thm:wiener-hopf}
Let~$\phi=\Phi^{kN}$ with~$\Phi$ given in~\eqref{eq:phi_recall}. Assume~$\psi_{0,+}$ and~$\psi_{0,-}$ are right and left nullvectors of~$\Phi(z)-wI$, respectively, that is,
\begin{equation}
\left(\Phi(z)-wI\right)\psi_{0,+}(z,w)=0, \quad \text{and} \quad \psi_{0,-}(z,w)\left(\Phi(z)-wI\right)=0.
\end{equation}
If~$\psi_{kN,+}$ and~$\psi_{kN,-}$ are obtained via the recursion~\eqref{eq:eigenvector_j}, then
\begin{equation}
z^{\ell N}C\psi_{k N,+}(z,w)=\widetilde \phi_+(z)\psi_{0,+}(z,w),
\end{equation}
and
\begin{equation}
\psi_{k N,-}(z,w)z^{-\ell N}C^{-1}=\psi_{0,-}(z,w)\widetilde \phi_-(z),
\end{equation}
where~$C$ is an invertible constant matrix.
\end{theorem}

The equalities in the previous theorem are naturally equalities on the spectral curve~$\mathcal R$, otherwise~$\psi_{j,\pm}$ must be trivial. Knowing~$\psi_{0,\pm}$ and~$\psi_{kN,\pm}$ as meromorphic vectors on~$\mathcal R$ is enough to recover~$\widetilde \phi_\pm$. We will see in Section~\ref{sec:asymptotic} that when writing the integral in Theorem~\ref{thm:bd_thm} as an integral on~$\mathcal R$, it is the products~$\widetilde \phi_+\psi_{0,+}$ and~$\psi_{0,-}\widetilde \phi_-$ that will be relevant for us. Understanding~$\psi_{0,\pm}$ and~$\psi_{kN,\pm}$ is the content of Sections~\ref{sec:linear_flow} and~\ref{sec:one_form_thetas}.

\subsection{The linear flow}\label{sec:linear_flow}
The first goal of this section is to define the vectors~$\psi_{0,+}$ and~$\psi_{0,-}$ and to determine their zeros and poles on~$\mathcal R$. These vectors will be defined in terms of the adjugate~$\adj (\Phi(z)-wI)$, which we will then express in terms of~$\adj K_{G_1}(z,w)$ using Lemma~\ref{lem:phi_kasteleyn}. The zeros and poles of~$\adj K_{G_1}(z,w)$ can then be computed using results of~\cite{BCT22}. The second goal is to determine the zeros and poles of~$\psi_{kN,+}$ and~$\psi_{kN,-}$. This is done using the relation~\eqref{eq:eigenvector_j}.

For~$(z,w)\in \mathcal R$ we define the column and row vectors of differential~$1$-forms~$\psi_{0,+}$ and~$\psi_{0,-}$ by
\begin{equation}\label{eq:eigenvector_right}
\psi_{0,+}(z,w)=\frac{\adj \left(\Phi(z)-wI\right)J_1^T\d z}{z\partial_w \det (\Phi(z)-wI)}
\end{equation}
and
\begin{equation}\label{eq:eigenvector_left}
\psi_{0,-}(z,w)=\frac{J_1\phi_1(z)^{-1}\adj \left(\Phi(z)-wI\right)\d z}{z\partial_w \det (\Phi(z)-wI)},
\end{equation}
where~$J_1=(1,0,\dots,0)$. By definition of the adjugate matrix
\begin{equation}
\adj(\Phi(z)-wI)(\Phi(z)-wI)=(\Phi(z)-wI)\adj(\Phi(z)-wI)=\det(\Phi(z)-wI)I=0,
\end{equation}
if~$(z,w)\in \mathcal R$. This shows that~$\psi_{0,+}$ and~$\psi_{0,-}$ are right and left nullvectors of~$(\Phi(z)-wI)$, respectively. The vectors are meromorphic differentials. The reason for this particular form is that, using the relation in Lemma~\ref{lem:phi_kasteleyn}, the vectors fit into the framework of~\cite{BCT22}.

Recall the definition of~$\Theta$ and the prime form~$E$ in Section~\ref{sec:abel_theta}, and also the definition of the angles in~\eqref{eq:angles_1} and~\eqref{eq:angles_2}.
\begin{lemma}\label{lem:kasteleyn_thetas}
Let~$K_{G_1}$ be the Kasteleyn matrix defined by~\eqref{eq:magnetic_kasteleyn_matrix} and let~$q=(z,w)\in \mathcal R$. For~$i,i'=0,\dots,\ell-1$ and~$j,j'=0,\dots,k-1$, there exist~$e_{\mathrm{b}_{i,j}},e_{\mathrm{w}_{i',j'}}\in \RR^g$ such that
\begin{equation}
\frac{\adj K_{G_1}(z,w)_{\mathrm b_{i,j}\mathrm w_{i',j'}}\d z}{wzP_w(z,w)}
=c_{iji'j'}\Theta\left(q;e_{\mathrm{b}_{i,j}}\right)\Theta\left(q;e_{\mathrm w_{i',j'}}\right)
\frac{\prod_{m=i+1}^{i'}E(q_{\infty,m},q)}{\prod_{m=i+1}^{i'+1}E(q_{0,m},q)}
\frac{\prod_{m=j+1}^{j'}E(p_{0,m},q)}{\prod_{m=j+1}^{j'+1}E(p_{\infty,m},q)},
\end{equation}
for some constant~$c_{iji'j'}\in \CC^*$.
\end{lemma}
Note that the factors in the right hand side of the equality in the previous lemma are not well-defined functions on~$\mathcal R$. The entire product, however, is a well-defined~$1$-form on~$\mathcal R$\footnote{This will be shown in Remark~\ref{rem:meromorphic_differential} below independently of the left side; an explicit form of~$e_{\mathrm{b}_{i,j}}$ and~$e_{\mathrm{w}_{i',j'}}$ will play a role.} and we therefore use~$q=(z,w)$ in the arguments. We will adopt this convention going forward, that is, to use~$q=(z,w)\in \mathcal R$ in the argument when the entire expression is a well-defined~$1$-form or function on~$\mathcal R$.

The lemma follows from~\cite{BCT22} and the proof is given in the following section. We also include an independent argument for~$k=2$ in Appendix~\ref{app:alt_proofs}. 

We are now ready to determine the zeros and poles of the nullvectors~$\psi_{0,+}$ and~$\psi_{0,-}$.
\begin{lemma}\label{lem:eigenvector_zeros_poles}
Let~$\psi_{0,+}$ and~$\psi_{0,-}$ be defined by~\eqref{eq:eigenvector_right} and~\eqref{eq:eigenvector_left}, and denote their entries by~$\psi_{0,\pm}^{(j+1)}$. For~$j=0,\dots,k-1$, there exist~$e_{\mathrm b}, e_{\mathrm w}, e_{\mathrm b_{0,j}}^{(0)}, e_{\mathrm w_{0,j}}^{(0)}\in\RR^g$ such that
\begin{equation}
\psi_{0,+}^{(j+1)}(z,w)=
c_{j,+}^{(0)}\Theta\left(q;e_{\mathrm{b}}\right)\Theta\left(q;e_{\mathrm w_{0,j}}^{(0)}\right)
\frac{\prod_{m=1}^{j-1}E(p_{0,m},q)}{\prod_{m=1}^{j+1}E(p_{\infty,m},q)},
\end{equation}
and
\begin{equation}
\psi_{0,-}^{(j+1)}(z,w)
=c_{j,-}^{(0)}\Theta\left(q;e_{\mathrm{b}_{0,j}}^{(0)}\right)\Theta\left(q;e_{\mathrm w}\right)
\frac{1}{E(q_{0,1},q)}
\frac{\prod_{m=2}^jE(p_{\infty,m},q)}{\prod_{m=1}^jE(p_{0,m},q)},
\end{equation}
for some constants~$c_{j,+}^{(0)},c_{j,-}^{(0)} \in \CC^*$.
\end{lemma}

\begin{proof}
Observe that the definition of~$\psi_{0,-}$~\eqref{eq:eigenvector_left} together with Lemma~\ref{lem:phi_kasteleyn} implies that
\begin{equation}\label{eq:eigenvector_as_kasteleyn}
\psi_{0,-}^{(j+1)}(z,w)=\frac{\left(\phi_1(z)^{-1}\adj(\Phi(z)-wI)\right)_{1,j+1}\d z}{z\partial_w\det(\Phi(z)-wI)}=\frac{\adj K_{G_1}(z,w)_{\mathrm b_{0,j}\mathrm w_{0,0}}\d z}{wzP_w(z,w)}.
\end{equation}
Hence, the desired formula for~$\psi_{0,-}$ follows from Lemma~\ref{lem:kasteleyn_thetas}, with~$c_{j,-}^{(0)}=c_{0j00}$,~$e_{\mathrm b_{0,j}}^{(0)}=e_{\mathrm b_{0,j}}$, and~$e_{\mathrm w}=e_{\mathrm w_{0,0}}$. Recall also the product convention from Remark~\ref{rem:products}.

To prove the expression for~$\psi_{0,+}$ we cannot rely on Lemma~\ref{lem:kasteleyn_thetas} directly. Instead, we introduce the column vector 
\begin{equation}
\psi_{-1,+}(z,w)=\frac{(\adj K_{G_1}(z,w)_{\mathrm b_{0,0}\mathrm w_{0,j}})_{j=0}^{k-1}\d z}{wzP_w(z,w)}.
\end{equation}
The zeros and poles of~$\psi_{-1,+}$ are determined in Lemma~\ref{lem:kasteleyn_thetas}, indeed, the~$(j+1)$st entry is given by
\begin{equation}\label{eq:eigenvalue_-1}
\psi_{-1,+}^{(j+1)}(z,w)
=c_{000j}\Theta\left(q;e_{\mathrm{b}_{0,0}}\right)\Theta\left(q;e_{\mathrm w_{0,j}}\right)
\frac{1}{E(q_{0,1},q)}
\frac{\prod_{i=1}^{j}E(p_{0,i},q)}{\prod_{i=1}^{j+1}E(p_{\infty,i},q)}.
\end{equation}
Similarly to~\eqref{eq:eigenvector_as_kasteleyn}, it follows from Lemma~\ref{lem:phi_kasteleyn} that~$\psi_{0,+}=\phi_1\psi_{-1,+}$ where~$\phi_1$ is defined by~\eqref{eq:def_phi_bernoulli}. This implies that
\begin{equation}
\Phi_{-1}(z)\psi_{-1,+}(z,w)=w\psi_{-1,+}(z,w),
\end{equation}
where~$\Phi_{-1}=\phi_1^{-1}\Phi_0\phi_1$, and~$\Phi_0=\Phi$. We will see that multiplying~$\psi_{-1,+}$ with~$\phi_1$ introduces a zero at~$q_{0,1}$ and a pole of the~$(j+1)$st entry at~$p_{0,j}$, or in other words, it removes the pole of~\eqref{eq:eigenvalue_-1} at~$q_{0,1}$ and the zero at~$p_{0,j}$. Moreover, the~$g$ zeros coming from the factor~$\Theta\left(q;e_{\mathrm w_{0,j}}\right)$ will move according to a relation explained below. 

We first show that the pole at~$q_{0,1}$ is removed. Let~$z_1=(-1)^k\alpha^v_1/\gamma^v_1$, where~$\alpha^v_1$ and~$\gamma^v_1$ are defined by~\eqref{eq:edge_weights_product}, so that~$q_{0,1}=(z_1,0)$. We also let~$\psi_{-1,+}^*=a \psi_{-1,+}$, where~$a$ is a meromorphic (scalar) function in a neighborhood of~$q_{0,1}$ which cancels out the pole of~$\psi_{-1,+}$ at~$q_{0,1}$. Since~$\psi_{-1,+}^*$ is a right null vector of~$\Phi_{-1}(z)-wI$, we get
\begin{equation}\label{eq:eigenvector_zeros_angle}
\Phi_{-1}(z_1)\psi_{-1,+}^*(q_{0,1})
=\left.\left(\Phi_{-1}(z)-wI\right)\right|_{(z,w)=q_{0,1}}\psi_{-1,+}^*(q_{0,1})=0.
\end{equation}
Recall Lemma~\ref{lem:transition_matrices_properties}~\eqref{eq:determinant_bern_geo} which implies that~$\det \phi_1(z_1)=0$ while~$\det \phi_i(z_1)\neq 0$ for~$i=2,\dots,\ell$, since the angles are distinct due to Assumption~\ref{ass:main_ass}\eqref{ass:non-singularity}. Thus,~$\Phi_{-1}\phi_1^{-1}=\prod_{i=2}^{2\ell}\phi_i$ is invertible at~$z_1$, and by multiplying~\eqref{eq:eigenvector_zeros_angle} by~$(\Phi_{-1}\phi_1^{-1})^{-1}$ we obtain that
\begin{equation}
\phi_1(z_1)\psi_{-1,+}^*(q_{0,1})=0.
\end{equation}
Since, by definition,~$\psi_{-1,+}^*(q_{0,1})\neq 0$, this means that multiplying~$\psi_{-1,+}$ by~$\phi_1$ removes the pole of~$\psi_{-1,+}$ at~$q_{0,1}$.

The final form of~$\psi_{0,+}$ now follows by an explicit computation, where we keep track of the number of zeros and poles. Recall that~$\phi_1$ is given by~\eqref{eq:def_phi_bernoulli} and has the form~\eqref{eq:bernoulli}. Thus, for~$j=1,\dots,k-1$, 
\begin{multline}\label{eq:eigenvector_poles_infty}
\psi_{0,+}^{(j+1)}(z,w)=\left(\phi_1\psi_{-1,+}\right)^{(j+1)}(z,w)=\alpha_{j,1}\psi_{-1,+}^{(j)}(z,w)+\gamma_{j+1,1}\psi_{-1,+}^{(j+1)}(z,w)\\
=\left(\alpha_{j,1}c_{000(j-1)}\Theta\left(q;e_{\mathrm w_{0,j-1}}\right)E(p_{\infty,j+1},q)
+\gamma_{j+1,1}c_{000j}\Theta\left(q;e_{\mathrm w_{0,j}}\right)E(p_{0,j},q)\right) \\
\times \frac{\Theta\left(q;e_{\mathrm{b}}\right)}{E(q_{0,1},q)}
\frac{\prod_{i=1}^{j-1}E(p_{0,i},q)}{\prod_{i=1}^{j+1}E(p_{\infty,i},q)},
\end{multline}
where~$e_{\mathrm b}=e_{\mathrm b_{0,0}}$. Recall~\eqref{eq:abels_thm_diff}; that is, the degree of a meromorphic differential~$1$-form, the number of zeros minus the number of poles, is~$2g-2$. The degree of the factor in the last row of~\eqref{eq:eigenvector_poles_infty} is~$g-3$, which tells us that the degree of the factor on the second row is~$g+1$. Since this factor has no poles we conclude that it has~$g+1$ zeros. We already know that~$q_{0,1}$ is one of these~$g+1$ zeros. Hence, 
\begin{equation}\label{eq:eigenvector_zeros}
\psi_{0,+}^{(j+1)}(z,w)
=c_{j,+}^{(0)}\Theta(q;e_{\mathrm w_{0,j}}^{(0)})E(q_{0,1},q)
\frac{\Theta\left(q;e_{\mathrm{b}}\right)}{E(q_{0,1},q)}
\frac{\prod_{i=1}^{j-1}E(p_{0,i},q)}{\prod_{i=1}^{j+1}E(p_{\infty,i},q)},
\end{equation}
for some constants~$c_{j+}^{(0)}$ and~$e_{\mathrm w_{0,j}}^{(0)}$. As explained in Section~\ref{sec:abel_theta}, the constant~$e_{\mathrm w_{0,j}}^{(0)}$ is defined from the final~$g$ zeros of the left hand side of~\eqref{eq:eigenvector_zeros} through~\eqref{eq:jacobi_inverse}. A similar computation also holds when~$j=0$.

What remains is to show that~$e_{\mathrm w_{0,j}}^{(0)}\in \RR^g$. By~\eqref{eq:abels_thm_diff}, the divisors~$(\psi_{0,+}^{(j+1)})$ and~$(\psi_{-1,+}^{(j+1)})$ are mapped by the Abel map to the same point, that is,~$u\left((\psi_{0,+}^{(j+1)})\right)=2\Delta=u\left((\psi_{-1,+}^{(j+1)})\right)$. In particular, this implies that 
\begin{equation}
e_{\mathrm w_{0,j}}^{(0)}+(u(q_{0,1})-u(p_{0,j}))=e_{\mathrm w_{0,j}}
\end{equation} 
in the Jacobi variety~$J(\mathcal R$), which, together with~\eqref{eq:positive_differentials}, shows that we can take~$e_{\mathrm w_{0,j}}^{(0)}\in \RR^g$.
\end{proof}

The previous lemma gives us an expression of~$\psi_{0,\pm}$ in terms of theta functions and prime forms. We proceed with a similar statement for~$\psi_{m,\pm}$, for all~$m$. In Section~\ref{sec:asymptotic} we will see that this expression is suitable for asymptotic analysis. Recall that~$\psi_{m,+}$ and~$\psi_{m,-}$ are defined recursively via~\eqref{eq:eigenvector_j}. 

\begin{proposition}\label{prop:zeros_poles}
For each~$m=0,1\dots$, and~$j=0,\dots,k-1$, there exist~$e_{\mathrm b}, e_{\mathrm w}, e_{\mathrm b_{0,j}}^{(m)}, e_{\mathrm w_{0,j}}^{(m)}\in\RR^g$ such that the~$(j+1)$st entry~$\psi_{m,\pm}^{(j+1)}$ of~$\psi_{m,\pm}$ are given by
\begin{equation}
\psi_{m,+}^{(j+1)}(z,w)=
c_{j,+}^{(m)}\Theta\left(q;e_{\mathrm{b}}\right)\Theta\left(q;e_{\mathrm w_{0,j}}^{(m)}\right)
\frac{\prod_{i=1}^{\ell m} E(q_{0,i},q)}{\prod_{i=1}^{\ell m} E(p_{0,j-i},q)}
\frac{\prod_{i=1}^{j-1}E(p_{0,i},q)}{\prod_{i=1}^{j+1}E(p_{\infty,i},q)},
\end{equation}
and
\begin{equation}
\psi_{m,-}^{(j+1)}(z,w)
=c_{j,-}^{(m)}
\frac{\Theta\left(q;e_{\mathrm{b}_{0,j}}^{(m)}\right)\Theta\left(q;e_{\mathrm w}\right)}{E(q_{0,1},q)}
\frac{\prod_{i=1}^{\ell m} E(p_{0,j+1-i},q)}{\prod_{i=1}^{\ell m} E(q_{\infty,i},q)}
\frac{\prod_{i=2}^jE(p_{\infty,i},q)}{\prod_{i=1}^jE(p_{0,i},q)}.
\end{equation}
for some constants~$c_{j,+}^{(m)}, c_{j,-}^{(m)}\in \CC^*$. The indices of the angles~$p_{0/\infty,i}$ and~$q_{0/\infty,i}$ are taken modulo~$k$ and~$\ell$, respectively.
\end{proposition}
\begin{remark}\label{rem:f_eigenvector}
For~$m=kN$, which is relevant for our purposes, we have, for~$\tilde q \in \widetilde{\mathcal R}$, 
\begin{equation}
\frac{\prod_{i=1}^{\ell kN} E(\tilde q_{0,i},\tilde q)}{\prod_{i=1}^{\ell kN} E(\tilde p_{0,j-i},\tilde q)}=f(\tilde q)^N,
\end{equation}
where~$f$ is as in Definition~\ref{def:action_function}, and it follows from Lemma~\ref{lem:parametrization_curve} bellow that
\begin{equation}
\frac{\prod_{i=1}^{\ell kN} E(\tilde p_{0,j+1-i},\tilde q)}{\prod_{i=1}^{\ell kN} E(\tilde q_{\infty,i},\tilde q)}=f(\tilde q)^{-N}w^{kN}.
\end{equation}
\end{remark}

\begin{proof}[Proof of Proposition~\ref{prop:zeros_poles}]
The proof goes by induction with Lemma~\ref{lem:eigenvector_zeros_poles} as the base case. The proof of the induction step is similar to the proof of the mentioned lemma, so we will be brief with details.

Recall that~$\psi_{m+1,+}=(\Phi_m)_+\psi_{m,+}$ and~$\Phi_m=(\Phi_m)_-(\Phi_m)_+$, where~$(\Phi_m)_+$ and~$(\Phi_m)_-$ are defined in~\eqref{eq:phi_+_phi_-}. Using the definitions of~$(\Phi_m)_+$ and~$(\Phi_m)_-$ together with~\eqref{eq:invariant_determinant} and~\eqref{eq:determinant_bern_geo} of Lemma~\ref{lem:transition_matrices_properties} we obtain
\begin{equation}
\det \left((\Phi_m)_+(z)\right)=\prod_{i=1}^\ell (\gamma^v_i-(-1)^k\alpha^v_i z^{-1})\quad \text{and} \quad 
\det \left((\Phi_m)_-(z)\right)=\prod_{i=1}^\ell \frac{1}{1-\beta^v_i z^{-1}},
\end{equation}
where~$\alpha^v_i$,~$\beta^v_i$ and~$\gamma^v_i$ are defined in~\eqref{eq:edge_weights_product}. In particular, if~$z_i=(-1)^k\alpha^v_i/\gamma^v_i$ (recall that~$\beta_i^v<1<\alpha_i^v/\gamma_i^v$) then~$\det \left((\Phi_m)_-(z_i)\right) \neq 0$ while~$\det\left((\Phi_m)_+(z_i)\right)=0$. A similar argument as in the proof of Lemma~\ref{lem:eigenvector_zeros_poles}, see the discussion around~\eqref{eq:eigenvector_zeros_angle}, now shows that multiplying~$\psi_{m,+}$ by~$(\Phi_m)_+$ increases the multiplicity of the zero at the angle~$q_{0,i}=(z_i,0)$ by one, for all~$i=1,\dots,\ell$. By an explicit computation, similar to~\eqref{eq:eigenvector_poles_infty} in the proof of Lemma~\ref{lem:eigenvector_zeros_poles}, using that each factor in~$(\Phi_m)_+$ has the same form as~\eqref{eq:bernoulli}, we obtain the extra poles of~$\psi_{m+1,+}^{(j+1)}$, as compared to~$\psi_{m,+}^{(j+1)}$, at~$p_{0,j-i}$ for~$i=\ell m+1,\dots,\ell (m+1)$. Finally, we obtain how the zeros from the factor~$\Theta\left(q;e_{\mathrm w_{0,j}}^{(m)}\right)$ are moving, using the fact that the Abel map takes the divisors~$(\psi_{m+1,+}^{(j+1)})$ and~$(\psi_{m,+}^{(j+1)})$ to the same point. It implies that
\begin{equation}\label{eq:linear_flow}
e_{\mathrm w_{0,j}}^{(m+1)}=e_{\mathrm w_{0,j}}^{(m)}-\left(\sum_{i=1}^\ell u(q_{0,i})-\sum_{i=\ell m+1}^{\ell (m+1)}u(p_{0,j-i})\right),
\end{equation} 
as vectors in~$J(\mathcal R)$, and, hence,~$e_{\mathrm w_{0,j}}^{(m+1)}\in \RR^g$.

Similarly, we obtain the zeros and poles of~$\psi_{m+1,-}$ by using the relation~$\psi_{m+1,-}=\psi_{m,-}(\Phi_m)_-$. Recall that the factors of~$(\Phi_m)_-$ are of type~\eqref{eq:geometric} and have poles at~$z_i=\beta^v_i$,~$i=1,\dots,\ell$. Multiplying~$\psi_{m,-}$ by~$(\Phi_m)_-$ may therefore introduce poles at~$(z_i,w_i)$ for all~$w_i\in \CC$ such that~$(z_i,w_i)\in \mathcal R$, with~$z_i=\beta^v_i$. However, since~$\psi_{m,-}(z,w)$ is a left nullvector of~$\Phi_m(z)-wI$ and~$\Phi_m=(\Phi_m)_-(\Phi_m)_+$, we obtain
\begin{equation}
\psi_{m,-}(z,w)(\Phi_m)_-(z)=\psi_{m,-}(z,w)w(\Phi_m)_+(z)^{-1},
\end{equation}
where~$(\Phi_m)_+^{-1}$ is analytic at~$z_i$. Hence, the only poles that are introduced are at~$q_{\infty,i}$ for~$i=1,\dots,\ell$. Recall that~$(\Phi_m)_-$ is a product of~$\ell$ factors and each is of the form~\eqref{eq:geometric}, see~\eqref{eq:phi_+_phi_-}. Using the explicit expression of~\eqref{eq:geometric}, in particular, that the leading term as~$z\to 0$ is upper triangular with zeros along the diagonal, we obtain, in the spirit of~\eqref{eq:eigenvector_poles_infty}, that multiplying~$\psi_{m,-}$ with~$n=1,\dots,\ell$ such factors from the right, the product as~$q\to p_{0,i'}$ for some~$i'$ behaves as
\begin{equation}
\left(\prod_{i=1}^{\ell m+n} E(p_{0,1-i},q),\frac{\prod_{i=1}^{\ell m+n} E(p_{0,2-i},q)}{E(p_{0,1},q)},\dots,\frac{\prod_{i=1}^{\ell m+n} E(p_{0,k-i},q)}{\prod_{i=1}^{k-1}E(p_{0,i},q)}\right).
\end{equation}
For example, for~$n=1$ and~$j=1,\dots,k-1$ the behavior of the product is
\begin{equation}
\Ordo\left(\frac{\prod_{i=1}^{\ell m} E(p_{0,j-i},q)}{\prod_{i=1}^{j-1}E(p_{0,i},q)}\right)=\Ordo\left(\frac{\prod_{i=1}^{\ell m+1} E(p_{0,j+1-i},q)}{\prod_{i=1}^jE(p_{0,i},q)}\right),
\end{equation}
and if~$j=0$ it is
\begin{equation}
\Ordo\left(z\frac{\prod_{i=1}^{\ell m} E(p_{0,k-i},q)}{\prod_{i=1}^{k-1}E(p_{0,i},q)}\right)=\Ordo\left(\prod_{i=0}^{\ell m} E(p_{0,k-i},q)\right)=\Ordo\left(\prod_{i=1}^{\ell m+1} E(p_{0,1-i},q)\right).
\end{equation}
In particular, with~$n=\ell$ we obtain the extra zeros of~$\psi_{m+1,-}^{(j+1)}$, as compared to~$\psi_{m,-}^{(j+1)}$, at~$p_{0,j+1-i}$ for~$i=\ell m+1,\dots, \ell(m+1)$. Moreover, the zeros coming from the factor~$\Theta\left(q;e_{\mathrm{b}_{0,j}}^{(m)}\right)$ are changing according to 
\begin{equation}
e_{\mathrm b_{0,j}}^{(m+1)}=e_{\mathrm b_{0,j}}^{(m)}-\left(\sum_{i=\ell m+1}^{\ell (m+1)}u(p_{0,j+1-i})-\sum_{i=1}^\ell u(q_{\infty,i})\right)
\end{equation} 
in~$J(\mathcal R)$, which shows that~$e_{\mathrm b_{0,j}}^{(m+1)}\in \RR^g$.
\end{proof}
\begin{remark}\label{rem:linear_flow}
In the proof of Proposition~\ref{prop:zeros_poles} we also obtained how~$e_{\mathrm w_{0,j}}^{(m)}$ evolves with~$m$, see~\eqref{eq:linear_flow}. In particular, the difference over~$k$ steps is 
\begin{equation}
e_{\mathrm w_{0,j}}^{(m+k)}-e_{\mathrm w_{0,j}}^{(m)}=\ell\sum_{i=1}^ku(p_{0,i})-k\sum_{i=1}^\ell u(q_{0,i})
\end{equation} 
as an equality in~$J(\mathcal R)$, and the total flow is given by
\begin{equation}
e_{\mathrm w_{0,j}}^{(kN)}-e_{\mathrm w_{0,j}}^{(0)}=N\left(\ell\sum_{i=1}^ku(p_{0,i})-k\sum_{i=1}^\ell u(q_{0,i})\right).
\end{equation} 
We recognize this flow from Section~\ref{sec:limit_shape_smooth}. Indeed, with~$H$ as in Corollary~\ref{cor:height_function_smooth} and~$(u_n,v_n)$ in the~$n$th smooth region,~$n=1,\dots,g$, 
\begin{equation}\label{eq:shift_height_function}
e_{\mathrm w_{0,j}}^{(kN)}-e_{\mathrm w_{0,j}}^{(0)}=-N\big(H(u_1,v_1),\dots,H(u_g,v_g)\big) \!\!\mod \ZZ^g.
\end{equation} 
In a similar way we obtain that
\begin{equation}
e_{\mathrm b_{0,j}}^{(kN)}-e_{\mathrm b_{0,j}}^{(0)}=N\big(H(u_1,v_1),\dots,H(u_g,v_g)\big) \!\!\mod \ZZ^g.
\end{equation} 
Recall that the right hand side is the leading order term of the unnormalized height function minus the linear term.
\end{remark}
\begin{remark}\label{rem:linear_flow_discussion}
The relation~\eqref{eq:shift_height_function} is rather surprising to us, and we will not attempt to fully explain it. However, we point out that both the left and right hand sides have relations to the domino-shuffling algorithm.
 
By the construction, the left hand side of~\eqref{eq:shift_height_function} is related to the switching described in Section~\ref{sec:switching}. In~\cite{CD22} the update rule of the weights~\eqref{eq:update_rule} was identified with the change of the weights under the domino-shuffling algorithm.

In~\cite{CT21} the stationary speed of growth of the stochastic process defined by the domino-shuffling algorithm was expressed in terms of the limit shape of the Aztec diamond, see~\cite[Equation (2.16)]{CT21}. In fact, up to matching the different conventions,~$H$ is the stationary speed of growth. However, their setting is slightly different from ours, their fundamental domain is rotated by~$\pi/4$ compared with the fundamental domain in this paper, and we will therefore not make this relation precise. 
\end{remark}

\subsection{Proof of Lemma~\ref{lem:kasteleyn_thetas}}\label{sec:one_form_thetas}

The proof relies on the results in~\cite{BCT22} (which, in their turn, partially rely on those of~\cite{Foc15, GK13, KO06}). We therefore begin by recalling a number of definitions from~\cite{BCT22}, see, in particular, Sections 3 and 4 therein. These definitions will only be used in this section.

Recall the definition of~$G$ in Section~\ref{sec:dimer}. Let~$G^{\quadgraph}$ be the \emph{quad-graph}, the graph consisting of the disjoint union of the vertices from~$G$ and from the dual graph of~$G$, and edges between a primary vertex~$\mathrm{v}$ and a dual vertex~$\mathrm{f}$ such that~$\mathrm{v}$ lies on the boundary of the face corresponding to~$\mathrm{f}$. A \emph{train-track} of~$G$ is a dual path, a path that joins adjacent faces, of~$G^{\quadgraph}$ with no endpoints, which enters and exits a face at opposite edges. We introduce an orientation of the train-tracks so that they have the black vertices of~$G$ to the right and white vertices to the left. We denote the set of all oriented train-tracks by~$\mathcal T$. See Figure~\ref{fig:quad_graph}.

Recall from Section~\ref{sec:spectral_curve} that~$G_1$ denotes the fundamental domain of~$G$ embedded in the torus and that~$\gamma_u$ and~$\gamma_v$ are two closed loops on the torus. We denote the homology classes of~$\gamma_u$ and~$\gamma_v$ by~$[\gamma_u]$ and~$[\gamma_v]$. Let~$\mathcal T_1$ be the set of train-tracks of~$G$ mapped to~$G_1$. Given a train-track~$T\in \mathcal T_1$, we define~$h_T$ and~$v_T$ as the coordinates of the homology class of~$T$ in the basis~$\{[\gamma_u],[\gamma_v]\}$,~$[T]=(h_T,v_T)$. For all~$T\in \mathcal T_1$ we have~$[T]\in \{(0,\pm 1),(\pm 1,0)\}$, see Figure~\ref{fig:quad_graph}. We proceed by enumerating the train-tracks and associating an angle to each train-track. 

Recall the definition of the product of the edge weights in~\eqref{eq:edge_weights_product} and the definition of the angles in~\eqref{eq:angles_1} and~\eqref{eq:angles_2}. Let~$T\in \mathcal T_1$ with~$[T]=(1,0)$. The alternating product of the signed edge weights along~$T$, that is, the product of the edge weights on the south edges along~$T$ divided by the edge weights on the west edges along~$T$, is equal to~$\alpha^h_i/\beta^h_i$, for some~$i=1,\dots,k$. We denote~$T$ by~$T_i$ and assign the angle~$p_{0,i}$ to~$T_i$. Recall that~$p_{0,i}=(0,(-1)^\ell\alpha^h_i/\beta^h_i)$. In a similar way we enumerate the train-tracks with homology class~$(0,1)$ by~$T_{k+i}$,~$i=1,\dots,\ell$, and associate the angles~$q_{0,i}$ to them, the train-tracks with homology class~$(-1,0)$ by~$T_{k+\ell+i}$,~$i=1,\dots,k$, with angles~$p_{\infty,i}$, and finally we enumerate the train-tracks with homology class~$(0,-1)$ by~$T_{2k+\ell+i}$,~$i=1,\dots,\ell$ and associate the angles~$q_{\infty,i}$ to those. See Figure~\ref{fig:quad_graph}. We will occasionally use the notation~$a_{T_i}$,~$i=1,\dots,2(k+\ell)$, for the angle associated to the train-track~$T_i$. We denote a train-track in~$\mathcal T$ which maps to~$T \in \mathcal T_1$ also by~$T$ and associate to it the angle~$a_{T}$.

The \emph{discrete Abel map}~$\mathbf d$ with values in the Jacobian of~$\mathcal R$ is defined on the vertices of~$G^{\quadgraph}$ by setting~$\mathbf d(\mathrm f_0)=0$ for some vertex of the dual graph of~$G$, and then defining it recursively. Let~$\mathrm v$ and~$\mathrm v'$ be two adjacent vertices in~$G^{\quadgraph}$ and assume the edge~$\mathrm v\mathrm v'$ intersects the train-track~$T$ so that~$\mathrm v$ is on the left of~$T$ and~$\mathrm v'$ is on the right. Then~$\mathbf d(\mathrm v')=\mathbf d(\mathrm v)+u(a_T)$, where~$u$ is the Abel map defined in Section~\ref{sec:abel_theta}. The only property of the discrete Abel map which will be relevant for us in the asymptotic analysis is that~$\mathbf d(\mathrm v)\in (\RR/\ZZ)^g$ for all~$\mathrm v\in G^{\quadgraph}$. 

Given an~$M$-curve, an angle map (a monotone map from the train-tracks to~$A_0$, the non-compact oval of the~$M$-map) and a parameter~$t\in (\RR/\ZZ)^g$ of the Jacobi variety, the authors of~\cite{BCT22} defined Fock's adjacency operator~$K_F$, see also~\cite{Foc15}. This is the data that is given to us from the dimer model on~$G_1$ via the so-called \emph{spectral transform}, see~\cite{GK13, KO06}. Recall that~$\mathcal R$ is an~$M$-curve~\cite[Definition 7]{BCT22}, and define the angle map as~$T_i\mapsto a_{T_i}$. We fix a white vertex~$\mathrm{w_0}\in G_1$; the spectral transform then gives us a divisor~$D$ on~$\mathcal R$ as follows. The divisor is the sum of the common zeros of
\begin{equation}
\adj K_{G_1}(z,w)_{\mathrm b\mathrm w_0}=0, \quad \mathrm b\in \mathcal B_1.
\end{equation}
It is proved in~\cite[Theorem 1]{KO06} that the divisor~$D$ is a standard divisor, that is, it is a sum of~$g$ points, one on each compact oval, see~\eqref{eq:standard_divisor}. As mentioned in Section~\ref{sec:abel_theta} after~\eqref{eq:jacobi_inverse}, this means that~$u(D)=-e_0+\Delta$, for some~$e_0\in (\RR/ \ZZ)^g$. We set~$t=e_0-\mathbf d(\mathrm w_0)\in (\RR/\ZZ)^g$. Let~$\mathrm{w} \in \mathcal W$,~$\mathrm{b} \in \mathcal B$ be two vertices in~$G$ connected with an edge~$\mathrm w\mathrm b$, and let~$\mathrm f_l$ and~$\mathrm f_r$ be the faces on the left and right, respectively, when we orient the edge from~$\mathrm w$ to~$\mathrm b$. Assume the train-tracks which cross the edge~$\mathrm w\mathrm b$ are~$T_i$ and~$T_j$, with angles~$a_{T_i}$ and~$a_{T_j}$. Then Fock's adjacency operator is given by
\begin{equation}\label{eq:fock_weight}
(K_F)_{\mathrm w \mathrm b}=\frac{E(\tilde a_{T_i},\tilde a_{T_j})}{\theta(\tilde t+\tilde{\textbf{d}}(\mathrm f_l))\theta(\tilde t+\tilde{\textbf{d}}(\mathrm f_r))},
\end{equation}
where~$\tilde a_{T_i}$ is a lift of~$a_{T_i}$ to the universal cover, and similarly for the other entries. The right hand side depends on the choice of the lift, however, any choices are in the same gauge class, which means that they define the same dimer model, see~\cite[Proposition 2 and Remark 30]{BCT22}.

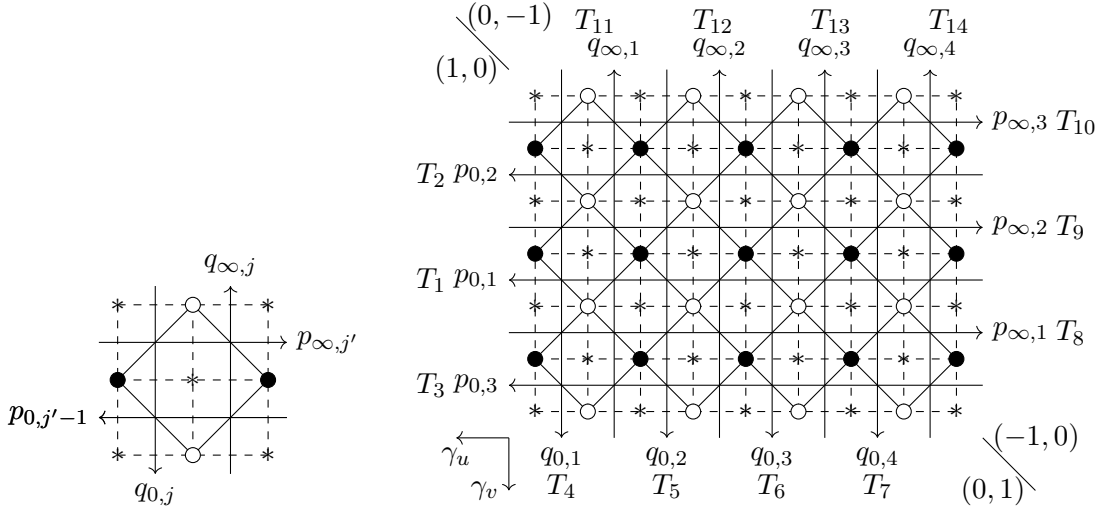
\begin{figure}
\begin{center}
\begin{tikzpicture}[scale=1]

% quad-graph
% horizontal
\foreach \y in {0,...,2}
{\draw [dashed](0,\y)--(2,\y);
}

% vertical
\foreach \x in {0,...,2}
{\draw [dashed](\x,0)--(\x,2);
}

% edges
\foreach \x/\y in {0/1,1/2}
{\draw (\x,\y)--(\x+1,\y-1);
}
\foreach \x/\y in {0/1,1/0}
{\draw (\x,\y)--(\x+1,\y+1);
}

% black points
\foreach \x in {0,2}
{\draw (\x,1) node[circle,draw=black,fill=black,inner sep=2pt]{};
}

% white points
\foreach \y in {0,2}
{\draw (1,\y) node[circle,draw=black,fill=white,inner sep=2pt]{};
}

%dual points
\foreach \x in {0,2}
{\foreach \y in {0,2}
{\draw (\x,\y) node{$*$};
}
}
\draw (1,1) node{$*$};

% train-tracks
% angles $q_{0,j}$
\foreach \x in {0}
{\draw [->](\x+.5,2.25)--(\x+.5,-.25);
\draw (\x+.5,-.25) node[below] {$q_{0,j}$};
}
% angles $q_{\infty,j}$
\foreach \x in {1}
{\draw [->](\x+.5,-.25)--(\x+.5,2.25);
\draw (\x+.5,2.25) node[above] {$q_{\infty,j}$};
}
% angles $p_{\infty,j}$
\foreach \y in {1}
{\draw [->](-.25,\y+.5)--(2.25,\y+.5);
\draw (2.25,\y+.5) node[right] {$p_{\infty,j'}$};
}
% angles $p_{0,j}$
\foreach \y in {0,}
{\draw [->](2.25,\y+.5)--(-.25,\y+.5);
\draw (-.25,\y.5) node[left]{$p_{0,j'-1}$};
}
\end{tikzpicture}
\quad 
\begin{tikzpicture}[scale=1.4]

% quad-graph
% horizontal
\foreach \y in {0,...,2}
{\draw [dashed](0,\y)--(4,\y);
}
\foreach \y in {0,...,3}
{\draw [dashed](0,\y-.5)--(4,\y-.5);
}
% vertical
\foreach \x in {0,...,4}
{\draw [dashed](\x,-.5)--(\x,2.5);
}
\foreach \x in {0,...,3}
{\draw [dashed](\x+.5,-.5)--(\x+.5,2.5);
}

% black points
\foreach \x in {0,...,4}
{\foreach \y in {0,...,2}
{\draw (\x,\y) node[circle,draw=black,fill=black,inner sep=2pt]{};
}
}

% white points and edges
\foreach \x in {1,...,4}
{\foreach \y in {1,2}
{\draw (\x-1,\y-1)--(\x,\y);
\draw (\x-1,\y)--(\x,\y-1);
\draw (\x-.5,\y-.5) node[circle,draw=black,fill=white,inner sep=2pt]{};
}
\foreach \y in {0}
{\draw (\x-.5,\y-.5)--(\x,\y);
\draw (\x-1,\y)--(\x-.5,\y-.5);
\draw (\x-.5,\y-.5) node[circle,draw=black,fill=white,inner sep=2pt]{};
}
\foreach \y in {3}
{\draw (\x-1,\y-1)--(\x-.5,\y-.5);
\draw (\x-.5,\y-.5)--(\x,\y-1);
\draw (\x-.5,\y-.5) node[circle,draw=black,fill=white,inner sep=2pt]{};
}
}

% Dual vertices
\foreach \x in {0,...,4}
{\foreach \y in {0,...,3}
{\draw (\x,\y-.5) node{$*$};
}
}
\foreach \x in {0,...,3}
{\foreach \y in {0,...,2}
{\draw (\x+.5,\y) node{$*$};
}
}

% train-tracks

% angles $p_{0,j}$
\foreach \y in {0,...,2}
{\draw [->](4.25,\y-.25)--(-.25,\y-.25);
}
\foreach \y in {1,2}
{\draw (-.25,\y-.25) node[left]{$p_{0,\y}$};
}
\draw (-.25,-.25) node[left]{$p_{0,3}$};
\foreach \y in {1,2}
{\draw (-.75,\y-.25) node[left]{$T_\y$};
}
\draw (-.75,-.25) node[left]{$T_3$};

% angles $q_{0,j}$
\foreach \x in {0,...,3}
{\draw [->](\x+.25,2.75)--(\x+.25,-.75);
}
\foreach \x in {1,...,4}
{\draw (\x-.75,-.75) node[below] {$q_{0,\x}$};
}
\foreach \x in {4,...,7}
{\draw (\x-3.75,-1.0) node[below] {$T_{\x}$};
}

% angles $p_{\infty,j}$
\foreach \y in {0,...,2}
{\draw [->](-.25,\y+.25)--(4.25,\y+.25);
}
\foreach \y in {1,...,3}
{\draw (4.25,\y-.75) node[right] {$p_{\infty,\y}$};
}
\foreach \y in {8,...,10}
{\draw (4.85,\y-7-.75) node[right] {$T_{\y}$};
}

%angles $q_{\infty,j}$
\foreach \x in {0,...,3}
{\draw [->](\x+.75,-.75)--(\x+.75,2.75);
}
\foreach \x in {1,...,4}
{\draw (\x-.25,2.75) node[above] {$q_{\infty,\x}$};
}
\foreach \x in {11,...,14}
{\draw (1.12*\x-11.5-.25,3.0) node[above] {$T_{\x}$};
}

% homology classes
% up
\draw (-.25,2.75)--(-.75,3.25);
\draw (-.75,3.25) node[right]{$(0,-1)$};
\draw (-.25,2.75) node[left]{$(1,0)$};
% down
\draw (4.25,-.75)--(4.75,-1.25);
\draw (4.25,-.75) node[right]{$(-1,0)$};
\draw (4.75,-1.25) node[left]{$(0,1)$};

% basis
% \gamma_u
\draw [->](-.25,-0.75)--(-.75,-.75);
\draw (-.75,-.75) node[below] {$\gamma_u$};
% \gamma_v
\draw [->](-.25,-0.75)--(-.25,-1.25);
\draw (-.25,-1.25) node[left] {$\gamma_v$};

\end{tikzpicture}
\end{center}
\caption{The picture shows~$G_1$ (solid edges) on top of the quad graph (dashed edges). The dual vertices are marked by a~$*$. The arrows represent the train-tracks, and at the tip of the arrows the corresponding angles are indicated. We also indicate the homology classes of the train tracks. \label{fig:quad_graph}}
\end{figure}

Finally, we define the \emph{geometric Newton polygon}, which is a rectangle with side lengths~$k$ and~$\ell$, by drawing the vector representations of the homology classes~$[T_{i}]$ starting from the tip of~$[T_{i-1}]$ for~$i=1,\dots,2(k+\ell)$. We denote the integer points of the boundary of the geometric Newton polygon by~$P_i$,~$i=1,\dots,2(k+\ell)$, so that~$P_{i+1}-P_i=[T_i]$.
\begin{lemma}
Fock's operator~$K_F$ defined in~\eqref{eq:fock_weight} is periodic with the same periods as the edge weights. 
\end{lemma}
\begin{proof}
By~\cite[Proposition 43]{BCT22} the operator~$K_F$ is periodic if and only if 
\begin{equation}
\sum_{i=1}^{2(k+\ell)}P_i(u(a_{T_i})-u(a_{T_{i-1}}))=0,
\end{equation}
where~$a_{T_i}$ is the angle of~$T_i$, and~$u$ is the Abel map defined in Section~\ref{sec:abel_theta}. The left hand side is equal to
\begin{multline}
\sum_{i=1}^{2(k+\ell)}(P_i-P_{i+1})u(a_{T_i})=-\sum_{i=1}^{2(k+\ell)}[T_i]u(a_{T_i}) \\
=\left(-\sum_{i=1}^ku(p_{0,i})+\sum_{i=1}^{k}u(p_{\infty,i}),-\sum_{i=1}^{\ell}u(q_{0,i})+\sum_{i=1}^{\ell}u(q_{\infty,i})\right) \in \left(\RR^g/\ZZ^g\right)^2.
\end{multline}
The entries of the vector on the right hand side are the images under the Abel map of the divisors of the functions~$(z,w)\mapsto z^{-1}$ and~$(z,w)\mapsto w^{-1}$. Thus by Abel's theorem~\eqref{eq:abels_thm}, the right hand side vanishes.
\end{proof}

The fact that~$K_F$ is periodic means that~\cite[Section 4]{BCT22} applies. Following the mentioned section, we define meromorphic functions~$z^*,w^*:\widetilde{\mathcal R}\to \CC$ by
\begin{equation}\label{eq:z}
z^*(\tilde q)=\prod_{T_i\in \mathcal T_1}E(\tilde a_{T_i},\tilde q)^{-v_{T_i}}=\prod_{i=1}^\ell \frac{E(\tilde q_{\infty,i},\tilde q)}{E(\tilde q_{0,i},\tilde q)},
\end{equation}
and
\begin{equation}\label{eq:w}
w^*(\tilde q)=\prod_{T_i\in \mathcal T_1}E(\tilde a_{T_i},\tilde q)^{h_{T_i}}=\prod_{i=1}^k \frac{E(\tilde p_{0,i},\tilde q)}{E(\tilde p_{\infty,i},\tilde q)}.
\end{equation}
The~$*$-notation is to distinguish the functions from~$z$ and~$w$. As explained in~\cite{BCT22}, the functions~$z^*$ and~$w^*$ project to well-defined functions on~$\mathcal R$, and we write~$z^*(q)$ and~$w^*(q)$. We define the magnetically altered Kasteleyn matrix~$K_F(z^*,w^*)$, as in Section~\ref{sec:gibbs_measure}, and the characteristic polynomial~$P_F(z^*,w^*)=\det K_F(z^*,w^*)$. Lemma~\ref{lem:kasteleyn_thetas} for~$K_F$ instead of~$K_{G_1}$ now follows from~\cite{BCT22}.

\begin{lemma}\label{lem:fock_thetas}
For~$i,i'=0,\dots,\ell-1$ and~$j,j'=0,\dots,k-1$, 
\begin{multline}
\frac{\adj K_F(z^*,w^*)_{\mathrm b_{i,j}\mathrm w_{i',j'}}\d z^*}{z^* w^* \partial_{w^*}P_F(z^*,w^*)} \\
=\Theta\left(q;-\tilde t-\tilde{\textbf{d}}(\mathrm b_{i,j})\right)\Theta\left(q;\tilde t+\tilde{\textbf{d}}(\mathrm w_{i',j'})\right)\frac{\prod_{m=i+1}^{i'}E(q_{\infty,m},q)}{\prod_{m=i+1}^{i'+1}E(q_{0,m},q)}
\frac{\prod_{m=j+1}^{j'}E(p_{0,m},q)}{\prod_{m=j+1}^{j'+1}E(p_{\infty,m},q)},
\end{multline}
where~$-\tilde t-\tilde{\textbf{d}}(\mathrm b_{i,j}), \tilde t+\tilde{\textbf{d}}(\mathrm w_{i',j'})\in \RR^g$.
\end{lemma}
\begin{proof}
This is a direct application of~\cite[Lemma 46]{BCT22}. The lemma states that
\begin{equation}\label{eq:fock_zeros_poles}
\frac{\adj K_F(z^*,w^*)_{\mathrm b_{i,j}\mathrm w_{i',j'}}\d z^*}{z^*w^* \partial_{w^*}P_F(z^*,w^*)} = g_{\mathrm b_{i,j}\mathrm w_{i',j'}}(q),
\end{equation}
where the right hand side is defined below.

For two adjacent vertices~$\mathrm{x}, \mathrm{y} \in G^{\quadgraph}$, we define~$g_{\mathrm{x},\mathrm{y}}$ depending on whether~$\mathrm x$ and~$\mathrm y$ correspond to a white vertex~$\mathrm w$, black vertex~$\mathrm b$, or a face~$\mathrm f$. It is defined as a function on the universal cover~$\widetilde{\mathcal R}$ through the following cases:
\begin{equation}
g_{\mathrm{f},\mathrm{w}}(\tilde q)=g_{\mathrm{w},\mathrm{f}}(\tilde q)^{-1}=\frac{\Theta\left(\tilde q;\tilde t+\tilde{\textbf{d}}(\mathrm w)\right)}{E(\tilde a_{T_i},\tilde q)}
\quad \text{and} \quad g_{\mathrm{b},\mathrm{f}}(\tilde q)=g_{\mathrm{f},\mathrm{b}}(\tilde q)^{-1}
=\frac{\Theta\left(\tilde q;-\tilde t-\tilde{\textbf{d}}(\mathrm b)\right)}{E(\tilde a_{T_i},\tilde q)}, 
\end{equation}
where~$T_i$ is the train-track crossing the edge in~$G^{\quadgraph}$ connecting~$\mathrm w$ with~$\mathrm f$, and~$\mathrm b$ with~$\mathrm f$, respectively. We define~$g_{\mathrm{x},\mathrm{y}}$ for all pairs of vertices in~$G^{\quadgraph}$ inductively, so that if~$\mathrm z\in G^{\quadgraph}$ then 
\begin{equation}
g_{\mathrm{x},\mathrm{z}}g_{\mathrm{z},\mathrm{y}}=g_{\mathrm{x},\mathrm{y}}.
\end{equation}
This is a well-defined function, see~\cite[Section 3.4]{BCT22}. The values~$g_{\mathrm x,\mathrm y}$ are defined on the universal cover, however, the right hand side of~\eqref{eq:fock_zeros_poles} is the projection to~$\mathcal R$, which is a well-defined differential~$1$-form.

We use the above definition to derive a formula for the right hand side of~\eqref{eq:fock_zeros_poles}. For~$i=0,\dots,\ell-1$ and~$j=1,\dots,k-1$, 
\begin{equation}
g_{\mathrm b_{i,j-1}\mathrm b_{i,j}}(\tilde q)
=\frac{\Theta\left(\tilde q;-\tilde t-\tilde{\textbf{d}}(\mathrm b_{i,j-1})\right)}{E(\tilde p_{\infty,j},\tilde q)}\frac{E(\tilde p_{0,j},\tilde q)}{\Theta\left(\tilde q;-\tilde t-\tilde{\textbf{d}}(\mathrm b_{i,j})\right)},
\end{equation}
for~$i=1,\dots,\ell-1$ and~$j=0,\dots,k-1$,
\begin{equation}
g_{\mathrm b_{i-1,j}\mathrm b_{i,j}}(\tilde q)
=\frac{\Theta\left(\tilde q;-\tilde t-\tilde{\textbf{d}}(\mathrm b_{i-1,j})\right)}{E(\tilde q_{0,i},\tilde q)}\frac{E(\tilde q_{\infty,i},\tilde q)}{\Theta\left(\tilde q;-\tilde t-\tilde{\textbf{d}}(\mathrm b_{i,j})\right)},
\end{equation}
and for~$i=0,\dots,\ell-1$ and~$j=0,\dots,k-1$, 
\begin{equation}
g_{\mathrm b_{i,j}\mathrm w_{i,j}}(\tilde q)
=\frac{\Theta\left(\tilde q;-\tilde t-\tilde{\textbf{d}}(\mathrm b_{i,j})\right)}{E(\tilde p_{\infty,j+1},\tilde q)}
\frac{\Theta\left(\tilde q;\tilde t+\tilde{\textbf{d}}(\mathrm w_{i,j})\right)}{E(\tilde q_{0,i+1},\tilde q)}.
\end{equation}
Hence, for~$i,i'=0,\dots,\ell-1$ and~$j,j'=0,\dots,k-1$, 
\begin{multline}
g_{\mathrm b_{i,j}\mathrm w_{i',j'}}(\tilde q)=\left(\prod_{m=j+1}^{j'}g_{\mathrm b_{i,m-1}\mathrm b_{i,m}}(\tilde q)\right)
\left(\prod_{m=i+1}^{i'}g_{\mathrm b_{m-1,j'}\mathrm b_{m,j'}}(\tilde q)\right)
g_{\mathrm b_{i',j'}\mathrm w_{i',j'}}(\tilde q) \\
=\Theta\left(\tilde q;-\tilde t-\tilde{\textbf{d}}(\mathrm b_{i,j})\right)\Theta\left(\tilde q;\tilde t+\tilde{\textbf{d}}(\mathrm w_{i',j'})\right)\frac{\prod_{m=i+1}^{i'}E(\tilde q_{\infty,m},\tilde q)}{\prod_{m=i+1}^{i'+1}E(\tilde q_{0,m},\tilde q)}
\frac{\prod_{m=j+1}^{j'}E(\tilde p_{0,m},\tilde q)}{\prod_{m=j+1}^{j'+1}E(\tilde p_{\infty,m},\tilde q)}.
\end{multline}
The statement now follows from~\eqref{eq:fock_zeros_poles}. 
\end{proof}

As explained in~\cite[parts~$2$ and~$3$ of Remark 50]{BCT22}, see also~\cite{GK13}, two periodic dimer models defined on the same graph are gauge equivalent (see explanation under~\eqref{eq:measure_dimer}) if they are associated with the same~$M$-curve, angle map and standard divisor. By construction of~$K_F$, this implies that~$K_F$ is gauge equivalent with~$K_{G_1}$, that is,~$K_F(z^*,w^*)=D_1K_{G_1}(z,w)D_2$, for some diagonal matrices with nonzero diagonal entries~$D_1$ and~$D_2$. Lemma~\ref{lem:kasteleyn_thetas} now almost follows from Lemma~\ref{lem:fock_thetas}. What remains is to relate~$z^*$ and~$w^*$ to~$z$ and~$w$.

\begin{lemma}\label{lem:parametrization_curve}
The functions defined by~\eqref{eq:z} and~\eqref{eq:w} are related to~$z(q)$ and~$w(q)$ by
\begin{equation}
z^*(q)=w(q)^{-1} \quad \text{and} \quad w^*(q)=z(q).
\end{equation}
\end{lemma}
\begin{proof}
By comparing the zeros and poles we find that~$z^*(q)=c_ww(q)^{-1}$ and~$w^*(q)=c_zz(q)$ for some constants~$c_z$ and~$c_w$. 

To compute the constants we consider the alternating product of the edge weights along train-tracks. As noted earlier, the alternating product of the edge weights along~$T_j$,~$j=1\dots,k$, is equal to~$\alpha^h_j/\beta^h_j$. The alternating product of the edge weights along a train-track is unchanged under a gauge transformation. This means that
\begin{equation}
\frac{\alpha^h_j}{\beta^h_j}=\frac{\prod_{i=1}^\ell E(p_{0,j},q_{0,i})}{\prod_{i=1}^\ell E(q_{\infty,i},p_{0,j})},
\end{equation}
where the right hand side is the alternating product of edge weights defined by~\eqref{eq:fock_weight} along~$T_j$. By definition of~$z^*$, and by the antisymmetry of the prime form, part~\eqref{eq:prime_form_antisym} of Fact~\ref{fact:prop_prime_form}, the previous equality yields
\begin{equation}
\frac{\alpha^h_j}{\beta^h_j}=(-1)^\ell z^*(p_{0,j})^{-1}.
\end{equation}
On the other hand, by the definition of the angle~$p_{0,j}$,
\begin{equation}
w(p_{0,j})=(-1)^\ell \frac{\alpha^h_j}{\beta^h_j}.
\end{equation}
Hence~$c_w=1$.

A similar computation shows that~$(-1)^kz(q_{0,i})=\alpha^v_i/\gamma^v_i=(-1)^kw^*(q_{0,i})$, yielding~$c_z=1$.
\end{proof}

Note that the previous lemma implies the second point in Remark~\ref{rem:f_eigenvector}. 

We are finally ready to prove Lemma~\ref{lem:kasteleyn_thetas}.
\begin{proof}[Proof of Lemma~\ref{lem:kasteleyn_thetas}]
As discussed just before Lemma~\ref{lem:parametrization_curve},~$K_F(z^*,w^*)=D_1K_{G_1}(z,w)D_2$ for some diagonal matrices~$D_1$ and~$D_2$ with nonzero entries. Lemma~\ref{lem:parametrization_curve} then implies that
\begin{equation}
\frac{\adj K_{G_1}(z,w)_{\mathrm b_{i,j}\mathrm w_{i',j'}}\d z}{z w \partial_{w}P(z,w)}=c_{iji'j'}\frac{\adj K_F(z^*,w^*)_{\mathrm b_{i,j}\mathrm w_{i',j'}}\d z^*}{z^* w^* \partial_{w^*}P_F(z^*,w^*)},
\end{equation}
for some constants~$c_{iji'j'}$. Here we have used that
\begin{equation}
\frac{\d z^*}{z^*w^*\partial_{w^*}P_F(z^*,w^*)}=-\frac{1}{\det D_1D_2}\frac{\d w}{wz\partial_zP(z,w)}=\frac{1}{\det D_1D_2}\frac{\d z}{wz\partial_wP(z,w)}.
\end{equation}
Lemma~\ref{lem:kasteleyn_thetas} now follows from Lemma~\ref{lem:fock_thetas}, with~$e_{\mathrm{b}_{i,j}}=-\tilde t-\tilde{\textbf{d}}(\mathrm b_{i,j})$ and~$e_{\mathrm{w}_{i',j'}}=\tilde t+\tilde{\textbf{d}}(\mathrm w_{i',j'})$.
\end{proof}

\begin{remark}\label{rem:meromorphic_differential}
Let~$D$ be the divisor of the zeros and poles of the~$1$-form of Lemma~\ref{lem:kasteleyn_thetas}. It follows from~\eqref{eq:abels_thm_diff} and~\eqref{eq:jacobi_inverse} that~$u(D)=2\Delta$ if and only if the~$1$-form is well defined on~$\mathcal R$. While we already know its well-definedness from the definition of the left hand side, we can also demonstrate that~$u(D)=2\Delta$ using the right hand side. Indeed, it follows from~\eqref{eq:jacobi_inverse} and the definition of~$e_{\mathrm{b}_{i,j}}=-\tilde t-\tilde{\textbf{d}}(\mathrm b_{i,j})$ and~$e_{\mathrm{w}_{i',j'}}=\tilde t+\tilde{\textbf{d}}(\mathrm w_{i',j'})$ that
\begin{equation}
u(D)=2\Delta+\tilde{\textbf{d}}(\mathrm b_{i,j})-\tilde{\textbf{d}}(\mathrm w_{i',j'})
+\sum_{m=i+1}^{i'}u(q_{\infty,m})-\sum_{m=i+1}^{i'+1}u(q_{0,m})+\sum_{m=j+1}^{j'}u(p_{0,m})-\sum_{m=j+1}^{j'+1}u(p_{\infty,m}).
\end{equation} 
We use a similar convention for the summation above as for the products in Remark~\ref{rem:products}. It follows from the definition of the discrete Abel map that
\begin{equation}
\tilde{\textbf{d}}(\mathrm w_{i',j'})-\tilde{\textbf{d}}(\mathrm b_{i,j})
=\sum_{m=i+1}^{i'}u(q_{\infty,m})-\sum_{m=i+1}^{i'+1}u(q_{0,m})+\sum_{m=j+1}^{j'}u(p_{0,m})-\sum_{m=j+1}^{j'+1}u(p_{\infty,m}),
\end{equation}
and hence,~$u(D)=2\Delta$. A general argument can be found in the proof of~\cite[Lemma 32]{BCT22}.
\end{remark}

\section{Asymptotic analysis} \label{sec:asymptotic}
This section contains the steep descent analysis of the correlation kernel in Theorem~\ref{thm:bd_thm}. In Section~\ref{sec:analysis_prep} we write the correlation kernel as an integral on~$\mathcal R$. In Section~\ref{sec:steepest_descent} we define the contours of steep descent and ascent, in the different regions, and prove the first equality of Theorem~\ref{thm:local_limit}. The second equality of that theorem is then proved in Section~\ref{sec:uniform_limt}. Finally, in Section~\ref{sec:height_function} we prove Proposition~\ref{prop:limit_shape}.

\subsection{Preparation for the steep descent analysis}\label{sec:analysis_prep}
In this section we use the result from the previous section to write the correlation kernel of Theorem~\ref{thm:bd_thm} as an integral on the Riemann surface~$\mathcal R$. We also make appropriate preparations for the steep descent analysis of that integral which will be performed in the subsequent sections. 

Recall~$\Phi$, defined by~\eqref{eq:phi_recall}, and the nullvectors~$\psi_{0,\pm}$ and~$\psi_{kN,\pm}$ defined on the Riemann surface~$\mathcal R$ and given by~\eqref{eq:eigenvector_right},~\eqref{eq:eigenvector_left} and~\eqref{eq:eigenvector_j}, respectively.

\begin{lemma}\label{lem:adjoint}
For all~$(z,w)\in \mathcal R$ the equality
\begin{equation}\label{eq:adjoint}
\frac{\psi_{0,+}(z,w)\psi_{0,-}(z,w)}{\psi_{0,-}(z,w)\psi_{0,+}(z,w)}=\frac{\adj(wI-\Phi(z))}{\partial_w\det(wI-\Phi(z))}
\end{equation}
holds. Moreover, for all~$z\in \CC^*$, 
\begin{equation}\label{eq:I_on_surface}
\sum_{w:(z,w)\in \mathcal R} \frac{\psi_{0,+}(z,w)\psi_{0,-}(z,w)}{\psi_{0,-}(z,w)\psi_{0,+}(z,w)}=I.
\end{equation}
\end{lemma}
\begin{proof}
Recall the definition of the characteristic polynomial~$P$ given in~\eqref{eq:characteristic_polynomial}. For all~$(z,w)\in \mathcal R$ with~$P_w(z,w)\neq 0$, the matrix~$wI-\Phi(z)$ is of rank~$k-1$, which means that~$\adj(wI-\Phi(z))$ is of rank~$1$. This implies that~$\adj(wI-\Phi(z))^2=\Tr (\adj(wI-\Phi(z)))\adj(wI-\Phi(z))$. Combining this equality with the explicit form of~$\psi_{0,+}$ and~$\psi_{0,-}$ given in~\eqref{eq:eigenvector_right} and~\eqref{eq:eigenvector_left}, we conclude that~$\psi_{0,+}(z,w)$ is a right eigenvector and~$\psi_{0,-}$ is a left eigenvector of~$\adj(wI-\Phi(z))$, both with eigenvalue~$\Tr(\adj(wI-\Phi(z)))$. This means that~$\adj(wI-\Phi(z))$ is, up to a scalar, the outer product of~$\psi_{0,+}$ and~$\psi_{0,-}$. Comparing the eigenvalues we conclude that
\begin{equation}\label{eq:adjoint_c}
\frac{\psi_{0,+}(z,w)\psi_{0,-}(z,w)}{\psi_{0,-}(z,w)\psi_{0,+}(z,w)}=\frac{\adj(wI-\Phi(z))}{\Tr(\adj(wI-\Phi(z)))}.
\end{equation}
By Jacobi's formula~\eqref{eq:jacobis_formula}, 
\begin{equation}
\Tr \adj(wI-\Phi(z))=\partial_w\det(wI-\Phi(z)),
\end{equation}
which proves~\eqref{eq:adjoint}.

If~$(z,w),(z,w^*)\in \mathcal R$, then 
\begin{equation}
w\psi_{0,-}(z,w)\psi_{0,+}(z,w^*)=\psi_{0,-}(z,w)\Phi(z)\psi_{0,+}(z,w^*)=w^*\psi_{0,-}(z,w)\psi_{0,+}(z,w^*).
\end{equation}
Hence, if~$w\neq w^*$ then~$\psi_{0,-}(z,w)\psi_{0,+}(z,w^*)=0$. It follows that if~$(z,w^*)\in \mathcal R$ then~$\psi_{0,+}(z,w^*)$ is an eigenvector of the left hand side of~\eqref{eq:I_on_surface}, with eigenvalue~$1$. For all but a finite set of~$z\in \CC$, there are~$k$ distinct~$w_i$ such that~$(z,w_i)\in \mathcal R$. We claim that the eigenvectors~$\psi_{0,+}(z,w_i)$ are linearly independent. Indeed,~$\psi_{0,+}(z,w_i)$ is an eigenvector of~$\Phi$ with eigenvalue~$w_i$. Since the eigenvalues are distinct,~$\psi_{0,+}(z,w_i)$ must be linearly independent. Hence, the left hand side of~\eqref{eq:I_on_surface} has~$k$ linearly independent eigenvectors all with eigenvalue~$1$, which means it is the identity matrix. 
\end{proof}

We use the previous lemma to write the integral in Theorem~\ref{thm:bd_thm} as an integral on~$\mathcal R$. The new integral is expressed in terms of the nullvectors, which we understand from the results of Section~\ref{sec:wh}.

Recall that we are working under the assumption that~$\beta_i^v<1<\alpha_i^v/\gamma_i^v$ (Assumption~\ref{ass:main_ass}\eqref{ass:wh}).

\begin{proposition}\label{prop:finite_kernel}
The correlation kernel associated with the~$n\to \infty$ limit of the determinantal point process~$\PP_\text{path}$~\eqref{eq:measure_on_points} is given by
\begin{multline}
 \left[K_\text{path}(2\ell x+i,ky+j;2\ell x'+i',ky'+j')\right]_{j',j=0}^{k-1}
=-\frac{\one_{2\ell x+i>2\ell x'+i'}}{2\pi\i}\int_{\Gamma} \prod_{m=2\ell x'+i'+1}^{2\ell x+i}\phi_m(z)z^{y'-y}\frac{\d z}{z} \\
 + \frac{1}{(2\pi\i)^2}\int_{\tilde \Gamma_s}\int_{\tilde \Gamma_l} \left(\prod_{m=1}^{i'}\phi_m(z_1)\right)^{-1} \frac{\psi_{0,+}(z_1,w_1)\psi_{kN,-}(z_1,w_1)}{\psi_{0,-}(z_1,w_1)\psi_{0,+}(z_1,w_1)}\\
 \times\frac{\psi_{kN,+}(z_2,w_2)\psi_{0,-}(z_2,w_2)}{\psi_{0,-}(z_2,w_2)\psi_{0,+}(z_2,w_2)}\prod_{m=1}^i\phi_m(z_2)\frac{z_1^{y'-\ell N}}{z_2^{y-\ell N}}\frac{w_2^{x-kN}}{w_1^{x'}}\frac{\d z_2\d z_1}{z_2(z_2-z_1)},
\end{multline}
for~$x,x'=0,\dots, kN-1$,~$i,i'=0,\dots 2\ell-1$,~$y,y'\in \ZZ$ and~$j,j'=0,\dots,k-1$. The contour~$\Gamma$ is a circle centered at zero with radius~$1$ and~$\tilde \Gamma_{s/l}=\{(z,w)\in \mathcal R:|z|=r_{s/l}\}$ where~$\beta_i^v<r_s<r_l<\alpha_i^v/\gamma_i^v$. All contours are oriented in the positive direction. 
\end{proposition}

\begin{proof}
There exists a Wiener--Hopf factorization of~$\phi=\Phi^{kN}$ by Assumption~\ref{ass:main_ass}\eqref{ass:wh}, which means that Theorem~\ref{thm:bd_thm} applies. The first integral in the statement is identical to the first integral in~\eqref{eq:bd_thm}, so we only need to consider the second integral. In order to use Theorem~\ref{thm:wiener-hopf} we write the second integral in~\eqref{eq:bd_thm} as an integral on~$\mathcal R$, using~\eqref{eq:I_on_surface}. We get that the integral is equal to
\begin{multline}
\frac{1}{(2\pi\i)^2}\int_{\Gamma_s}\int_{\Gamma_l} \left(\prod_{m=1}^{i'}\phi_m(z_1)\right)^{-1} \Phi(z_1)^{-x'}\sum_{w_1:(z_1,w_1)\in \mathcal R}\frac{\psi_{0,+}(z_1,w_1)\psi_{0,-}(z_1,w_1)}{\psi_{0,-}(z_1,w_1)\psi_{0,+}(z_1,w_1)}\,\widetilde \phi_-(z_1)\\
 \times\widetilde\phi_+(z_2)\sum_{w_2:(z_2,w_2)\in \mathcal R}\frac{\psi_{0,+}(z_2,w_2)\psi_{0,-}(z_2,w_2)}{\psi_{0,-}(z_2,w_2)\psi_{0,+}(z_2,w_2)}\,\Phi(z_2)^{x-kN}\prod_{m=1}^i\phi_m(z_2)\frac{z_1^{y'}}{z_2^{y}}\frac{\d z_2\d z_1}{z_2(z_2-z_1)} \\
 = \frac{1}{(2\pi\i)^2}\int_{\tilde \Gamma_s}\int_{\tilde \Gamma_l} \left(\prod_{m=1}^{i'}\phi_m(z_1)\right)^{-1} \Phi(z_1)^{-x'}\frac{\psi_{0,+}(z_1,w_1)\psi_{0,-}(z_1,w_1)}{\psi_{0,-}(z_1,w_1)\psi_{0,+}(z_1,w_1)}\,\widetilde \phi_-(z_1)\\
 \times\widetilde\phi_+(z_2)\frac{\psi_{0,+}(z_2,w_2)\psi_{0,-}(z_2,w_2)}{\psi_{0,-}(z_2,w_2)\psi_{0,+}(z_2,w_2)}\,\Phi(z_2)^{x-kN}\prod_{m=1}^i\phi_m(z_2)\frac{z_1^{y'}}{z_2^{y}}\frac{\d z_2\d z_1}{z_2(z_2-z_1)},
\end{multline}
where~$\tilde \Gamma_s$ and~$\tilde \Gamma_l$ are given in the statement. The proposition now follows from Theorem~\ref{thm:wiener-hopf} since~$\Phi(z_1)\psi_{0,+}(z_1,w_1)=w_1\psi_{0,+}(z_1,w_1)$ and~$\psi_{0,-}(z_2,w_2)\Phi(z_2)=w_2\psi_{0,-}(z_2,w_2)$.
\end{proof}

To see that the integral in the previous proposition is suitable for a steep descent analysis, we define for~$\tilde q\in \widetilde{\mathcal R}$ the following diagonal matrices, using the constants from Proposition~\ref{prop:zeros_poles},
\begin{equation}
D_{m,\mathrm w/\mathrm b}(\tilde q)=\diag\left\{\Theta\left(\tilde q;e_{\mathrm w_{0,j}/\mathrm b_{0,j}}^{(m)}\right)\right\}_{j=1}^k, \quad \text{and} \quad C_{m,\pm}=\diag\left\{c_{j,\pm}^{(m)}\right\}_{j=1}^k,
\end{equation}
for~$m=1,\dots,kN$. By Proposition~\ref{prop:zeros_poles} and Remark~\ref{rem:f_eigenvector},
\begin{equation}\label{eq:psi_n_psi_0+}
\psi_{kN,+}=f^ND_{kN,\mathrm w}D_{0,\mathrm w}^{-1}C_{kN,+}C_{0,+}^{-1}\psi_{0,+},
\end{equation}
and
\begin{equation}\label{eq:psi_n_psi_0-}
\psi_{kN,-}=\psi_{0,-}D_{kN,\mathrm b}D_{0,\mathrm b}^{-1}C_{kN,-}C_{0,-}^{-1}f^{-N}w^{kN},
\end{equation}
where~$f$ is as in Definition~\ref{def:action_function}. The following two lemmas show that the integral in Proposition~\ref{prop:finite_kernel} is suitable for a steep descent analysis. More precisely, Lemma~\ref{lem:meromorphic_one-forms} tells us how the contours in the integral can be deformed, and Lemma~\ref{lem:behavior_one-forms} tells us that the non-exponential part of the integral stays bounded. 

\begin{lemma}\label{lem:meromorphic_one-forms}
Assume that~$i,i'=0,\dots,2\ell-1$,~$0\leq x<kN$ and~$0\leq y'<\ell N-1$, and consider the differential forms 
\begin{equation}\label{eq:meromorphic_one-form_1}
\left(\prod_{m=1}^{i'}\phi_m(z_1)\right)^{-1} \frac{\psi_{0,+}(z_1,w_1)\psi_{kN,-}(z_1,w_1)}{\psi_{0,-}(z_1,w_1)\psi_{0,+}(z_1,w_1)}\frac{z_1^{y'-\ell N}}{w_1^{x'}}\d z_1,
\end{equation}
and
\begin{equation}\label{eq:meromorphic_one-form_2}
\frac{\psi_{kN,+}(z_2,w_2)\psi_{0,-}(z_2,w_2)}{\psi_{0,-}(z_2,w_2)\psi_{0,+}(z_2,w_2)}\prod_{m=1}^i\phi_m(z_2)\frac{w_2^{x-kN}}{z_2^{y-\ell N}}
\d z_2,
\end{equation}
both defined on~$\mathcal R$. These differential forms are meromorphic with possible poles only at~$q_{0,m}$ and~$q_{\infty,m}$,~$m=1,\dots,\ell$, and at~$p_{0,m}$ and~$p_{\infty,m}$,~$m=1,\dots,k$, respectively. 
\end{lemma}
\begin{proof}
By the definition of~$\psi_{m,\pm}$,~$m=0$ and~$m=kN$, it is clear that the differential forms are meromorphic. 

Let us determine the poles of~\eqref{eq:meromorphic_one-form_1}. By~\eqref{eq:adjoint} and~\eqref{eq:psi_n_psi_0-}, the form~\eqref{eq:meromorphic_one-form_1} is given by
\begin{equation}\label{eq:meromorphic_one_form_rewritten}
f^{-N}\left(\prod_{m=1}^{i'}\phi_m\right)^{-1}\frac{\adj(w_1I-\Phi)}{z_1\partial_{w_1}\det(w_1I-\Phi)}D_{kN,\mathrm b}D_{0,\mathrm b}^{-1}C_{kN,-}C_{0,-}^{-1}
\frac{z_1^{y'-\ell N+1}}{w_1^{x'-kN}}\d z_1.
\end{equation}
First we assume that~$i'$ is odd, say,~$i'=2i^*+1$,~$i^*=0,\dots, \ell-1$. To determine the poles we write~\eqref{eq:meromorphic_one_form_rewritten} in terms of prime forms and theta functions. Lemmas~\ref{lem:phi_kasteleyn} and~\ref{lem:kasteleyn_thetas} imply that
\begin{multline}\label{eq:adjugate_matrix_prime_form}
\left(\left(\prod_{m=1}^{2i^*+1}\phi_m(z_1)\right)^{-1}\frac{\adj(w_1I-\Phi(z_1))\d z_1}{z_1\partial_{w_1}\det(w_1I-\Phi(z_1))}\right)_{j'+1,j+1}\\
=-c_{0ji^*j'}
\Theta\left(q_1;e_{\mathrm{b}_{0,j}}\right)\Theta\left(q_1;e_{\mathrm w_{i^*,j'}}\right)
\frac{\prod_{m=1}^{i^*}E(q_{\infty,m},q_1)}{\prod_{m=1}^{i^*+1}E(q_{0,m},q_1)}
\frac{\prod_{m=j+1}^{j'}E(p_{0,m},q_1)}{\prod_{m=j+1}^{j'+1}E(p_{\infty,m},q_1)},
\end{multline}
where~$q_1=(z_1,w_1)\in \mathcal R$. The functions~$z_1$ and~$w_1$ can be written in terms of~$E$ using Lemma~\ref{lem:parametrization_curve}, and the function~$f$ is also expressible through~$E$'s via Definition~\ref{def:action_function}, see also Remark~\ref{rem:f_eigenvector}. By definition of~$e_{\mathrm b_{0,j}}^{(0)}$, which is given in the beginning of the proof of Lemma~\ref{lem:eigenvector_zeros_poles}, we have that~$e_{\mathrm b_{0,j}}^{(0)}=e_{\mathrm b_{0,j}}$. This implies that the factor~$\Theta\left(q_1;e_{\mathrm{b}_{0,j}}\right)$ in~\eqref{eq:adjugate_matrix_prime_form} is canceled by~$D_{0,\mathrm b}^{-1}$. Altogether, we find that the~$(j'+1,j+1)$-entry of~\eqref{eq:meromorphic_one_form_rewritten} is equal to
\begin{multline}\label{eq:meromrophic_one_form_entry}
-\frac{c_{j,-}^{(kN)}c_{0ji^*j'}}{c_{j,-}^{(0)}}
\Theta\left(q;e_{\mathrm{b}_{0,j}}^{(kN)}\right)\Theta\left(q;e_{\mathrm w_{i^*,j'}}\right)
\frac{\prod_{m=1}^{i^*}E(q_{\infty,m},q)}{\prod_{m=1}^{i^*+1}E(q_{0,m},q)}
\frac{\prod_{m=j+1}^{j'}E(p_{0,m},q)}{\prod_{m=j+1}^{j'+1}E(p_{\infty,m},q)}\\
\times\frac{\prod_{m=1}^kE(p_{0,m},q)^{y'+1}E(p_{\infty,m},q)^{\ell N-y'-1}}{\prod_{m=1}^\ell E(q_{0,m},q)^{x'}E(q_{\infty,m},q)^{kN-x'}}.
\end{multline}
It is now clear that if~$0\leq y'<\ell N-1$ and~$i'$ is odd, then~\eqref{eq:meromorphic_one_form_rewritten} only has poles at~$q_{0,j}$ and~$q_{\infty,j}$. 

If~$i'$ is even instead,~$i'=2i^*$,~$i^*=0,\dots,\ell-1$, then we multiply~\eqref{eq:meromorphic_one_form_rewritten} by~$\phi_{2i^*+1}\phi_{2i^*+1}^{-1}$ from the left, and we are back to the odd case times the matrix~$\phi_{2i^*+1}$. The effect of this multiplication in~\eqref{eq:meromrophic_one_form_entry} is the factor~$E(q_{0,i^*+1},q)/E(p_{0,j'},q)$, and the constant~$e_{\mathrm w_{i^*,j'}}$ will change so that the expression remains meromorphic on~$\mathcal R$. This follows from a similar argument as in the proof of Lemma~\ref{lem:eigenvector_zeros_poles} (cf.~\eqref{eq:eigenvector_zeros_angle} and~\eqref{eq:eigenvector_poles_infty}). In particular, it is still true that~\eqref{eq:meromorphic_one-form_1} may only have poles at~$q_{0,j}$ and~$q_{\infty,j}$, if~$0\leq y'<\ell N-1$.

The proof for~\eqref{eq:meromorphic_one-form_2} follows by a similar argument.
\end{proof}

\begin{lemma}\label{lem:behavior_one-forms}
Let~$F$ be as in Definition~\ref{def:action_function}. In the coordinate system~\eqref{eq:local_coordinates_x} and~\eqref{eq:local_coordinates_y}, the product of the two~$1$-forms in Lemma~\ref{lem:meromorphic_one-forms} times~$z_2^{-1}(z_2-z_1)^{-1}$ is equal to
\begin{equation}
\e^{N(F(q_1;\xi_N,\eta_N)-F(q_2;\xi_N,\eta_N))}G(q_1;q_2)\frac{z_1^{\zeta'}}{w_1^{\kappa'}} \frac{w_2^{\kappa}}{z_2^{\zeta}} \frac{\d z_2\d z_1}{z_2(z_2-z_1)},
\end{equation}
where
\begin{multline}
G(\tilde q_1;\tilde q_2)=\left(\prod_{m=1}^{i'}\phi_m(z_1)\right)^{-1}\frac{\psi_{0,+}(z_1,w_1)\psi_{0,-}(z_1,w_1)}{\psi_{0,-}(z_1,w_1)\psi_{0,+}(z_1,w_1)}D_{kN,\mathrm b}(\tilde q_1)D_{0,\mathrm b}(\tilde q_1)^{-1}C_{kN,-}C_{0,-}^{-1} \\
\times C_{0,+}^{-1}C_{kN,+}D_{0,\mathrm w}(\tilde q_2)^{-1}D_{kN,\mathrm w}(\tilde q_2)\frac{\psi_{0,+}(z_2,w_2)\psi_{0,-}(z_2,w_2)}{\psi_{0,-}(z_2,w_2)\psi_{0,+}(z_2,w_2)}\prod_{m=1}^{i}\phi_m(z_2),
\end{multline}
where~$q_i=(z_i,w_i)\in \mathcal R$ for~$i=1,2$, and~$\tilde q_i\in \widetilde{\mathcal R}$ is an arbitrary lift of~$q_i$ to the universal cover. Moreover, if~$z_1=z_2$ then~$G(q_1,q_2)=0$ if~$w_1\neq w_2$, and 
\begin{equation}
G(q_1;q_2)=\left(\prod_{m=1}^{i'}\phi_m(z_1)\right)^{-1}\frac{\adj(w_1I-\Phi(z_1))}{\partial_{w_1}\det(w_1I-\Phi(z_1))}\prod_{m=1}^{i}\phi_m(z_1)
\end{equation}
if~$w_1=w_2$. If~$U$ and~$V$ are compact subsets of~$\widetilde{\mathcal R}$ such that~$U$ and~$V$ do not contain any lift of the angles, then the function~$G$ is uniformly bounded on~$U\times V$.
\end{lemma}
As in the previous sections, we use~$q\in \mathcal R$ when the entire expression is well-defined on~$\mathcal R$ and~$\tilde q\in \widetilde{\mathcal R}$ when the expression is only valid on the universal cover.

A key motivation for examining the two linear flows, that is, considering both right and left nullvectors, lies in the proof of this lemma, see, in particular,~\eqref{eq:bound_cn}.
\begin{proof}
The first formula in the statement is straightforward to check via~\eqref{eq:psi_n_psi_0+} and~\eqref{eq:psi_n_psi_0-}. 

To obtain the formula for~$G$ if~$z_1=z_2$ we use Theorem~\ref{thm:wiener-hopf}, recalling also that~$\widetilde \phi_-\widetilde \phi_+=\Phi^{kN}$, to get that
\begin{equation}\label{eq:inner_product}
\psi_{kN,-}(z_1,w_1)\psi_{kN,+}(z_1,w_2)=w_1^{kN} \psi_{0,-}(z_1,w_1)\psi_{0,+}(z_1,w_2).
\end{equation} 
We recall also from the proof of Lemma~\ref{lem:adjoint} that the right hand side of the above equality is zero if~$w_1\neq w_2$. The expressions of~$G$ as~$z_1=z_2$ then follows by a straightforward computation using~\eqref{eq:psi_n_psi_0+} and~\eqref{eq:psi_n_psi_0-} together with~\eqref{eq:adjoint}.

The~$N$ dependence of~$G$ lies in~$D_{kN,\mathrm b/\mathrm w}$ and in~$C_{kN,\pm}$. Recall that~$\theta$ is an entire function, so it follows from the definition of~$\Theta(\tilde q;e_{\mathrm b_{0,j}}^{(kN)})$ and~\eqref{eq:quasi-periodic}, that the entries in~$D_{kN,\mathrm b/\mathrm w}$ are bounded from above. What remains is to show that the product ~$c_{j,+}^{(kN)}\cdot c_{j,-}^{(kN)}$ is bounded. 

Using Proposition~\ref{prop:zeros_poles} and Lemma~\ref{lem:parametrization_curve}, recall also our convention for products in Remark~\ref{rem:products}, we get for~$q\in \mathcal R$ that
\begin{multline}
\psi_{km,-}(q)\psi_{km,+}(q)=\frac{w^{km}}{\prod_{i=1}^kE(p_{0,i},q)E(p_{\infty,i},q)}\frac{\Theta\left(q;e_{\mathrm b}\right)\Theta\left(q;e_{\mathrm w}\right)}{E(q_{0,1},q)E(p_{\infty,1},q)} \\
\times\sum_{j=0}^{k-1}c_{j,+}^{(km)}c_{j,-}^{(km)}\Theta\left(q;e_{\mathrm b_{0,j}}^{(km)}\right)\Theta\left(q;e_{\mathrm w_{0,j}}^{(km)}\right)\prod_{\substack{i\neq j\\ i=1,\dots,k}}E(p_{0,i},q)\prod_{\substack{i\neq j+1 \\ i=1,\dots,k}}E(p_{\infty,i},q).
\end{multline}
Applying the above formula with~$m=0$ and~$m=N$ to equality~\eqref{eq:inner_product} we obtain
\begin{multline}
\sum_{j=0}^{k-1}c_{j,+}^{(kN)}c_{j,-}^{(kN)}\Theta\left(\tilde q;e_{\mathrm b_{0,j}}^{(kN)}\right)\Theta\left(\tilde q;e_{\mathrm w_{0,j}}^{(kN)}\right)\prod_{i\neq j}E(\tilde p_{0,i},\tilde q)\prod_{i\neq j+1}E(\tilde p_{\infty,i},\tilde q)\\
=\sum_{j=0}^{k-1}c_{j,+}^{(0)}c_{j,-}^{(0)}\Theta\left(\tilde q;e_{\mathrm b_{0,j}}^{(0)}\right)\Theta\left(\tilde q;e_{\mathrm w_{0,j}}^{(0)}\right)\prod_{i\neq j}E(\tilde p_{0,i},\tilde q)\prod_{i\neq j+1}E(\tilde p_{\infty,i},\tilde q),
\end{multline}
for some fixed choice of lifts~$\tilde q$,~$\tilde p_{0,i}$ and~$\tilde p_{\infty,i}$. Most terms in the last equality vanish when~$\tilde q=\tilde p_{\infty,j+1}$ for~$j=0,\dots,k-1$, and we obtain the equality
\begin{multline}\label{eq:bound_cn}
c_{j,+}^{(kN)}c_{j,-}^{(kN)}\Theta\left(\tilde p_{\infty,j+1};e_{\mathrm b_{0,j}}^{(kN)}\right)\Theta\left(\tilde p_{\infty,j+1};e_{\mathrm w_{0,j}}^{(kN)}\right)\prod_{i\neq j}E(\tilde p_{0,i},\tilde p_{\infty,j+1})\prod_{i\neq j+1}E(\tilde p_{\infty,i},\tilde p_{\infty,j+1})\\
=c_{j,+}^{(0)}c_{j,-}^{(0)}\Theta\left(\tilde p_{\infty,j+1};e_{\mathrm b_{0,j}}^{(0)}\right)\Theta\left(\tilde p_{\infty,j+1};e_{\mathrm w_{0,j}}^{(0)}\right)\prod_{i\neq j}E(\tilde p_{0,i},\tilde p_{\infty,j+1})\prod_{i\neq j+1}E(\tilde p_{\infty,i},\tilde p_{\infty,j+1}).
\end{multline}
Since~$e_{\mathrm b_{0,j}/\mathrm w_{0,j}}^{(kN)}\in \RR^g$, the discussion after~\eqref{eq:jacobi_inverse} implies that the zeros of~$\Theta\left(\tilde q;e_{\mathrm b_{0,j}/\mathrm w_{0,j}}^{(kN)}\right)$ lie on the compact ovals, which are bounded away from the angles. Thus,~$\Theta\left(\tilde p_{\infty,j+1};e_{\mathrm b_{0,j}}^{(kN)}\right)\Theta\left(\tilde p_{\infty,j+1};e_{\mathrm w_{0,j}}^{(kN)}\right)$ is bounded from below. Moreover, since the angles are assumed to be different, it follows from~\eqref{eq:prime_form_zero} of Fact~\ref{fact:prop_prime_form} that~$E(p_{\infty,i},p_{\infty,j+1})\neq 0$, if~$i\neq j+1$. Hence,~\eqref{eq:bound_cn} gives us a uniform bound from above on the product~$\left|c_{j,+}^{(kN)}c_{j,-}^{(kN)}\right|$.
\end{proof}

\subsection{The steep descent analysis}\label{sec:steepest_descent}

In this section we prove the first equality in Theorem~\ref{thm:local_limit} using a steep descent analysis. The details of the steep descent analysis depend on whether we consider the rough, smooth or frozen region, and we will therefore perform the analysis for one region at a time. The strategy, however, will be the same for all three cases. Namely, we prove, using Lemma~\ref{lem:behavior_one-forms}, that there exist suitable contours of steep descent and ascent,~$\gamma_s$ and~$\gamma_l$, so that if we take the double integral in Proposition~\ref{prop:finite_kernel} along those contours, then it tends to zero as~$N\to \infty$. These contours will intersect at some of the critical points defining the regions, recall Definition~\ref{def:regions}, and pass through some of the angles~$p_{0,j}$ and~$p_{\infty,j}$, and~$q_{0,i}$ and~$q_{\infty,i}$, respectively. We then use Lemma~\ref{lem:meromorphic_one-forms} to show that we can deform~$\tilde \Gamma_s$ and~$\tilde \Gamma_l$, given in Proposition~\ref{prop:finite_kernel} to~$\gamma_s$ and~$\gamma_l$, respectively, and that the only contribution from this deformation comes from the residue at~$z_1=z_2$ along the curve~$\gamma_{\xi,\eta}$ from Definition~\ref{def:curve_integration}. For a sketch of the image of the contours in the amoeba in the different regions, see Figures~\ref{fig:amoeba_curves_rough} and~\ref{fig:amoeba_curves_smooth_frozen}.

\subsubsection{The rough region}\label{sec:rough}

We begin by proving the equality of~\eqref{eq:limiting_kernel_lhs} and~\eqref{eq:limiting_kernel} in Theorem~\ref{thm:local_limit} in the rough region.

Let
\begin{equation}
\mathcal D=\mathcal D(\xi,\eta)=\{(z,w)\in \mathcal R:\re F(z,w;\xi,\eta)\geq \re F(q_c;\xi,\eta)\},
\end{equation}
where~$q_c=\Omega(\xi,\eta)$ and~$\Omega$ is, as in Definition~\ref{def:omega}, a the unique zero of~$\d F$ in~$\mathcal R_0$. The set~$\mathcal D$ is defined to help us understand the curves of steep descent and ascent. Since~$q_c$ is a critical point of a meromorphic function of order one, the level lines of~$\re F$ divide the plane locally at~$q_c$ into four parts, two of which lie in~$\mathcal D$. We denote the connected components of~$\mathcal D$ which contain these parts by~$D_1$ and~$D_2$. Similarly, we denote the connected components of~$\mathcal D^c$ which contain the other two parts by~$C_1$ and~$C_2$. At this point it is not clear that~$D_1\neq D_2$ and~$C_1\neq C_2$, but we will see in Lemma~\ref{lem:rough_sets} below that this is indeed the case. See Figure~\ref{fig:amoeba_curves_rough}, right panel, for a sketch of the images of these sets in the amoeba.

 \begin{figure}[t]
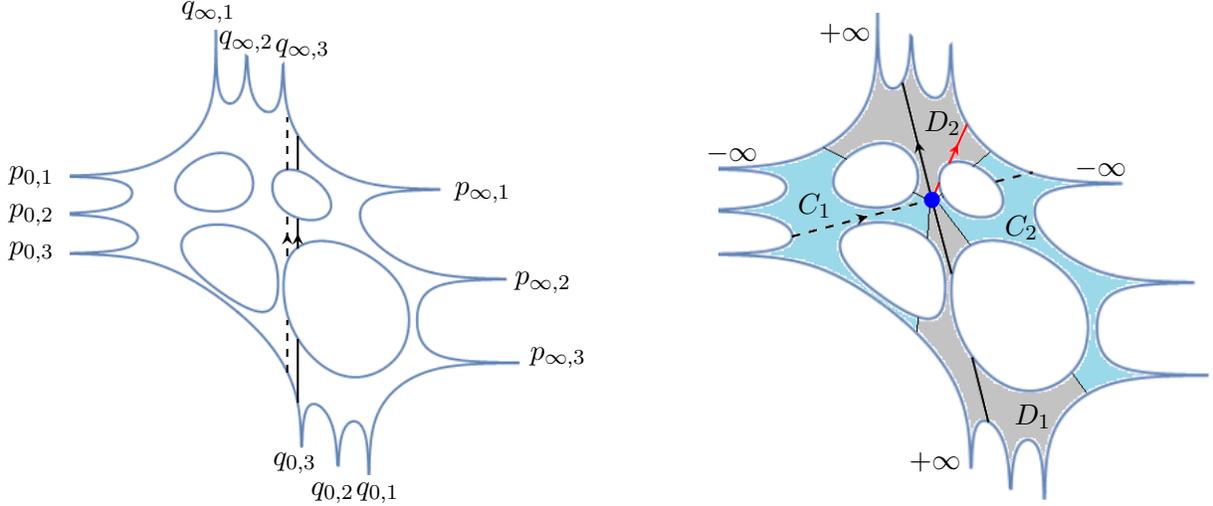

 \begin{center}
 \begin{tikzpicture}[scale=1]
  \tikzset{->-/.style={decoration={
  markings, mark=at position .5 with {\arrow{stealth}}},postaction={decorate}}}
    \draw (0,0) node {\includegraphics[trim={1cm, 1cm, 1cm, 1cm}, clip, angle=180, scale=.5]{amoeba3x3applicable.png}};
    %\tilde \Gamma_l
   \draw[thick](0.04,-2)--(0.04,-1.15);
   \draw[->-,thick](0.04,0.05)--(0.04,.45);
   \draw[thick](0.04,1.1)--(0.04,1.55);
   %\tilde \Gamma_s
   \draw[dashed,thick] (-.1,-1.6)--(-.1,-.9);
   \draw[->-,dashed,thick] (-.1,-.05)--(-.1,.55);
   \draw[dashed,thick] (-.1,1.1)--(-.1,1.8);
	%angles
	\draw (.1,2.7) node {$q_{\infty,3}$};
	\draw (-.65,2.8) node {$q_{\infty,2}$};
	\draw (-1.15,3.2) node {$q_{\infty,1}$};
	
	\draw (-3.5,1) node {$p_{0,1}$};
	\draw (-3.5,.5) node {$p_{0,2}$};
	\draw (-3.5,0) node {$p_{0,3}$};

	\draw (0,-2.8) node {$q_{0,3}$};
	\draw (.5,-3.2) node {$q_{0,2}$};
	\draw (1.1,-3.2) node {$q_{0,1}$};

	\draw (2.5,.8) node {$p_{\infty,1}$};
	\draw (3.3,-.4) node {$p_{\infty,2}$};
	\draw (3.5,-1.4) node {$p_{\infty,3}$};
  \end{tikzpicture}
  \qquad
   \begin{tikzpicture}[scale=1]
	\draw (0,0) node {\includegraphics[trim={0cm, .1cm, .8cm, .5cm}, clip, scale=.65]{amoeba3x3curvesRough2.png}};
	% regions
	\draw (1.05,.41) node {$C_2$};
	\draw (-1.7,.7) node {$C_1$};
	\draw (0,1.8) node {$D_2$};
	\draw (1.2,-2.1) node {$D_1$};
	
	%angles
	\draw (-1.3,3.0) node {$+\infty$};
	\draw (-2.8,1.4) node {$-\infty$};
	\draw (-.1,-2.7) node {$+\infty$};
	\draw (2.1,1.2) node {$-\infty$};

%	\draw (0,0) node {\includegraphics[trim={0cm, 0cm, 1cm, 0cm}, clip, scale=.7]{amoeba3x3curvesRough.png}};
%	% regions
%	\draw (1.05,.45) node {$C_2$};
%	\draw (-1.5,.7) node {$C_1$};
%	\draw (0,1.7) node {$D_2$};
%	\draw (1.2,-1.9) node {$D_1$};
%	
%	%angles
%	\draw (-1.2,2.9) node {$+\infty$};
%	\draw (-2.5,1.4) node {$-\infty$};
%	\draw (0,-2.5) node {$+\infty$};
%	\draw (2.1,1.2) node {$-\infty$};

%	\draw (.1,2.7) node {$+\infty$};
%	\draw (-.65,2.8) node {$+\infty$};
%	\draw (-1.15,3.2) node {$+\infty$};
%	
%	\draw (-3.5,1) node {$-\infty$};
%	\draw (-3.5,.5) node {$-\infty$};
%	\draw (-3.5,0) node {$-\infty$};
%
%	\draw (0,-2.8) node {$+\infty$};
%	\draw (.5,-3.2) node {$+\infty$};
%	\draw (1.1,-3.2) node {$+\infty$};

%	\draw (2.5,.8) node {$-\infty$};
%	\draw (3.3,-.4) node {$-\infty$};
%	\draw (3.5,-1.4) node {$-\infty$};
\end{tikzpicture}
%   \begin{tikzpicture}[scale=1]
%    \tikzset{->-/.style={decoration={
%  markings, mark=at position .5 with {\arrow{stealth}}},postaction={decorate}}}
%    \draw (0,0) node {\includegraphics[trim={1cm, 1cm, 1cm, 1cm}, clip, angle=180, scale=.5]{amoeba3x3applicable.png}};
%    %\gamma_l
%   \draw(0.3,-2.03)--(0.1,-1.25);
%   \draw(-.15,-.25)--(-.39,.64);
%   \draw[->-](-.39,.64)--(-.75,2.05);
%   %\gamma_s
%   \draw[->-,dashed] (-2.1,.2)--(-.25,.68);
%   \draw[dashed] (.39,.86)--(.84,.97);
%   %\gamma_{\xi,\eta}
%   \draw[red](-.39,.64)--(-.3,.85);
%   \draw[->-,red](-.18,1.08)--(0.04,1.55);
%   % Critical point
%   % Rough
%	\draw (-.39,.64) node[blue,circle,fill,inner sep=1.5pt]{};
%  \end{tikzpicture}
 \end{center}
  \caption{An example of an amoeba and a sketch of the curves of integration in the steep descent analysis. Left: original curves,~$\tilde \Gamma_s$ (dashed),~$\tilde \Gamma_l$ (solid). Right: curves in the rough region,~$\gamma_s$ (dashed),~$\gamma_l$ (solid),~$\gamma_{\xi,\eta}$ (red). The gray and blue regions represent the regions in Lemma~\ref{lem:rough_sets}, and the behavior of~$\re F$ at the angles is indicated by~$\pm \infty$. \label{fig:amoeba_curves_rough}}
\end{figure}

\begin{lemma}\label{lem:rough_sets}
After a possible relabeling of~$D_1$ and~$D_2$, there exist~$j, j'\in \{1,\dots,\ell\}$ such that~$q_{0,j} \in D_1$ and~$q_{\infty,j'} \in D_2$, while~$q_{0,i}\notin D_2$ and~$q_{\infty,i}\notin D_1$ for all~$i$. Similarly,~$p_{0,j}\in C_1$ and~$p_{\infty,j'}\in C_2$, for some (possibly different)~$j,j'\in \{1,\dots,k\}$ while~$p_{0,i}\notin C_2$ and~$p_{\infty,i}\notin C_1$ for all~$i$.   
\end{lemma} 
\begin{proof}
Since~$\re F$ is harmonic, see Lemma~\ref{lem:conjugate_functions}, non-constant in~$D_1$ and~$D_2$, and constant along the boundaries,~$\re F$ is unbounded in~$D_1$ and~$D_2$. The function~$\re F$ is only unbounded at the angles, and by~\eqref{eq:limit_f_1}-\eqref{eq:limit_f_4} we conclude that both sets have to contain at least one of the points~$q_{0,j}$,~$q_{\infty,j}$. By a similar argument,~$C_1$ and~$C_2$ have to contain one of the points~$p_{0,j}$,~$p_{\infty,j}$.

By~\eqref{eq:conjugate_functions} the set~$\mathcal D$ is symmetric with respect to complex conjugation. It is therefore sufficient to understand the part of the sets in~$\mathcal R_0$, which we can think of as subsets of the amoeba. The images of~$D_1$,~$D_2$,~$C_1$ and~$C_2$ in the amoeba are still connected components that have to contain the tentacle that corresponds to any angle contained in the set. If, say,~$q_{0,j}\in D_1$ and~$q_{0,j'}\in D_2$, then the images of~$D_1$ and~$D_2$ in the amoeba separate one of the images of~$C_1$ and~$C_2$ from the tentacles corresponding to~$p_{0,j}$ and~$p_{\infty,j}$. Since this cannot happen, we find that one of the sets, say~$D_2$, cannot contain~$q_{0,j'}$ and has to contain a point~$q_{\infty,j}$. See Figure~\ref{fig:amoeba_curves_rough}, right panel.

A similar argument also applies to the sets~$C_1$ and~$C_2$.
\end{proof}

The idea is now to deform one of the contours of the double integral in Proposition~\ref{prop:finite_kernel} to~$\mathcal D$ and the other one to~$\mathcal D^c$. We then use Lemma~\ref{lem:behavior_one-forms} to see that the integrand is infinitesimally small on these contours, as~$N\to \infty$.\footnote{In fact, the integrand is uniformly exponentially small outside of any neighborhood of~$q_c$.} The final expression will come from the residues we pick up in these deformations.  

\begin{lemma}\label{lem:limit_kernel_rough}
The equality of~\eqref{eq:limiting_kernel_lhs} and~\eqref{eq:limiting_kernel} in Theorem~\ref{thm:local_limit} holds for~$(\xi,\eta)$ in the rough region, and the limit is uniform for~$(\xi_N,\eta_N)$ in compact subsets of the rough region.
\end{lemma}

\begin{proof}
Let~$\tilde \gamma_l$ be a simple curve in the interior of~$\mathcal D \cap \mathcal R_0$ that starts in a neighborhood of~$q_{0,j}$ for some~$j$, passes through~$q_c$, and ends in a neighborhood of~$q_{\infty,j'}$ for some~$j'$. The fact that such a curve exists follows from Lemma~\ref{lem:rough_sets}. The curve~$\tilde \gamma_l$ uniquely determines a closed oriented curve in~$\mathcal D\cap \mathcal R$, which is invariant, up to orientation, under complex conjugation. We denote this curve by~$\gamma_l$. We similarly define the curve~$\gamma_s$, but now taking a curve in the interior of~$\mathcal D^c \cap \mathcal R_0$ that starts in a neighborhood of one of~$p_{0,j}$, passes through~$q_c$ and ends in a neighborhood of~$p_{\infty,j'}$. See Figure~\ref{fig:amoeba_curves_rough} for a sketch of the images of the curves in the amoeba.

We deform the contours~$\tilde \Gamma_l$ and~$\tilde \Gamma_s$ in the integral of Proposition~\ref{prop:finite_kernel} to~$\gamma_l$ and~$\gamma_s$, respectively. Note that we might need to deform~$\tilde \Gamma_s$ ``through'' one of the compact ovals~$A_i$ for some~$i=1,\dots,g$, that is, the image in the amoeba is deformed to the other side of the image of~$A_i$. Such deformation contributes an integral along~$A_i$. However, since the curves are invariant under complex conjugation, we also get another contribution of the integral along~$A_i$, but with the opposite orientation. This means that the total contribution of the oval is zero. Moreover, we do not need to pass through the points~$p_{0,j}$ and~$p_{\infty,j}$ when we deform~$\tilde \Gamma_l$ to~$\gamma_l$ or the points~$q_{0,j}$ and~$q_{\infty,j}$ when we deform~$\tilde \Gamma_s$ to~$\gamma_s$, which follows from Assumption~\ref{ass:main_ass}\eqref{ass:wh}, see Figure~\ref{fig:amoeba_curves_rough}. Thus, Lemma~\ref{lem:meromorphic_one-forms} implies that the only contribution in the deformation of the contours comes from the residue at~$z_1=z_2$. By Lemma~\ref{lem:behavior_one-forms} the residue gives us a single integral along a curve~$\gamma_{\xi,\eta}$ which is as in Definition~\ref{def:curve_integration}, see Figure~\ref{fig:amoeba_curves_rough}. By Proposition~\ref{prop:finite_kernel} and Lemma~\ref{lem:behavior_one-forms}, the correlation kernel is given by 
\begin{multline}
 \left[K_\text{path}(2\ell x+i,ky+j;2\ell x'+i',ky'+j')\right]_{j',j=0}^{k-1} \\
 = -\frac{\one_{2\ell x+i>2\ell x'+i'}}{2\pi\i}\int_{\Gamma} \prod_{m=2\ell x'+i'+1}^{2\ell x+i}\phi_m(z)z^{y'-y}\frac{\d z}{z} \\
 + \frac{1}{2\pi\i}\int_{\gamma_{\xi,\eta}} \left(\prod_{m=1}^{i'}\phi_m(z_2)\right)^{-1}\frac{\adj (w_2I-\Phi(z_2))}{\partial_{w_2}\det (w_2I-\Phi(z_2))}
\prod_{m=1}^{i}\phi_m(z_2)\frac{z_2^{\zeta'-\zeta}}{w_2^{\kappa'-\kappa}}\frac{\d z_2}{z_2} \\
 + \frac{1}{(2\pi\i)^2}\int_{\gamma_s}\int_{\gamma_l}\e^{N(F(z_1,w_1;\xi,\eta)-F(z_2,w_2;\xi,\eta))}G(z_1,w_1;z_2,w_2)\frac{z_1^{\zeta'}}{w_1^{\kappa'}} \frac{w_2^{\kappa}}{z_2^{\zeta}} \frac{\d z_2\d z_1}{z_2(z_2-z_1)}.
\end{multline}
The first two terms on the right hand side are equal to the limiting kernel in the statement. Hence, what remains is to show that the last integral tends to zero uniformly on compact subsets as~$N\to \infty$. This follows from the definition of~$\mathcal D$ and standard arguments (see, for instance,~\cite{Oko03}) using the fact that~$G$ is bounded by Lemma~\ref{lem:behavior_one-forms}. 
\end{proof}

\subsubsection{The smooth regions}\label{sec:smooth}

We proceed by proving the equality of~\eqref{eq:limiting_kernel_lhs} and~\eqref{eq:limiting_kernel} in Theorem~\ref{thm:local_limit} in the smooth regions.

By definition of the smooth regions, there is an~$A_\iota$,~$\iota=1,\dots,g$, containing four simple zeros of~$\d F$. Two of these points are local maxima when we view~$\re F$ as a function on~$A_\iota$, we denote them by~$q_{M_1}$ and~$q_{M_2}$, and two are local minima and we denote them by~$q_{m_1}$ and~$q_{m_2}$. Note that the critical points come with a cyclic order along~$A_\iota$ which is alternating between local maxima and local minima. Similarly to the previous section, we define
\begin{equation}
\mathcal D_M=\mathcal D_M(\xi,\eta)=\big\{(z,w)\in \mathcal R:\re F(z,w;\xi,\eta)\geq \textstyle{\min_{i\in \{1,2\}}}\{\re F(q_{M_i};\xi,\eta)\}\big\}.
\end{equation}
and
\begin{equation}
\mathcal D_m=\mathcal D_m(\xi,\eta)=\big\{(z,w)\in \mathcal R:\re F(z,w;\xi,\eta)\leq \textstyle{\max_{i\in\{1,2\}}}\{\re F(q_{m_i};\xi,\eta)\}\big\}.
\end{equation}
Note that~$\mathcal D_M\cap \mathcal D_m=\emptyset$. We denote the connected component of~$\mathcal D_M$ containing~$q_{M_i}$ by~$D_i$, for~$i=1,2$, and the connected component of~$\mathcal D_m$ containing~$q_{m_i}$ by~$C_i$, for~$i=1,2$. See Figure~\ref{fig:amoeba_curves_smooth_frozen}, left panel.

\begin{lemma}\label{lem:smooth_sets}
After a possible relabeling of~$D_1$ and~$D_2$, there exist~$j,j'\in \{1,\dots,\ell\}$ such that~$q_{0,j} \in D_1$ and~$q_{\infty,j'} \in D_2$, while~$q_{0,i}\notin D_2$ and~$q_{\infty,i}\notin D_1$ for all~$i$. Similarly,~$p_{0,j}\in C_1$ and~$p_{\infty,j'}\in C_2$, for some (possibly different)~$j,j'\in\{1,\dots,k\}$, while~$p_{0,i}\notin C_2$ and~$p_{\infty,i}\notin C_1$ for all~$i$.   
\end{lemma} 

\begin{proof}
The argument is very similar to the proof of Lemma~\ref{lem:rough_sets}.
\end{proof}

\begin{lemma}\label{lem:limit_kernel_smooth}
The equality of~\eqref{eq:limiting_kernel_lhs} and~\eqref{eq:limiting_kernel} in Theorem~\ref{thm:local_limit} holds for~$(\xi,\eta)$ in the smooth regions, and the limit is uniform for~$(\xi_N,\eta_N)$ in compact subsets of the smooth regions.
\end{lemma}

\begin{proof}
We pick two simple curves~$\tilde \gamma_{l,1}$ and~$\tilde \gamma_{l,2}$ in the interior of~$\mathcal D_M \cap \mathcal R_0$ where~$\tilde \gamma_{l,1}$ starts in a neighborhood of~$q_{0,j}$ for some~$j$ and ends at~$q_{M_1}$ and~$\tilde \gamma_{l,2}$ starts at~$q_{M_2}$ and ends in a neighborhood of~$q_{\infty,j'}$ for some~$j'$. The existence of such curves follows from Lemma~\ref{lem:smooth_sets}. As described in the proof of Lemma~\ref{lem:limit_kernel_rough}, these curves define two curves in~$\mathcal R$,~$\gamma_{l,i}$,~$i=1,2$, which are invariant, up to orientation, under conjugation. In a similar way, we define the curves~$\gamma_{s,1}$ and~$\gamma_{s,2}$, but now taking the defining curves~$\tilde \gamma_{s,1}$ and~$\tilde \gamma_{s,2}$ in the interior of~$\mathcal D_m \cap \mathcal R_0$, going from a neighborhood of~$p_{0,j}$, for some~$j$, and ending at~$q_{m_1}$, and going from~$q_{m_2}$ and ending at~$p_{\infty,j'}$ for some~$j'$, respectively. See Figure~\ref{fig:amoeba_curves_smooth_frozen}, left panel. 

Let us now deform the contours~$\tilde \Gamma_l$ and~$\tilde \Gamma_s$ in the integral of Proposition~\ref{prop:finite_kernel} to~$\gamma_{l,1} \cup \gamma_{l,2}$ and~$\gamma_{s,1}\cup\gamma_{s,2}$, respectively. In the deformation we also obtain a contour going between~$q_{M_1}$ and~$q_{M_2}$ along~$\partial \mathcal R_0$, and also a contour going between~$q_{m_1}$ and~$q_{m_2}$. These contours do not contribute to the value of the integral since the curves are, up to orientation, invariant under conjugation, which means we will integrate along these parts in both directions.

The rest of the proof follows the proof of Lemma~\ref{lem:limit_kernel_rough} word for word.
\end{proof}

 \begin{figure}[t]
 \begin{center}
   \begin{tikzpicture}[scale=.8]
    \draw (0,0) node {\includegraphics[trim={1cm, 0cm, 1cm, 1cm}, clip, scale=.7]{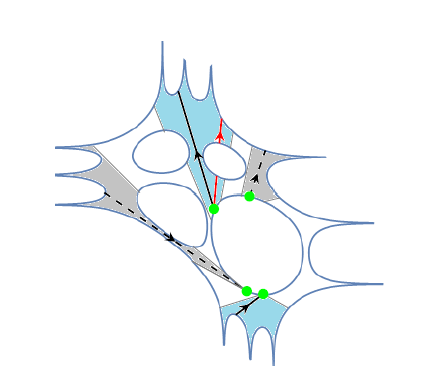}};
    	%angles
	\draw (-1.6,2.7) node {$+\infty$};
	\draw (-2.6,1.5) node {$-\infty$};
	\draw (-.5,-2.2) node {$+\infty$};
	\draw (1.8,1.3) node {$-\infty$};
	\end{tikzpicture}
%	\quad
	\begin{tikzpicture}[scale=.8]
    \draw (0,0) node {\includegraphics[trim={1cm, 0cm, 1cm, 1cm}, clip, scale=.7]{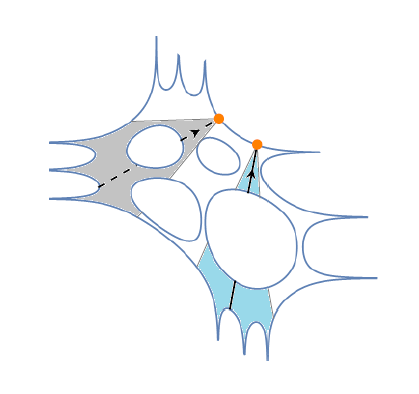}};
    	%angles
	\draw (-1.6,2.9) node {$+\infty$};
	\draw (-2.6,1.7) node {$-\infty$};
	\draw (-.5,-2.0) node {$+\infty$};
	\draw (1.8,1.5) node {$-\infty$};
	\end{tikzpicture}
%	\quad
	\begin{tikzpicture}[scale=.8]
    \draw (0,0) node {\includegraphics[trim={1cm, 0cm, 1cm, 1cm}, clip, scale=.7]{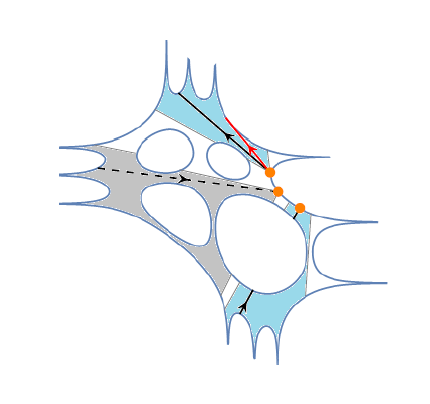}};
    	%angles
	\draw (-1.6,2.9) node {$+\infty$};
	\draw (-2.6,1.7) node {$-\infty$};
	\draw (-.5,-2.0) node {$+\infty$};
	\draw (1.8,1.5) node {$-\infty$};
	\end{tikzpicture}
 \end{center}
  \caption{The integration contours in the steep descent analysis. Left: smooth region,~$\gamma_{s,1}$,~$\gamma_{s,2}$ (dashed),~$\gamma_{l,1}$,~$\gamma_{l,2}$ (solid),~$\gamma_{\xi,\eta}$ (red). Middle and Right: frozen region,~$\gamma_{s,i}$ (dashed),~$\gamma_{l,i}$ (solid),~$\gamma_{\xi,\eta}$ (red). In the middle picture the curve~$\gamma_{\xi,\eta}$ is empty. The blue shaded regions are the sets~$D_i$ and the gray shaded regions are the sets~$C_i$, given in Lemmas~\ref{lem:smooth_sets} and~\ref{lem:frozen_sets}. At the tentacles the behavior of~$\re F$ is indicated. \label{fig:amoeba_curves_smooth_frozen}}
\end{figure}
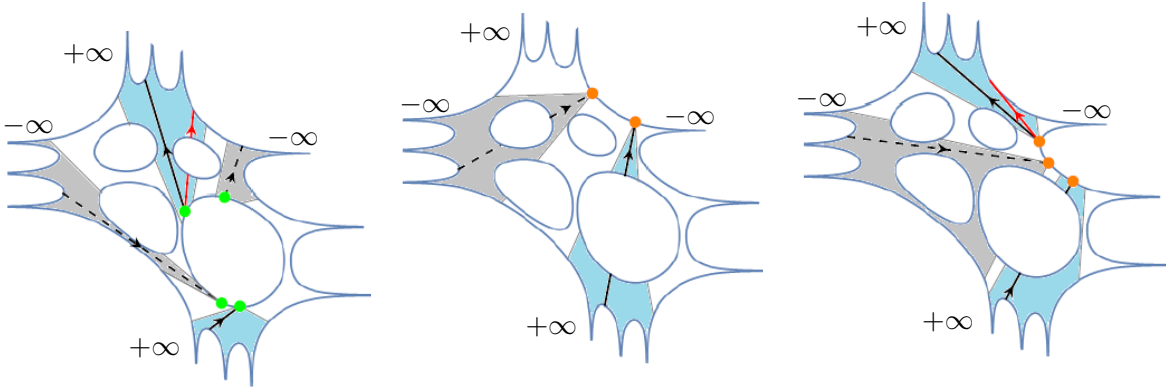

\subsubsection{The frozen regions}\label{sec:frozen}
Finally, let us prove the equality of~\eqref{eq:limiting_kernel_lhs} and~\eqref{eq:limiting_kernel} in Theorem~\ref{thm:local_limit} in the frozen regions.

If~$(\xi,\eta)$ is in a frozen region, then there is, by Definition~\ref{def:regions}, a component of the unbounded oval~$A_{0,\iota}$, for some~$\iota=1,\dots,2(k+\ell)$, containing two or three simple zeros of~$\d F$. If the endpoints of~$A_{0,\iota}$ are two different types of angles, that is, if~$\iota\in\{1,\ell+1,\ell+k+1,2\ell+k+1\}$ (these correspond to the north, east, south and west frozen regions), then~$A_{0,\iota}$ contains two critical points, otherwise the component contains three critical points. See Figures~\ref{fig:critical_points} and~\ref{fig:amoeba_curves_smooth_frozen} (middle and right panels). If we view~$\re F$ as a function along~$A_{0,\iota}$, then the zeros of~$\d F$ are local minima or local maxima. We denote the local maxima of~$\re F$ by~$q_{M_i}$ and the local minima by~$q_{m_i}$. If~$A_{0,\iota}$ contains two critical points, then there is one local maximum and one local minimum. If~$A_{0,\iota}$ contains three critical points, then there are two local maxima and one local minimum or the other way around. More precisely, if the endpoints of~$A_{0,\iota}$ are~$p_{0,j}$ and~$p_{0,j'}$, or~$p_{\infty,j}$ and~$p_{\infty,j'}$, then there are two local maxima and one local minimum, while if the endpoints are~$q_{0,j}$ and~$q_{0,j'}$, or~$q_{\infty,j}$ and~$q_{\infty,j'}$, then there are two local minima and one local maximum. This follows from~\eqref{eq:limit_f_1}-\eqref{eq:limit_f_4}, see also Figure~\ref{fig:critical_points}.

Proceeding as in Section~\ref{sec:smooth}, we define the sets
\begin{equation}
\mathcal D_M=\mathcal D_M(\xi,\eta)=\big\{(z,w)\in \mathcal R:\re F(z,w;\xi,\eta)\geq \textstyle{\min_i}\{\re F(q_{M_i};\xi,\eta)\}\big\}.
\end{equation}
and
\begin{equation}
\mathcal D_m=\mathcal D_m(\xi,\eta)=\big\{(z,w)\in \mathcal R:\re F(z,w;\xi,\eta)\leq \textstyle{\max_i}\{\re F(q_{m_i};\xi,\eta)\}\big\}.
\end{equation}
We denote the connected component of~$\mathcal D_M$ containing~$q_{M_i}$ by~$D_i$, and the connected component of~$\mathcal D_m$ containing~$q_{m_i}$ by~$C_i$. See Figure~\ref{fig:amoeba_curves_smooth_frozen} for a sketch of the components.

\begin{lemma}\label{lem:frozen_sets}
If~$\iota=1,\dots,\ell+1$, then there exists~$j\in\{1,\dots,\ell\}$ such that~$q_{0,j}\in D_1$ and if~$\iota=k+\ell+1,\dots,k+2\ell+1$, then there exists~$j\in\{1,\dots,\ell\}$ such that~$q_{\infty,j}\in D_1$. Otherwise, after a possible relabeling of~$D_1$ and~$D_2$, there exists~$j\in\{1,\dots,\ell\}$ such that~$q_{\infty,j}\in D_1$ and a~$j'$ so that~$q_{0,j'} \in D_2$. Similarly, if~$\iota=\ell+1,\dots,k+\ell+1$, then there exists~$j\in\{1,\dots,\ell\}$ such that~$p_{\infty,j}\in C_1$ and if~$\iota=k+2\ell+1,\dots,2k+2\ell,1$, then there is a~$j$ so that~$p_{0,j}\in C_1$. Otherwise, after a possible relabeling of~$C_1$ and~$C_2$, there is a~$j$ so that~$p_{\infty,j}\in C_1$ and a~$j'$ so that~$p_{0,j'} \in C_2$.
\end{lemma} 
\begin{proof}
The proof is similar to the proof of Lemma~\ref{lem:rough_sets}.
\end{proof}

\begin{lemma}\label{lem:limit_kernel_frozen}
The equality of~\eqref{eq:limiting_kernel_lhs} and~\eqref{eq:limiting_kernel} in Theorem~\ref{thm:local_limit} holds for~$(\xi,\eta)$ in the frozen regions, and the limit is uniform for~$(\xi_N,\eta_N)$ in compact subsets of the frozen regions.
\end{lemma}

\begin{proof}
As we did for the rough and smooth regions, we pick a simple closed curve~$\gamma_{l,i}$ in~$D_i$ which connects the critical point in~$D_i$ with an angle in~$D_i$. Moreover, we take the orientation to be compatible with~$\tilde \Gamma_l$. Similarly we define~$\gamma_{s,i}$ in~$C_i$, with an orientation that is compatible with~$\tilde \Gamma_s$. See Figure~\ref{fig:amoeba_curves_smooth_frozen}. Above~$i$ runs over~$\{1,2\}$ or only over~$\{1\}$.

We deform the contours~$\tilde \Gamma_l$ and~$\tilde \Gamma_s$ in the integral of Proposition~\ref{prop:finite_kernel} to~$\cup_i\gamma_{l,i}$ and~$\cup_i\gamma_{s,i}$, respectively. The rest of the proof follows the proof of Lemma~\ref{lem:limit_kernel_rough} word for word.
\end{proof}

\subsection{Proof of Theorem~\ref{thm:local_limit}}\label{sec:uniform_limt}

The first equality~\eqref{eq:limiting_kernel} of Theorem~\ref{thm:local_limit} is proved in Lemmas~\ref{lem:limit_kernel_rough},~\ref{lem:limit_kernel_smooth} and~\ref{lem:limit_kernel_frozen} above. What remains to complete the proof of Theorem~\ref{thm:local_limit} is to prove the equality of~\eqref{eq:limiting_kernel} and~\eqref{eq:local_limit}. The argument is similar to the proof of~\cite[Theorem 51]{BCT22}.

\begin{proof}[Proof of Theorem~\ref{thm:local_limit}]
We start from the integral on the torus~\eqref{eq:local_limit} and prove that it is equal to the integrals on the Riemann surface~\eqref{eq:limiting_kernel}. Up to details, the equality follows from the residue theorem. 

Let~$r_1$ and~$r_2$ be as in the statement, cf. Definition~\ref{def:point_integration}. Let us fix~$z\in \CC$ with~$|z|=\e^{r_1}$. If~$2\ell\kappa+i> 2\ell\kappa'+i'$, then, for~$|w|<\e^{r_2}$, the meromorphic function
\begin{equation}
w\mapsto \left(\prod_{m=1}^{i'}\phi_m(z)\right)^{-1}(\Phi(z)-wI)^{-1}\Phi(z)^{\one_{i'\geq i}}w^{\one_{i'< i}}
\prod_{m=1}^{i}\phi_m(z)\frac{z^{\zeta'-\zeta}}{w^{\kappa'-\kappa+1}} 
\end{equation}
has a pole in~$w$ if and only if~$w\mapsto \det (\Phi(z)-wI)$ vanishes, that is, if~$(z,w)\in \mathcal R$. By the residue theorem,
\begin{multline}\label{eq:unified_residue}
\frac{1}{2\pi\i}\int_{|w|=\e^{r_2}}\left(\prod_{m=1}^{i'}\phi_m(z)\right)^{-1}\frac{\adj(\Phi(z)-wI)}{\det (\Phi(z)-wI)}\cdot\Phi(z)^{\one_{i'\geq i}}w^{\one_{i'< i}}
\prod_{m=1}^{i}\phi_m(z)\frac{z^{\zeta'-\zeta}}{w^{\kappa'-\kappa}}\frac{\d w}{w} \\
= -\sum_{w:(z,w)\in \mathcal R, |w|<\e^{r_2}}\left(\prod_{m=1}^{i'}\phi_m(z)\right)^{-1}\frac{\adj(wI-\Phi(z))}{\partial_w\det (wI-\Phi(z))}
\prod_{m=1}^{i}\phi_m(z)\frac{z^{\zeta'-\zeta}}{w^{\kappa'-\kappa}}.
\end{multline}
Here we used the identity~$\adj(wI-\Phi(z))\Phi(z)=\adj(wI-\Phi(z))w$ if~$(z,w)\in \mathcal R$. The points~$(z,w)\in \mathcal R$ in the sum all lie in the part of the curve~$|z|=\e^{r_1}$ with~$|w|<\e^{r_2}$.

If instead~$2\ell\kappa+i\leq 2\ell\kappa'+i'$, then we compute the integral in~\eqref{eq:unified_residue} by summing the residues in the exterior of the contour. Since~$(wI-\Phi(z))^{-1}=\Ordo(w^{-1})$ as~$w\to \infty$, the only contribution comes, again, from the zeros of~$\det(wI-\Phi(z))$. Thus,~\eqref{eq:unified_residue} still holds, up to a sign due to the orientation of the curve, and the sum runs instead over the set~$\{w:(z,w)\in \mathcal R, |w|>\e^{r_2}\}$. The points~$(z,w)$ in the sum now lie in the part of the curve~$|z|=\e^{r_1}$ with~$|w|>\e^{r_2}$. 

 \begin{figure}[t]
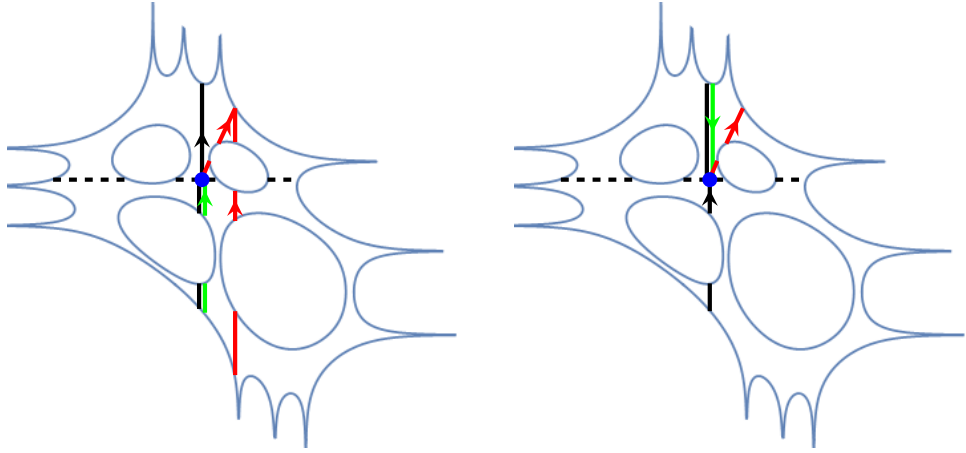

 \begin{center}
  \begin{tikzpicture}[scale=1]
    \tikzset{->-/.style={decoration={
  markings, mark=at position .5 with {\arrow{stealth}}},postaction={decorate}}}
    \tikzset{-->-/.style={decoration={
  markings, mark=at position .7 with {\arrow{stealth}}},postaction={decorate}}}
    \draw (0,0) node {\includegraphics[trim={1cm, 1cm, 1cm, 1cm}, clip, angle=180, scale=.5]{amoeba3x3applicable.png}};
	%|z|=\e^x
%   \draw[line width=1.2mm](-.4,-1.15)--(-.4,-.78);
%   \draw[line width=1.2mm](-.4,.15)--(-.4,.6);
%   \draw[->-,ultra thick](-.4,.6)--(-.4,1.88);
%   \draw[green,line width=.5mm](-.4,-1.15)--(-.4,-.78);
%   \draw[-->-,green,line width=.5mm](-.4,.15)--(-.4,.6);
\foreach \x in {0.039}
{   \draw[ultra thick](-.4-\x,-1.12)--(-.4-\x,-.78);
   \draw[ultra thick](-.4-\x,.15)--(-.4-\x,.6);
   \draw[->-,ultra thick](-.4,.6)--(-.4,1.88);
   \draw[green,ultra thick](-.4+\x,-1.17)--(-.4+\x,-.78);
   \draw[-->-,green,ultra thick](-.4+\x,.12)--(-.4+\x,.6);
   }
   %|w|=\e^y
   \draw [dashed,ultra thick](-2.38,.6)--(-1.35,.6);
   \draw [dashed,ultra thick](-.75,.6)--(-.2,.6);
   \draw [dashed,ultra thick](.47,.6)--(.89,.6);
   %\gamma_{\xi,\eta}
   \draw[red,ultra thick](-.39,.64)--(-.3,.85);
   \draw[-->-,red,ultra thick](-.18,1.08)--(0.04,1.55);
   % |z|=1
   \draw[red,ultra thick](0.04,-2)--(0.04,-1.15);
   \draw[-->-,red,ultra thick](0.04,0.05)--(0.04,.45);
   \draw[red,ultra thick](0.04,1.1)--(0.04,1.55);
   % Critical point
   % Rough
	\draw (-.4,.6) node[blue,circle,fill,inner sep=2.pt]{};
  \end{tikzpicture}
  \quad
   \begin{tikzpicture}[scale=1]
    \tikzset{->-/.style={decoration={
  markings, mark=at position .5 with {\arrow{stealth}}},postaction={decorate}}}
    \tikzset{-->-/.style={decoration={
  markings, mark=at position .7 with {\arrow{stealth}}},postaction={decorate}}}
    \draw (0,0) node {\includegraphics[trim={1cm, 1cm, 1cm, 1cm}, clip, angle=180, scale=.5]{amoeba3x3applicable.png}};
    %|z|=\e^x
\foreach \x in {0.039}
{   \draw[ultra thick](-.4,-1.15)--(-.4,-.78);
   \draw[-->-,ultra thick](-.4,.15)--(-.4,.6);
   \draw[ultra thick](-.4-\x,1.88)--(-.4-\x,.6);
   \draw[->-,green,ultra thick](-.4+\x,1.88)--(-.4+\x,.6);
   }
   %|w|=\e^y
   \draw [dashed,ultra thick](-2.38,.6)--(-1.35,.6);
   \draw [dashed,ultra thick](-.75,.6)--(-.2,.6);
   \draw [dashed,ultra thick](.47,.6)--(.89,.6);
   %\gamma_{\xi,\eta}
   \draw[red,ultra thick](-.39,.64)--(-.3,.85);
   \draw[-->-,red,ultra thick](-.18,1.08)--(0.04,1.55);
   % Critical point
   % Rough
	\draw (-.4,.6) node[blue,circle,fill,inner sep=2pt]{};
  \end{tikzpicture}
 \end{center}
  \caption{An example of the curves in Section~\ref{sec:uniform_limt} with~$(r_1,r_2)$ in the interior of the amoeba, that is, with~$(\xi,\eta)$ in the rough region. The curves are:~$|z|=\e^{r_1}$ (solid, black),~$|w|=\e^{r_2}$ (dashed),~$\gamma_{r_1,r_2}$ (green). Left: ($2\ell\kappa+i> 2\ell\kappa'+i'$)~$\gamma_{\xi,\eta}$ and~$\gamma_0$ (red). Right: ($2\ell\kappa+i\leq 2\ell\kappa'+i'$)~$\gamma_{\xi,\eta}$ (red). The adjacent black and green curves do, in reality, coincide. \label{fig:amoeba_uniformed_curves}}
\end{figure}

We saw in the proof of Lemma~\ref{lem:adjoint} that if~$(z,w),(z,\tilde w)\in \mathcal R$ then
\begin{equation}\label{eq:projection}
\frac{\adj(wI-\Phi(z))}{\partial_w\det (wI-\Phi(z))}\frac{\adj(\tilde wI-\Phi(z))}{\partial_{\tilde w}\det (\tilde wI-\Phi(z))}=\frac{\adj(\tilde wI-\Phi(z))}{\partial_{\tilde w}\det (\tilde wI-\Phi(z))},
\end{equation}
if~$w=\tilde w$, and the left hand side is zero if~$w\neq \tilde w$. Let~$\gamma_{r_1}=\{(z,\tilde w)\in \mathcal R: |z|=\e^{r_1}\}$ be oriented so that the projection to the~$z$ variable is oriented in the positive direction. As in the proof of Proposition~\ref{prop:finite_kernel}, we use Lemma~\ref{lem:adjoint} to write the double integral~\eqref{eq:local_limit} as an integral on~$\mathcal R$:
\begin{multline}\label{eq:unified_double_integral}
\frac{1}{(2\pi\i)^2}\int_{|z|=\e^{r_1}}\int_{|w|=\e^{r_2}}\left(\prod_{m=1}^{i'}\phi_m(z)\right)^{-1}\frac{\adj(\Phi(z)-wI)}{\det (\Phi(z)-wI)}\Phi(z)^{\one_{i'\geq i}}w^{\one_{i'< i}}
\prod_{m=1}^{i}\phi_m(z)\frac{z^{\zeta'-\zeta}}{w^{\kappa'-\kappa}}\frac{\d w}{w}\frac{\d z}{z} \\
=\frac{1}{(2\pi\i)^2}\int_{(z,\tilde w)\in \gamma_{r_1}}\int_{|w|=\e^{r_2}}\left(\prod_{m=1}^{i'}\phi_m(z)\right)^{-1}\frac{\adj(\Phi(z)-wI)}{\det (\Phi(z)-wI)}\frac{\adj(\tilde wI-\Phi(z))}{\partial_{\tilde w}\det (\tilde wI-\Phi(z))} \\
\times\Phi(z)^{\one_{i'\geq i}}w^{\one_{i'< i}}\prod_{m=1}^{i}\phi_m(z)\frac{z^{\zeta'-\zeta}}{w^{\kappa'-\kappa}} \frac{\d w}{w}\frac{\d z}{z}.
\end{multline}
We emphasize that~$(z,\tilde w) \in \mathcal R$ while~$w\in \CC$. By~\eqref{eq:unified_residue} and~\eqref{eq:projection}, the right hand side of~\eqref{eq:unified_double_integral} is equal to
\begin{equation}\label{eq:single_integral}
\frac{-1}{(2\pi\i)^2}\int_{\gamma_{r_1,r_2}}\left(\prod_{m=1}^{i'}\phi_m(z)\right)^{-1}\frac{\adj(\tilde wI-\Phi(z))}{\partial_{\tilde w}\det (\tilde wI-\Phi(z))}
\prod_{m=1}^i\phi_m(z)\frac{z^{\zeta'-\zeta}}{(\tilde w)^{\kappa'-\kappa}} \frac{\d z}{z},
\end{equation}
where~$\gamma_{r_1,r_2}$ is the part of~$\gamma_{r_1}$ with~$|\tilde w|<\e^{r_2}$ if~$2\ell\kappa+i>2\ell\kappa'+i'$ and the part with~$|\tilde w|>\e^{r_2}$ if~$2\ell\kappa+i\leq 2\ell\kappa'+i'$, but oriented in opposite direction. See Figure~\ref{fig:amoeba_uniformed_curves}. 

Using an argument as in the proof of Lemma~\ref{lem:meromorphic_one-forms}, see in particular~\eqref{eq:adjugate_matrix_prime_form}, we obtain that
\begin{equation}\label{eq:integrand_at_angles}
\left(\prod_{m=1}^{i'}\phi_m(z)\right)^{-1}\frac{\adj(\tilde wI-\Phi(z))}{\partial_{\tilde w}\det (\tilde wI-\Phi(z))}\prod_{m=1}^i\phi_m(z)\frac{\d z}{z}=
\begin{cases}
\Ordo\left(\prod_{m=\tilde i+\iota+1}^{\tilde i'+\iota'}E(q_{0,m},q)^{-1}\right),& q \to q_{0,m}, \\
\Ordo\left(\prod_{m=\tilde i+1}^{\tilde i'}E(q_{\infty,m},q)\right),& q \to q_{\infty,m},
\end{cases}
\end{equation}
where~$i=2\tilde i+\iota$, and~$i'=2\tilde i'+\iota'$, with~$\iota,\iota'\in \{0,1\}$. If~$2\ell\kappa+i\leq 2\ell\kappa'+i'$ the integrand in~\eqref{eq:single_integral} is analytic at~$q_{\infty,m}$, for all~$m$. Indeed, if~$\kappa<\kappa'$, then the power of~$\tilde w$ in~\eqref{eq:single_integral} cancels out any pole and the entire expression is analytic at~$q_{\infty,m}$, and if~$\kappa=\kappa'$, then~$i\leq i'$, and it follows from~\eqref{eq:integrand_at_angles} that~\eqref{eq:single_integral} is analytic at~$q_{\infty,m}$. This means that we can deform~$\gamma_{r_1,r_2}$ to~$-\gamma_{\xi,\eta}$, the contour~$\gamma_{\xi,\eta}$ from Definition~\ref{def:curve_integration}, oriented in opposite direction. We can change the orientation with the extra minus sign in~\eqref{eq:single_integral}, and this proves the theorem for~$2\ell\kappa+i\leq 2\ell\kappa'+i'$. 

When~$2\ell\kappa+i>2\ell\kappa'+i'$ there is an extra term in~\eqref{eq:limiting_kernel}. We write the extra integral as an integral on~$\mathcal R$ using Lemma~\ref{lem:adjoint},
\begin{multline}\label{eq:extra_integral}
-\frac{1}{2\pi\i}\int_{\Gamma} \left(\prod_{m=1}^{i'}\phi_m(z)\right)^{-1}\Phi(z)^{\kappa-\kappa'}\prod_{m=1}^{i}\phi_m(z)z^{\zeta'-\zeta}\frac{\d z}{z} \\
=-\frac{1}{(2\pi\i)^2}\int_{\gamma_0}\left(\prod_{m=1}^{i'}\phi_m(z)\right)^{-1}\frac{\adj(wI-\Phi(z))}{\partial_{w}\det (wI-\Phi(z))}
\prod_{m=1}^i\phi_m(z)\frac{z^{\zeta'-\zeta}}{w^{\kappa'-\kappa}}\frac{\d z}{z},
\end{multline}
where~$\gamma_0=\{(z,w)\in \mathcal R:|z|=1\}$ is oriented so that its projection to the~$z$ variable is oriented in the positive direction. By~\eqref{eq:integrand_at_angles} we may deform~$\gamma_{r_1,r_2}$ through~$q_{0,j}$ for all~$j$. Hence, we deform~$\gamma_{r_1,r_2}$ in~\eqref{eq:single_integral} to~$(-\gamma_{\xi,\eta})\cup \gamma_0$ (see the left panel of Figure~\ref{fig:amoeba_uniformed_curves}) and obtain the integral in the statement. This proves the theorem for~$2\ell\kappa+i>2\ell\kappa'+i'$, and concludes the proof of Theorem~\ref{thm:local_limit}.
\end{proof}

\subsection{Proof of Proposition~\ref{prop:limit_shape}}\label{sec:height_function}
In this final section we prove Proposition~\ref{prop:limit_shape}. The proposition does not follow from Theorem~\ref{thm:local_limit}, however, after writing the expectation of the height function appropriately, the proof of the proposition is obtained by tweaking the proof of that theorem slightly. 

\begin{proof}[Proof of Proposition~\ref{prop:limit_shape}]
Since we are interested in the normalized height function, it suffices to determine the limit
\begin{equation}
\lim_{N\to\infty}\frac{1}{k\ell N}\EE\left[h(2\ell x,2k y)\right].
\end{equation}

By definition of the height function~\eqref{eq:height_difference_aztec}, we express it at the face~$(2\ell x,2ky)$ by first going from~$(0,0)$ to~$(2\ell x,0)$ and then to~$(2\ell x,2ky)$, that is,
\begin{equation}
h(2\ell x,2k y)=\left(h(2\ell x,2k y)-h(2\ell x,0)\right)+\left(h(2\ell x,0)-h(0,0)\right).
\end{equation}
The second term~$h(2\ell x,0)-h(0,0)=\ell x$, see Figure~\ref{fig:height_function}. Note that for any~$j\geq 0$,~$h(2\ell x,2(j+1))=h(2\ell x, 2j)+1$ if~$\mathrm{b}_{\ell x,j}$ is covered by a west or south dimer, and~$h(2\ell x,2(j+1))=h(2\ell x, 2j)$ if it is covered by an east or north dimer. Equivalently, the value of~$h$ increases by~$1$ if there is a particle at~$\mathrm{b}_{\ell x,j}$ from the point process defined in Section~\ref{sec:paths}, and the value of~$h$ does not change if there is no particle. This means we can use~$K_\text{path}$ to express the expectation of the height function. In particular
\begin{equation}
\EE\left[h(2\ell x,2k y)\right]=\sum_{0\leq ky_0+j_0<k y} \PP_\text{path}\left[\text{there is a particle at }\mathrm{b}_{\ell x,ky_0+j_0}\right]+\ell x,
\end{equation}
where the probability measure is the point process~\eqref{eq:measure_on_points} with particles on the west and south dimers. Recall~\eqref{eq:dpp} which implies that
\begin{equation}
\PP_\text{path}\left[\text{there is a particle at }\mathrm{b}_{\ell x+i,ky_0+j_0}\right]=K_\text{path}(2\ell x+2i,ky_0+j_0,2\ell x+2i,ky_0+j_0),
\end{equation}
and hence,
\begin{equation}\label{eq:height_function_trace}
\EE\left[h(2\ell x,2k y)\right]
=\ell x+\sum_{0\leq y_0< y} \Tr \left[K_\text{path}(2\ell x,ky_0+j_0,2\ell x,ky_0+j_0')\right]_{j_0',j_0}^{k-1}.
\end{equation}
Using the formula for the correlation kernel in Theorem~\ref{thm:bd_thm}, we find that the sum in~\eqref{eq:height_function_trace} is equal to
\begin{multline}\label{eq:height_function_sum}
 \frac{1}{(2\pi\i)^2}\int_{\Gamma_s}\int_{\Gamma_l} \Tr \left[\Phi(z_1)^{-x}\widetilde \phi_-(z_1)\widetilde\phi_+(z_2)\Phi(z_2)^{x-kN}\right]\frac{\d z_2\d z_1}{(z_2-z_1)^2} \\
-\frac{1}{(2\pi\i)^2}\int_{\Gamma_s}\int_{\Gamma_l} \Tr \left[\Phi(z_1)^{-x}\widetilde \phi_-(z_1)\widetilde\phi_+(z_2)\Phi(z_2)^{x-kN}\frac{z_1^{y}}{z_2^{y}}\right]\frac{\d z_2\d z_1}{(z_2-z_1)^2},
\end{multline}
where we have taken the summation over~$y_0$ inside the integral. 

We begin by evaluating the first integral in~\eqref{eq:height_function_sum}. For the sake of brevity, let us use the notation~$A=\widetilde \phi_+\Phi^{x-kN}$. Note that~$A^{-1}=\Phi^{-x}\widetilde \phi_-$. The matrix~$A$ is analytic in the interior of~$\Gamma_l$, which means we may deform~$\Gamma_l$ to a point, and the only contribution comes from the double pole at~$z_2=z_1$. By Jacobi's formula~\eqref{eq:jacobis_formula}, 
\begin{equation}
\Tr\left[A(z)^{-1}A'(z)\right]=\frac{(\det A(z))'}{\det A(z)}.
\end{equation}
Thus, the first integral in~\eqref{eq:height_function_sum} is equal to
\begin{equation}
\frac{1}{2\pi\i}\int_{\Gamma_s}\frac{(\det A(z))'}{\det A(z)}\d z=\ell k N-\ell x,
\end{equation}
where the right hand side is the number of zeros of~$\det A$ inside~$\Gamma_s$ (all coming from~$\Phi^{x-kN}$, cf.~\eqref{eq:matrix_valued_function}), and the equality follows from the argument principle.

We proceed with the second term in~\eqref{eq:height_function_sum}. If we ignore the trace, it is similar to the expression in Theorem~\ref{thm:bd_thm} with~$x'=x$,~$y'=y$ and~$i'=i=0$, it only differs by a factor~$z_2(z_2-z_1)^{-1}$. This means we can follow the proof of~\eqref{eq:limiting_kernel} word for word, and find that the leading order term of the desired integral as~$N\to \infty$ comes from the residue at~$z_2=z_1$ along the curve~$\gamma_{\xi,\eta}$ given in Definition~\ref{def:curve_integration}. The residue will, however, be different from what we got in the proof of~\eqref{eq:limiting_kernel} since the pole is of order~$2$ instead of order~$1$. Hence, the leading order term of the second term of~\eqref{eq:height_function_sum} is given by
\begin{equation}\label{eq:height_function_leading_term}
-\frac{1}{2\pi\i}\int_{\gamma_{\xi,\eta}}\left.\frac{\d}{\d z_2}\left(\e^{N(F(q_1;\xi_N,\eta_N)-F(q_2;\xi_N,\eta_N))}\Tr G(q_1;q_2)\right)\right|_{q_2=q_1}\d z_1,
\end{equation}
where~$q_i=(z_i,w_i)\in \mathcal R$. The integral~\eqref{eq:height_function_leading_term} can be computed as~$N\to \infty$ using that
\begin{equation}
\Tr G(q_2;q_2)=1, \quad \text{and} \quad \left.\frac{\d }{\d z_2}\Tr G(\tilde q_1;\tilde q_2)\right|_{\tilde q_2=\tilde q_2}=\Ordo(1),
\end{equation}
where the error term is uniform on compact subsets of~$\widetilde{\mathcal R}$ which do not contain any lifts of the angles, and~$\tilde q_i$ is a lift of~$q_i$ to~$\widetilde{\mathcal R}$. For the first equality we have used~\eqref{eq:adjoint}, and for the second equality we have used the explicit expression of~$G$ and the fact that a similar bound is true for the derivative of~$\Theta\left(\tilde q;e_{\mathrm w_{0,j}/\mathrm b_{0,j}}^{(kN)}\right)$. This implies that~\eqref{eq:height_function_leading_term} is equal to
\begin{equation}
\frac{N}{2\pi\i}\int_{\gamma_{\xi,\eta}}\d F
-\frac{1}{2\pi\i}\int_{\gamma_{\xi,\eta}}\left.\frac{\d}{\d z_2}\Tr G(q_1;q_2)\right|_{q_2=q_1}\d z_1,
\end{equation}
where the second term is bounded. Hence,
\begin{equation}
\frac{1}{k\ell N}\EE\left[h(2\ell x,2k y)\right]=\frac{1}{2\pi\i k\ell}\int_{\gamma_{\xi,\eta}}\d F+1+\Ordo(N^{-1}),
\end{equation}
as~$N\to \infty$.

We end the proof by differentiating~$\bar h$. In the rough region the endpoints of~$\gamma_{\xi,\eta}$ lie at a zero of~$\d F$. In the smooth and frozen regions we can deform~$\gamma_{\xi,\eta}$ so this is still true. It means that (cf.~\eqref{eq:def_action_function})
\begin{equation}
\partial_\xi \bar h(\xi,\eta)=\frac{1}{k\ell}\frac{1}{2\pi\i}\int_{\gamma_{\xi,\eta}}\partial_\xi\d F=-\frac{1}{2\ell}\frac{1}{2\pi\i}\int_{\gamma_{\xi,\eta}}\frac{\d w}{w}.
\end{equation} 
The expression for the derivative with respect to~$\eta$ is obtained similarly.
\end{proof}

\appendix

\section{Connection with~\cite{ADPZ20} and~\cite{KO07}}\label{app:burgers}
It turns out that the map~$\Omega$ from Section~\ref{sec:global_limit} can be related to some of the objects studied in~\cite{ADPZ20} and in~\cite{KO07}, and we will discuss these relations in this section.

We first prove that the function~$\Omega$ solves the version of the complex Burgers equation as presented in~\cite[Theorem 1]{KO07}. The proof is similar to the one in~\cite{Ber21}.

\begin{proposition}\label{prop:app:burgers}
Let~$(\xi,\eta)\in \mathcal F_R$, the rough region, and let us denote~$(z(u,v),w(u,v))=\Omega(u,v)=\Omega(\xi,\eta)$, where~$u=-\frac{\xi+1}{2\ell}$ and~$v=-\frac{\eta+1}{2k}$, and~$\Omega$ is as in Definition~\ref{def:omega}. Then 
\begin{equation}
\frac{z_u}{z}+\frac{w_v}{w}=0
\end{equation}
and~$P(z,w)=0$.
\end{proposition} 
\begin{proof}
The identity~$P(z,w)=0$ follows by definition. 

Let~$(z,w(z))$ be a local coordinate in a neighborhood of~$\Omega(u,v)\in \mathcal R_0$. Since~$\Omega$ is a simple critical point of~$F$~\eqref{eq:def_action_function}, we get, by differentiating~$\d F=0$ with respect to~$u$,
\begin{equation}
0=\frac{\d^2 }{\d z^2}\left(F(z,w(z);u,v)\right)z_u+k\ell\frac{w'(z)}{w},
\end{equation}
where~$\frac{\d^2 }{\d z^2}\left(F(z,w(z);u,v)\right)\neq 0$. Similarly, by differentiating~$\d F=0$ with respect to~$v$, we obtain 
\begin{equation}
0=\frac{\d^2 }{\d z^2}\left(F(z,w(z);u,v)\right)z_v-\frac{k\ell}{z}.
\end{equation}
We multiply the first equality with~$z_v$ and then substitute the second equality into the first. The proposition now follows since~$w'(z)z_v=w_v$. 
\end{proof}

In the coordinate system of the previous proposition there is an alternative expression of the map~$\Log \circ \, \Omega$ of Theorem~\ref{thm:arctic_curves}. In fact, we will show that~$\Log \circ\, \Omega$ is closely related to the map studied in~\cite{ADPZ20}, see also Remark~\ref{rem:relation_to_adpz}. 

The \emph{Ronkin function}~$R:\RR^2\to \RR$ of~$P$ is given by
\begin{equation}
R(r_1,r_2)=\frac{1}{(2\pi\i)^2}\int_{|z|=1}\int_{|w|=1} \log\left|P\left(\e^{r_1}z,\e^{r_2}w\right)\right| \frac{\d w}{w}\frac{\d z}{z}.
\end{equation}
For properties of the Ronkin function, we refer to~\cite{Mik04} and references therein. The Ronkin function is strictly convex in the interior of the amoeba and linear in the exterior. Moreover,~$\Int N(P) \subset \nabla R \subset N(P)$,~where~$\Int N(P)$ is the interior of the Newton polygon~$N(P)$. The \emph{surface tension}~$\sigma$ is the Legendre transform of the Ronkin function:
\begin{equation}
\sigma (s,t)=\max_{(r_1,r_2)\in \RR^2}\left(-R(r_1,r_2)+sr_1+tr_2\right). 
\end{equation}
In the interior of the amoeba the gradient~$\nabla R$ is an invertible function, and it follows from the definition of the Legendre transform that~$\nabla \sigma =(\nabla R)^{-1}$, see, for instance,~\cite[§ 26]{Roc97}.

Recall that~$\nabla \bar h$ is the slope of the limiting height function~\eqref{eq:slope_new_coord}. We define~$L(s,t)=(s,t)+(0,k)$ so that~$L\circ \nabla \bar h\in N(P)$, see Remark~\ref{rem:possible_slopes}. As noted in the same remark, the composition corresponds to changing the definition of the height function by choosing a different reference configuration. The composition~$(I+\nabla \sigma) \circ L \circ \nabla \bar h$ was used in~\cite{ADPZ20}, see, in particular, Section~3 therein, to study the limit shape.
\begin{proposition}\label{prop:app:maps_relations}
Let~$(u,v)$ be as in Proposition~\ref{prop:app:burgers}, then
\begin{equation}\label{eq:maps_relations}
\nabla \sigma \circ L \circ \nabla \bar h(u,v)=\Log \circ \, \Omega(u,v).
\end{equation}
\end{proposition} 
\begin{proof}
It follows from the residue theorem, see also the proof of Theorem~\ref{thm:local_limit} in Section~\ref{sec:uniform_limt} for a similar computation with more details provided, that, if~$(r_1,r_2)=\Log \circ \, \Omega(u,v)$, then
\begin{multline}
\nabla R(r_1,r_2)=
\left(\frac{1}{(2\pi\i)^2}\int_{|z|=\e^{r_1}}\int_{|w|=\e^{r_2}} \frac{P_z(z,w)}{P(z,w)} \frac{\d w}{w}\d z,
\frac{1}{(2\pi\i)^2}\int_{|z|=\e^{r_1}}\int_{|w|=\e^{r_2}} \frac{P_w(z,w)}{P(z,w)} \d w\frac{\d z}{z} \right) \\
=\left(\frac{1}{2\pi\i}\int_{\gamma_{\xi,\eta}}\frac{\d w}{w}, -\frac{1}{2\pi\i}\int_{\gamma_{\xi,\eta}}\frac{\d z}{z}+k\right)=L\circ \nabla \bar h(u,v),
\end{multline}
with~$\gamma_{\xi,\eta}$ as in Definition~\ref{def:curve_integration}. Hence,
\begin{equation}
\nabla \sigma \circ L \circ \nabla \bar h(u,v)=(\nabla R)^{-1}\circ \nabla R\circ \Log \circ \, \Omega(u,v) = \Log \circ \, \Omega(u,v).
\end{equation}
\end{proof}
The mapping of the left hand side of~\eqref{eq:maps_relations} may be interpreted as the composition
\begin{equation}
(u,v)\mapsto (u,v,\bar h(u,v)) \mapsto (r_1,r_2,R(r_1,r_2)) \mapsto (r_1,r_2),
\end{equation}
where the second mapping takes a point on the graph of~$\bar h$ to the unique point on the graph of~$R$ with a parallel tangent plane. The equality~\eqref{eq:maps_relations} together with this interpretation provides an intuitive reason, although not a proof, why Corollary~\ref{cor:arctic_curve_slopes} hold.

\section{Alternative proofs of Theorem~\ref{thm:homeomorphism} and Lemma~\ref{lem:phi_kasteleyn}}\label{app:alt_proofs}
Several proofs in the present paper rely on profound algebraic geometry theory. While this is indeed necessary, some of the arguments can be alternatively proved in a more concrete manner. In this section, we provide two such alternative proofs.

In the proof of Theorem~\ref{thm:homeomorphism}, the logarithmic Gauss map was used to demonstrate the invertibility of the map~$\Omega$, more precisely, to prove that~\eqref{eq:non-zero_determinant_2} holds. Below we present an alternative proof of this fact, which does not rely on the logarithmic Gauss map.
\begin{lemma}
For all~$(z,w)\in \mathcal R_0$ the following holds:
\begin{equation}
\im \left(\frac{zP_z}{wP_w}\right)\neq 0.
\end{equation}
\end{lemma}
\begin{proof}
For~$(z,w)\in \mathcal R_0$ let~$(r_1,r_2)=(\log|z|,\log|w|)\in \mathcal A$. Then
\begin{equation}
(z,w)=(\e^{r_1+\i\lambda(r_1,r_2)},\e^{r_2+\i\mu(r_1,r_2)}),
\end{equation}
for some functions~$\lambda$ and~$\mu$. The functions~$\lambda$ and~$\mu$ are explicitly given in~\cite{Pas16}, but we will not need that here. 

We differentiate the equality~$P(z,w)=0$ with respect to~$r_1$,
\begin{equation}\label{eq:app:diff_spectral_curve}
0=P_z(z,w)z_{r_1}+P_w(z,w)w_{r_1},
\end{equation}
to obtain the equalities  
\begin{equation}
-\frac{zP_z}{wP_w}=\frac{w_{r_1}}{z_{r_1}}\frac{z}{w}=\frac{\i\mu_{r_1}}{1+\i\lambda_{r_1}},
\end{equation}
where in the last equality we have used that 
\begin{equation}
z_{r_1}=(1+\i\lambda_{r_1})z, \quad \text{and} \quad w_{r_1}=\i\mu_{r_1} w.
\end{equation}
Hence,
\begin{equation}
\im \left(-\frac{zP_z}{wP_w}\right)=\frac{\mu_{r_1}}{|1+\i\lambda_{r_1}|^2}.
\end{equation}
The statement now follows since~$\mu_{r_1}\neq 0$. Indeed, if~$\mu_{r_1}=0$, then~$w_{r_1}=0$, which implies, by~\eqref{eq:app:diff_spectral_curve}, that~$P_z(z,w)=0$ which is false for~$(z,w)\in \mathcal R_0$.
\end{proof}

The proof of Lemma~\ref{lem:phi_kasteleyn} in Section~\ref{sec:one_form_thetas} is based on a general formalism, even though we have explicit formulas in our specific setting. It is therefore reasonable to ask for a more explicit proof. To address this, we provide such proof in the special case of~$k=2$.

\begin{lemma}
Let~$K_{G_1}$ be the Kasteleyn matrix defined by~\eqref{eq:magnetic_kasteleyn_matrix} and let~$q=(z,w)\in \mathcal R$. For~$i,i'=0,\dots,\ell-1$ and~$j,j'=0,\dots,k-1$, there exist~$e_{\mathrm{b}_{i,j}},e_{\mathrm{w}_{i',j'}}\in \RR^g$ such that
\begin{equation}\label{eq:app:kasteleyn_thetas}
\frac{\adj K_{G_1}(z,w)_{\mathrm b_{i,j}\mathrm w_{i',j'}}\d z}{wzP_w(z,w)}
=c_{iji'j'}\Theta\left(q;e_{\mathrm{b}_{i,j}}\right)\Theta\left(q;e_{\mathrm w_{i',j'}}\right)
\frac{\prod_{m=i+1}^{i'}E(q_{\infty,m},q)}{\prod_{m=i+1}^{i'+1}E(q_{0,m},q)}
\frac{\prod_{m=j+1}^{j'}E(p_{0,m},q)}{\prod_{m=j+1}^{j'+1}E(p_{\infty,m},q)},
\end{equation}
for some constant~$c_{iji'j'}\in \CC^*$.
\end{lemma}

\begin{proof}[Proof for~$k=2$]
The differential form
\begin{equation}\label{eq:app:lhs}
\frac{\adj K_{G_1}(z,w)_{\mathrm b_{i,j}\mathrm w_{i',j'}}\d z}{wzP_w(z,w)}
\end{equation}
is meromorphic on~$\mathcal R$, thus, to prove the equality we only need to show that the zeros and poles agree. The proof relies on the fact that if~$(\omega_1,\omega_2)$ is an interior integer point of the Newton polygon, then
\begin{equation}\label{eq:app:holomorphic_differential}
\frac{z^{\omega_1}w^{\omega_2}\d z}{wzP_w(z,w)}
\end{equation}
is a holomorphic differential, see~\cite{CL18, KO06, Kho01}. In fact, the set of these differential forms is a basis of the holomorphic differentials on~$\mathcal R$. 

The entries of~$\adj K_{G_1}$ are, up to a possible sign, determinants of submatrices of~$K_{G_1}$, which means they are all linear combinations of~$z^{\omega_1}w^{\omega_2}$ with~$(\omega_1,\omega_2)\in N(P)\cap\ZZ^2$. Recall that~$N(P)$ is the Newton polygon. Thus, the possible poles of~\eqref{eq:app:lhs} come from the terms with~$(\omega_1,\omega_2)$ at the boundary of~$N(P)$. In particular, the possible poles are at the angles. Recall that~$N(P)$ is the rectangle with vertices~$(0,0)$,~$(k,0)$,~$(k,-\ell)$ and~$(0,-\ell)$. It follows from~\eqref{eq:app:holomorphic_differential} that if 
\begin{equation}
\adj K_{G_1}(z,w)_{\mathrm b_{i,j}\mathrm w_{i',j'}}=\Ordo\left(z^{\omega_1}w^{\omega_2}\right),
\end{equation}
as~$(z,w)\to q_{0,n}$,~$n=1,\dots,\ell$, then~\eqref{eq:app:lhs} has a pole at~$q_{0,n}$ if~$\omega_2=0$, a zero if~$\omega_2=2$ and neither pole nor zero if~$\omega_2=1$. Similarly,~\eqref{eq:app:lhs} has a pole/zero/neither at~$q_{\infty,n}$ if~$\omega_2=k/(k-2)/(k-1)$, at~$p_{\infty,n}$,~$n=1,\dots,k$, if~$\omega_1=0/(-2)/(-1)$ and at~$p_{0,n}$ if~$\omega_1=(-\ell)/(-\ell+2)/(-\ell+1)$. 

To obtain the behavior of~$\adj K_{G_1}(z,w)_{\mathrm b_{i,j}\mathrm w_{i',j'}}$ at the angles we use Lemma~\ref{lem:phi_kasteleyn}; for~$(z,w)\in \mathcal R$,
\begin{multline}\label{eq:app:periodic_adjoint}
\adj K_{G_1}(z,w)_{\mathrm b_{i,j}\mathrm w_{i',j'}} \\
=w\prod_{m=1}^\ell (1-\beta_m^v z^{-1})\left(\left(\prod_{m=1}^{2i'+1}\phi_m(z)\right)^{-1}\adj(\Phi(z)-wI)\prod_{m=1}^{2i}\phi_m(z)\right)^T_{j+1,j'+1}.
\end{multline}

We begin with the behavior of~\eqref{eq:app:periodic_adjoint} at~$q_{0,n}$. This behavior is determined by the products~$\prod_{m=1}^{2i'+1}\phi_m$ and~$\prod_{m=1}^{2i}\phi_m$. Indeed, recall that
\begin{equation}
\Phi(z)\adj(\Phi(z)-wI)=w\adj(\Phi(z)-wI)=\adj(\Phi(z)-wI)\Phi(z),
\end{equation}
which implies that we may replace the second and third product of the right hand side of~\eqref{eq:app:periodic_adjoint} by
\begin{equation}\label{eq:app:replace}
w^{-1}\left(\prod_{m=2(i'+1)}^{2\ell}\phi_m(z)\right) \quad \text{and} \quad w \left(\prod_{m=2i+1}^{2\ell}\phi_m(z)\right)^{-1}.
\end{equation}
Now, to determine the behavior of~\eqref{eq:app:periodic_adjoint} at~$q_{0,n}$, we make the above replacement of the second product if~$1\leq n\leq i'+1$ and of the third product if~$1\leq n \leq i$. If there is no factor~$w^{\pm 1}$ coming from the possible replacements, then~\eqref{eq:app:periodic_adjoint} is of order~$w$ as~$w\to 0$. In general we find that~\eqref{eq:app:periodic_adjoint} is of order~$w^{\omega_2}$ with~$\omega_2=0$ if~$i<n\leq i'+1$,~$\omega_2=1$ if~$n\leq i'+1$ and~$n\leq i$ or~$i'+1<n$ and~$i<n$, and~$\omega_2=2$ if~$i'+1<n\leq i$. This shows that the zeros and poles we have found at~$q_{0,n}$ agree with the right hand side of~\eqref{eq:app:kasteleyn_thetas}. Note, however, that the above calculations exclude additional poles at~$q_{0,n}$, but not additional zeros. Nevertheless, a counting argument at the end of the proof will prove that these indeed are all zeros.

The behavior at~$q_{\infty,n}$,~$n=1,\dots,\ell$, is again determined by the products~$\prod_{m=1}^{2i'+1}\phi_m$ and~$\prod_{m=1}^{2i}\phi_m$. Note that~$w(1-\beta_n^v z^{-1})$ is bounded at~$q_{\infty,n}$, by the definition of the angles~\eqref{eq:angles_1}. The behavior of~\eqref{eq:app:periodic_adjoint}, if we ignore the factor~$w^{\pm 1}$ coming from~\eqref{eq:app:replace}, is of order~$w^{k-1}$. We replace the second product in~\eqref{eq:app:periodic_adjoint} if~$1\leq n\leq i'$ and the third product if~$1\leq n \leq i$. This gives us the zeros and poles of the right hand side of~\eqref{eq:app:kasteleyn_thetas} at~$q_{0,n}$. 

We continue with the angles~$p_{\infty,n}$,~$n=1,\dots,k$. Note that the above is valid for any~$k$, from now on, however, we take~$k=2$. 

It follows from the explicit expression of~$\phi_m$, see~\eqref{eq:bernoulli} and~\eqref{eq:geometric}, that the products~$\prod_{m=1}^{2i'+1}\phi_m$ and~$\prod_{m=1}^{2i}\phi_m$ are lower triangular as~$(z,w)\to p_{\infty,n}$. Moreover,
\begin{equation}\label{eq:app:adjoint}
\adj(\Phi(z)-wI)=
\begin{pmatrix}
\Phi_{22}(z)-w & -\Phi_{12}(z) \\
-\Phi_{21}(z) & \Phi_{11}(z)-w
\end{pmatrix}
=
\begin{pmatrix}
\Ordo\left(E(p_{\infty,2},q)\right) & \Ordo\left(z^{-1}\right) \\
\Ordo(1) & \Ordo\left(E(p_{\infty,1},q)\right)
\end{pmatrix},
\end{equation}
where~$\Phi_{mm'}$ is the~$mm'$ entry of~$\Phi$. Here we have used the explicit expression of the factors~$\phi_m$ in the product~$\Phi=\prod_{m=1}^{2\ell}\phi_m$ to deduce that~$\Phi_{nn}(z)-w\to 0$ as~$(z,w)\to p_{\infty,n}$. It follows that~\eqref{eq:app:periodic_adjoint}, recall the transpose, is of order~$z^{\omega_1}$ at~$p_{\infty,n}$, where~$\omega_1=0$ if~$j=0$ and~$j'=1$ or if~$j=j'=n-1$, and~$\omega_1=1$ otherwise. This shows the right number of zeros and poles of~\eqref{eq:app:lhs} at~$p_{\infty,n}$. 

Similarly, the products~$\prod_{m=1}^{2i'+2}\phi_m$ (note that we include an extra factor) and~$\prod_{m=1}^{2i}\phi_m$ are upper triangular as~$q=(z,w)\to p_{0,n}$. Moreover, 
\begin{multline}
\phi_{2i'+2}(z)\adj(\Phi(z)-wI)=
\begin{pmatrix}
\Ordo(z) & \Ordo(1) \\
\Ordo(z) & \Ordo(z)
\end{pmatrix}
\begin{pmatrix}
\Ordo\left(E(p_{0,1},q)\right) & \Ordo(1) \\
\Ordo(z) & \Ordo\left(E(p_{0,2},q)\right)
\end{pmatrix} \\
=
\begin{pmatrix}
\Ordo(z) & \Ordo\left(E(p_{0,2},q)\right) \\
\Ordo\left(zE(p_{0,1},q)\right) & \Ordo(z)
\end{pmatrix}.
\end{multline}
This shows that~\eqref{eq:app:periodic_adjoint} has the right behavior at~$p_{0,n}$ to conclude that the zeros and poles of the right and left hand side of~\eqref{eq:app:kasteleyn_thetas} match.

Summing up the number of zeros and subtracting the number of poles we have found so far gives us~$-2$. Since we know that we have found all poles, it means (recall that the degree of a differential form is~$(2g-2)$~\eqref{eq:abels_thm_diff}) that we are missing exactly~$2g$ zeros. 

We obtain the last zeros using the spectral data. Indeed, as discussed in Section~\ref{sec:one_form_thetas}, the spectral transform contains a divisor on~$\mathcal R$, more precisely, for each white vertex~$\mathrm{w_{i',j'}}\in G_1$, there exists a divisor~$D_{\mathrm w}$ on~$\mathcal R$ which is defined as the sum of the common zeros of 
\begin{equation}
\adj K_{G_1}(z,w)_{\mathrm b\mathrm w_{i',j'}}=0, \quad \mathrm b\in \mathcal B_1.
\end{equation}
Moreover, the divisor~$D_{\mathrm w}$ is a standard divisor, see~\cite[Theorem 1]{KO06}. Recall that a standard divisor is a sum of~$g$ points, with one point on each compact oval, see~\eqref{eq:standard_divisor}. Since the angles lie on the non-compact oval, these zeros are distinct from the zeros we have found at the angles. Moreover, the fact that~$D_{\mathrm w}$ is standard implies that~$u(D_{\mathrm w})=-e_0+\Delta$, for some~$e_0\in (\RR/ \ZZ)^g$, where~$u$ is the Abel map and~$\Delta$ is the vector of Riemann constants. Hence, with~$e_{\mathrm w_{i',j'}}=\tilde e_0 \in \RR^g$, the lift of~$e_0$ to the universal cover, we have obtained the~$g$ zeros coming from the second theta factor of the right hand side of~\eqref{eq:app:kasteleyn_thetas}.

The final~$g$ zeros are determined by Abel's theorem~\eqref{eq:abels_thm_diff}. We denote the divisor of these final zeros by~$D_{\mathrm b}$. Since~$u(p)\in (\RR/ \ZZ)^g$ for any angle~$p=q_{0,n}, q_{\infty,n}, p_{0,n}, p_{\infty,n}$, Abel's theorem implies that~$u(D_{\mathrm b})=-e+\Delta$ for some~$e\in (\RR/ \ZZ)^g$. This gives us the first theta factor of the right hand side of~\eqref{eq:app:kasteleyn_thetas} with~$e_{\mathrm b_{i,j}}=\tilde e \in \RR^g$.
\end{proof}

\begin{remark}
What is missing to extend the previous argument to all~$k$ is the behavior of~\eqref{eq:app:periodic_adjoint} at~$p_{0,n}$ and~$p_{\infty,n}$,~$n=1,\dots,k$. An explicit computation can still be performed to obtain the leading order term of~\eqref{eq:app:adjoint}, and, hence, determine the poles at these angles. However, obtaining the second order term, which is necessary to determine the zeros, requires a more detailed computation, which we will not pursue here.
\end{remark}

\bibliographystyle{plain}
\bibliography{bibliotek}

\end{document}